\documentclass[11pt]{article}
\usepackage[utf8]{inputenc}
\usepackage{fullpage}

\usepackage{amsthm,mathtools,scalerel}
\usepackage{color,soul,latexsym,amsmath,amssymb,amsfonts,amsthm,dsfont,enumitem,xcolor,bbm,threeparttable,graphicx,float,subcaption}
\usepackage{authblk}
\usepackage[round]{natbib}
\PassOptionsToPackage{hyphens}{url}
\RequirePackage[colorlinks=true,allcolors=blue,hypertexnames=false]{hyperref}%
\usepackage[symbol]{footmisc}
\usepackage{graphicx,todonotes}
\usepackage{xcolor,stmaryrd}
\usepackage{enumitem}

\usepackage{wrapfig}
\definecolor{Green}{HTML}{069937}
\newif \ifnotes

\usepackage[most]{tcolorbox}
\colorlet{shadecolor}{gray!15}
\usepackage{outlines}
\usepackage{array}
\usepackage{selectp}
\usepackage[ruled,linesnumbered, vlined, noend]{algorithm2e}
\usepackage{titlesec}

\newtheorem{theorem}{Theorem}%[section]

\newtheorem{lemma}{Lemma}
\newtheorem{corollary}{Corollary}
\newtheorem{proposition}{Proposition}

\newtheorem*{example}{Example}

\newtheorem*{remark}{Remark}

\DeclareMathOperator*{\argmin}{argmin}

\DeclareMathOperator*{\esssup}{ess\,sup}
\newcommand{\E}{\mathbb{E}}

\renewcommand{\hat}{\widehat}
\let\tilde\widetilde

\allowdisplaybreaks
\usepackage{tikz}
\usetikzlibrary{positioning,shapes.geometric,calc,shapes,arrows.meta,arrows,decorations.markings, external, trees}
\tikzstyle{Arrow} = [thick,decoration={markings,mark=at position 1 with {\arrow[thick]{latex}}},shorten >= 3pt, preaction = {decorate}]

\title{\bf Statistical Properties of Rectified Flow}
\author[1]{Gonzalo Mena}
\author[2]{Arun Kumar Kuchibhotla}
\author[3]{Larry Wasserman}
\affil[1]{\texttt{gmena@andrew.cmu.edu}}
\affil[2]{{\tt arunku@cmu.edu}}
\affil[3]{{\tt larry@cmu.edu}}
\affil[1,2,3]{Department of Statistics \& Data Science, Carnegie Mellon University}

\date{}                   
\begin{document}

\maketitle

\begin{abstract}
Rectified flow
\citep{liu2022flow, liu2022rectified, wu2023fast}
is a method for defining a transport map between
two distributions, and enjoys popularity in machine learning, although theoretical results supporting the validity of these methods are scant.
The rectified flow can be regarded as an approximation to optimal transport, but in contrast to other transport methods that require optimization over a function space, computing the rectified flow only requires
standard statistical tools such as regression or density estimation, that we leverage to develop empirical versions of transport maps.
We study some structural properties of the rectified flow, including existence, uniqueness, and regularity, as well as the related statistical properties, 
such as rates of convergence and central limit theorems, for some selected estimators. To do so, we analyze the bounded and unbounded cases separately as each presents unique challenges. In both cases, we are able to establish convergence at faster rates than those for the usual nonparametric regression and density estimation.
%The velocity field $v_t$ that defines flow,
%for $(0 \leq t \leq 1)$,
%is based on a simple regression equation but
%has some surprising properties.
%The validity of the method requires that
%a certain differential equation has a unique solution.
%It appears that this has not been verified in any generality
%before. We verify the existence and uniqueness of the solution.

%This turns out to require nonstandard methods.
%Indeed, if the space is compact, then
%the boundaries of the original space
%are pulled into the interior of the field
%resulting in the existence of submanifolds
%on which the field is non-differentiable and, in fact,
%is not even continuous no matter how smooth the distributions are.
%Strangely, a density based estimator obtains better rates than
%a regression estimator and these rates are unusual.
%Another odd property is that $v_t$ can be estimated at an $n^{-1/2}$ rate
%for all fixed $0<t<1$ and at $t=0,1$ but not at $t$ close to 0 and 1.
%We show that a
%smoothed version of 
%rectified flow can also be estimated at a parametric rate, is especially 
%easy to implement and can be estimated with a simple $U$-statistic.
\end{abstract}

\newpage

\tableofcontents

\newpage

\section{Introduction}

In some statistical problems,
we need to find a map
$T$ that transforms one distribution $\mu_0$
into another distribution $\mu_1$.
Examples include
generative modeling \citep{balaji2020, rout2021},
transfer learning \citep{lu2017},
domain adaptation \citep{yan2018semi, courty2017joint, courty2014domain},
causal inference \citep{li2021causal, torous2021optimal} and
image analysis \citep{kolouri2016} among others.
If $X_0$ denotes a draw from $\mu_0$,
then we want a map $T$ such that
$T(X_0)\sim \mu_1$.
Such a map is called a
{\em transport map}.
More generally,
a {\em coupling}
is a joint distribution $J$
for a pair $(X_0,X_1)$
such that
$J$ has marginals $\mu_0$ and $\mu_1$.
A transport map is a
degenerate coupling
of the form $(X_0,T(X_0))$ where $X_0 \sim \mu_0$.

A commonly used transport map
is the {\em Monge map} or {\em optimal transport map}
which minimizes the average cost 
$\E[c(X_0,T(X_0))]$ over all maps $T$ such that $T(X_0)\sim \mu_1$.
Often one uses
$c(x,y) = \|x-y\|^2$ (the Euclidean norm)
so the optimal transport map minimizes
$\E[\|X_0 - T(X_0)\|^2]$ over all maps $T$ such that $T(X_0)\sim \mu_1$. (Note that unlike the general definition of transport maps, the use of the Euclidean norm implicitly assumes $X_0$ and $X_1$ lie in the same vector space.)
If $T_0$ is such a minimizer, then
the minimum value
$W^2(\mu_0,\mu_1) = \E[\|X_0 - T_0(X_0)\|^2]$
defines the well known Wasserstein distance.
Despite its intuitive appeal,
the optimal transport map
can be difficult to deal with, both theoretically and computationally.
Computation, estimation, and inference
for the optimal transport map are all
challenging.
Under smoothness conditions, the Monge map
can, in principle, be estimated at a fast rate \citep{manole2023clt}.
Yet, there is no practical way to compute the map
under these smoothness conditions.
Other maps and couplings that are useful
include the Schrodinger coupling
\citep{tong2023simulation, chen2016} and
diffusion models \citep{liu2022diff}.

Throughout, we assume that $\mu_0$ and $\mu_1$ are supported on the same convex set $\Omega\subseteq\mathbb{R}^d$ with a non-empty interior. (The convexity assumption is crucially used in most of our results, and significantly different proofs would be required to relax convexity. The assumption of the same support can be easily relaxed, but is assumed for simplicity.)

Recently, 
\cite{liu2022flow}, \cite{liu2022rectified}
proposed
a new method for constructing
couplings called 
{\em rectified flow}. This leads to a transport map using
an ordinary differential equation. 
We start with any convenient coupling between $X_0$ and $X_1$;
usually this is just the independence coupling.
Set
\begin{equation}\label{eq:velocity-field}
v_t(z) \equiv v(t, z) := \E\bigl[X_1-X_0 \,\bigm| \,(1-t)X_0 + tX_1=z\bigr],\quad{z\in\Omega}.
\end{equation}
For each $x\in\Omega$, let $t\mapsto \mathfrak{R}(t, x)$ solve the ODE
\begin{equation}\label{eq:rectified-flow-ODE}
\frac{d}{d t}\mathfrak{R}(t, x) = v(t, \mathfrak{R}(t, x)),\quad t\in[0, 1],\quad \mbox{with}\quad\mathfrak{R}(0, x) = x.
\end{equation}
This leads to a transport map called the rectified map, defined by
\begin{equation}\label{eq:rectified-flow}
R(x) \equiv \mathfrak{R}(1, x) = x + \int_0^1 v(t, \mathfrak{R}(t, x))dt.
\end{equation}
(We use
$v_t(z)$ and $v(t,z)$ interchangeably to
to denote the velocity field~\eqref{eq:velocity-field}. Likewise, we denote by
$v_j(t,z)$ or $v_{t,j}(z)$ the  $j^{\rm th}$ coordinate of this vector.)
When we do not need to stress the starting point $x$, we write $z_t = \mathfrak{R}(t, x)$,
and refer to $t\mapsto z_t$ as the path.
Theorem 3.2 of~\cite{liu2022rectified} has shown that
if the ODE~\eqref{eq:rectified-flow-ODE} has a unique solution and $z\mapsto v(t, z)$ is locally bounded, then
$x\mapsto R(x)$ is a valid transport map, that is,
$R(X_0)\sim \mu_1$ if $X_0\sim \mu_0$; see Theorem \ref{thm:uniquene-ae-implies-transport-map} below.
We can think of rectified flow
as a method for converting an initial coupling
(usually the independence coupling) into a valid
transport map.

Repeating the rectified flow operation produces a straight coupling; a path $t\mapsto Z_t$ is said to be straight if $Z_t = (1 - t)Z_0 + tZ_1$ for all $t\in(0, 1)$, or equivalently, $\partial Z_t/\partial t = C$ for all $t\in(0, 1)$. See Theorem 3.7 of~\cite{liu2022flow} for a precise result. If the velocity field is constructed to be a conservative field (i.e., derivative of a function), then the rectified flow map 
after iterations
is the optimal transport map, under some regularity conditions; see Section 4.2 of~\cite{hertrich2025relation}.   

Given that the velocity field~\eqref{eq:velocity-field} is defined as 
a conditional expectation, one can readily leverage the wealth of literature
on regression methods to estimate $z\mapsto v_t(z)$ for all $t\in(0, 1)$ and obtain $\widehat{R}$ by
solving the empirical ODE $\partial z_t/\partial t = \widehat{v}_t(z_t), t\in (0, 1)$ with $z_0 = x$.
These facts hold true as long as $(X_0, X_1)$ is {\em any} valid coupling of $\mu_0$ and $\mu_1$; in particular, $X_0$ and $X_1$ need not be independent. 
If one uses the independent coupling, i.e., $X_0$ and $X_1$ are independent, then the velocity field
can be written explicitly in terms of the Lebesgue densities of $\mu_0$ and $\mu_1$ (assuming the existence of Lebesgue densities); see Lemma~\ref{lem:representation-of-velocity-field}. Such an alternative representation implies that by estimating the densities of $\mu_0$ and $\mu_1$, one can obtain an estimator of $z\mapsto v_t(z)$ for all $t\in(0, 1)$ simultaneously, that in turn yields an estimator of the transport map. Throughout the manuscript, we assume that $X_0$ and $X_1$ are independent. 
% \cite{liu2022rectified} and~\cite{liu2022flow} further show that iterating the rectified flow produces a straight coupling, that is, start with independent coupling $(X_0, X_1)$ to obtain a rectified flow map $R_1(\cdot)$, compute the rectified flow map $R_2(\cdot)$ using the coupling $(X_0, R_1(X_0))$, and so on to obtain $R_k(\cdot)$ for $k\ge1$. Then $R_k(\cdot)$ converges to the optimal transport map $T_0(\cdot)$ as $k\to\infty$.

Given the simplicity of rectified flow, it has attracted a lot of interest recently. Rectified flow has been 
independently rediscovered
at least twice since the work
of~\cite{liu2022flow}.
\cite{heitz2023} introduced
iterative $\alpha$-deblending
and
\cite{delbracio2023}
introduced inversion by direct iteration.
These are equivalent to rectified flow.
Since the original work,
applications have appeared in 
\cite{hu2023rf, liu2023instaflow, zhang2024language}.
%If the method is iterated, the resulting map converges
%to the Monge map.
%Converting the transport problem
%to a regression problem
%has significant statistical advantages.
%It means that all the tools of
%regression that have been developed over the years
%can now be used for estimating transport maps.

The proof that $R(\cdot)$ is a valid transport map crucially relies on the assumption that the ODE~\eqref{eq:rectified-flow-ODE} has a unique solution. To the best of our knowledge, this assumption has not been formally verified for a general collection of distributions. Even the existence of a solution to this ODE is not obvious. From the general theory of differential equations (Cauchy-Lipschitz theorem), we know that if $z\mapsto v_t(z)$ is smooth enough (uniformly in $t$), then the existence and uniqueness of the solution follow. In this paper, we study the regularity of the velocity field $v_t(z)$ as a function of $(t, z)$ under general smoothness conditions on the densities and their supports. Rather surprisingly, even if the densities of $\mu_0$ and $\mu_1$ are infinitely smooth on their (bounded) support, $z\mapsto v_t(z)$ need not be Lipschitz continuous with a uniformly bounded Lipschitz constant. This makes the application of existing theory of ODEs difficult for the study of~\eqref{eq:rectified-flow-ODE}. For a more specialized analysis, we consider the bounded and unbounded cases separately, as the proof techniques and assumptions differ substantially in these two cases. We study structural properties of the velocity field and introduce several estimators for the velocity and the map. Based on these structural properties, we derive convergence rates and the central limit theorem for our estimators.

We can summarize our general contributions as follows.
\begin{enumerate}
    \item Starting from an independent coupling, we provide four representations for the velocity field that define the rectified flow. This allows us to define different estimators for the velocity field based on the routinely used non-parametric density and regression estimators. Such procedures also enable us to employ dimension reduction techniques, allowing for faster estimation rates of a transport map. 
    
    \item We show, via examples, that the velocity field~\eqref{eq:velocity-field} need not be a smooth function of $t$ and $z$, even if the Lebesgue densities $p_0$ and $p_1$ are infinitely differentiable. This is especially the case if $p_0$ and $p_1$ are compactly supported. In fact,
    in the compact case, we show that the velocity field is not even continuous.
    This is significant 
    since most proofs of the existence and uniqueness of the ODE assume that the velocity field
    is smooth.
    
%    contribution for several reasons: (1) smoothness assumptions on the velocity field are 
%often used to %claim the existence and uniqueness of solutions to 
%the ODE~\eqref{eq:rectified-flow-ODE}, see, for 
%example,~\citet[Theorem 7.3, Chapter I]{HairerWannerNorsett1993},~
%\citet[Chapter 8]%{AmbrosioGigliSavare2008},~\cite{maniglia2007probabilistic}, %and~\cite{gwiazda2018differentiability}; 
%(2) several authors have mistakenly claimed or assumed smoothness of the velocity %field~\citep{hertrich2025relation,bansal2025wassersteinconvergencestraightnessrectified}.  
    
    \item When $\mu_0$ and $\mu_1$ are log-concave, log-H{\"o}lder continuous, and are supported on all of $\mathbb{R}^d$, we prove that the velocity field is smooth, enabling the application of classical existence and uniqueness results from the ODE literature. While traditional non-parametric methods can be used to analyze the convergence rates of the velocity field, studying the rectified flow estimator is more complicated. This complexity arises because the estimator is derived through a non-linear operation, specifically, solving an ordinary differential equation (ODE). To address this, we use perturbation theory of ODEs (specifically, an analogue of the Alekseev-Gr{\"o}bner formula~\citep[Theorem 14.5]{HairerWannerNorsett1993}) to determine the convergence rates and asymptotic normality of our regression-based rectified flow map estimator when $\mu_0$ and $\mu_1$ are strongly log-concave supported on all of $\mathbb{R}^d$. This is, to our knowledge, the first such result in the literature regarding the rectified map estimator. 
    \item When $\mu_0$ and $\mu_1$ are compactly supported, the theory of rectified flow is significantly complicated by the lack of smoothness or even continuity. In fact, no solution satisfying~\eqref{eq:rectified-flow-ODE} for all $t\in[0,1]$ might exist. For this reason, we consider Carath{\'e}odory solutions~\eqref{eq:rectified-flow-integral-equation} and prove their existence and uniqueness under mild conditions. Additionally, the lack of smoothness prohibits the use of Alekseev-Gr{\"o}bner formula, but by a direct method, we provide convergence rates and asymptotic normality of the density-based rectified flow map. This requires the development of a new kernel density estimator for arbitrary convex sets (with no assumptions on the smoothness of the boundaries).

\item The estimators have some
surprising properties. For example, in the compact case, a regression-based estimator
would have poor rates due to the lack of smoothness
of the velocity. Surprisingly, a density-based estimator
has fast rates. 

\item Another surprising finding is
the following. For any fixed $0 < t < 1$, the velocity field $v_t$ can be estimated at $n^{-1/2}$
rates using semiparametric estimators. Also, $v_0$ and $v_1$ can be estimated at
$n^{-1/2}$ rates. But, the $n^{-1/2}$ rate is not obtainable uniformly over
$0 \leq t \leq 1$.
    
\end{enumerate}

\paragraph{Paper Outline.}
We review estimation rates and computation of the optimal transport which motivated us to a statistical study of rectified flow
in Section \ref{section::transport}.
In Section~\ref{sec:rectified-flow}, we provide four representations of the velocity field~\eqref{eq:velocity-field} under independence coupling, and provide some examples where the velocity field and the map itself can be explicitly computed. This gives us insight into the regularity properties of the velocity field, and also the differences between the optimal transport map and the rectified map. We also discuss different estimation schemes for the velocity field based on the four representations, and show how one can incorporate structure into these estimators to avoid the curse of dimensionality. 
% We briefly discuss the rectified Wasserstein distance in Section \ref{section::distance}.
% We consider how to iterate the rectified flow in
% Section \ref{section::iteration}.
To derive the regularity properties and the estimation rates, we consider the cases of unbounded and bounded supports separately. 
In Section~\ref{sec:review-ODE}, we review some results regarding the existence, uniqueness, and stability of solutions from the literature of ordinary differential equations.
In Section~\ref{sec:unbounded}, we consider the case of unbounded $\Omega$ and derive regularity properties, rates of convergence, and a CLT assuming $\mu_0$ and $\mu_1$ are log-concave distributions. In Section~\ref{sec:bounded}, we consider the case of bounded $\Omega$ and derive regularity, rates of convergence, and a CLT assuming $\mu_0$ and $\mu_1$ have smooth densities on their support $\Omega$.

\section{Optimal Transport}
\label{section::transport}

In this section we provide a brief
review of optimal transport.

{\bf Optimal Transport.}
Let $\mu_0$ and $\mu_1$ be
probability distributions.
We say that $T$ is a {\em transport map}
if $X\sim \mu_0$ implies that
$T(X)\sim \mu_1$.
If $\mu_0$ is absolutely continuous then
there are many such maps.
The {\em Monge map} $T_0$, or {\em optimal transport map},
(under the squared $L_2$ norm)
is the transport map $T_0$ that minimizes
$$
\int \|T(x) -x\|^2 d\mu_0(x).
$$
The (quadratic) Wasserstein distance $W(\mu_0,\mu_1)$ is defined by
$$
W^2(\mu_0,\mu_1) = \int \|T_0(x) -x\|^2 d\mu_0(x).
$$

{\bf Dynamical Representation}.
Optimal transport can be expressed
in dynamic form as 
a flow from $\mu_0$ to $\mu_1$, as described in~\cite{benamou2000computational}.
Consider a flow defined by the differential equation
$$
\frac{\partial}{\partial t}\mathfrak{R}_{u}(t, x) = u(t, \mathfrak{R}_u(t, x)),\quad\mbox{for all}\quad t\in[0, 1],
$$
for a velocity field $u(t, z)$, with $\mathfrak{R}_u(0, x)=x$.
Let $\mathcal{U}$ be the collection of all velocity fields such that $\mathfrak{R}_u(t, x), t\in[0,1]$
is uniquely defined and $\mathfrak{R}(1, X_0) \overset{d}{=} X_1$. Then $x\mapsto \mathfrak{R}_{u^*}(1, x)$ is the optimal transport map, where $u^*(\cdot, \cdot)$ is the minimizer of
$$
\E \Biggl[ \int_0^1 \|u(t, \mathfrak{R}_u(t, X_0))\|^2 dt \Biggr],
$$
over all $u\in\mathcal{U}$.
If $\mu_t$ denotes the law of $\mathfrak{R}(t, X_0)$,
then the solution satisfies the
{\em continuity equation}
$$
\partial_t p_t + \text{div}(u_t^* p_t) = 0
$$
where $p_t$ is the density of $\mu_t$
where $\text{div}$ denotes the divergence.
Hence, the optimization can be written as
$$
\inf_{u,p}\int_0^1 \int \|u_t(x)\|^2 p_t(x) dx dt
$$
subject to
$\partial_t p_t + \text{div}(u_t p_t) = 0$. (See Proposition 1.1 of~\cite{benamou2000computational}.)

{\bf Estimation.}
The primary statistical task in
optimal transport is to estimate the transport map
and Wasserstein distance
from samples
$X_1,\ldots, X_n\sim \mu_0$ and
$Y_1,\ldots, Y_n\sim \mu_1$.
The simplest estimate
is the transport
from the empirical distribution
$\mu_{0,n} = n^{-1}\sum_i \delta_{X_{0i}}$
to
$\mu_{1,n} = n^{-1}\sum_i \delta_{X_{1i}}$
where $\delta_a$ denotes a point mass at $a$.
The solution is
$\hat T(X_{0i}) = Y_{\hat \pi(i)}$
where $\hat\pi$ is the permutation that minimizes
$\sum_i \|X_{0i} - Y_{\pi(i)}\|^2$.
This defines that map $\hat T$ at the data points $X_{0i}$.
The map can be extended to all $x$ by taking
$\hat T(x) = \hat T (X_{0i}(x))$
where $X_{0i}(x)$ is the closest data point to $x$.
Computing this map takes time $O(n^3)$
and requires the use of linear programming software.
The estimate has the slow rate $n^{-2/d}$
which is dominated by bias.
Inference is difficult due to the large bias.
It is natural to ask
if we can improve estimation and inference
when $\mu_0$ and $\mu_1$ have smooth densities.
In principle,
the answer is yes.
The minimax rate for estimating $T$
is much faster under smoothness assumptions.
\citep{hutter2019minimax, deb2021rates, manole2021sharp, 
gunsilius2022convergence, manole2021}.
Furthermore,
\cite{manole2023clt} obtained a centered, central limit
theorem for the transport map under smoothness.
Unfortunately,
finding practical methods to 
compute these estimators
is still unsolved.
Furthermore,
the theoretical results rely on fairly
restrictive conditions.
Hence, the goal of
finding simple, practical methods
to estimate transport maps and to
quantify the uncertainty of the estimates
is still largely open.
Due to these challenges,
other maps and couplings have been considered, such
as regularized transport \citep{cuturi2013},
minibatch transport \citep{fatras2021}
and sliced transport \citep{bai2023, manole2022sliced}.

\section{Rectified Flow: Some Examples and Estimation Methods}
\label{sec:rectified-flow}

We start by stating a useful property that the velocity can be represented in several equivalent ways. Recall that $X_0\sim \mu_0$ and $X_1\sim \mu_1$ are independent and 
\[
v_t(z) := \mathbb{E}[X_1 - X_0|X_t = z]\quad\mbox{where}\quad X_t := (1-t)X_0 + tX_1.
\]
\begin{lemma}\label{lem:representation-of-velocity-field}
Assuming $\mu_0$ and $\mu_1$ have Lebesgue densities $p_0$ and $p_1$, the velocity field can be equivalently written as follows:
\begin{align}
\label{eq::vzt}
v_t(z) &= \E\bigl[X_1-X_0\,\bigm|\,(1-t)X_0+t X_1 =z\bigr]\\
\label{eq:veldelta} v_t(z)
&= 
\frac{\int \delta \, p_0(z-t \delta) p_1(z + (1-t)\delta) d\delta }
{\int p_0(z-t \delta) p_1(z + (1-t)\delta) d \delta} = \frac{f_t(z)}{p_t(z)}\\
\label{eq:velp0}
v_t(z)
& =
\frac{\E_0 \left[ t^{-1}(z-X_0)\ p_1\left(t^{-1}{(z - X_0(1 - t))}\right)\right]}
{\E_0 \left[ p_1\left( t^{-1}({z - X_0(1 - t)})\right)\right]}\\
\label{eq:velp1}
v_t(z)
&= \frac{\E_1 \left[ (1-t)^{-1}({X_1-z})\ 
p_0\left( (1-t)^{-1}{(z - tX_1)}\right)\right]}
{\E_1 \left[ p_0\left( (1-t)^{-1}(z - tX_1)\right)\right]},
\end{align}
where $\mathbb{E}_0[\cdot]$ and $\mathbb{E}_1[\cdot]$ represent expectations with respect to $\mu_0$ and $\mu_1$, respectively, and
\begin{equation}\label{eq:notation-numerator-denominator-velocity-field}
f_t(z) = \int \delta p_0(z-t\delta) p_1(z + (1-t)\delta) d\delta,\ \ \ 
p_t(z) = \int p_0(z-t\delta) p_1(z + (1-t)\delta) d\delta.
\end{equation}
\end{lemma}

The proof is a straightforward calculation and is omitted.
The closed form expression~\eqref{eq:velp1} for $v_t(z)$
was mentioned in Eq. (4) of~\cite{liu2022flow},
along with an expression for the derivative of the velocity
field with respect to $z$ under the assumption
of a differentiable log-density.
From Lemma~\ref{lem:representation-of-velocity-field}, it maybe tempting to
conclude that $v_t$ has the same smoothness as $p_0$ and $p_1$ but, as we shall see,
this is not the case.

The following result guarantees that the rectified transport is a valid transport map whenever there is almost sure uniqueness of solutions that satisfy the ODE~\eqref{eq:rectified-flow-ODE} almost surely. Formally, instead of the ODE~\eqref{eq:rectified-flow-ODE}, consider the integral equation
\begin{equation}\label{eq:rectified-flow-integral-uniqueness}
    \mathfrak{R}(t, x) = x + \int_0^t v(s, \mathfrak{R}(s, x))ds\quad\mbox{for all}\quad t\in[0, 1].
\end{equation}
Note that if $\mathfrak{R}(t, x)$ satisfies the ODE~\eqref{eq:rectified-flow-ODE}, then it also satisfies~\eqref{eq:rectified-flow-integral-uniqueness}. However, $\mathfrak{R}(t, x)$ satisfying~\eqref{eq:rectified-flow-integral-uniqueness} may only satisfy~\eqref{eq:rectified-flow-ODE} for almost all $t\in[0, 1]$. In the case of bounded supports, we could only prove the existence of unique solutions to~\eqref{eq:rectified-flow-integral-uniqueness} (not~\eqref{eq:rectified-flow-ODE}), that too for almost all $x\in\Omega$. Hence, we provide an alternative result that shows that such almost sure unique solutions still yield transport maps. 
\begin{theorem}\label{thm:uniquene-ae-implies-transport-map}
Suppose that $A\subseteq\Omega$ with $\mu_0(A)=1$ such that the integral equation~\eqref{eq:rectified-flow-integral-uniqueness} has a unique solution for each $x\in A$. Then, for all $t\in[0,1]$, the law of $\mathfrak{R}(t, X_0)$ is the same as that of $X_t = (1-t)X_0 + t X_1$.
\end{theorem}
A proof of Theorem~\ref{thm:uniquene-ae-implies-transport-map} can be found in Section~\ref{appsubsec:proof-unique-almost-surely}.

\subsection{Explicit examples}
In this section, we consider several examples in which the velocity field can be explicitly computed. In general, even when $p_0$ and $p_1$ are known, the velocity field does not exhibit a closed-form expression. This is possible in specific examples such as Gaussians, mixtures of Gaussians, and uniform distributions. These examples will be discussed in the following and provide insights into the differences between the rectified map and the optimal transport map, and also into the regularity that the velocity field can be expected to satisfy, in general. We note here that the closed-form expressions for the velocity field in some of the following examples can also be found elsewhere~\citep{hertrich2025relation}.
\paragraph{$\mu_0$ and $\mu_1$ are Gaussian.}
To gain further insight into the rectified flow,
we now consider the case where
$\mu_0$ and $\mu_1$ are Gaussian.

\begin{lemma}
\label{lemma::gaussianvel}
Let 
$X_0 \sim N(m_0,\Sigma_0)$
and
$X_1 \sim N(m_1,\Sigma_1)$ be independent.
Then, for $t\in[0,1]$, assuming the invertibility of $t^2\Sigma_1 + (1 - t)^2\Sigma_0$, we have
\begin{equation}
\label{eq:gaussianvel}
v_t(z) = m_1 - m_0 + (t\Sigma_1 - (1-t)\Sigma_0)(t^2\Sigma_1 +(1-t)^2\Sigma_0)^{-1}(z - m_t),
\end{equation}
where $m_t = (1-t)m_0 + tm_1$. 
\end{lemma}
It is interesting to note that it is possible to construct a pair of singular covariance matrices $\Sigma_0$ and $\Sigma_1$ such that the velocity field is well-defined for all $t\in(0, 1)$, but not at $t = 0, 1$. This shows that even in the case of Gaussian distributions, the map $t\mapsto v_t(z)$ need not be continuous. Note that with singular covariance matrices, $\mu_0$ and $\mu_1$ do not exhibit Lebesgue densities.
% {\color{red} Is the point of the following lemma to provide an alternative formula for $v_t$? If so, shouldn't it be part of previous lemma. Also, the previous lemma only requires $\Sigma_0$ to be invertible and the following requires both to be invertible. $b$ does not play a role in the statement of the following lemma, right?}
% \vspace{0in}

\begin{figure}[ht!]
    \centering
    \begin{minipage}{0.48\textwidth}
        \centering
        \includegraphics[width=\textwidth]{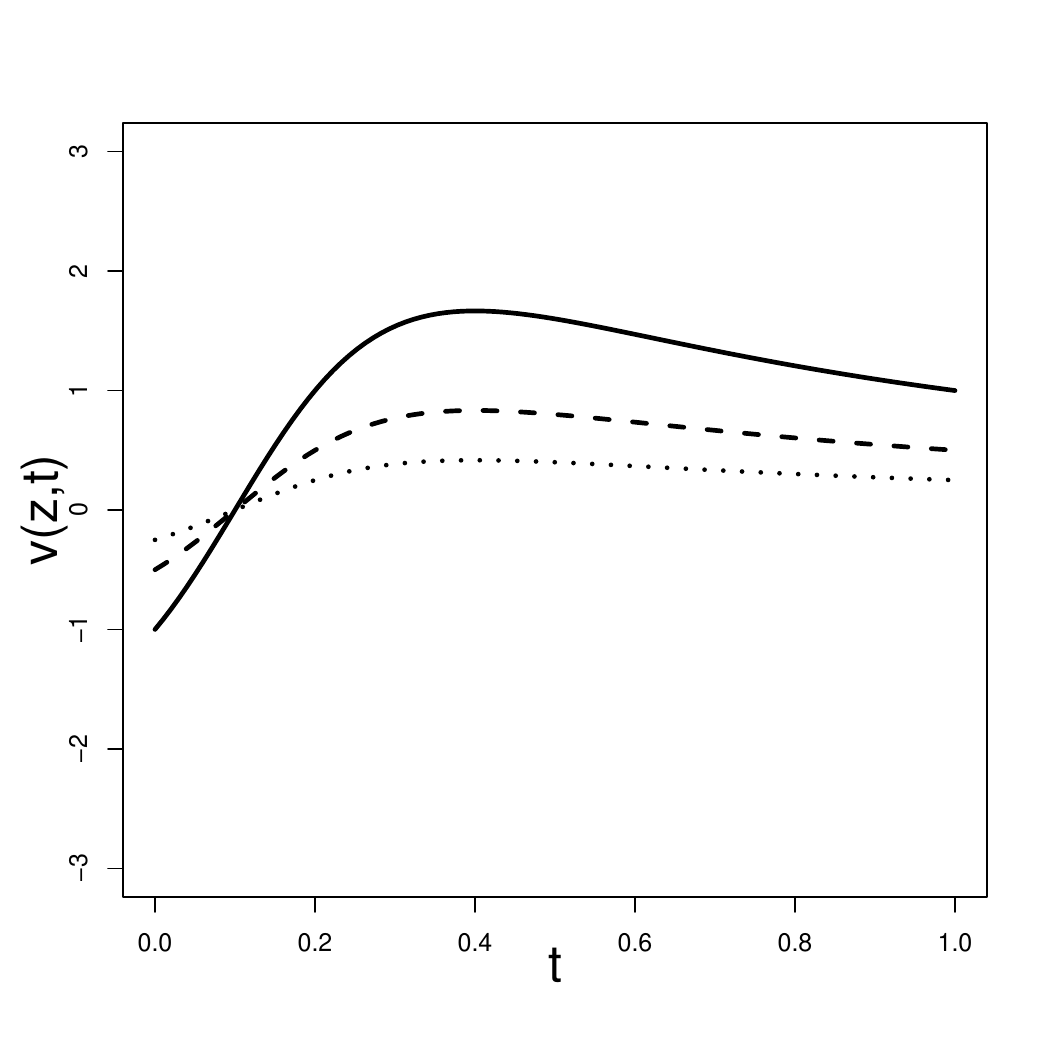}
    \end{minipage}\hfill
    \begin{minipage}{0.48\textwidth}
        \centering
        \includegraphics[width=\textwidth]{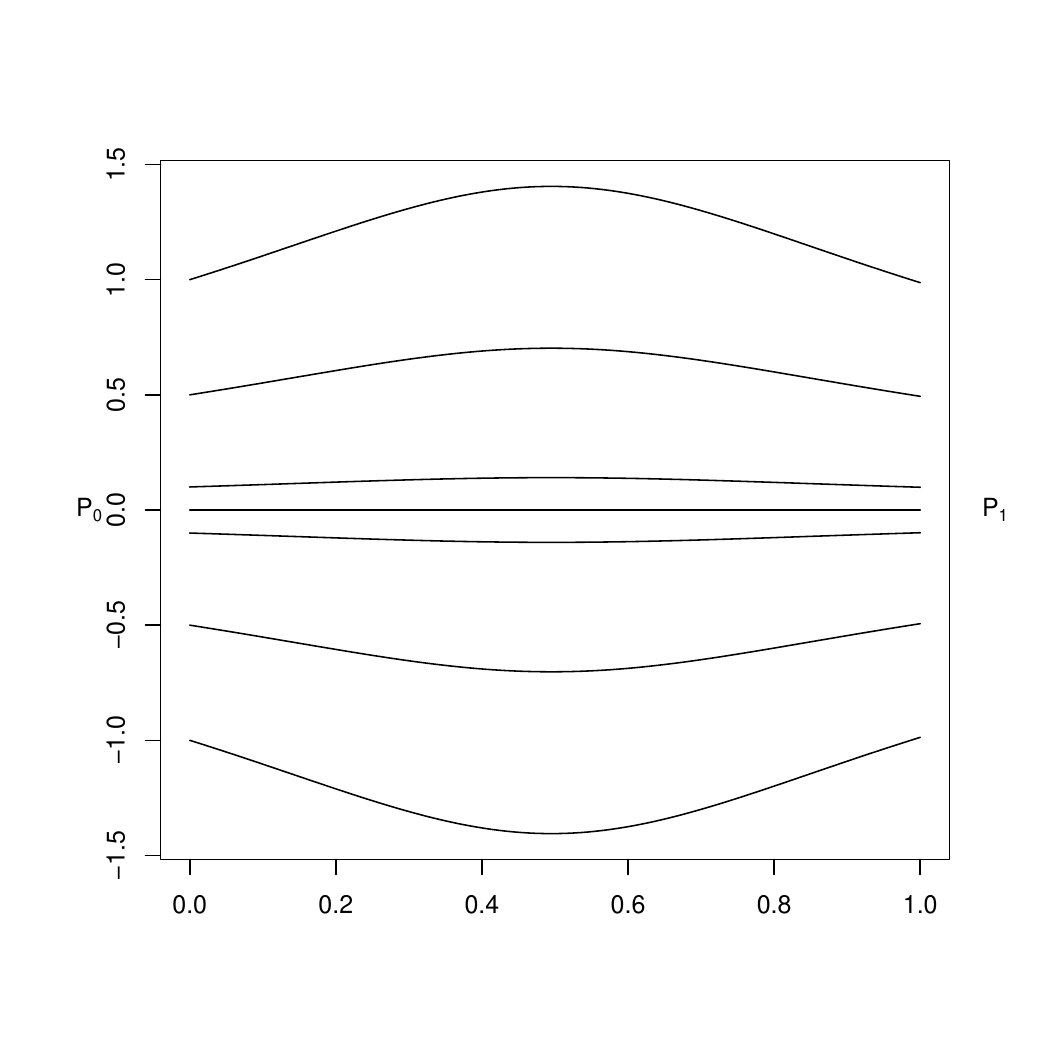}
    \end{minipage}
    \caption{The figure on the left shows the velocity for the rectified flow when $z = 1$ (solid), $z = 1/2$ (dotted), and $z = 1/4$ (dashed). The figure on the right shows the paths $z(t)$
when $\mu_0 = \mu_1 = N(0,1)$.
The resulting map $R(x)$ is the identity map
but the paths are nonlinear.
The optimal transport path is simply constant.}
    \label{fig:side-by-side}
\end{figure}
% {\color{red}Edit the following sentence for clarity.}
Given the explicit form for the velocity field, one can obtain a closed-form expression for the rectified flow starting from an independent coupling. This allows us to compute the iterated rectified flow and also find a closed-form expression for the infinite-iteration rectified flow. Interestingly, in this case of Gaussian-to-Gaussian transport maps, the rectified flow does not change after the first iteration. This is the content of the following result.

\begin{proposition}\label{prop:flow-from-Gaussian-to-Gaussian}
Consider the setting of Lemma~\ref{lemma::gaussianvel}. The rectified map between $X_0$ and $X_1$ (starting from an independent coupling) is 
\begin{equation}
\label{eq:rectgaussian}
R(x)=m_1+\Sigma_0^{1/2}\left(\Sigma_0^{-1/2}\Sigma_1 \Sigma_0^{-1/2}\right)^{1/2}\Sigma_0^{-1/2}(x-m_0),
\end{equation}
where $\Sigma^{1/2}$ denotes the positive square root matrix of the positive semidefinite symmetric matrix $\Sigma$, i.e., $\Sigma^{1/2}=PD^{1/2}P^\top$ where $PDP^\top$ is the eigendecomposition of $\Sigma$.
Moreover, the map $R(\cdot)$ does not change upon iteration, i.e., the rectified flow map starting from the coupling $(X_0, R(X_0))$ is $R(\cdot)$.
\end{proposition}

\begin{remark}
The resulting map is, in general, different from the well-known optimal transport between Gaussians given by \citep{peyre2019computational}: 
$$
T_0(x)=m_1+\Sigma_0^{-1/2}\left(\Sigma_0^{1/2}\Sigma_1\Sigma_0^{1/2}\right)^{1/2}
\Sigma_0^{-1/2}\left(x-m_0\right).
$$
The rectified map and the optimal transport map are equivalent only if $\Sigma_0$ and $\Sigma_1$ commute (i.e., simultaneously diagonalizable). This fact is also well-known; see~\citet[Theorem 2(ii)]{hertrich2025relation}.
As a simple example,
Figure \ref{fig:side-by-side} (left)
shows
$v(t,z)$
when
$X_0 \sim N(0,1)$
and $X_1 \sim N(0,9)$.
\end{remark}

\begin{remark}
It is interesting to see what happens when
$X_0\overset{d}{=} X_1$. This is called self-transport.
\cite{mordant2024entropic} considered this
in the case of entropic transport.
If $X_0\overset{d}{=} X_1\sim N(m, \Sigma)$, from Proposition~\ref{prop:flow-from-Gaussian-to-Gaussian}, it follows that
$R(x) = x$ is the identity map,
but the path $t\mapsto Z_t$ (the solution of~\eqref{eq:rectified-flow-ODE}) is nonlinear (Figure~\ref{fig:side-by-side}, right).
\end{remark}

\paragraph{$\mu_0$ and $\mu_1$ are mixture of Gaussians.}

Another special case 
where $v_t(z)$ has a simple closed form
is the mixture of Gaussians. The following result is also known; see~\citet[Theorem 4]{hertrich2025relation}.

\begin{lemma}
\label{lemma::mixture}
Suppose that
$X_0 \sim \sum_{i=1}^{I_0}\pi^i_{0}N(m^i_0,\Sigma^i_{0})$
and
$X_1 \sim \sum_{j=1}^{I_1}\pi^j_{1}N(m^j_{1},\Sigma^j_{1})$
where, $X_0$ and $X_1$ are independent of each other.
Then
\[
    v_t(z) = \frac{\sum_{i,j}^{I_0,I_1}\pi^i_{0}\pi^j_{1}v^{i,j}_t(z)\tau^{i,j}_t(z)}{\sum_{i,j}^{I_0,I_1}\pi^i_{0}\pi^j_{1}\tau^{i,j}_t(z)}
\]
where $v^{i,j}_t(z)$ is the velocity field between $N(m^i_0,\Sigma^i_{0})$ and $N(m^j_{1},\Sigma^j_{1})$ that can be computed by Equation \eqref{eq:gaussianvel}, and
$\tau^{i,j}_t(z) = N(z;m^{i,j}_t,\Sigma^{i,j}_t)$ where $m^{i,j}_t=tm_1^j+(1-t)m_0^i$ and $\Sigma_t^{i,j}=t^2\Sigma_1^j+(1-t)^2\Sigma_0^i$.
\end{lemma}

While the rectified flow itself cannot be obtained in a closed-form from Lemma~\ref{lemma::mixture}, it provides a simple method to estimate the velocity field in practice. The parameters of the mixtures
can be estimated using the EM algorithm, and subsequently, the velocity field can be estimated using a simple plug-in.
This implies that estimating the rectified flow map between two mixtures of Gaussians can be computationally efficient, unlike the optimal transport map.
Again, this is in sharp contrast to
optimal transport, where no closed-form
expression is available for
mixtures of Gaussians.
\begin{wrapfigure}{r}{0.3\textwidth} 
% 'r' for right, 0.4\textwidth for width
  \centering
  \includegraphics[width=0.3\textwidth]{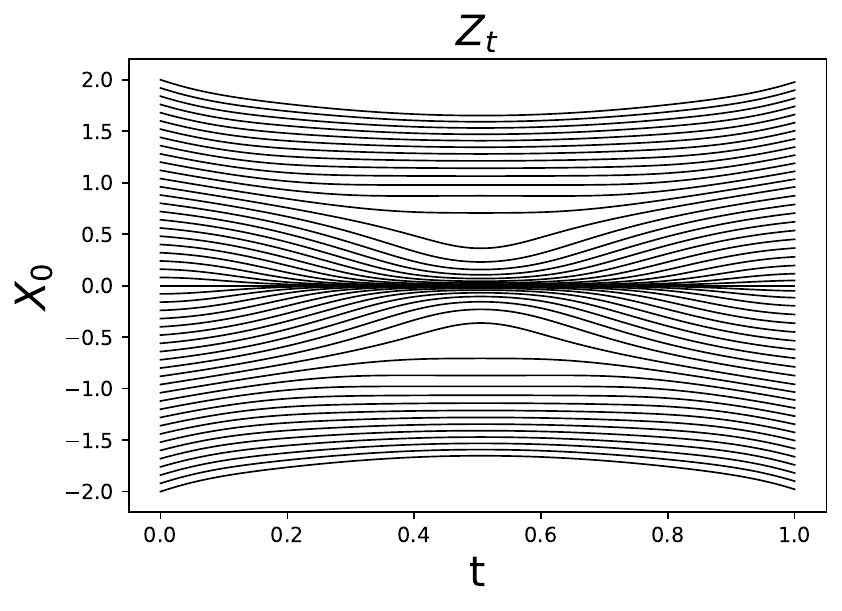} % Replace with your image
  \caption{Plot of the rectified flow map $x\mapsto R(x)$ in transporting $X_0\sim 0.5 N(1, 1) + 0.5 N(-1, 1)$ to itself.}
  \label{fig::mixture}
\end{wrapfigure}
In Fig. \ref{fig::mixture} we show the trajectories $Z_t$ for different starting points $X_0$ when transporting from the mixture $0.5N(1,1)+0.5N(-1,1)$ to itself.

\paragraph{$\mu_0$ and $\mu_1$ are Uniform.}

Another special case that provides some interesting insight
is the case where $\mu_0 = \mu_1 = \mathrm{Unif}(0,1)$.
It is easy to see that $p_0(z - tx)p_1(z + (1 - t)x) > 0$ if and only if $x$ lies in $S_t(z)$ where
$S_t(z) = [a_t(z),b_t(z)],$
with
\begin{align*}
a_t(z) &= \max\Biggl\{\frac{z-1}{t}, - \frac{z}{1-t} \Biggr\},\ \ \ 
b_t(z) = \min\Biggl\{\frac{z}{t}, \frac{1-z}{1-t} \Biggr\}
\end{align*}
Given that the densities are constant on $[0, 1]$, we get that $v_t(z) = (a_t(z)+b_t(z))/2.$
In particular, if $t\in(0, 1/2]$, we have
\begin{equation}\label{eq:uniform-velocity-field}
v_t(z) {=} \frac{1}{2t(1-t)}\times
\begin{cases}
z(1-2t), &\mbox{if }z \le t,\\
(1-2z)t, &\mbox{if }t \le z \le 1 - t,\\
(1-z)(2t-1), &\mbox{if }z \ge 1 - t.
\end{cases}
\end{equation}
A plot of $v_t(z)$ for a few values of $t$ is given in Figure 
\ref{fig::uniform_lipschitz}. While not represented in these plots, it is noteworthy that $(t, z)\mapsto v(t, z)$ is not a continuous function at $\{0, 1\}^2$. In fact, 
\[
\lim_{z\to0}\lim_{t\to0} v_t(z) ~\neq~ \lim_{t\to0}\lim_{z\to0} v_t(z). 
\]
Take, for example, $z = 0$ and then $t\to0$, and $z = t$ and then $t \to 0$.
Moreover, from the figures, it is clear that $z\mapsto v_t(z)$ is not differentiable everywhere
and, in fact, is not even uniformly (in $t$) Lipschitz with the Lipschitz constant tending to infinity as $t(1-t)\to 0$.
This example shows that one cannot expect smoothness of $z\mapsto v_t(z)$ when the support $\Omega$ is bounded.
% , and we shall see that this is indeed the case more generally. While the Lipschitz constant diverges as $t(1-t)\to0$, one can find a more refined bound for $|v_t(z_1) - v_t(z_2)|$ as a function of $t$ and $|z_1 - z_2|$. In fact, for $t\in[0, 1/2]$, we have
% \[
% v_t(z_1) - v_t(z_2) = 
% \frac{1}{2t(1 - t)}\times\begin{cases}
% (z_1 - z_2)(1 - 2t), &\mbox{if }z_1 \le z_2 \le t,\\
% 2t|z_1 - z_2| - |z_1 - t|, &\mbox{if }z_1 \le t \le z_2 \le 1 - t,\\
% (z_1 - z_2 + 1)(1 - 2t), &\mbox{if }z_1 \le t\le 1 - t\le z_2,\\
% -2(z_1 - z_2)t, &\mbox{if }t \le z_1 \le z_2 \le 1 - t,\\
% 2t|z_1 - z_2| + (1 - t - z_2), &\mbox{if }t \le z_1 \le 1 - t \le z_2.
% \end{cases}
% \]
It is easy to see that for $t\in[0, 1/2]$,
\begin{equation}\label{eq:uniform-Lipschitz-cond}
|v_t(z_1) - v_t(z_2)| \le a(t)\kappa(|z_1 - z_2|),\quad\mbox{where}\quad a(t) = \frac{1}{2t(1-t)},\; \kappa(u) = u.
\end{equation}
This implies that the Lipschitz constant diverges as $t(1-t)\to0$.
% Even though the integrability of the Lipschitz constant fails, Nagumo's uniqueness theorem~\citep[Sec. 2]{Bihari1956} implies that there is a unique solution to~\eqref{eq:rectified-flow-ODE} and hence, the rectified flow map is uniquely defined. Also, see~\cite{bownds1971extension} and Section 1.6 of~\cite{agarwal1993uniqueness}.

\begin{figure}[H]
\begin{center}
\includegraphics[width=0.8\textwidth]{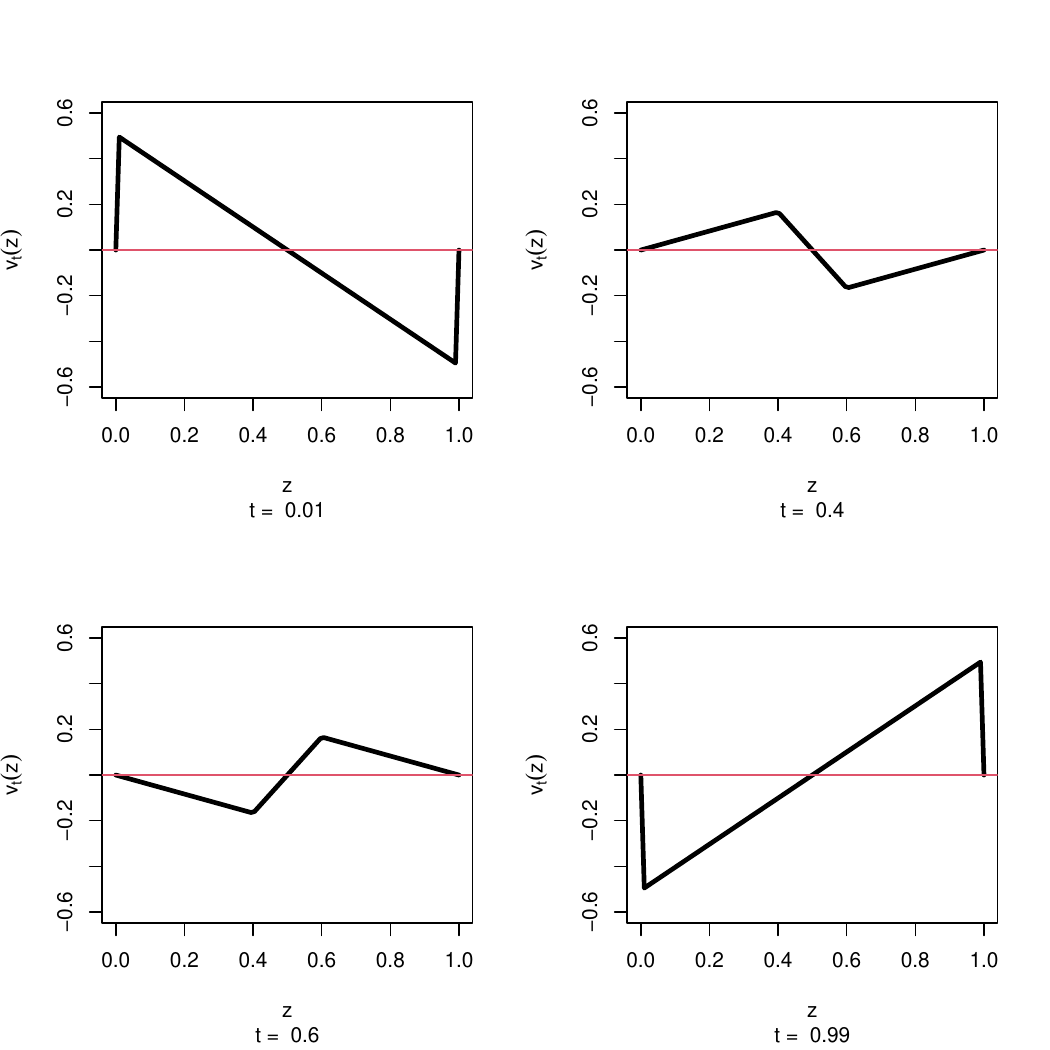}
\end{center}
\caption{\em
A plot of $v_t(z)$ versus $z$ for four values of $t$
in the case where $\mu_0 = \mu_1 = \mathrm{Unif}[0,1]$.
We see that $v_t(z)$ is piecewise smooth.
As $t$ approaches 0 and 1, the Lipschitz constant
approaches infinity near the boundary.}
\label{fig::uniform_lipschitz}
\end{figure}

\subsection{Estimating the Velocity}
\label{section::estimation}

Now we turn to the nonparametric estimation
of $v_t(z)\equiv v(t, z)$.
Throughout this subsection, we assume access to independent observations
$X_{01},\ldots, X_{0n}\sim p_0$
and
$X_{11},\ldots, X_{1n}\sim p_1$.
The implicit assumption that the number of observations from $p_0$ is the same as those from $p_1$ is not significant and is only made for notational ease.
We assume that the support of $p_0$ and $p_1$ is $\Omega\subseteq\mathbb{R}^d$.
We consider
density-based,
regression-based, and semiparametric estimation.
In the case where $\Omega$ is compact,
the density based approach is superior to regression.
Examination of the regularity needed for these results
is deferred until later sections.

\subsubsection{Density-Based Estimator}
\label{section::density}

Based on the second representation in Lemma~\ref{lem:representation-of-velocity-field},
we define a density-based estimator as follows.
Define
\begin{equation}\label{eq:product-of-densities}
r_t(z, \delta) = p_0(z - t\delta)p_1(z + (1 - t)\delta).
\end{equation}
Let
$\hat p_0$ and $\hat p_1$
be estimators
of $p_0$ and $p_1$ and
$$
\hat r_t(z,\delta) = \hat p_0(z-t\delta)\hat p_1(z+(1-t)\delta).
$$
Then the density-based estimator of the velocity field is given by 
\begin{equation}\label{eq:density-based-velocity-estimator}
\hat v_t^{\mathrm{den}}(z) = \frac{\hat f_t(z)}{\hat p_t(z)},
\end{equation}
where
$$
\hat f_t(z) = \int_{\mathbb{R}^d} \delta \hat p_0(z-t\delta)\hat p_1(z+(1-t)\delta)d\delta
$$
and
$$
\hat p_t(z) = \int_{\mathbb{R}^d} \hat p_0(z-t\delta)\hat p_1(z+(1-t)\delta)d\delta.
$$
If the support of the densities, $\Omega$, is known, then it is important to use a density estimator that is not subject to boundary bias; see
\cite{bouezmarni2010, bertin2025new, muller1999}
and references therein. 
Because we require
some specific properties,
we propose a 
new estimator in~\eqref{eq:new-boundary-corrected-estimators} that works for arbitrary domains.

In Section \ref{sec:bounded},
we show that this estimate
has a certain robustness to the lack of smoothness of $z\mapsto v_t(z)$.
Specifically, if $\Omega$ is bounded, then
$v_t(z)$ is not even Lipschitz continuous, 
but the rate of convergence of $\widehat{v}_t^{\mathrm{den}}(z)$ depends on the smoothness of $p_0, p_1$.
This is in contrast to the regression approach (discussed below)
which relies on the smoothness of $z\mapsto v_t(z)$.

The following is a simple result that demonstrates the consistency of $\hat{f}_t(z)$ and $\hat{p}_t(z)$ under some primitive consistency conditions on $\hat{p}_0$ and $\hat{p}_1$. Note that consistency is implied only for $t\in(0, 1)$ but not for $t$ such that $t(1-t)\to0$. This result does not make any assumptions on the support of $p_0, p_1$. In Sections~\ref{sec:unbounded} and~\ref{sec:bounded}, we provide refined rates of convergence for the velocity field estimate, under some assumptions on the support. 

Define
\begin{align*}
    f_t(z) &= \int_{\mathbb{R}^d} \delta p_0(z - t\delta)p_1(z + (1 - t)\delta)d\delta,\\
    p_t(z) &= \int_{\mathbb{R}^d} p_0(z - t\delta)p_1(z + (1 - t)\delta)d\delta,\\
    \|\widehat{p}_t - p_t\|_2 &= \left(\int_{\mathbb{R}^d} |\hat{p}_t(\delta) - p_t(\delta)|^2d\delta\right)^{1/2},\quad t\in[0, 1],\\
    \bar{\zeta}_2(\hat{p}_t, p_t) &= \sup\left\{\left|\int_{\mathbb{R}^d} g(\delta)(\hat{p}_t(\delta) - p_t(\delta))d\delta\right|:\, g:\mathbb{R}^d\to\mathbb{R}\mbox{ is a quadratic function}\right\},\quad t\in[0, 1].
\end{align*}
Note that $p_t(z) = p_0(z)$ if $t = 0$ and $p_t(z)=p_1(z)$ if $t = 1$. In fact, $p_t(\cdot)$ is the Lebesgue density of $X_t = (1 - t)X_0 + tX_1$ for independent $X_0, X_1$. In addition, we note that $\bar{\zeta}_2(\widehat{p}_t, p_t)$ measures how well the first two moments of $p_t$ can be estimated by those of $\widehat{p}_t$.
\begin{proposition}\label{prop:consistency-of-density-estimator}
    Fix any $t\in[0, 1]$ and any $z\in\mathbb{R}^d$, then 
    {\small\[
    \|\hat{f}_t(z) - f_t(z)\| \le \frac{\|\hat{p}_1 - p_1\|_2\|p_0\|_{\infty}(\mathbb{E}\|z - X_0\|^2)^{1/2} + \|\hat{p}_0 - p_0\|_2\|\hat{p}_1\|_{\infty}(\mathbb{E}\|z - X_1\|^2 + \bar{\zeta}_2(\hat{p}_1, p_1))^{1/2}}{t^{d/2}(1 - t)^{d/2}}. 
    \]}
    and
    \[
    |\hat{p}_t(z) - p_t(z)| \le \frac{\|p_0\|_{\infty}\|\hat{p}_1 - p_1\|_2 + \|\hat{p}_1\|_{\infty}\|\hat{p}_0 - p_0\|_2}{t^{d/2}(1 - t)^{d/2}}.
    \]
\end{proposition}
The upper bounds presented in Proposition~\ref{prop:consistency-of-density-estimator} are asymmetric in $(p_0, \hat{p}_0)$ and $(p_1, \hat{p}_1)$, and one can take the minimum of the two bounds obtained by swapping the roles of $(p_0, \hat{p}_0)$ and $(p_1, \hat{p}_1)$.
Proposition~\ref{prop:consistency-of-density-estimator} implies if 
\begin{equation}\label{eq:density-uniform-boundedness}
\max\{\|p_0\|_{\infty}, \|p_1\|_{\infty}, \|\hat{p}_0\|_{\infty}, \|\hat{p}_1\|_{\infty}\} = O_p(1),
\end{equation}
and $\max\{\mathbb{E}[\|z - X_0\|^2], \mathbb{E}[\|z - X_1\|^2]\} = O(1)$, then 
\[
\|\hat{f}_t(z) - f_t(z)\| = O_p\left(\frac{\|\hat{p}_1 - p_1\|_2 + \|\hat{p}_0 - p_0\|_2}{t^{d/2}(1 - t)^{d/2}}\right),
\]
and
\[
\sup_{z\in\mathbb{R}^d}\,|\hat{p}_t(z) - p_t(z)| = O_p\left(\frac{\|\hat{p}_1 - p_1\|_2 + \|\hat{p}_0 - p_0\|_2}{t^{d/2}(1 - t)^{d/2}}\right).
\]
Therefore, assuming $L_2$-consistency of $\hat{p}_1$ to $p_1$ and $\hat{p}_0$ to $p_0$, $\hat{f}_t(z)$ and $\hat{p}_t(z)$ are consistent for $f_t(z)$ and $p_t(z)$, respectively. Finally, an application of Slutsky's theorem implies consistency of $\hat{v}_t^{\mathrm{den}}(z)$ for all $z$ such that $p_t(z) > 0$ and all $t\in(0, 1)$. The consistency is non-uniform in $t\in(0, 1)$. Formally, we have
\[
\|\hat{v}_t^{\mathrm{den}}(z) - v_t(z)\| \le \frac{\|\hat{f}_t(z) - f_t(z)\|}{p_t(z) - \|\hat{p}_t(z) - p_t(z)\|_{\infty}} + \frac{\|f_t(z)\|}{p_t(z)(p_t(z) - \|\hat{p}_t - p_t\|_{\infty})}|\hat{p}_t(z) - p_t(z)|.
\]
Hence, for all $(t, z)$ such that $p_t(z) \ge 2\|\hat{p}_t - p_t\|_{\infty}$, we get
\begin{equation}\label{eq:rate-of-convergence-density-field}
\frac{\|\hat{v}_t^{\mathrm{den}}(z) - v_t(z)\|}{1 + \|v_t(z)\|} =  O_p\left(\frac{\|\hat{p}_1 - p_1\|_2 + \|\hat{p}_0 - p_0\|_2}{p_t(z)t^{d/2}(1 - t)^{d/2}}\right).
\end{equation}
Note that under assumption~\eqref{eq:density-uniform-boundedness}, $L_2$ consistency of the density estimators is implied by $L_1$ consistency. It is well-known that without any (smoothness) assumptions on the densities, $L_1$-consistent estimators can be constructed~\citep{devroye-density-estimation}.

The rate of convergence in~\eqref{eq:rate-of-convergence-density-field} also implies that $\hat{v}_t(\cdot)$ inherits the rate of convergence from each of the density estimators, no matter the smoothness of the velocity field. In particular, if $p_0$ and $p_1$ can be estimated at a parametric rate, that is, 
\begin{equation}\label{eq:parametric-rate}
\|\hat{p}_1 - p_1\|_2 = O_p(n^{-1/2})\quad\mbox{ and }\|\hat{p}_0 - p_0\|_2 = O_p(n^{-1/2}), 
\end{equation}
then the velocity field can be estimated (pointwise) at a parametric rate. Assumption~\eqref{eq:parametric-rate} holds, for example, if $p_0$ and $p_1$ are known to belong to a parametric model~\citep{dasgupta2012density}. In the non-parametric case, approximate parametric rates are possible with extreme smoothness on the densities $p_0, p_1$.  Note that Proposition~\ref{prop:consistency-of-density-estimator} does not make any assumption on the structure of density estimators $\hat{p}_0$ and $\hat{p}_1$. With no additional structural assumptions, one can use non-parametric density estimators such as the kernel density estimator or the $k$-nearest neighbor estimators. We note here that these estimators can adapt to the intrinsic volume dimension of the data~\citep{dasgupta2014optimal,kim2019uniformconvergenceratekernel,zhao2022analysis}. If some additional structure is assumed, then specialized density estimators can be used. Here are some examples of such an additional structure:
\begin{enumerate}
    \item If the densities $p_0$ and $p_1$ are assumed to be (homothetic) log-concave, $\hat{p}_0$ and $\hat{p}_1$ can be taken to be the non-parametric MLE~\citep{samworth2018recent,kubal2024log,xu2021high}. 
    \item If the densities $p_0$ and $p_1$ have depend only $x$ only through low-dimensional projections, then $\hat{p}_0$ and $\hat{p}_1$ can be taken to be the projection pursuit density estimators~\citep{friedman1984projection,vandermeulen2024dimension}.  
    \item If the densities $p_0$ and $p_1$ are expected to be close to a parametric family, then $\hat{p}_0$ and $\hat{p}_1$ can be taken to be quasi-MLEs with a non-parametic adjustment as in~\citep{hjort1995nonparametric,hjort1996locally}.
\end{enumerate}
\begin{remark}
    It is worth stressing here that Proposition~\ref{prop:consistency-of-density-estimator} does not require the underlying data $(X_{01}$, $\ldots$, $X_{0n})$ and $(X_{11}, \ldots, X_{1n})$ to be independent. Even if these two data vectors are dependent, one can construct estimators of $p_0$ and $p_1$ and use the velocity field estimator $\hat{v}_t^{\mathrm{den}}(\cdot)$. This means that even if the given data is not from an independent coupling, we can get a rectified flow map with an independent coupling by using $\hat{v}_t^{\mathrm{den}}(\cdot)$.
\end{remark}

\subsubsection{Regression-based estimator}

The expression
$v_t(z) = \E[ X_1-X_0|(1-t) X_0 + t X_1 = z]$
implies that we can estimate
$\hat v_t(z)$
by performing nonparametric regression for the mean function $\mathbb{E}(X_1 - X_0|X_t=z)$
where 
$X_t\equiv (1-t)X_0 + t X_1$.
This regression is non-standard
for two reasons.
First, the outcome $X_1 - X_0$
and the features $X_t$
are deterministic functions of the same
underlying variables $(X_0,X_1)$.
Second, the features $X_t$ are correlated for different values of $t$.

In the context of regression-based estimation of the velocity field, several estimators are possible. Firstly, following the logic of the Nadaraya-Watson regression estimator, we can approximate $\mathbb{E}[X_1 - X_0|X_t = z]$ with $\mathbb{E}[(X_1 - X_0) K_h(X_t - z)]/\mathbb{E}[K_h(X_t - z)]$ for a function $K_h(\cdot)$ that approximates a bump at zero. The numerator is the expected value of a two-sample second-order $U$-statistic kernel $(X_1 - X_0)K_h((1 - t)X_0 + tX_1 - z)$. Hence, given access to independent data vectors $(X_{01}, \ldots, X_{0n})$ and $(X_{11}, \ldots, X_{1n})$, we obtain our first regression based estimator
\begin{equation}\label{eq:U-statistics-regression-field}
\hat{v}_t^{\mathrm{reg0}}(z) = \frac{\sum_{i,j} (X_{1i} - X_{0j})K_h((1 - t)X_{0j} + tX_{1i} - z)}{\sum_{i,j} K_h((1 - t)X_{0j} + tX_{1i} - z)}.
\end{equation}
Here, $\sum_{i,j}$ is a summation over $i\in\{1, 2, \ldots, n\}$ and $j\in\{1, 2, \ldots, n\}$. If $z\mapsto v_t(z)$ is expected to have some structure encoded through the assumption that $v_t\in\mathcal{V}$ for a function class $\mathcal{V}$, then one can consider the following regression-based estimator of $v_t$:
\begin{equation}\label{eq:ERM-regression-field}
    \hat{v}_t^{\mathrm{reg1}}(z) = \argmin_{v\in\mathcal{V}}\, \sum_{i,j} (X_{1i} - X_{0j} - v((1 - t)X_{0j} + tX_{1i}))^2.
\end{equation}
For example, $\mathcal{V}$ can encode a parametric model, or a single-index model, or a multiple-index model, or a neural network, or even qualitative constraints such as convexity. 
In general, one can create the artificial (regression) data set $((1 - t)X_{0j} + tX_{1i},\, X_{1i} - X_{0j}), i,j\in\{1, 2, \ldots, n\}$ and apply arbitrary nonparametric regression techniques. If computing the regression estimator based on $n^2$ observations is computationally intensive, then one can compute an estimator using a random pairing of observations, e.g., $(X_{0i}, X_{1i}), 1\le i\le n$.
For example, the kernel regression estimator with one such pairing is
\begin{equation}\label{eq:kerreg}
\hat v_t^{\mathrm{reg2}}(z) = 
\dfrac{\sum_{i=1}^n (X_{1i}-X_{0i})
K_h\left(X_{ti}-z\right)}{\sum_{i=1}^n K_h\left( X_{ti}-z\right)}.
\end{equation}
Although some efficiency is lost, an advantage of $\hat{v}_t^{\mathrm{reg2}}$ over the others is that it is computed using $n$ independent observations. 
For a gain in efficiency, one can consider estimators computed from several such pairings and average them. 

Assuming $v_t(\cdot)$ is either smooth or close enough to $\mathcal{V}$, it is possible to get rates of convergence of the aforementioned estimators of $v_t$. Unfortunately, such results are of little practical value because the smoothness of $z\mapsto v_t(z)$ is a delicate issue, as already exemplified in Figure~\ref{fig::uniform_lipschitz}. In Section~\ref{sec:unbounded}, we study the regularity of $z\mapsto v_t(z)$ when $\mu_0$ and $\mu_1$ are strongly log-concave distributions, and provide rates of convergence for $\widehat{v}_t^{\mathrm{reg2}}.$

The above is not the only possible regression-based estimator.  We could, for example, use the fact that $v_t(z)=z/t - \mathbb{E}(X_0|X_t=z)/t$ and regress $X_0$ on $X_t$. This yields the analogue of $\hat{v}_t^{\mathrm{reg2}}$ as
$${\hat{v}}_t^{\mathrm{reg3}}(z)=\frac{z}{t}-\frac{1}{t}\dfrac{\sum_{i=1}^nX_{0i} K_{h}\left(tX_{1i}+(1-t)X_{0i}-z\right)}{\sum_{i=1}^nK_{h}\left(tX_{1i}+(1-t)X_{0i}-z\right)}.$$

\paragraph{Connection between density- and regression-based estimators.}
Recall that the denominator of $v_t(z)$ is
\begin{equation}\label{eq:representation-of-density-of-X_t}
p_t(z) =
\int p_0(z-t\delta)p_1(z+(1-t)\delta) d \delta =
\frac{1}{t^d}\int p_1\left(\frac{z-x}{t}+x\right)p_0(x) dx.
\end{equation}
We have two estimators:
one from regression and one from density estimation.
Consider the denominator of the kernel regression based estimator $\widehat{v}_t^{\mathrm{reg0}}(z)$:
$$
\hat{p}_t^{\mathrm{reg0}}(z) =  \frac{1}{n^2}\sum_{i,j} \frac{1}{h^d} K \left(\frac{(1-t)X_{0j}+tX_{1i}-z}{h}\right).
$$
Now consider the density based estimator using the second equality in~\eqref{eq:representation-of-density-of-X_t}:
$$
\hat{p}_t^{\mathrm{den}}(z) = \frac{1}{nt^d}\sum_i \hat p_1\left(\frac{z-X_{0i}}{t}+X_{0i}\right).
$$
Taking $\hat p_1$ to be a kernel density estimator (with a symmetric kernel $K(\cdot)$), we get
\begin{align*}
\hat{p}_t^{\mathrm{den}}(z) &= \frac{1}{nt^d}\sum_i \hat p_1\left(\frac{z-X_{0i}}{t}+X_{0i}\right) =  \frac{1}{n^2t^d}\sum_{i,j} \frac{1}{h^d} K \left(\frac{X_{1j} - \frac{z-X_{0i}}{t}-X_{0i}}{h}\right)\\
&=  \frac{1}{n^2}\sum_{i,j} \frac{1}{(th)^d} K \left(\frac{(1-t)X_{0i}+tX_{1j}-z}{th}\right).
\end{align*}
This is equivalent to $\hat{p}_t^{\mathrm{reg0}}(z)$ but with a time-dependent bandwidth $th$.

\subsubsection{Substitution Estimators}
Using the third and fourth representations for the velocity field in Lemma~\ref{lem:representation-of-velocity-field}, one can define estimators that only use the density estimator for one of the measures. Formally, from the third representation in Lemma~\ref{lem:representation-of-velocity-field}, we can define
the estimator $\hat{v}_t^{(3)}(z) = \hat{f}_t^{(3)}(z)/\hat{p}_t^{(3)}(z)$, where
\begin{align*}
\hat f_t^{(3)}(z) &= \frac{1}{nt^d}\sum_{i=1}^n \frac{z-X_{0i}}{t}\hat p_1\left( \frac{z-X_{0i}}{t}+X_{0i} \right)\\
\hat p_t^{(3)}(z) &= \frac{1}{nt^d}\sum_{i=1}^n  \hat p_1\left( \frac{z-X_{0i}}{t}+X_{0i} \right).
\end{align*}
This can be particularly useful if we, for example, know that $p_1$ is very smooth and therefore can be estimated at a fast rate, while $p_0$ cannot be estimated as well. On the flip side, the fourth representation in Lemma~\ref{lem:representation-of-velocity-field} yields the following estimator
\begin{align*}
    \hat f_t^{(4)}(z) &= \frac{1}{n(1-t)^d}\sum_{i=1}^n \frac{X_{1i} - z}{1 - t}\hat{p}_0\left(\frac{z - X_{1i}}{1 - t} + X_{1i}\right),\\
    \hat{p}_t^{(4)}(z) &= \frac{1}{n(1-t)^d}\sum_{i=1}^n \hat{p}_0\left(\frac{z - X_{1i}}{1 - t} + X_{1i}\right).
\end{align*}
The explicit dependence on $t, 1 - t$ in the denominator here implies that these estimators have variances blowing up as $t\to0$ or $1-t\to 0$ for $\widehat{p}_t^{(3)}$ and $\widehat{p}_t^{(4)}$, respectively. 
\subsubsection{Semiparametric Estimators}

For fixed $t$ and $z$,
we can view
the velocity as a ratio
of linear functionals,
namely,
$v_t(z) = \psi_N/\psi_D$ 
where
$$
\psi_D = \int p_0(z-t\delta)p_1(z+ (1-t)\delta)d\delta = p_t(z)
$$
and
$$
\psi_N = \int \delta p_0(z-t\delta)p_1(z+ (1-t)\delta)d\delta = f_t(z).
$$
Since these are
bilinear functionals of $p_0$ and $p_1$
we might expect to construct
a semiparametric efficient estimator for $v_t(z)$.
This turns out to be true but with some caveats.

Recall that the one-step, semiparametric estimator of a 
pathwise differentiable functional
$\psi=T(p)$ is the plugin estimator plus
the efficient influence function:
$$
\hat\psi = T(\hat p) + \frac{1}{n}\sum_i \varphi(X_i,\hat p)
$$
where $\hat p$ and 
$n^{-1}\sum_i \varphi(X_i,p)$ are 
usually computed from separate parts of the data
and $\varphi(x,p)$ is the efficient influence function.
Often,
$\varphi(x,p)$ is simply the Gateaux derivative of $T(p)$
which will be the case in what follows.
If $\|\widehat{p} - p\| = o_p(n^{-1/4})$,
then
$\sqrt{n}(\hat\psi - \psi)\overset{d}{\to} N(0,\sigma^2)$
where
$\sigma^2 = \E[\varphi^2(X,p)]$.
Now we apply this approach to
the velocity.

Let $\hat p_0$ and $\hat p_1$
be estimators of $p_0$ and $p_1$.
The plugin estimator of $v_t(z)$ is
$$
\hat\psi_{pi}=
\frac{\int_{S_t(z)} \delta  \hat p_0(z-t\delta)\hat p_1(z+ (1-t)\delta)d\delta}
{ \int_{S_t(z)} \hat p_0(z-t\delta)\hat p_1(z+ (1-t)\delta)d\delta}.
$$
The one-step estimator is
$$
\hat v_t^{1\mbox{-}\mathrm{step}}(z) = \hat\psi_{pi} +
\frac{1}{n}\sum_i \varphi_0(X_{0i},\hat p_1) + 
\frac{1}{n}\sum_i \varphi_1(X_{1i},\hat p_0)
$$
where
$\varphi_0$ is the efficient influence
function of $v_t(z)$ with respect to $\mu_0$ and
$\varphi_1$ is the efficient influence
function of $v_t(z)$ with respect to $\mu_1$ 
which are given in the next lemma.

\begin{theorem}\label{thm:semiparametric-rate}
For $0 < t < 1$,
the efficient influence function
for $v_t(z)$ is
$\varphi = \varphi_0 + \varphi_1$ where
$\varphi_0$ and $\varphi_1$ 
are the Gateuax derivatives of $v_t(z)$
with respect to $\mu_0$ and $\mu_1$ and
are given by
\begin{align*}
\varphi_0(x,p_1) &= \frac{\varphi_{N0}(x,p_1)}{p_t(z)} - 
v_t(z) \frac{\varphi_{D0}(x,p_1)}{p_t(z)} \\
\varphi_0(y,p_0) &= \frac{\varphi_{N1}(y,p_0)}{p_t(z)} - 
v_t(z) \frac{\varphi_{D1}(y,p_0)}{p_t(z)}
\end{align*}
where
\begin{align*}
\varphi_{N0}(x,p_1) &= 
\frac{1}{t^d}\left(\frac{z-x}{t}\right)p_1\left(\frac{z-x}{t} + x\right) - f_t(z)\\
\varphi_{N1}(y,p_0) &= 
\frac{1}{(1-t)^d}\left(\frac{z-y}{1-t}\right)p_0\left(\frac{z-y}{1-t} + y\right) - f_t(z)\\
\varphi_{D0}(x,p_1) &= \frac{1}{t^d}p_1\left(\frac{z-x}{t} + x\right) - p_t(z)\\
\varphi_{D1}(y,p_0) &= \frac{1}{(1-t)^d}p_0\left(\frac{z-y}{1-t} + y\right) - p_t(z).
\end{align*}
If $\|\hat p_0-p_0\|_2=o_p(n^{-1/4})$ and
$\|\hat p_1-p_1\|_2=o_p(n^{-1/4})$ and
are computed from an independent sample, then
$\sqrt{n}(\hat v_t(z) - v_t(z)) \overset{d}{\to} N(0,\Sigma)$
where
$\Sigma = \mathrm{Var}(\varphi_0) + \mathrm{Var}(\varphi_1)$.
\end{theorem}

The above shows that $v_t(z)$
can be estimated at a $n^{-1/2}$ rate
which seems to contradict our other results.
However,
this only applies when $t\in(0,1)$ is fixed.
The following result shows that the semiparametric efficient estimator behaves poorly as $t(1-t)\to 0$.

\begin{lemma}\label{lem:semiparametric-bad}
Under the setting of Theorem~\ref{thm:semiparametric-rate},
\[
\lim_{n\to\infty}n\mathrm{Var}(\hat v_t^{1\mbox{-}\mathrm{step}}(z)) \asymp \frac{1}{t^d(1-t)^d},\quad\mbox{as}\quad t(1-t)\to 0.
\]
\end{lemma}

Given that the variance explodes as $t\to 0,1$,
how should we use this estimator?
One possibility is to proceed as follow.
Note that
$v_0(z)=\E[X_1]-z$ and
$v_1(z)=z - \E[X_0]$
which can be estimated at an $n^{-1/2}$ rate.
Fix a small $t_0>0$.
We can use a hybrid model that treats
$v_t$ nonparametrically using the one-step estimator
over $t\in [t_0,1-t_0]$
and we estimate $v_t$ using a parametric model otherwise.
For example,
$$
\hat v_t(z) = 
\begin{cases}
\hat v_0(z) = n^{-1}\sum_{i=1}^n X_{1i}-z & 0 \leq t \leq t_0\\
\hat v_t^{1\mbox{-}\mathrm{step}}(z) & t_0 < t < 1-t_0\\
\hat v_1(z) = z - n^{-1}\sum_{i=1}^n X_{0i} & 1-t_0 \leq t \leq 1.
\end{cases}
$$
More generally, we could use a polynomial
or some other smooth function for $t\in[0,t_0]$ or $t\in[1-t_0, 1]$.
Then
$\hat v_t(z) - v_t(z) = O_p(n^{-1/2})$.
Another possibility is to use the estimator
$\hat R(x) = x + \int_{\xi}^{1-\xi} \hat v_t^{1\mbox{-}\mathrm{step}}(\hat z_t)dt$
which provides a $\sqrt{n}$-estimate of an approximate rectified flow.

\subsubsection{Smoothed Transport}
\label{section::smoothed}

Some researchers have focused on
transporting a smoothed version of $\mu_0$
to a smoothed version of $\mu_1$
\citep{goldfeld2020, goldfeld2021, chen2021asymptotics}.
The resulting transport map and Wasserstein distance
can be estimated at a $n^{-1/2}$ rate.
Let
$K_\sigma$ denote a Normal with variance $\sigma^2 I$.
Similarly, the smoothed Wasserstein distance
is defined to be
$W(\mu_0\star K_\sigma,\mu_1\star K_\sigma)$
where $\star$ denotes convolution.
We define the smooth rectified map
$R_\sigma$ 
to be the rectified map from
$\mu_0\star K_\sigma$ to
$\mu_1\star K_\sigma$.
Then,
$$
v_{t,\sigma}(z) =
\frac{\int \delta \int K_\sigma(z-t\delta-u)d\mu_0(u) 
\int K_\sigma(z+(1-t)\delta-v)d\mu_1(v) d\delta }
{\int  \int K_\sigma(z-t\delta-u)d\mu_0(u) 
\int K_\sigma(z+(1-t)\delta-v)d\mu_1(v) d\delta } \equiv \frac{N}{D}.
$$
We estimate the velocity field with
the following ratio of two sample $U$-statistics:
\begin{align*}
\hat v_t(z) &=
\frac{\sum_i\sum_j h_{t,z}(X_{0i},X_{1j})}
{\sum_i\sum_j \tilde h_{t,z}(X_{0i},X_{1j})}
\equiv \frac{\hat N}{\hat D}
\end{align*}
where
\begin{align*}
h_{t,z}(u,v) &= \int_{\mathbb{R}^d} \delta K_\sigma(z - t \delta -u)
K_\sigma(z + (1-t) \delta -v)d\delta\\
\tilde h_{t,z}(u, v) &= \int_{\mathbb{R}^d}  K_\sigma(z - t \delta -u)K_{\sigma}(z +(1-t)\delta - v)d\delta.
\end{align*}
Then, using the fact that $K_{\sigma}$ is a Gaussian distribution, we obtain
\begin{align*}
\hat v_t(z) 
&=
% \frac{\sum_i\sum_j \int \delta K_\sigma(z-t\delta-X_{0i}) 
% K_\sigma(z+(1-t)\delta-X_{1i}) }
% {\sum_i\sum_j 
% \int  K_\sigma(z-t\delta-X_{0i}) K_\sigma(z+(1-t)\delta-X_{1i}) }\\
% &=
\frac{\sum_i\sum_j\frac{(1-t)(z-X_{1i})-t(z-X_{0i})}{2t^2 - 2t + 1}
\exp\left\{ - \frac{z-2zt+t(X_{0i}+X_{1i})-X_{1i}}{2 h^2 (2t^2 - 2t + 1)}\right\}}
{ \sum_i\sum_j
\exp\left\{ - \frac{z-2zt+t(X_{0i}+X_{1i})-X_{1i}}{2 h^2 (2t^2 - 2t + 1)}\right\}}
\end{align*}
By standard asymptotic normality results on $U$-statistics,
we see that 
$$
\sqrt{2n}
\left(
\begin{array}{c}
\hat N - N\\
\hat D - D 
\end{array}
\right)
\overset{d}{\to} N(0,\Sigma)
$$
where
$\Sigma_{11}  = 2 \text{Cov}[h(X,Y),h(X,Y')] + 2 \text{Cov}[h(X,Y),h(X',Y)]$,
$\Sigma_{22} = 2 \text{Cov}[\tilde h(X,Y),\tilde h(X,Y')] + 
2 \text{Cov}[\tilde h(X,Y),\tilde h(X',Y)]$.
The limiting distribution of
$\hat v_t(z)$ follows by the delta method.

This approach is notable
for producing an approximate transport map
requiring no optimization.

\section{Auxiliary Results on Ordinary Differential Equations}\label{sec:review-ODE}
In the previous section, we have provided various estimators of the velocity field and discussed their relative strengths. It is not as easy to discuss the corresponding rectified flow estimators because the rectified flow is a complicated, non-linear function of the velocity field. As mentioned earlier, it is not obvious to claim the existence or uniqueness of solutions to the ODE of the type~\eqref{eq:rectified-flow-ODE}. In this section, we review some results from the literature on ordinary differential equations regarding existence, uniqueness, and stability. Many of these results are known in the ODE literature, but scattered enough that it is worth gathering them here.

Throughout this section, we deal with the following generic ODE: for a function $F:[0, 1]\times\mathcal{S}\to\mathbb{R}^d$,
\begin{equation}\label{eq:generic-ODE}
    \frac{dy(t)}{dt} = F(t, y(t)),\; t\in[0, 1]\;\mbox{with the initial condition}\; y(0) = x. 
\end{equation}
% One of the most common techniques to numerically solve ODEs of the type~\eqref{eq:generic-ODE} is the Euler discretization method, which works as follows. Fix any integer $M \ge 1$. Set $y^{(M)}(0) = x$. Define  
% \[
% y^{(M)}\left(\frac{k}{M}\right) - y^{(M)}\left(\frac{k-1}{M}\right) = \frac{1}{M} F\left(\frac{k-1}{M}, y^{(M)}\left(\frac{k-1}{M}\right)\right),\quad\mbox{for}\quad 1\le k\le M.
% \]
% Finally, for $t\in [k/M, (k+1)/M]$ for some $0 \le k\le M-1$, set
% \begin{equation}\label{eq:Euler-Disc}
% y^{(M)}(t) = y^{(M)}\left(\frac{k}{M}\right) + \left(t - \frac{k}{M}\right)F\left(\frac{k}{M},\, y^{(M)}\left(\frac{k}{M}\right)\right).
% \end{equation}
If the function $(t, y)\mapsto F(t, y)$ is continuous and bounded, then the classical Peano existence theorem~\citep[Theorem 7.6]{HairerWannerNorsett1993} implies the existence of a solution in the neighborhood of zero. The solution can then be extended to all of $[0, 1]$ using the Peano continuation theorem~\citep[Theorem 1.2.15]{KanschatScheichl2021}. Although continuity might appear a weak condition in our context, especially when $\mu_0$ and $\mu_1$ have smooth Lebesgue densities, we shall see that the velocity field $v(t, z)$ in~\eqref{eq:velocity-field} cannot be continuous in both $t, z$ when $\mu_0, \mu_1$ are compactly supported. When $\mu_0, \mu_1$ are supported on $\mathbb{R}^d$, then $v(t, z)$ can be proved to be continuous under weak regularity conditions. 

For this reason, we present the Carath{\'e}odory existence theorem, which weakens the continuity assumption. This comes with the caveat that the solution need not be differentiable everywhere. (A simple example is $dy(t)/dt = \mbox{sign}(t)$ with $y(0) = 0$. Then $y(t) = |t|$, which is not differentiable at zero.) Given the potential non-differentiability, it is easier to think of the following integral equation than the differential equation~\eqref{eq:generic-ODE}:
\begin{equation}\label{eq:generic-ODE-integral}
    y(t) = x + \int_0^t F(s, y(s))ds\quad\mbox{for all}\quad t \in [0, 1].
\end{equation}
It is easy to see that any such function $y(\cdot)$ is absolutely continuous (but need not always be differentiable). Furthermore, any solution to the ODE~\eqref{eq:generic-ODE} is also a solution to~\eqref{eq:generic-ODE-integral}, and any solution to~\eqref{eq:generic-ODE-integral} satisfies~\eqref{eq:generic-ODE} but only almost everywhere $t\in[0, 1]$.

In the context of rectified flow between compactly supported $\mu_0, \mu_1$, it is also important to show that the solutions lie in that support. For example, in the context of the generic ODE~\eqref{eq:generic-ODE}, this means that $y(0) \in \mathcal{S}$ implies $y(t)\in\mathcal{S}$ for all $t\in[0, 1]$. Viability theory~\citep{aubin2009viability} provides results of this kind. Unfortunately, the standard results (e.g., Nagumo's theorem or Theorem 2 of~\cite{hartman1972invariant}) require continuity. These continuity assumptions have been weakened in the differential inclusions literature to those of Carath{\'e}odory's existence theorem~\citep{tallos1991viability}. For an accessible presentation, we present the following result, which proves the existence of a solution and, under certain conditions, the existence in $\mathcal{S}$. We use the following notation: $\mathcal{S}^\circ$ represents the interior of $\mathcal{S}$, $\partial \mathcal{S}$ denotes the boundary of $\mathcal{S}$, and $\mathcal{B}(x, r) = \{y\in\mathbb{R}^d:\, \|x - y\| \le r\}$ with $\mathcal{B}_0 = \mathcal{B}(0, 1)$.

The Carath{\'e}odory conditions~\citep[Chapter 1]{filippov2013differential} are as follows:
\begin{enumerate}[label=(C\arabic*)]
    \item For almost all $t\in[0, 1]$, $\mathcal{S}\ni x\mapsto F(t, x)$ is well-defined and continuous.\label{eq:continuous-in-x}
    \item For each $x\in\mathcal{S}$, the function $[0, 1]\ni t\mapsto F(t, x)$ is measurable.\label{eq:meas-in-t}
    \item For all $x\in\mathcal{S}, t\in[0, 1]$, $\|F(t, x)\| \le B < \infty$.\label{eq:bounded-in-x}
    \item For almost all $t\in[0, 1]$, $F:[0, 1]\times\mathbb{R}^d\to\mathbb{R}^d$ is well-defined and for some $B < \infty$ satisfies $\|F(t, x)\| \le B(1 + \|x\|)$ for all $t\in[0, 1], x\in\mathbb{R}^d$.\label{eq:integrably-bounded}
\end{enumerate}
For any $x\in\mathcal{S}$, define the contingent cone of $\mathcal{S}$ at $x$ as
For any $x\in\mathcal{S}$, define the 
 intermediate cone of $\mathcal{S}$ at $x$ as
\begin{align*}
T_{\mathcal{S}}(x) &:= \left\{v\in\mathbb{R}^d:\, \liminf_{h\downarrow 0}\frac{\mbox{dist}(x + hv, \mathcal{S})}{h} = 0\right\}.
\end{align*}
See Definition 4.1.1 of~\cite{aubin2009set}.
\begin{enumerate}[label=(V\arabic*)]
    \item The set $\mathcal{S}$ is closed and convex, and $F(t, x)\in T_{\mathcal{S}}(x)$ for all $t\in[0, 1]$ and $x\in\mathcal{S}$.\label{eq:tangent-cone-viability}
\end{enumerate}
Examples of tangent cones can be found in Chapter 4 of~\cite{aubin2009set} and also Section~\ref{sec:bounded}. A simple result worth recalling is that $T_{\mathcal{S}}(x) = \mathbb{R}^d$ whenever $x\in\mathcal{S}^\circ$ (i.e., $x$ is in the interior of $\mathcal{S}$). 

The following result only assumes $F(\cdot, \cdot)$ is defined on $[0, 1]\times\mathcal{S}$, except for part 2.
\begin{theorem}[Existence of Solutions (in $\mathcal{S}$)]\label{thm:Peano-existence} 
    Consider the integral equation~\eqref{eq:generic-ODE-integral}.
    \begin{enumerate}
        \item Suppose assumptions~\ref{eq:continuous-in-x}--\ref{eq:bounded-in-x} hold with $x\in\mathcal{S}^\circ$. Then there exists $T\in(0, 1]$ and an absolutely continuous function $y^*:[0, T]\to\mathcal{S}$ satisfying~\eqref{eq:generic-ODE-integral}. Here, $T$ can be chosen to be $\min\{1, \mathrm{dist}(x,\partial\mathcal{S})/B\}$.
        \item If, instead of~\ref{eq:bounded-in-x}, assumption~\ref{eq:integrably-bounded} holds (and~\ref{eq:continuous-in-x} and~\ref{eq:meas-in-t} hold with $\mathcal{S} = \mathbb{R}^d$), then there exists an absolutely continuous function $y^*:[0,1]\to\mathbb{R}^d$ satisfying~\eqref{eq:generic-ODE-integral} and moreover, all solutions of~\eqref{eq:generic-ODE-integral} satisfy $\|y^*(t)\| \le (1 + \|x\|)e^{Bt} - 1$ for all $t\in[0, 1]$. 
        \item If assumptions~\ref{eq:continuous-in-x},~\ref{eq:meas-in-t},~\ref{eq:bounded-in-x}, and~\ref{eq:tangent-cone-viability} hold and $x\in\mathcal{S}$, then there exists an absolutely continuous function $y^*:[0,1]\to\mathcal{S}$ satisfying~\eqref{eq:generic-ODE-integral}. 
    \end{enumerate}
\end{theorem}
A proof of Theorem~\ref{thm:Peano-existence} can be found in Section~\ref{appsubsec:proof-Peano-existence}.

The following result proves the uniqueness of the solution. Consider the following smoothness condition on $F(\cdot, \cdot)$.
\begin{enumerate}[label=(W\arabic*)]
    \item There exist measurable functions $a:[0, 1]\to\mathbb{R}_+\cup\{\infty\}$ and $\kappa:\mathbb{R}_+\to\mathbb{R}_+$ such that $\lim_{u\to0}\kappa(u) = 0$ and for any $t\in[0, 1]$\label{eq:Lipschitz-Osgood}
    \begin{equation}\label{eq:velocity-Lipschitz-in-z}
    \|F(t, y) - F(t, y')\| \le a(t)\kappa(\|y - y'\|)\quad\mbox{for all}\quad y, y'\in\mathcal{S},
    \end{equation}
    and for every $\delta > 0$,
    \[
    \lim_{\gamma\to0}\Psi^{-1}\left(\Psi\left(\int_0^{\gamma} a(s)\kappa(\delta s)ds\right) + \int_{\gamma}^t a(s)ds\right) = 0,\quad\mbox{where}\quad\Psi(u) = \int \frac{du}{\kappa(u)}. 
    \]
    (Here $\Psi(\cdot)$ is the indefinite integral.)
\end{enumerate}
Note that assumption~\ref{eq:Lipschitz-Osgood} implies assumption~\ref{eq:continuous-in-x}. This can be seen as a special case of the assumption of Theorem 3.1 of~\cite{liu2024uniqueness}, which itself is a generalization of Theorem 2.1 of~\cite{constantin2023uniqueness}. For other uniqueness conditions for ODEs, see~\cite{Bernfeld1975Uniqueness},~\cite{banas1981relations}, and Chapter 1 of~\cite{agarwal1993uniqueness}. Finally, we note that~\eqref{eq:velocity-Lipschitz-in-z} is stronger than the one required for our results. In particular, inequality~\eqref{eq:velocity-Lipschitz-in-z} can be replaced with
\[
\|F(t, y_1(t)) - F(t, y_2(t))\| \le a(t)\kappa(\|y_1(t) - y_2(t)\|)\quad\mbox{for any two solutions}\; y_j(\cdot)\;\mbox{of~\eqref{eq:generic-ODE-integral},} 
\]
for the proof of Theorem~\ref{thm:uniqueness-viability}. Similarly, inequality~\eqref{eq:velocity-Lipschitz-in-z} can be replaced with
\[
\|F(t, y(t)) - F(t, w(t))\| \le a(t)\kappa(\|y(t) - w(t)\|)\quad\mbox{for any two solutions}\; y(\cdot), w(\cdot)\;\mbox{of~\eqref{eq:generic-ODE-integral},~\eqref{eq:generic-ODE2-integral}}, 
\]
for the proof of Theorem~\ref{thm:stability}. While stating the assumptions in terms of the solutions we want to study might seem circular, these relaxed versions are helpful when we can prove apriori that any solution at time $t$ belongs to a much smaller set than $\mathcal{S}$ itself, which in turn helps in reducing the constant factor $a(t)$; see, for example, Lemma~\ref{lem:distance-to-boundary-latest}.

We provide two specific instances $a(\cdot)$ and $\kappa(\cdot)$ that satisfy assumption~\ref{eq:Lipschitz-Osgood}.
\paragraph{Example 1. (Osgood functions)}\label{exam:Osgood-function} Consider $a(s) = L$ and any (Osgood) function $\kappa(\cdot)$ such that $\kappa(0) = 0$ and $\lim_{\varepsilon\to0}\,\Psi(\varepsilon) = -\infty.$ Some examples of such functions are $u$, $u\log(1/u)$, $u\log(1/u)\log(\log(1/u))$ (for $u<1$). (Note that functions of the type $\kappa(u) = u^{\alpha}$ or $u(\log(1/u))^{\alpha}$ for $\alpha > 1$ do {\em not} satisfy this assumption.)
Then
\[
\int_0^{\gamma} a(s)\kappa(\delta s)ds = \frac{L}{\delta}\int_0^{\delta\gamma} \kappa(s)ds,\quad \int_{\gamma}^t a(s)ds = L(t - \gamma).
\]
As $\gamma \to 0$, the first integral converges to zero and hence,
\[
\Psi\left(\int_0^{\gamma} a(s)\kappa(\delta s)ds\right) \to -\infty\quad\mbox{as}\quad \gamma\to 0.
\]
This implies that
\[
\Psi\left(\int_0^{\gamma} a(s)\kappa(\delta s)ds\right) + \int_{\gamma}^t a(s)ds \to -\infty\quad\mbox{as}\quad \gamma\to 0.
\]
This yields assumption~\ref{eq:Lipschitz-Osgood}. This example is the Osgood uniqueness theorem.
\paragraph{Example 2. (Nagumo functions)} Consider $a(s) = c/s$ and $\kappa(s) = s$. Then
\[
\int_0^{\gamma} a(s)\kappa(\delta s)ds = c\delta\gamma, \quad\int_{\gamma}^t a(s)ds = c\log(t/\gamma),\quad\mbox{and}\quad \Psi(u) = \log(u).
\]
This implies
\[
\Psi^{-1}\left(\Psi\left(\int_0^{\gamma} a(s)\kappa(\delta s)ds\right) + \int_{\gamma}^t a(s)ds\right) = \exp\left(\left[\log\left({c\delta\gamma}\right) + c\log(t/\gamma)\right]\right) = t^c\gamma^{1-c}\delta,
\]
which converges to zero as $\gamma\to 0,$ for any $c < 1$. This almost recovers the Nagumo uniqueness theorem, which allows for $c = 1$ but requires continuity of $t\mapsto F(t, x)$.
\begin{theorem}[Uniqueness of Solutions]\label{thm:uniqueness-viability} Consider the integral equation~\eqref{eq:generic-ODE-integral}.
\begin{enumerate}
    \item Suppose assumptions~\ref{eq:meas-in-t},~\ref{eq:bounded-in-x} and~\ref{eq:Lipschitz-Osgood} hold. Fix $x\in\mathcal{S}^\circ$. Then there exists $T\in(0, 1]$ such that there is a unique solution $y^*:[0, T]\to\mathcal{S}$ that satisfies~\eqref{eq:generic-ODE-integral} for $t\in[0, T]$. Here $T$ can be chosen to be $\min\{1, \mathrm{dist}(x, \partial\mathcal{S})/B\}$.
    \item If, instead,~\ref{eq:meas-in-t},~\ref{eq:integrably-bounded}, and~\ref{eq:Lipschitz-Osgood} hold, then for any $x\in\mathbb{R}^d$, there exists a unique solution $y^*:[0, 1]\to \mathcal{B}(x, Be^B(1 + \|x\|))$ that satisfies~\eqref{eq:generic-ODE-integral} for $t\in[0, 1]$.
    \item If assumptions~\ref{eq:meas-in-t},~\ref{eq:bounded-in-x},~\ref{eq:tangent-cone-viability} and~\ref{eq:Lipschitz-Osgood} hold and $x\in\mathcal{S}$, then there exists a unique solution $y^*:[0, 1]\to\mathcal{S}$ that satisfies~\eqref{eq:generic-ODE-integral} for $t\in[0,1]$.
\end{enumerate}
\end{theorem}
A proof of Theorem~\ref{thm:uniqueness-viability} can be found in Section~\ref{appsubsec:proof-uniqueness-viability}.

The following result provides a stability bound, i.e., a bound on the difference of solutions when the $F$ changes. Here again, traditional results such as those in~\cite{brauer1966perturbations} and Proposition 20.1 of~\cite{Soderlind2024LogarithmicNorms} assume continuity with respect to $t$.

For any $a, b > 0$, set $a\vee b = \max\{a, b\}$.
\begin{theorem}[Stability of the Solution]\label{thm:stability}
    For a function $G:[0,1]\times\mathcal{S}\to\mathbb{R}^d$ and $x'\in\mathcal{S}$, consider the integral equation
    \begin{equation}\label{eq:generic-ODE2-integral}
        w(t) = x' + \int_0^t G(s, w(s))ds,\quad\mbox{for}\quad t\in[0, 1].
    \end{equation}
    Suppose $F(\cdot, \cdot)$ and $G(\cdot, \cdot)$ satisfy assumptions~\ref{eq:continuous-in-x},~\ref{eq:meas-in-t}, and~\ref{eq:bounded-in-x}. Suppose, additionally, $F(\cdot, \cdot)$ satisfies~\ref{eq:Lipschitz-Osgood} with a non-decreasing function $\kappa(\cdot)$.
    For any $\delta \geq 0$, set
    \begin{equation}\label{eq:error-estimate-of-F-G}
    \mathcal{E}_{\delta}(t) ~:=~ 2B\delta + \|x' - x\| + \int_{\delta}^{t\vee\delta} \|F(s, w(s)) - G(s, w(s))\|ds.
    \end{equation}
    Let 
    \[
    T = \sup\{\tau\in[0, 1]:\, w(s), y(s)\in\mathcal{S}\mbox{ for all }s\in[0, \tau]\}.
    \]
    If $x, x'\in\mathcal{S}^\circ$, then the (unique) solution of~\eqref{eq:generic-ODE-integral} and $w:[0,T]\to\mathcal{S}$, any solution of~\eqref{eq:generic-ODE2-integral},
    \begin{equation}\label{eq:perturbation-bound}
    \|y(t) - w(t)\| \le \Psi^{-1}\left(\Psi\left(\mathcal{E}_{\delta}(t)\right) + \int_{\delta}^{t\vee\delta} a(s)ds\right),\quad\mbox{for all}\quad t,\delta\in[0, T].
    \end{equation}
    
    Moreover, if $F(\cdot,\cdot)$ and $G(\cdot, \cdot)$ satisfy~\ref{eq:tangent-cone-viability}, then, for any solution $w(\cdot)$ of~\eqref{eq:generic-ODE2-integral} such that $w(s)\in\mathcal{S}$ for all $s\in[0, 1]$, inequality~\eqref{eq:perturbation-bound} holds for all $t\in[0, 1]$.
\end{theorem}
A proof of Theorem~\ref{thm:stability} can be found in Section~\ref{appsubsec:proof-of-stability}. Note that $\mathcal{E}_{\delta}(t)$ in~\eqref{eq:error-estimate-of-F-G} is well-defined only for $t$ such that $w(s)\in\mathcal{S}$ for all $s\in[0, t]$, when the functions $F, G$ are defined only on $[0, 1]\times\mathcal{S}$. If, instead of~\ref{eq:bounded-in-x}, one assumes~\ref{eq:integrably-bounded} for $F(\cdot, \cdot), G(\cdot, \cdot)$, then inequality~\eqref{eq:perturbation-bound} continues to hold with $B$ in the definition~\eqref{eq:error-estimate-of-F-G} of $\mathcal{E}_{\delta}(t)$ replaced with $(1 + \|x\|\vee\|x'\|)Be^B$.

To illustrate Theorem~\ref{thm:stability}, consider $a(s) = c/s$ and $\kappa(s) = s$. We get $\Psi(u) = \log(u)$ and $\int_{\delta}^t a(s)ds = c\log(t/\delta)$. Therefore,~\eqref{eq:perturbation-bound} becomes
\[
\|y(t) - w(t)\| \le \exp\left(\log(\mathcal{E}_{\delta}(t)) + c\log(t/\delta)\right) = t^c\delta^{-c}\mathcal{E}_{\delta}(t).
\]
Because $\mathcal{E}_{\delta}(t) \ge 2B\delta$, the upper bound cannot converge to zero if $c \ge 1$ as $\delta\to0$. If $c < 1$, then assuming $\sup_{x\in\mathcal{S}}\|F(s, x) - G(s, x)\| \to 0$ and $\|x - x'\|\to0$, one can choose $\delta\to0$ so that the upper bound converges to zero. 

Finally, the following result combines techniques from stability to prove existence and unique of solutions for perturbed ODEs. This is useful for proving the existence of the empirical rectified flow.
\begin{theorem}[Existence under small uniform perturbations of the vector field]\label{prop:extension}
Consider the integral equation~\eqref{eq:generic-ODE-integral}. Suppose assumptions~\ref{eq:meas-in-t} and~\ref{eq:Lipschitz-Osgood} hold with $\kappa(u) = u$ and $a(t) = L$. Set $R := e^L(\int_0^1 \|F(s, x)\|ds) + 10.$
Let $G : [0,1]\times\mathbb{R}^d \to \mathbb{R}^d$ be a  continuous and
locally Lipschitz in the second argument, i.e., each $y\in\mathbb{R}^d$ there exists constants $L_y, \varepsilon_y > 0$ such that $\sup_{t\in[0,1]}\|G(t, y_1) - G(t, y_2)\| \le L_y\|y_1 - y_2\|$ for all $y_1, y_2\in\mathcal{B}(y, \varepsilon_y)$. Set
\[
    \Delta := \sup_{(t,y)\in [0,1]\times \mathcal{B}(x, R)} \|F(t,y) - G(t,y)\|.
\]
If ${\Delta}(e^{L}-1)/{L} \le 9,$
then the integral equation~\eqref{eq:generic-ODE2-integral} (with $x' = x$) admits a unique solution $w : [0,1] \to \mathbb{R}^d$. Moreover, the unique solutions $y(\cdot)$ of~\eqref{eq:generic-ODE-integral} and $w(\cdot)$ of~\eqref{eq:generic-ODE2-integral} satisfy
\[
    \|y(t)-w(t)\| \;\le\; \frac{\Delta}{L}\,\bigl(e^{Lt}-1\bigr),
    \qquad t\in[0,1].
\]
In particular, $w(t)\in\mathcal{B}(x, R)$ for $t\leq 1$.
\end{theorem}
A proof can be found in Section~\ref{appsubsec:proof-of-extension}. 

\paragraph{Application to the case of uniform distribution}
Suppose $\mu_0 = \mu_1 = \mbox{Unif}[0,1]$. From~\eqref{eq:uniform-velocity-field}, we get for $t\in[0, 1/2]$,
\[
v(t, z) = \frac{1}{2t(1-t)}\times\begin{cases}z(1-2t), &\mbox{if }z \le t,\\
(1-2z)t, &\mbox{if }t \le z \le 1 - t,\\
(1-z)(2t-1), &\mbox{if }z \ge 1 - t.\end{cases}
\]
The expression is analogous for $t\in[1/2, 1]$.
It is easy to see that $|v(t, z)| \le 1$ for all $t\in[0, 1], z\in[0, 1]$. This can also be seen from the fact that $v(t, z) = \mathbb{E}[X_1 - X_0|X_t = z]$ which can at most be 1 in absolute value because $|X_1 - X_0| \le 1$. Moreover, inequality~\eqref{eq:uniform-Lipschitz-cond} verifies assumption~\ref{eq:Lipschitz-Osgood} with $a(s) = 1/(2s(1-s))$ and $\kappa(u) = u$ so that there exists a unique solution $z^*:[0,1]\to[0, 1]$ satisfying~\eqref{eq:rectified-flow-ODE} almost everywhere.
% Given this uniqueness, the following result provides the solution.
% {\color{red}needs to be written.}
\section{Existence, Regularity, and Estimation in the Unbounded Case}

\label{sec:unbounded}

Here, we address the existence and regularity of the population rectified flow, as well as statistical rates for the regression-based estimator, when $\Omega=\mathbb{R}^d$. To do so, we need to assume that the underlying densities 
have full support, have sufficient H{\"o}lder regularity, and satisfy a strongly log-concave bound. Specifically, we assume that
\begin{enumerate}[label=(U\arabic*),ref=U\arabic*,series=myU]
\item $p_0(x),p_1(x)>0$, for all $x \in \mathbb{R}^d$.\label{assump:unbnd-densitiy-positive}
\item $\phi_0(x) = -\log p_0(x), \phi_1(x) = -\log p_1(x)$, are twice differentiable and their Hessians $x\mapsto \nabla^2\phi_j(x),$ $j=0,1$ are $(\beta-2)$-H{\"o}lder smooth, i.e., for $j = 0,1$,\label{assump:unbnd-log-Holder}
\begin{align*}
\sup_{x\in\mathbb{R}^d,\,\|k\|_1 = m}\,\left|\frac{\partial^{m}}{\partial x_1^{k_1}\partial x_2^{k_2}\cdots \partial x_d^{k_d}}\phi_j(x)\right| \le C_m < \infty,\quad 2 \le m \le \lceil\beta\rceil - 1,\;\mbox{and}\\
\max_{\|k\|_1 = \lceil\beta\rceil - 1}\sup_{x\neq y}\left|\frac{\partial^{\|k\|_1}}{\partial x_1^{k_1}\partial x_2^{k_2}\cdots \partial x_d^{k_d}}\phi_j(x) - \frac{\partial^{\|k\|_1}}{\partial y_1^{k_1}\partial y_2^{k_2}\cdots \partial y_d^{k_d}}\phi_j(y)\right| \le C_{\beta}\|x - y\|^{\beta - \lceil\beta\rceil + 1}.
\end{align*}
Here $k = (k_1, \ldots, k_d)$ is a vector of non-negative integers and $\|k\|_1 = \sum_{j=1}^d k_j$.
\item $p_0,p_1$ are strongly log-concave with parameter $\alpha$, i.e., for $j=0,1$ and $x\in\mathbb{R}^d$,\label{assump:strong-log-concave}
\[
\phi_j(x) \ge \phi_j(x') + \nabla\phi_j(x')^{\top}(x' - x) + \frac{\alpha}{2}\|x - x'\|^2\quad\mbox{for all}\quad x, x'\in\mathbb{R}^d, j = 0, 1.
\]
\end{enumerate}

In Section \ref{sub:reguexistanceunbounded} we establish several regularity properties of the velocity and related quantities, from which we will conclude in Theorem \ref{teo:regularity} the existence and regularity of the population rectified transport map. This result bears some resemblance with Theorem 5.1 in \cite{gao2024gaussian}, that also establishes existence of the flow under strong log-concavity. However, their setup is substantially different so that none of these results can be inferred from the other. In Section \ref{subsection:estimationrect} we establish asymptotic bias, variance, and a central limit theorem for the regression-based estimator of $R(x)$. The main result of this section is Theorem \ref{teo:CLT}, which establishes a central limit theorem at a faster rate than the one for usual kernel regression estimators. Throughout this section, we use the notation $\Delta = X_1 - X_0$.

\subsection{Regularity and existence}\label{sub:reguexistanceunbounded}
One question that arises with the previous assumptions is whether they will imply the usual notion of $\beta$--H{\"o}lder smoothness in $p_0,p_1$ and whether this will translate into smoothness of the density of $X_t=tX_1+(1-t)X_0$. These two questions have positive answers: $\beta$-H{\"o}lder smoothness of $p_0,p_1$ is addressed in Lemma \ref{lemma:holderpj}. Additionally, $\beta$--H{\"o}lder smoothness of $p_t$ (with bounds that are independent of $t$), and other related quantities are stated in Lemma \ref{lemma:holdersmooth}. These are proved in Sections \ref{sub:holderpap} and in Section \ref{sub:proofholdersmooth}, respectively.
\begin{lemma}\label{lemma:holderpj}

For any vector $k = (k_1, \ldots, k_d)$ of non-negative integers such that $\|k\|_1 \le \lceil\beta\rceil - 1$, there exist constants $c_{0,k},\ldots, c_{\beta,k}$, such that for all $x\in\mathbb{R}^d$ and $j\in\{0,1\}$
\begin{equation}\label{eq:bound-on-derivatives}
\left|\frac{\partial^{\|k\|_1}p_j(x)}{\partial x_1^{k_1}\partial x_2^{k_2}\cdots\partial x_d^{k_d}}\right| \leq p_j(x)\left(\sum_{l=0}^{\max\{2,\lceil\beta\rceil -\|k\|_1 - 1\}} c_{l,k}\|x\|^l + c_{\beta,k}\|x\|^{\beta - \|k\|_1}\right).
\end{equation}
In particular, if $p_j$ are log-concave then $p_j$ is $\beta$-H{\"o}lder continuous.
\end{lemma}

\begin{lemma}\label{lemma:holdersmooth}
For any $k>0$ define the function $$f_k(t,z)=\int\delta^kp_0(z-t\delta)p_0(z+(1-t)\delta)d\delta=\mathbb{E}\left(\Delta|X_s=z\right)p_t(z).$$ 

Under assumptions \eqref{assump:unbnd-log-Holder} and \eqref{assump:strong-log-concave}, 
for each $t$, each component of the map $z\to f_k(t,z)$ is $\beta$-H{\"o}lder continuous with a constant independent on $t$. In particular, $z\to p_t(z)$ and $z\to v(t,z)p_t(z)$ are $\beta$-H{\"o}lder functions.
\end{lemma}

These bounds will be helpful for the analysis of estimators in Section \ref{subsection:estimationrect}. We also show in Proposition \ref{prop:derzbound} that the first and second derivatives of the velocity are bounded. In turn, this implies the main result of our section, Theorem \ref{teo:regularity}, establishing that the population rectified transport map is well-defined and preserves marginals.

\begin{proposition}
    
\label{prop:derzbound}
Under assumptions \eqref{assump:unbnd-log-Holder} and \eqref{assump:strong-log-concave}, $\nabla v(t,z)$ is uniformly bounded over $t$ and $z$. Also, for each $i$, $\nabla^2 v_i(t,z)$ is uniformly bounded on $t$ and $z$ as well.
\end{proposition}
\begin{theorem}\label{teo:regularity}
Under assumptions \eqref{assump:unbnd-log-Holder} and \eqref{assump:strong-log-concave}, the rectified map $R(x)$ is uniquely defined as the solution to \eqref{eq:rectified-flow}, and is twice continuously differentiable. Moreover, it satisfies the marginal preserving property: if $X_0\sim \mu_0$, for each $t\in[0,1]$, $\mathfrak{R}(t, X_0)$ has the same law as that of $X_t=(1-t)X_0+tX_1$.
\end{theorem}
\begin{proof}
Since $v(t,z)$ has bounded derivatives (Proposition \ref{prop:derzbound}), by the classical Cauchy-Lipschitz theory of ordinary differential equations \cite{hartman2002ordinary}, the ODE defining $R(x)$ \eqref{eq:rectified-flow-ODE} has unique solutions for each $x\in\mathbb{R}^d$. Moreover, the map $x\to R(x)$ inherits the regularity of the velocity, which in our case is guaranteed to be twice continuously differentiable, by Lemma \ref{lemma:holdersmooth}. The marginal preserving property is a direct consequence of Theorem \ref{thm:uniquene-ae-implies-transport-map}.
\end{proof}
\subsection{Estimation}
\label{subsection:estimationrect}

We now turn to estimating the rectified map
using kernel regression.
 In this section, we will assume that $\hat{v}(t,z)=\hat{v}_{h_n}(t,z)$ is the regression-based estimator \eqref{eq:kerreg}. Based on this estimators, we consider the empirical rectified flow $\hat{R}(x)=\hat{\mathfrak{R}}(1,x)$ where $\hat{\mathfrak{R}}(t,x)$ is the solution to ODE \eqref{eq:rectified-flow-ODE} with velocity $\hat{v}(t,z)$ for a fixed $x\in\mathbb{R}^d$. We will make the following additional assumptions
\begin{enumerate}[label=(K\arabic*),ref=K\arabic*]
\item \label{assump:K} The kernel $K$ is of order $\lfloor \beta\rfloor$, has bounded support and has continuous and bounded first and second derivatives. We further assume that $K$ is regular in the sense of \cite{einmahl2005uniform}.
\end{enumerate}
\begin{enumerate}[label=(U\arabic*),ref=U\arabic*,resume*=myU]
%\item\label{assump:boundprob} We assume that there is a solution $\mathfrak{\hat{R}}(t,x)$ to the empirical ODE \eqref{eq:rectified-flow-ODE} with velocity $\hat{v}$. Moreover, we assume that this solution is bounded in probability, i.e.  
%$$\sup_{0\leq t\leq 1}||\mathfrak{\hat{R}}(t,x)||= O_p(1).$$
\item \label{assump:morse} For a starting point $x$ define the following functions for $1\leq i\leq d$ (the discrete time derivatives of the coordinates of the velocity vector) 
$$F^i_t(y):=\frac{v_i(t+y,z_{t+y})-v_i(t,z_t)}{y},$$
where the trajectories $z_t$ start at $x$. These functions must satisfy
\begin{enumerate}
\item Are twice differentiable, and their second derivatives are equicontinuous in $t$.
\item For each $t$, $F_t^i$ has at most $M$ critical points, where $M$ is uniform in $d$ and $t\in[0,1]$.
\item The set of critical points is uniformly non-degenerate; i.e., the second derivatives of $F^i_t$ are bounded away from zero at those points.
\end{enumerate}
\end{enumerate}

%Note that the kernel regression estimator satisfies $\lVert\hat{v}(z,t)\rVert \leq\max_{i=1,\ldots n} ||\Delta_i||$ which by sub-Gaussianity (recall that $X,Y$ are independent and their densities $p_0,p_1$ can be bounded by a Gaussian) implies that $\hat{\mathfrak{R}}(t,x)=O_p(\sqrt{2\alpha\log n})$. In turn, this implies that $\hat{z}(t,0,x)=O_p(\sqrt{2\alpha\log n})$, uniformly over $t\leq 1$. Therefore, \eqref{assump:boundprob} is a strengthening of this fact.

Under these assumptions, we show that $\hat{\mathfrak{R}}(t,x)$  is well-defined on an event with a probability converging to one. Before stating our main distributional result, let us first state pointwise consistency: that $\hat{\mathfrak{R}}(t,x)$ converges to $\mathfrak{R}(t,x)$ at least as fast as the uniform consistency rate of $\hat{v}(t,z)$ towards $v(t,z)$. Under assumptions
\begin{proposition}\label{prop:consistR}
Under assumptions \eqref{assump:K}, \eqref{assump:unbnd-densitiy-positive}, \eqref{assump:unbnd-log-Holder} and \eqref{assump:strong-log-concave}, for every compact set $C\subseteq\mathbb{R}^d$, there exists a compact set $B\subseteq\mathbb{R}^d$ such that
$$\sup_{x\in C}\|\mathfrak{R}(t,x)-\hat{\mathfrak{R}}(t,x)\|=O_p\left(\sup_{0\leq s\leq 1,\,z\in B} \lVert \hat{v}(s,z)-v(s,z)\rVert\right)=O_p\left(\sqrt{\frac{\log(1/h_n)}{nh_n^d}}+h_n^\beta\right).$$
% where the last equation holds under the leftmost inequality in \eqref{eq:bandwidth}.

\end{proposition}

We can finally state the main result of this section, a pointwise CLT for $\hat{R}(x)$ around $R(x)$, assuming that \eqref{assump:morse} holds at $x$. 

\begin{theorem}\label{teo:CLT}
Suppose that assumptions \eqref{assump:K}, \eqref{assump:unbnd-densitiy-positive}, \eqref{assump:unbnd-log-Holder}, \eqref{assump:strong-log-concave} and \eqref{assump:morse} hold. Let $\varepsilon>0$ be any arbitrary positive number, and $d>1$. Suppose that $h=h_n$ is such that (the rightmost inequality is an undersmoothing condition)
 \begin{equation}\label{eq:bandwidth} n^{-\frac{1}{d+2+\varepsilon}}\ll h_n\ll n^{-\frac{1}{d-1+2\beta}}.\end{equation}
 where $n^{-a}\ll \alpha_n\ll n^{-b}$ if there are $c>0, b<c'<a$, such that $\alpha_n=cn^{-c'}$. Then,
 %\sqrt{nh_n^{d-1}}
\begin{equation}\label{eq:clt} \sqrt{nh_n^{d-1}}\left(\hat{R}(x)-R(x)\right)~\overset{d}{\to}~N\left(0,\Sigma(x)\right),\end{equation}
where
 \begin{align}\nonumber 
\Sigma(x) = \int_0^1  \frac{\tilde{\Phi}(t)}{p_t(z_t)}  \mathbb{E}\left[ \left(\Delta-v_{t}(z_{t})\right)\left(\Delta-v_t(z_t)\right)^\top \iint\limits_{\mathbb{R}^d\times\mathbb{R}} K\left(u+\omega(\Delta -v_t(z_t))\right) K\left(u\right)dud\omega   \Big |X_t=z_t\right] \tilde{\Phi}(t)^\top dt,\\ \label{eq:sigma}&&
\end{align}
$\tilde{\Phi}(t)=\Phi(1)\Phi(t)^{-1}$ and $\Phi(t)$ is the fundamental 
matrix associated to the ODE \eqref{eq:rectified-flow-ODE}, i.e., it is the unique invertible solution to the matrix differential equation
    \begin{equation}\label{eq:fundamental-derivative-velocity}
    \Phi(t) = I + \int_0^t \partial_zv(s, \mathfrak{R}(s, x))^\top\Phi(s)ds.
    \end{equation}
    
In the case of a symmetric kernel, this simplifies to
     \begin{eqnarray*}
\Sigma(x) = \int_0^1\frac{\tilde{\Phi}(t)}{p_t(z_t)}
\mathbb{E}\left[\left(\Delta-v_{t}(z_{t})\right)\left(\Delta-v_t(z_t)\right)^\top   \int_{-\infty}^\infty (K\ast K)\left(\omega(v_t(z_t)-\Delta)\right)d\omega  \Big |X_t=z_t\right]  \tilde{\Phi}(t)^\top dt
\end{eqnarray*}
 \end{theorem}

Under the condition $n^{-\frac{1}{d+2+\varepsilon}}\ll h_n$ only (i.e., no undersmoothing), we can derive a central limit theorem with an asymptotic bias of order $h_n^\beta$. In this case, since the asymptotic variance is of order $1/(nh_n^{d-1})$, the optimal bandwidth is $h_n=n^{-\frac{1}{d-1+2\beta}}$, and the corresponding asymptotic MSE is of order $n^{-\frac{2\beta}{2\beta+d-1}}$, faster than the usual nonparametric regression asymptotic MSE rate $n^{-2\beta/(2\beta+d)}$. We note that, based on classical results on nonparametric regression (Lemma \ref{lemma:nonreg} in Section \ref{sub:addunbounded}), we would be able to estimate the velocity at each time at this rate. While our rectified estimator enjoys a faster rate, it is still slower than the one established in \cite{manole2023clt} for the optimal transport estimator. They demonstrate that both bias and variance decay faster than the usual kernel estimators, at rates $h_n^{\beta+1}$ and $1/(nh^{d-2})$, respectively \cite{manole2023clt}. In this case, the optimal MSE rate $n^{\frac{-2(\beta+1)}{2\beta+d}}$ \cite{hutter2019minimax,manole2021}. Therefore, our rates interpolate between the classical kernel estimation rates and the ones for optimal transport. Although it remains an open question, we don't expect better rates in our setup: the analysis of \cite{manole2023clt} heavily exploits the fact that the optimal transport map has $\beta+1$-H{\"o}lder regularity \citep{caffarelli1}, a statement that may not hold for the rectified transport map.

 While the technical condition \eqref{assump:morse} is generally hard to verify, it is easy to demonstrate that it holds for each $x$ in the case of arbitrary rectified transport between independent Gaussians. In this case, the analysis reduces to individual components, and for each of these components, the set of critical points in \eqref{assump:morse} is characterized by the solutions to a polynomial equation; the non-degeneracy condition clearly holds.
If we are not able to verify \eqref{assump:morse}, we can still state a CLT
$$  \sqrt{nh_n^{d-1}}\Sigma_h^{-1/2}(x)\left(\hat{R}(x)-R(x)\right) ~\overset{d}{\to}~N\left(0,I_d\right).$$
for some matrix $\Sigma_h$ that appears in the proof of Theorem \ref{teo:CLTlinear}.

 \begin{corollary}\label{cor:1d}
 In the one-dimensional case, $R(x)$ coincides with the optimal transport. The bias rate can be improved to order $h_n^{\beta+1}$ if we choose a kernel of order $\lfloor \beta+1\rfloor $. Additionally, we have 
$$  \sqrt{n}\left(\hat{R}(x)-R(x)\right)~\overset{d}{\to}~ N\left(0,\Sigma(x)\right), $$
whenever  $n^{-\frac{1}{3+\varepsilon}}\ll h_n\ll n^{-\frac{1}{2\beta+2}}$.
 The asymptotic variance in this case doesn't depend on the kernel and has a more explicit formula:
 \begin{equation}\label{eq:sigmax1d} \Sigma(x) = \int_0^1\frac{1}{p_t(z_t)}\tilde{\Phi}(t)^2 
\mathbb{E}\left[\lvert X_1-X_0-v_{t}(z_{t})\rvert  \Big |X_t=z_t\right]dt.\end{equation}
 \end{corollary}
 
The asymptotic behavior of our estimator in this case resembles that of optimal transport, where kernel estimators in one dimension achieve the parametric rate \citep{ponnoprat2024uniform}. The gains in bias rate don't come as a surprise because of the extra regularity of $R(x)$ in this case \citep{manole2023clt}. In the case where $X_0,X_1$ are Gaussian, we can even derive an explicit formula for the asymptotic variance.

\begin{example}\label{example:1dgaussian}
One-dimensional Gaussians. If $X_0\sim N\left(m_0,\sigma_0^2\right),X_1\sim N\left(m_1,\sigma_1^2\right)$ then
$$\Sigma(x)= 2\frac{\sigma_1}{\sigma_0}\left(\arctan\left(\frac{\sigma_1}{\sigma_0}\right)+\arctan\left(\frac{\sigma_0}{\sigma_1}\right)\right)\exp\left(\frac{1}{2\sigma^2_0}(x-m_0)^2\right).$$
\end{example}

We conclude this section with an abbreviated version of the proof of Theorem \ref{teo:CLT}.
\subsubsection*{Proof sketch of Proposition \ref{prop:consistR} and Theorem \ref{teo:CLT}}

 We write
 \begin{eqnarray} \hat{R}(x)-R(x)&=&
 \int_0^1 \tilde{\Phi}(s) \left(\hat{v}_s(z_s)-v_s(z_s)\right)ds+S\\
&=&\label{eq:linearization} \int_0^1 \tilde{\Phi}(s) \left(\frac{\hat{f}_s(z_s)-v_s(z_s)\hat{p}_s(z_s)}{p_s(z_s)}\right)ds+\tilde{S},\end{eqnarray}

 where $S,\tilde{S}$ are remainder terms. In the first line we exploit a variation of parameters-type bound. The second line is the usual linearization of ratio estimators such as the ones for Kernel regression. The main benefit of the above expression is that the integral term in \eqref{eq:linearization} is much simpler to analyze, as it reduces to a linear combination of kernel-like quantities evaluated over the deterministic trajectory $\mathfrak{R}(s,x)$. The bulk of the proof consists of first showing that this linearized component enjoys the asymptotic bias/variance bounds and CLT in Theorem \ref{teo:CLT} and subsequent remark. This is done in Proposition \ref{prop:linearized} and Theorem \ref{teo:CLTlinear}. Then, we show that the remainder $\tilde{S}$ decays faster than this rate. We defer all the proofs and details to Section \ref{sec:unboundedproof}.

\section{Existence and Regularity in the Bounded Case}
\label{sec:bounded}
When $\Omega$ is a proper convex subset of $\mathbb{R}^d$, it is non-trivial to show that the velocity field~\eqref{eq:velocity-field} is well-defined on $\Omega$, especially for $z\in\partial\Omega$. This is because the product of densities $p_0(z - t\delta)p_1(z + (1-t)\delta)$ can be strictly positive only if $z - t\delta\in\Omega$ and $z + (1-t)\delta\in\Omega$. When $z\in\partial\Omega$, the set of all $\delta$ satisfying such constraints can be a singleton (containing zero) or a face of the convex set $\Omega.$ If it is a singleton, then the velocity field becomes ill-defined, and one has to consider derivatives to show that a unique continuous extension exists. This can be successfully resolved for any strictly convex set $\Omega$ as shown in Section~\ref{sec:strictly-convex-set-definition}. When $\Omega$ has non-trivial faces (line segments on the boundary), then a generic proof of unique continuous extension to the boundary seems out of reach. The remaining section is organized as follows. In Section~\ref{subsec:exist-uni-rates-bounded}, we prove existence and uniqueness of solutions to~\eqref{eq:rectified-flow-ODE} under relatively mild conditions on the densities. Under the same conditions, we also prove that rectified flow map starting in the interior stays in the interior. In Section~\ref{subsec:rates-of-convergence-bnd}, we provide rates of convergence for estimation of the velocity field~\eqref{eq:velocity-field} and for the rectified flow. In Section~\ref{subsec:linearization-asymp-norm-bounded}, we prove asymptotic normality for the rectified flow map. In Section~\ref{sec:strictly-convex-set-definition}, we show that the velocity field has a unique continuous extension to the boundary of $\Omega$ if $\Omega$ is strictly convex. This, in particular, implies that the rectified flow map may be chosen to be the identity on the boundary of a strictly convex set $\Omega$.
\subsection{Existence, uniqueness, and regularity of rectified flow}\label{subsec:exist-uni-rates-bounded}
When $z\in\Omega^\circ$, such continuity issues do not arise if $p_0$ and $p_1$ are assumed to be continuous and bounded away from zero on $\Omega$. To avoid the trouble of defining a unique continuous extension, we only prove the existence and uniqueness of solutions whenever $x\in\Omega^\circ$. We need the following notation and assumptions. For any set $A\subseteq\mathbb{R}^d$, define
\begin{equation}\label{eq:inflated-imploded-set}
    \begin{split}
        A^{\varepsilon} &= \{x\in\mathbb{R}^d:\, \mbox{dist}(x, A) \le \varepsilon\},\\
        A^{-\varepsilon} &= \{x\in \mathbb{R}^d:\, \mathcal{B}(x, \varepsilon)\subseteq A\}.
    \end{split}
\end{equation}
\begin{enumerate}[label=(B\arabic*)]
    \item $\Omega$ is a compact, convex subset of $\mathbb{R}^d$ with a non-empty interior.\label{eq:compact-support}
    \item The Lebesgue densities $p_0$ and $p_1$ of $\mu_0$ and $\mu_1$ are bounded away from zero on $\Omega$. Moreover, $x\mapsto p_0(x)$ and $x\mapsto p_1(x)$ are continuous and uniformly bounded on $\Omega$.\label{eq:bounded-away-densities}
\end{enumerate}
Under assumption~\ref{eq:compact-support}, set
\begin{equation}\label{eq:inradius}
    r_{\mathrm{in}} ~:=~ \sup\{r\ge 0:\, \mathcal{B}(z, r) \subseteq \Omega\mbox{ for some }z\in\Omega\} ~=~ \sup_{z\in\Omega}\,\mbox{dist}(z,\, \partial\Omega) > 0.
\end{equation}
Under assumption~\ref{eq:bounded-away-densities}, set
\begin{equation}\label{eq:density-lower-bound-notation}
\underline{\mathfrak{p}} := \inf_{x\in\Omega}\min\{p_0(x), p_1(x)\} \le \sup_{x\in\Omega}\max\{p_0(x),\,p_1(x)\} =: \overline{\mathfrak{p}},
\end{equation}
and for any $\eta\in[0,\mbox{diam}(\Omega)]$,
\begin{equation}\label{eq:modulus-density}
\omega(\eta) := \sup_{\substack{z, z'\in\Omega,\\\|z - z'\| \le \eta}}\max\left\{\left|\frac{p_0(z)}{p_0(z')} - 1\right|,\, \left|\frac{p_1(z)}{p_1(z')} - 1\right|\right\}. 
\end{equation}
Note that the compactness of $\Omega$ combined with continuity of $p_0(\cdot)$ and $p_1(\cdot)$ implies that $\omega(\eta) \to 0$ as $\eta\to0$.
\begin{enumerate}[label=(B3)]
\item The function $\eta\mapsto \omega(\eta) + \eta$ is a non-decreasing Osgood function, i.e., $\Psi(u) = \int du/(u + \omega(u))$ satisfies $\Psi(\varepsilon) \to -\infty$ as $\varepsilon \to 0$.\label{eq:Osgood-condition-densities}
\end{enumerate}

See Example 1 (page~\pageref{exam:Osgood-function}) for examples of Osgood functions. Note that~\ref{eq:Osgood-condition-densities} is weaker than the assumption that $\omega(\cdot)$ is an Osgood function. 
For example, for uniform densities $p_0, p_1$, $\omega \equiv 0$, and hence, does not satisfy the Osgood condition.
\begin{theorem}\label{thm:rectified-flow-ODE-unique-sol}
    Suppose assumptions~\ref{eq:compact-support},~\ref{eq:bounded-away-densities}, and~\ref{eq:Osgood-condition-densities} hold. Then for any $x\in\Omega^\circ$, there exists a unique Carath{\'e}odory solution satisfying~\eqref{eq:rectified-flow-ODE}.
\end{theorem}
\begin{proof}
As noted in the example of the rectified flow from standard uniform to itself, one cannot expect the velocity field $(t, z) \mapsto v(t, z)$ to be jointly continuous on $[0,1]\times\Omega$. This prompts us to consider Carath{\'e}odory solutions to~\eqref{eq:rectified-flow-ODE}, i.e., we consider solutions satisfying
\begin{equation}\label{eq:rectified-flow-integral-equation}
z(t) = x + \int_0^t v(s, z(s))ds,\quad\mbox{for}\quad t\in[0, 1].
\end{equation}
Because we have not defined $v(t, z)$ for $z\in\partial\Omega$, this integral equation is not well-defined if $z(s)\in\partial\Omega$ for any $s\in[0, 1)$. We proceed with the following steps to prove the existence and uniqueness of solutions to~\eqref{eq:rectified-flow-integral-equation}. Because $x\in\Omega^\circ$, $\mbox{dist}(x,\, \partial\Omega) > 0$.
\begin{enumerate}
    \item Lemma~\ref{lem:boundedness-and-local-Lipschitz} proves that for any $t\in[0, 1]$, the function $z\mapsto v(t, z)$ is uniformly bounded on $\Omega$ and uniformly continuous on every compact subset of $\Omega^\circ$. This verifies the assumptions for Theorem~\ref{thm:Peano-existence}(1), and implies the existence of $\bar{t}\in(0, 1]$ and a Carath{\'e}odory solution $z^*(\cdot)$ of~\eqref{eq:rectified-flow-integral-equation} on $[0, \bar{t})$ that lies in $\Omega^\circ$.   
    \item Note that every solution to~\eqref{eq:rectified-flow-integral-equation} starting at $x\in\Omega^\circ$ has to either stay in $\Omega^\circ$ for $t\in[0,1)$ or reach the boundary at some $\tau < 1$ and terminate (because we have not defined $z\mapsto v(t, z)$ on $\partial\Omega$). 
    Lemma~\ref{lem:distance-to-boundary-latest} proves that every solution $z(\cdot)$ of~\eqref{eq:rectified-flow-integral-equation} must satisfy $\mbox{dist}(z(t), \partial\Omega) \ge (1-t)\mbox{dist}(x, \partial\Omega)$ for all $t\in[0, 1]$, which implies that no solution of~\eqref{eq:rectified-flow-integral-equation} starting in the interior can reach the boundary, except maybe at $t = 1$.
    \item Using the modulus of continuity bound on the velocity field from Lemma~\ref{lem:boundedness-and-local-Lipschitz}, Lemma~\ref{lem:uniqueness-final} proves that for any $x\in\Omega^\circ$, there exists a unique solution $z^*(\cdot)$ satisfying~\eqref{eq:rectified-flow-integral-equation} and $z^*(t)\in \Omega^\circ$ for all $t\in[0, 1)$.
    \item Hence, for any $x\notin\partial\Omega$, there exists a unique solution. This implies that the rectified flow provides a valid transport map (Theorem~\ref{thm:uniquene-ae-implies-transport-map}). 
\end{enumerate}
\end{proof}
The following lemma provides some basic boundedness and uniform continuity properties of the velocity field~\eqref{eq:velocity-field} under assumptions~\ref{eq:compact-support} and~\ref{eq:bounded-away-densities}.
\begin{lemma}\label{lem:boundedness-and-local-Lipschitz}
    Under assumption~\ref{eq:compact-support},
    \begin{equation}\label{eq:bounded-velocity}
    \sup_{t\in[0,1]}\|v(t, z)\| \le \mathrm{diam}(\Omega),\quad\mbox{for all}\quad z\in\Omega^\circ.
    \end{equation}
    Under assumptions~\ref{eq:compact-support} and~\ref{eq:bounded-away-densities}, for all $t\in[0, 1], \varepsilon, \eta > 0$,
    \begin{equation}\label{eq:locally-Lipschitz}
    \begin{split}
    &\sup_{\substack{z, z'\in\Omega^{-\varepsilon},\\\|z - z'\| \le \eta}}\,{\|v(t, z) - v(t, z')\|}
    % \\ 
    % &
    ~\le~ \mathfrak{L}_1\omega(\eta) + \mathfrak{L}_2(\varepsilon)\eta,
    % \mathrm{Leb}\left(\mathcal{B}\left(0,\,\frac{\mathrm{diam}^2(\Omega)}{\varepsilon^2}\right)\right)\left[\omega(\eta)\frac{2\overline{\mathfrak{p}}\mathrm{diam}(\Omega)}{\underline{\mathfrak{p}}^2} + \eta\frac{4d\overline{\mathfrak{p}}^2}{\underline{\mathfrak{p}}^2}\right].
    \end{split}
    \end{equation}
    where
    \begin{equation}\label{eq:Lipschitz-multiplers}
    \begin{split}
        \mathfrak{L}_1 &= 9\mathrm{diam}(\Omega),\quad\mathrm{and}
        % \\ 
        \quad\mathfrak{L}_2(\varepsilon) = \frac{3\mathrm{diam}(\Omega)}{\omega^{-1}(1)} + \frac{1}{\varepsilon^2}\frac{\overline{\mathfrak{p}}^2}{\underline{\mathfrak{p}}^2}3d5^{d+1}\mathrm{diam}^2(\Omega).
    \end{split}
    \end{equation}
\end{lemma}
A proof of Lemma~\ref{lem:boundedness-and-local-Lipschitz} can be found in Section~\ref{appsubsec:proof-boundedness-and-local-Lipschitz}.
Note that if $\omega(\eta) = C\eta$ for some constant $C$, then inequality~\eqref{eq:locally-Lipschitz} implies that $z\mapsto v(t, z)$ is Lipschitz on $\Omega^{-\varepsilon}$ for each $t\in[0, 1]$.
Better bounds on the modulus of continuity of $z\mapsto v(t, z)$ (depending on $t$) are available in the proof of Lemma~\ref{lem:boundedness-and-local-Lipschitz}. In particular, from the proof of Lemma~\ref{lem:boundedness-and-local-Lipschitz}, it follows that $\mathfrak{L}_2(\varepsilon)$ can be replaced with a time-dependent function
\begin{equation}\label{eq:time-dependent-Lipschitz-constant}
\mathfrak{L}_2(\varepsilon; t) = \frac{3\mathrm{diam}(\Omega)}{\omega^{-1}(1)} + \frac{d}{\varepsilon}\frac{3\overline{\mathfrak{p}}^2\mbox{diam}(\Omega)}{\underline{\mathfrak{p}}^2}
\begin{cases}
1, &\mbox{if }\min\{t, 1 - t\} \le \varepsilon/\mbox{diam}(\Omega),\\
5^{d+1}/\min\{t, 1 - t\}, &\mbox{if }\min\{t, 1 - t\} \ge \varepsilon/\mbox{diam}(\Omega).
\end{cases}
\end{equation}
See inequalities~\eqref{eq:low-values-of-t},~\eqref{eq:high-value-of-t},~\eqref{eq:unrestrictied-Lipschitz}, and~\eqref{eq:correct-restricted-Lipschitz-constant}. It is easy to see  that $\sup_{t\in[0, 1]} \mathfrak{L}_2(\varepsilon; t) = \mathfrak{L}_2(\varepsilon).$

\begin{remark}[(Sub-)Optimality of Lemma~\ref{lem:boundedness-and-local-Lipschitz}]
    The modulus of continuity bound~\eqref{eq:locally-Lipschitz} is obtained as a result of 
    \[
    \|v(t, z) - v(t, z')\| \le 3\mathrm{diam}(\Omega)\left[3\omega(\eta) + \frac{\eta}{\omega^{-1}(1)} + \frac{\overline{\mathfrak{p}}^2}{\underline{\mathfrak{p}}^2}\,\frac{\mathrm{Vol}(S_t(z)\Delta S_t(z'))}{\mathrm{Vol}(S_t(z))}\right],
    \]
    where $\Delta$ represents the symmetric difference of sets. This bound holds without any assumptions on $z, z'\in\Omega$ (i.e., we do not need $z, z'\in\Omega^{-\varepsilon}$). In controlling the second term on the right hand side, the in-radius of $S_t(z)$ is used, and this is obtained from $z\in \Omega^{-\varepsilon}.$ Lemma~\ref{lem:boundedness-and-local-Lipschitz} controls this term when $z, z'\in\Omega^{-\varepsilon}$ for arbitrary convex sets with non-empty interior (in particular, with no assumption on the smoothness of the boundary). It is possible that for specific convex sets such as polygons and ellipsoids, better bounds could be obtained. We believe, however, that in the worst case Lemma~\ref{lem:boundedness-and-local-Lipschitz} is sharp. 
\end{remark}

The following lemma, deriving bounds on distance to boundary, plays the most crucial role in the derivation of uniqueness and regularity properties of rectified flow.
\begin{lemma}\label{lem:distance-to-boundary-latest}
    Suppose assumptions~\ref{eq:compact-support} and~\ref{eq:bounded-away-densities} hold. For $x\in\Omega^\circ$, let $y(\cdot)$ be any function satisfying
    \begin{equation}\label{eq:sub-rectified-flow-equation}
    y(t) = x + \int_0^t v(s, y(s))ds\quad\mbox{for all}\quad t\in[0, T),
    \end{equation}
    with $T = \sup\{t\in[0, 1]:\, y(s)\in\Omega^\circ\mbox{ for all }s\in[0,t]\}$. Then $T = 1$. (This implies that every solution exists on $[0, 1]$.) 
    Furthermore, any solution $y:[0,1]\to\mathbb{R}^d$ must satisfy
    \begin{equation}\label{eq:distance-from-boundary}
    \mathrm{dist}(y(t),\,\partial\Omega) \ge (1 - t)\mathrm{dist}(x,\partial\Omega)\quad\mbox{for all}\quad t\in[0, 1].
    \end{equation}
\end{lemma}
A proof can be found in Section~\ref{appsubsec:proof-of-dist-to-bdry}. Lemma~\ref{lem:distance-to-boundary-latest} provides a lower bound on the distance to boundary of the path. Unfortunately, that quantitative bound does not imply that $z(1) \in \Omega^\circ$ if $x\in\Omega^\circ$. The following result argues indirectly that $z(1)\in\Omega^\circ$ for almost all $x$, which will be useful in establishing better rates of convergence for the rectified flow.

To state the result, we need some notation. For each $t\in[0,1]$ and $x\in\Omega^\circ$, let $\mathfrak{R}(t, x)$ be the unique solution to~\eqref{eq:rectified-flow-integral-equation} in the sense that
\begin{equation}\label{eq:whole-path-integral}
\mathfrak{R}(t, x) = x + \int_0^t v(s, \mathfrak{R}(s, x))ds\quad\mbox{for all}\quad t\in[0, 1], \; x\in\Omega^\circ.
\end{equation}
Note that Theorem~\ref{thm:rectified-flow-ODE-unique-sol} proves such a function is uniquely defined under assumptions~\ref{eq:compact-support},~\ref{eq:bounded-away-densities}, and~\ref{eq:Osgood-condition-densities}, which are also used in the following result. Also, define
\begin{equation}\label{eq:distance-to-boundary-simultaneous}
  \mathfrak{d}(x) := \inf_{t\in[0, 1]}\, \mathrm{dist}(\mathfrak{R}(t, x),\,\partial\Omega)\quad\mbox{for all}\quad x\in\Omega^\circ.   
\end{equation}
and for $t\in[0, 1]$,
\begin{equation}\label{eq:non-differentiability}
\begin{split}
    B(t) &:= \{z^*\in\Omega:\, z\mapsto v(t, z)\mbox{ is not differentiable at }z^*\}\\
    &= \left\{z^*\in\Omega:\, \limsup_{h\to0}\,\frac{v(t, z^* + h) - v(t, z^*)}{h} \neq \liminf_{h\to0}\,\frac{v(t, z^* + h) - v(t, z^*)}{h}\right\}.
\end{split}
\end{equation}
\begin{lemma}\label{lem:distance-to-boundary-at-1}
    Suppose assumptions~\ref{eq:compact-support},~\ref{eq:bounded-away-densities}, and~\ref{eq:Osgood-condition-densities} hold. Consider the events
    \begin{equation}\label{eq:measure-1-sets}
        \begin{split}
            \mathcal{E}_1 &:= \{z\in\Omega:\, \mathrm{dist}(z, \partial\Omega) > 0\},\\
            \mathcal{E}_2 &:= \{z\in\Omega:\, \mathrm{dist}(\mathfrak{R}(1, z),\,\partial\Omega) > 0\},\\
            \mathcal{E}_3 &:= \{(t, z)\in[0, 1]\times\Omega:\, \mathfrak{R}(t, z)\notin B(t)\},\\
            \mathcal{E}_4 &:= \{z\in\Omega:\, \mathrm{Leb}(\{t\in[0,1]:\, \mathfrak{R}(t, z)\in B(t)\}) = 0\}.
        \end{split}
    \end{equation}
    Then $\mu_0(\mathcal{E}_1\cap\mathcal{E}_2) = 1$. Furthermore, for any $\varepsilon > 0$, we have
    \begin{equation}\label{eq:dist-to-boundary-path}
    \mathfrak{d}(x) ~\ge~ \frac{\mathrm{dist}(\mathfrak{R}(1, x),\,\partial\Omega)\varepsilon}{\mathrm{diam}(\Omega) + \varepsilon},\quad\mbox{for all}\quad x\in\Omega^{-\varepsilon},
    \end{equation}
    and for any $\gamma \in [0, \mathrm{diam}(\Omega)/2]$, setting
    \begin{equation}
        \mathcal{A}_{\gamma} := \{x\in\Omega:\, \mathrm{dist}(\mathfrak{R}(1, x),\,\partial\Omega) \le \gamma\},
    \end{equation}
    we have  
    \begin{equation}\label{eq:dist-to-boundary-almost-everywhere}
    \mu_0(\mathcal{A}_{\gamma}) \le \gamma \frac{2^dd\overline{\mathfrak{p}}}{r_{\mathrm{in}}\underline{\mathfrak{p}}},\quad\mbox{and}\quad \inf_{x\in\mathcal{A}_{\gamma}^c\cap\Omega^{-\varepsilon}}\,\mathfrak{d}(x) ~\ge~ \frac{\gamma\varepsilon}{\mathrm{diam}(\Omega) + \varepsilon}.
    \end{equation}
    Finally, if for some $C \ge 0$ such that $\omega(\eta) \le C\eta$ for all $\eta > 0$, then $(\mathrm{Leb}\times\mu_0)(\mathcal{E}_3) = 1,$ and $\mu_0(\mathcal{E}_4) = 1$. 
\end{lemma}
A proof of Lemma~\ref{lem:distance-to-boundary-at-1} can be found in Section~\ref{appsubsec:proof-of-dist-to-bnry-at-1}. The first part of the result proves that $\mathfrak{d}(x) > 0$ for almost all $x\in\Omega$. Inequality~\eqref{eq:dist-to-boundary-path} proves that if $\mathfrak{R}(0, x)$ and $\mathfrak{R}(1, x)$ are away from the boundary, then the entire path is away from the boundary. This follows from Lemma~\ref{lem:distance-to-boundary-latest} and the boundedness of the velocity field. Inequality~\eqref{eq:dist-to-boundary-almost-everywhere}, on the other hand, proves that for almost all $x\in \Omega$, $\mathfrak{R}(1, x)$ is away from the boundary. It is not obvious if one can obtain quantitative lower bounds on the distance of $\mathfrak{R}(1, x)$ from the boundary for all $x\in\Omega$, without additional assumptions on $\Omega$. The last part of Lemma~\ref{lem:distance-to-boundary-at-1} proves that for almost all $(s, x)\in[0,1]\times\Omega$, $\mathfrak{R}(s, x)$ is a differentiability point of $z\mapsto v(s, z)$. 

Following Lemma~\ref{lem:distance-to-boundary-at-1}, we now present a result on the (smoothness) regularity of the rectified flow. From the stability/perturbation results of ODEs, uniform continuity of the rectified flow is not hard to derive. The modulus of continuity, however, depends very strongly on the distance to boundary of rectified flow path. 
\begin{theorem}[Uniform Continuity of Rectified Flow]\label{thm:continuity-of-paths}
For any $\varepsilon > 0$, $x\mapsto \mathfrak{R}(t, x)$ is uniformly continuous on $\Omega^{-\varepsilon}$. Indeed, for all $\eta \ge 0$,
\begin{align*}
\sup_{\substack{x, x'\in\Omega^{-\varepsilon}, t\in[0, 1],\\
\|x - x'\| \le \eta}}\, \|\mathfrak{R}(t, x) - \mathfrak{R}(t, x')\|
% \\ 
% &\qquad
~&\le~ \Psi^{-1}(\Psi(\eta)/2 + \mathfrak{C}_1 + \frac{\mathfrak{C}_2}{\varepsilon}\ln\left(\frac{\mathrm{diam}(\Omega)}{\varepsilon}\right)\\ 
&\qquad+ \mathrm{diam}(\Omega)\max\left\{\frac{4\mathfrak{C}_2}{2\mathfrak{C}_2 - \varepsilon\Psi(\eta)},\, \frac{\eta}{\Psi^{-1}(-2\mathfrak{C}_2/\varepsilon)}\right\},
\end{align*}
for some constants $\mathfrak{C}_1, \mathfrak{C}_2$ depending only on $d, \mathrm{diam}(\Omega),$ and $\omega^{-1}(1)$.

Moreover, for any $\gamma \in [0, \mathrm{diam}(\Omega)/2]$,
\begin{align*}
\sup_{\substack{x, x'\in\mathcal{A}_{\gamma}^c\cap\Omega^{-\varepsilon}, t\in[0, 1],\\
\|x - x'\| \le \eta}}\, \|\mathfrak{R}(t, x) - \mathfrak{R}(t, x')\| ~&\le~ \Psi^{-1}\left(\Psi(\eta) + \mathfrak{C}_1 + \frac{\mathfrak{C}_2}{\min\{\varepsilon, \gamma\}}\ln\left(\frac{\mathrm{diam}(\Omega)}{\varepsilon}\right) + \frac{\mathfrak{C}_3}{\varepsilon\gamma}\right),
\end{align*}
for some constants $\mathfrak{C}_1, \mathfrak{C}_2, \mathfrak{C}_3$ depending only on $d, \mathrm{diam}(\Omega), \omega^{-1}(1)$. (These constants may be different from the ones above.)
\end{theorem}
A proof of Theorem~\ref{thm:continuity-of-paths} can be found in Section~\ref{appsubsec:proof-continuity-of-paths}. To better understand the implications of Theorem~\ref{thm:continuity-of-paths}, consider the case where the densities $p_0$ and $p_1$ are Lipschitz continuous so that $\omega(\eta) \le \eta$.\footnote{We assume the Lipschitz constant of 1, for simplicity.} In this case, $\Psi(u) = \ln(u)/2$ which diverges to $-\infty$ as $u\to 0$. In this case, the first part of Theorem~\ref{thm:continuity-of-paths} implies
\begin{equation}\label{eq:modulus-Lipschitz-densities}
\begin{split}
    \sup_{\substack{x, x'\in\Omega^{-\varepsilon}, t\in[0, 1],\\
\|x - x'\| \le \eta}}\, \|\mathfrak{R}(t, x) - \mathfrak{R}(t, x')\| ~&\le~ \frac{8\mathfrak{C}_2\mbox{diam}(\Omega)}{4\mathfrak{C}_2 + \varepsilon\ln(1/\eta)} + \mbox{diam}(\Omega)\exp\left(-\frac{4\mathfrak{C}_2}{\varepsilon}\right)\eta\\ 
&\quad + \exp\left(2\mathfrak{C}_1 + \frac{2\mathfrak{C}_2\ln(\mbox{diam}(\Omega)/\varepsilon)}{\varepsilon}\right)\eta^{1/2}. 
\end{split}
\end{equation}
Although the right hand side converges to zero as $\eta\to0$, it converges at an extremely slow logarithmic rate (from the first term on the right hand side). Additionally, the bound diverges exponentially fast with $\varepsilon\to 0$. It is unclear if this is bound on the modulus of continuity can be improved without further assumptions on $\Omega$. Any improvement in the Lipschitz constant (in Lemma~\ref{lem:boundedness-and-local-Lipschitz}) of the velocity field yields an improvement of the modulus of continuity. 

While the bound in general is pessimistic, the second part of Theorem~\ref{thm:continuity-of-paths} provides a better bound for almost all $x\in\Omega^{-\varepsilon}$. Indeed, it yields
\begin{equation}\label{eq:modulus-Lipschitz-densities-almost-everywhere}
    \begin{split}
        \sup_{\substack{x, x'\in\mathcal{A}_{\gamma}^c\cap\Omega^{-\varepsilon}, t\in[0, 1],\\
\|x - x'\| \le \eta}}\, \|\mathfrak{R}(t, x) - \mathfrak{R}(t, x')\| ~&\le~ \exp\left(2\mathfrak{C}_1 + \frac{2\mathfrak{C}_2}{\min\{\varepsilon, \gamma\}}\ln\left(\frac{\mbox{diam}(\Omega)}{\varepsilon}\right) + \frac{2\mathfrak{C}_3}{\varepsilon\gamma}\right)\eta.
    \end{split}
\end{equation}
In comparison with~\eqref{eq:modulus-Lipschitz-densities}, inequality~\eqref{eq:modulus-Lipschitz-densities-almost-everywhere} shows that for almost all $x\in\Omega^{-\varepsilon}$, the rectified flow map is Lipschitz continuous. But note that the Lipschitz constant is exponentially quickly growing with $1/\varepsilon$ and $1/\gamma$.

\subsection{Rates of convergence of estimates of the velocity and rectified flow}
\label{subsec:rates-of-convergence-bnd}

Because $v_t$ is not smooth, we do not use
a regression estimator.
Instead we use a density-based estimator
based on the form given in
Lemma \ref{lem:representation-of-velocity-field}.
The rate given in
Section \ref{sec:rectified-flow}
is not uniform in $t$
so we will derive better rates in this section.
We define a new density estimator that accomodates
arbitrary boundaries.
To derive the rates of convergence of the rectified flow, we need perturbation bounds as in Theorem~\ref{thm:stability}. Unfortunately, our Lipschitz continuity bound in Lemma~\ref{lem:boundedness-and-local-Lipschitz} is not strong enough to directly apply Theorem~\ref{thm:stability}. Hence, we follow the proof of Theorem~\ref{thm:stability} and derive the rates of convergence for the rectified flow. We define the following collection of velocity fields on $[0,1]\times\Omega$ and prove a general equicontinuity property. 

A velocity field $\nu:[0,1]\times\Omega\to\mathbb{R}^d$ is said to belong to $\mathcal{V}$ if there exists $\varpi(\cdot)$ satisfying~\ref{eq:Osgood-condition-densities} and a constants $\mathfrak{C}$ such that:
\begin{enumerate}[label=(P\arabic*)]
    \item $\|\nu(t, z)\| \le \mbox{diam}(\Omega)$ for all $z\in\Omega^\circ$ and $t\in[0, 1]$.\label{eq:bounded-v-class-prop}
    \item For every $\varepsilon > 0$,\label{eq:Lipschitz-v-class-prop}
    \begin{align*}
    &\sup_{\substack{z,z'\in\Omega^{-\varepsilon},\\\|z - z'\| \le \eta}}\,\|\nu(t, z) - \nu(t, z')\|\\ ~&\le~ \mathfrak{C}\varpi(\eta)
    % \\ 
    % &\quad
    +~ \frac{\mathfrak{C}\eta}{\varepsilon}\times\begin{cases}1,&\mbox{if }\min\{t, 1 - t\} \le \varepsilon/\mbox{diam}(\Omega),\\
    1/\min\{t, 1 - t\}, &\mbox{otherwise.}\end{cases} 
    \end{align*}
    \item For every $t\in[0, 1]$ and $z\in\Omega$, \label{eq:belonging-v-class-prop}
    \[
    \nu(t, z) ~\in~ S_t(z) ~=~ \frac{z - \Omega}{t}\cap\frac{\Omega - z}{1 - t}.
    \]
\end{enumerate}
(The constant $\mathfrak{C}$ and the function $\varpi(\cdot)$ are the same for all functions in $\mathcal{V}$.) Note that if $\widetilde{\nu}:[0,1]\times\Omega\to\mathbb{R}^d$ satisfies~\ref{eq:Lipschitz-v-class-prop} but not~\ref{eq:bounded-v-class-prop} or~\ref{eq:belonging-v-class-prop}, then one can consider 
\[
\nu(t, x) := \mbox{Proj}_{(\Omega-\Omega)\cap S_t(x)}(\widetilde{\nu}(t, x)), 
\]
which would satisfy all the assumptions. Here $\Omega - \Omega = \{z - z':\, z, z'\in\Omega\}$. In other words, assumptions~\ref{eq:bounded-v-class-prop}--\ref{eq:belonging-v-class-prop} are not restrictive.

From the proof of Theorem~\ref{thm:rectified-flow-ODE-unique-sol}, it follows that for every $\nu\in\mathcal{V}$, the integral equation~\eqref{eq:rectified-flow-integral-equation} (with $\nu$ replacing $v$) has a unique Carath{\'e}odory solution that lies entirely in $\Omega$. The following result proves an equicontinuity result for the solutions in terms of $\nu(\cdot, \cdot)$. For each $\nu\in\mathcal{V}$, let $\mathfrak{R}_{\nu}(\cdot, \cdot)$ be the unique function satisfying
\begin{equation}\label{eq:new-integral-equation-equicontinuity}
\mathfrak{R}_{\nu}(t, x) = x + \int_0^t \nu(s, \mathfrak{R}_{\nu}(s, x))ds,\quad\mbox{for all}\quad x\in\Omega^\circ, \, t\in[0, 1].
\end{equation}
Define the supremum distance on $\mathcal{V}$ as follows. For any two functions $\nu_1, \nu_2\in\mathcal{V}$, set
\[
\|\nu_1 - \nu_2\|_{\infty} ~:=~ \sup_{t\in[0, 1],\,z\in\Omega^\circ}\, \|\nu_1(t, z) - \nu_2(t, z)\|.
\]
Finally, set
\[
\overline{\Psi}(u) := \int \frac{du}{\varpi(u) + u}du.
\]
(This is an indefinite integral.)
\begin{theorem}\label{thm:equi-continuity-in-velocity}
For every $x\in\Omega^{-\varepsilon}$ and $\nu_1, \nu_2\in\mathcal{V}$ with $\|\nu_1 - \nu_2\|_{\infty} \le \Delta$, we have
\begin{equation}\label{eq:equicontinuity-in-velocity-weak}
\begin{split}
\sup_{t\in[0, 1]}\,\|\mathfrak{R}_{\nu_1}(t, x) - \mathfrak{R}_{\nu_2}(t, x)\|
    &\le \overline{\Psi}^{-1}(\overline{\Psi}(\Delta)/2 + \frac{C}{\varepsilon}\ln\left(\frac{\mathrm{diam}(\Omega)}{\varepsilon}\right) \\
    &\quad+ \mathrm{diam}(\Omega)\max\left\{\frac{4C}{2C - \varepsilon\overline{\Psi}(\Delta)},\,\frac{\Delta}{\overline{\Psi}^{-1}(-C/\varepsilon)}\right\},
\end{split}
\end{equation}
where $C$ is a constant depending only on $\mathfrak{C}$ (in~\ref{eq:Lipschitz-v-class-prop}). Moreover, if $x\in\Omega^{-\varepsilon}$ and there exists a $\gamma > 0$ such that
\begin{equation}\label{eq:dist-to-boundary-nu_2}
\mathrm{dist}(\mathfrak{R}_{\nu_2}(1, x),\,\partial\Omega) \ge 2\gamma,
\end{equation}
then, for all small enough $\Delta$ (so that the right hand side of~\eqref{eq:equicontinuity-in-velocity-weak} is less than $\gamma$) and $\nu_1\in\mathcal{V}$ satisfying $\|\nu_1 - \nu_2\|_{\infty} \le \Delta$, we have
\begin{equation}\label{eq:equicontinuity-in-velocity-strong}
\sup_{t\in[0, 1]}\|\mathfrak{R}_{\nu_1}(t, x) - \mathfrak{R}_{\nu_2}(t, x)\| \le \overline{\Psi}^{-1}\left(\overline{\Psi}(\Delta) + \frac{C}{\min\{\varepsilon, \gamma\}}\ln\left(\frac{\mathrm{diam}(\Omega)}{\varepsilon}\right) + \frac{C}{\varepsilon\gamma}\right),
\end{equation}
for a constant $C$ depending only on $\mathfrak{C}$ (in~\ref{eq:Lipschitz-v-class-prop}).
\end{theorem}
The proof of Theorem~\ref{thm:equi-continuity-in-velocity} is almost the same as that of Theorem~\ref{thm:continuity-of-paths}, and a condensed version of the proof can be found in Section~\ref{appsubsec:proof-equicontinuity-in-velocity}. The assumptions of Theorem~\ref{thm:equi-continuity-in-velocity} can be weakened slightly. It suffices to assume that both $\nu_j, j = 1, 2$ satisfy Properties~\ref{eq:bounded-v-class-prop} and~\ref{eq:belonging-v-class-prop} so as to ensure that $\mathfrak{R}_{\nu_j}(t, x)\in\Omega^{-(1-t)\varepsilon}, j = 1, 2$. Additionally, it suffices that one of $\nu_1, \nu_2$ satisfy property~\ref{eq:Lipschitz-v-class-prop} for the validity of~\eqref{eq:equicontinuity-in-velocity-weak} and~\eqref{eq:equicontinuity-in-velocity-strong}. Finally, under~\eqref{eq:dist-to-boundary-nu_2}, we do not need $\|\nu_1 - \nu_2\|_{\infty} \le \Delta$, we just need $\sup_{t\in[0,1],\,z\in\Omega^{-\kappa}}\|\nu_1(t, z) - \nu_2(t, z)\| \le \Delta$ for some $\kappa > 0$ depending on $\varepsilon, \gamma$. (This $\kappa$ is precisely the right-hand side of~\eqref{eq:dist-to-boundary-path}.)

Theorem~\ref{thm:equi-continuity-in-velocity} proves that (uniform) closeness of velocity fields implies the closeness of the unique solutions. As remarked after Theorem~\ref{thm:continuity-of-paths}, the modulus of continuity with respect to the velocity field is not optimistic for arbitrary $x\in\Omega^{-\varepsilon}$. When the distance of the solution to the boundary is bounded away from zero, then one obtains a better modulus of continuity. In fact, if $\varpi(\eta) = \eta$, then inequality~\eqref{eq:equicontinuity-in-velocity-strong} implies
\[
\sup_{t\in[0, 1]}\|\mathfrak{R}_{\nu_1}(t, x) - \mathfrak{R}_{\nu_2}(t, x)\| \le \Delta\exp\left(\frac{C}{\min\{\varepsilon, \gamma\}}\ln\left(\frac{\mathrm{diam}(\Omega)}{\varepsilon}\right) + \frac{C}{\varepsilon\gamma}\right).
\]
A simple application of Theorem~\ref{thm:equi-continuity-in-velocity} is to understand the effect of discretization algorithms in estimating the solutions of~\eqref{eq:rectified-flow-integral-equation}. For example, Euler discretization corresponds to taking (for some $k \ge 1$)
\[
\nu_1(t, x) = v(\lfloor kt\rfloor/k,\, x)\quad\mbox{for}\quad x\in\Omega,\, t\in[0, 1].
\]
One can apply Theorem~\ref{thm:equi-continuity-in-velocity} with $\nu_2 \equiv v$ and obtain an error bound in terms of $k$.
Note that, in this case, $\Delta = \sup_{x\in\Omega^\circ}\sup_{t\in[0, 1]}\|v(t, x) - v(\lfloor kt\rfloor/k,\, x)\|$, which can be controlled if $v(\cdot, \cdot)$ is Lipschitz continuous in the first argument, uniformly in the second argument. The following result provides such Lipschitz continuity.
\begin{proposition}\label{prop:Lipschitz-cont-velocity-in-time}
    Suppose assumptions~\ref{eq:compact-support},~\ref{eq:bounded-away-densities} hold. Also, suppose $\omega(\eta) \le L\eta$ for all $\eta > 0$ for some $L \ge 0$. Then for all $t\in[0, 1]$ and $h > 0$ such that $t + h\in[0, 1]$, we have
    \[
    \sup_{z\in\Omega^{-\varepsilon}}\,\|v(t, z) - v(t + h, z)\| \le 2\mathrm{diam}(\Omega)\left(8L\mathrm{diam}(\Omega) + \frac{\overline{\mathfrak{p}}^2}{\underline{\mathfrak{p}}^2}\frac{4^{d+1}d\mathrm{diam}^2(\Omega)}{\varepsilon^2\bar{t}}\right).
    \]
\end{proposition}
A proof can be found in Section~\ref{appsubsec:proof-Lipschitz-velocity-in-time}.

Theorem~\ref{thm:equi-continuity-in-velocity} also allows us to consider rates of convergence of rectified flow estimators based on the rates of convergence of estimators of the velocity field. In particular, consider the density-based estimator of the velocity field (as discussed in Section~\ref{section::density}). Proposition~\ref{prop:consistency-of-density-estimator} provides a simple result on the rate of convergence that is unfortunately inapplicable for Theorem~\ref{thm:equi-continuity-in-velocity} because the rate is not uniform in $t$. The following result provides a better rate of convergence using assumptions~\ref{eq:compact-support}--\ref{eq:bounded-away-densities}.

Suppose $\widehat{p}_0$ and $\widehat{p}_1$ are the estimators of densities $p_0$ and $p_1$, respectively. (We do not require $\widehat{p}_0$ and $\widehat{p}_1$ to be densities or to be supported on $\Omega$. In case their support is not $\Omega$, we redefine $\widehat{p}_j(z)$ as $\widehat{p}_j(z)\mathbf{1}\{z\in\Omega\}$.) Set
\begin{equation}\label{eq:density-ratio-estimation-error}
    r_n := \sup_{x\in\Omega}\,\max\left\{\left|\ln\frac{\widehat{p}_0(x)}{p_0(x)}\right|,\,\left|\ln\frac{\widehat{p}_1(x)}{p_1(x)}\right|\right\}.
\end{equation}
It is easy to see that, by definition, 
\begin{equation}\label{eq:estimted-denisty-ratio-implication}
e^{-r_n}p_j(x) \le \widehat{p}_j(x)\le e^{r_n}p_j(x)\quad\mbox{for}\quad j = 1, 2,\, x\in\Omega.
\end{equation}
Additionally, if $r_n\to 0$, then (under assumption~\ref{eq:bounded-away-densities})
\[
\max_{j=1,2}\,\|\widehat{p}_j - p_j\|_{\infty} ~\asymp~ r_n.
\]
Define, as in~\eqref{eq:density-based-velocity-estimator}, 
\[
\widehat{v}^{\mathrm{den}}(t, z) := \frac{\int_{S_t(z)} \delta \widehat{p}_0(z - t\delta)\widehat{p}_1(z + (1-t)\delta)d\delta}{\int_{S_t(z)} \widehat{p}_0(z - t\delta)\widehat{p}_1(z + (1-t)\delta)d\delta}.
\]
The domain of integration here is $S_t(z)$ because $\widehat{p}_0$ and $\widehat{p}_1$ are assumed to be supported on $\Omega$ (which is possible only if $\Omega$ is known). Under this setting, the following result holds.
\begin{theorem}[Rate of convergence of velocity field]\label{thm:rate-of-conv-velocity}
    Suppose $\widehat{p}_0$ and $\widehat{p}_1$ are the estimators of densities $p_0$ and $p_1$, respectively. Then, for all $n\ge1$,
    \[
    \|\widehat{v}^{\mathrm{den}} - v\|_{\infty} = \sup_{t\in[0, 1],\,z\in\Omega^\circ}\|\widehat{v}^{\mathrm{den}}(t, z) - v(t, z)\| ~\le~ 2\mathrm{diam}(\Omega)(e^{4r_n} - 1). 
    \]
    If $r_n = o_p(1)$ as $n\to\infty$, then
    \[
    \sup_{t\in[0, 1],\,z\in\Omega^\circ}\|\widehat{v}^{\mathrm{den}}(t, z) - v(t, z)\| = O_p(r_n),\quad\mbox{as}\quad n\to\infty.
    \]
\end{theorem}
See Section~\ref{appsubsec:proof-rate-of-conv-velocity} for a proof.
\begin{theorem}[Rate of convergence of rectified flow]\label{thm:rate-of-conv-rectified-flow}
    Consider the setting of Theorem~\ref{thm:rate-of-conv-velocity}. Suppose that
    \[
    \widehat{\omega}(\eta) := \sup_{x, x'\in\Omega,\, \|x - x'\| \le \eta}\, \max_{j = 1, 2}\left|\frac{\widehat{p}_j(x)}{\widehat{p}_j(x')} - 1\right|\quad\mbox{for all}\quad \eta > 0,
    \]
    with $\widehat{\omega}(\cdot)$ satisfying assumption~\ref{eq:Osgood-condition-densities}. Then $\widehat{\mathfrak{R}}(\cdot, \cdot)$ satisfying
    \begin{equation}\label{eq:uniqueness-solution}
    \widehat{\mathfrak{R}}(t, x) = x + \int_0^t \widehat{v}^{\mathrm{den}}(s,\,\widehat{\mathfrak{R}}(s, x))ds,
    \end{equation}
    is well-defined (i.e., uniquely) for all $x\in\Omega^\circ$. If, for some $\varepsilon, \gamma > 0$,
    \begin{equation}\label{eq:x-assumption}
    x\in\Omega^{-\varepsilon}\quad\mbox{and}\quad \mathrm{dist}(\mathfrak{R}(1, x),\,\partial\Omega) \ge 2 \gamma, 
    \end{equation}
    then, for $n$ large enough,
    \[
    \|\widehat{\mathfrak{R}}(1, x) - \mathfrak{R}(1, x)\| \le  {\Psi}^{-1}\left({\Psi}(2\mathrm{diam}(\Omega)(e^{4r_n} - 1)) + \frac{C}{\min\{\varepsilon, \gamma\}}\ln\left(\frac{\mathrm{diam}(\Omega)}{\varepsilon}\right) + \frac{C}{\varepsilon\gamma}\right).
    \]
    (Here $\Psi(u) = \int du/(\omega(u) + u)$ from assumption~\ref{eq:Osgood-condition-densities}.)
\end{theorem}
\begin{proof}[Proof of Theorem~\ref{thm:rate-of-conv-rectified-flow}]
    From the setting of Theorem~\ref{thm:rate-of-conv-velocity} and the assumption on $\widehat{\omega}(\cdot)$, we get that $\widehat{p}_0, \widehat{p}_1$ satisfy assumptions~\ref{eq:compact-support},~\ref{eq:bounded-away-densities}, and~\ref{eq:Lipschitz-Osgood}, and hence, Theorem~\ref{thm:rectified-flow-ODE-unique-sol} implies the uniqueness of the solution to~\eqref{eq:uniqueness-solution}. The second part of the result follows from~\eqref{eq:x-assumption} and Theorem~\ref{thm:equi-continuity-in-velocity} (in particular,~\eqref{eq:equicontinuity-in-velocity-strong}).
\end{proof}
Theorem~\ref{thm:rate-of-conv-rectified-flow} implies that the rectified flow shares the same rate of convergence as the velocity field if $\omega(u) \le Lu$ for all $u\ge 0$ (i.e., the densities are Lipschitz continuous).
\begin{corollary}[Rate of convergence of rectified flow]\label{cor:rate-of-conv-rectified-lipschitz}
    Consider the setting of Theorem~\ref{thm:rate-of-conv-velocity}. If $\omega(u) \le Lu$ for all $u \ge 0$, then 
    \[
    \|\widehat{\mathfrak{R}}(1, x) - \mathfrak{R}(1, x)\| \le 2\mathrm{diam}(\Omega)(e^{4r_n} - 1)\exp\left(\frac{C(L+1)}{\min\{\varepsilon,\,\gamma\}}\ln\left(\frac{\mathrm{diam}(\Omega)}{\varepsilon}\right) + \frac{C(L+1)}{\varepsilon\gamma}\right).
    \]
\end{corollary}

\subsection{Linearization and Asymptotic Normality}\label{subsec:linearization-asymp-norm-bounded}
In this section, we prove an expansion for $\widehat{\mathfrak{R}}(1, x) - \mathfrak{R}(1, x)$ in terms of $\widehat{v}^{\mathrm{den}} - v$. The following lemma on the Lipschitz continuity of $\widehat{v}^{\mathrm{den}} - v$ is crucial for such an expansion. 
\begin{lemma}\label{lem:equicontinuity-vhat-v}
    Suppose $\max\{\widehat{\omega}(\eta),\,\omega(\eta)\} \le L\eta$ for all $\eta > 0$. 
    Define
    \[
    s_n = \esssup_{x\in\Omega^\circ}\,\max\left\{\left|\frac{d\log \widehat{p}_0(x)}{dx} - \frac{d\log p_0(x)}{dx}\right|,\,\left|\frac{d\log\widehat{p}_1(x)}{dx} - \frac{d\log p_1(x)}{dx}\right|\right\}.
    \]
    For any $\varepsilon > 0$ and $\eta \le \min\{1/L,\,\varepsilon^2/\mathrm{diam}(\Omega)\}$, we have
    \begin{align*}
        &\sup_{\substack{z, z'\in\Omega^{-\varepsilon},\\\|z - z'\| \le \eta}}\|\widehat{v}^{\mathrm{den}}(t, z) - v(t,z) - \widehat{v}^{\mathrm{den}}(t, z') + v(t, z')\|
        \\
        &\le C'\eta(e^{4r_n} - 1)(\varepsilon^{-2} + \eta\varepsilon^{-4})
    %     \\
    % &\quad
    + C''(1 + \eta\varepsilon^{-2})^2\left[\exp(s_n\eta) - 1 + (e^{r_n} - 1)L\eta\right],
    \end{align*}
    for some constants $C'$ and $C''$ depending only on $d, L, \mathrm{diam}(\Omega),$ and $(\overline{\mathfrak{p}},\,\underline{\mathfrak{p}})$. In particular, for $\eta$ small enough and $n$ large enough (so that $r_n$ is small enough), we have
    \[
    \sup_{\substack{z, z'\in\Omega^{-\varepsilon},\\\|z - z'\| \le \eta}}\|\widehat{v}^{\mathrm{den}}(t, z) - v(t,z) - \widehat{v}^{\mathrm{den}}(t, z') + v(t, z')\| \le C'_{\varepsilon}\eta (r_n + s_n),
    \]
    for a constant $C'_{\varepsilon}$ depending on $C', C'',$ and $\varepsilon$.
\end{lemma}
\begin{theorem}[Linearization]\label{thm:linearization-bounded-case}
    Fix $x\in\mathcal{E}_1\cap\mathcal{E}_2\cap\mathcal{E}_4$. (Recall the definitions from Lemma~\ref{lem:distance-to-boundary-at-1}.) Then $z\mapsto v(s, z)$ is differentiable at $z = \mathfrak{R}(s, x)$ for almost all $s\in[0,1]$. Let the derivative be denoted by $\partial_zv(s,\mathfrak{R}(s, x))$. Define $\widetilde{E}(t, x)$ as the unique solution of the integral equation
    \begin{equation} \label{eq:influence-function-rectified-flow0} \widetilde{E}(t, x)= \int_0^t \{\widehat{v}^{\mathrm{den}}(s, \mathfrak{R}(s, x)) - v(s, \mathfrak{R}(s, x))\}ds + \int_0^t \partial_zv(s, \mathfrak{R}(s, x))\widetilde{E}(s,x)ds.\end{equation}
    The existence of a unique solution to~\eqref{eq:influence-function-rectified-flow0} follows from Proposition~\ref{prop:Volterra}. Indeed, this unique solution can be written as
    \begin{equation}\label{eq:influence-function-rectified-flow}
    \widetilde{E}(t, x) = \Phi(t)\int_0^t (\Phi(s))^{-1}\{\widehat{v}^{\mathrm{den}}(s, \mathfrak{R}(s, x)) - v(s, \mathfrak{R}(s, x))\}ds,
    \end{equation}
    where $\Phi(\cdot)$ is the unique invertible solution to the matrix differential equation
    \begin{equation}\label{eq:fundamental-derivative-velocity}
    \Phi(t) = I + \int_0^t \partial_zv(s, \mathfrak{R}(s, x))^\top\Phi(s)ds.
    \end{equation}
    (Here $I$ is the identity matrix in $\mathbb{R}^{d\times d}$.)
    Then, assuming $r_n + s_n = o_p(1)$, we get
    \begin{equation}\label{eq:linearization-first}
    \sup_{t\in[0,1]}\|\widehat{\mathfrak{R}}(t, x) - \mathfrak{R}(t, x) - \widetilde{E}(t, x)\| = o_p\left(\sup_{t\in[0,1]}\|\widetilde{E}(t, x)\|\right).
    \end{equation}
\end{theorem}
A proof of Theorem~\ref{thm:linearization-bounded-case} can be found in Section~\ref{appsubsec:proof-of-linearization}. An important implication of Theorem~\ref{thm:linearization-bounded-case} is that the rate of convergence and the limiting distribution of $\widehat{\mathfrak{R}}(t, x) - \mathfrak{R}(t, x)$ matches that of $\widetilde{E}(t, x)$, if $\widetilde{E}(t, x)$ and $\sup_{t\in[0,1]}\|\widetilde{E}(t, x)\|$ share the same rate of convergence. Theorem~\ref{thm:linearization-bounded-case} does not require any specific structure on the density estimators as long as $r_n + s_n = o_p(1)$ as $n\to\infty$. 

The condition of same rate of convergence can be verified by proving tightness of the process $t\mapsto a_n\widetilde{E}(t, x)$ where $a_n$ is such that $a_n\widetilde{E}(1, x) = O_p(1)$. Note that $\widetilde{E}(0, x) = 0$ and hence, verifying certain moment condition on the increment $a_n(\widetilde{E}(t, x) - \widetilde{E}(s, x))$ would imply $a_n\sup_{t\in[0,1]}\|\widetilde{E}(t, x)\| = O_p(1)$; see, for example, Example 2.2.7 of~\cite{VanderVaartWellner2023} or Section 1.3 of~\cite{Talagrand2022}. However, $\widehat{v}^{\mathrm{den}}(s, z) - v(s, z)$ is a ratio of random quantities which makes it difficult to directly apply this technique. This can be resolved by a further linearlization of $\widehat{v}^{\mathrm{den}}(s, z) - v(s, z)$ as shown in the proof of Theorem~\ref{thm:asymp-normality-bounded-case}. 
% If there exists a function $(s, x) \mapsto \widehat{\gamma}(s, x)$ with finite moments such that 
% \begin{equation}\label{eq:linearized-velocity-field}
% \|\widehat{v}^{\mathrm{den}}(s, \mathfrak{R}(s, x)) - v(s, \mathfrak{R}(s, x)) - \widehat{\gamma}(s, x)\| = o_p(1/a_n),
% \end{equation}
% and 
% \begin{equation}
%     \sup_{|t_1 - t_2| \le \delta}
% \mathbb{E}\left[a_n^2\left\|\int_{t_1}^{t_2} (\Phi(s))^{-1}\widehat{\gamma}(s, x)ds\right\|^2\right] 
%     \le C\delta^2,
% \end{equation}
% for some constant $C$ and all $\delta > 0$, we get from Example 2.2.7 of~\cite{VanderVaartWellner2023} that~\eqref{eq:linearization-first} implies
% \[
% \sup_{t\in[0,1]}\,\left\|a_n(\widehat{\mathfrak{R}}(t, x) - \mathfrak{R}(t, x)) - a_n\Phi(t)\int_0^t (\Phi(s))^{-1}\widehat{\gamma}(s, x)ds\right\| = o_p(1).
% \]
% Define
% \begin{equation}\label{eq:linearlized-linearization}
%     \widetilde{E}^*(t, x) = \Phi(1)\int_0^t (\Phi(s))^{-1}\int_{S_s(\mathfrak{R}(s, x))} \frac{\delta - v(s,\mathfrak{R}(s, x))}{p_s(\mathfrak{R}(s, x))}\sum_{j=0}^1 \{\widehat{p}_j(\mathfrak{R}(s, x) + (j-s)\delta) - p_j(\mathfrak{R}(s, x) + (j-s)\delta)\}p_{1-j}(z + (1-j-s)\delta)d\delta ds.
% \end{equation}
% This implies that $a_n(\widehat{\mathfrak{R}}(t, x) - \mathfrak{R}(t, x))$ and $a_n\Phi(t)\int_0^t (\Phi(s))^{-1}\widehat{\gamma}(s, x)ds$ have the same limiting distribution for all $t\in[0, 1]$.

{\bf The Density Estimator.}
The limiting distribution of properly standardized $\widetilde{E}(t, x)$ depends heavily on the density estimators. We provide the following general limiting distribution result for a specific higher-order kernel-based density estimator. 

We describe our estimator of densities $p_0, p_1$. (Although we expect this estimator to be known, we could not find an explicit reference.) For any $z\in\Omega$ and bandwidth $h > 0$, set $V_{z,h} = (\Omega - z)\cap \mathcal{B}(0, h)$. Note that if $z\in\Omega^\circ$ and $h$ is small enough, then $V_{z, h} = \mathcal{B}(0, h)$. Our assumption~\ref{eq:compact-support} on $\Omega$ implies that $V_{z, h}$ is a convex set with a non-empty interior for any $z\in\Omega$ and $h > 0$. Let $K_{z,h}:V_{z,h}\to\mathbb{R}$ be an order $m \ge 1$ kernel function satisfying
\begin{equation}\label{eq:m-order-kernel}
    \int_{V_{z,h}} K_{z, h}(u)du = \mbox{Vol}(V_{z,h}),\quad\mbox{and}\quad \sup_{\substack{(\alpha_1, \ldots, \alpha_d)\in\{0,1,\ldots\}^d,\\\sum_{j=1}^d \alpha_j \le m}}\,\left|\int_{V_{z,h}} K_{z,h}(u)\left(\prod_{j=1}^d u_j^{\alpha_j}\right)du\right| = 0.
\end{equation}
The existence of such a function follows from defining an $L_2$ space on $V_{z,h}$ and applying Gram-Schmidt orthogonalization. In fact, set $\widetilde{K}(u) = \mathbf{1}\{u\in V_{z,h}\}$. For any $u\in\mathbb{R}^d$ and $\alpha\in\{0, 1, \ldots\}^d$ with $\|\alpha\|_1 = \sum_{j=1}^d \alpha_j \le m$, define $\varphi_{\alpha}(u) = \prod_{j=1}^d u_j^{\alpha_j}.$
Set $\varphi_0(u) = 1$ for all $u\in\mathbb{R}^d$. For $u\in V_{z,h}$, define
$\Upsilon_{m}(u) ~=~ (\varphi_{\alpha}(u))_{1 \le \|\alpha\|_1 \le m}$,
where the elements are organized with respect to the total ordering on $\alpha\in\{0, 1, \ldots\}^d$ defined for $\alpha\neq \alpha'$ by $\alpha\prec\alpha'$ if and only if $\|\alpha\|_1 < \|\alpha'\|_1$, or $\|\alpha\|_1 = \|\alpha'\|_1$ and $\alpha_{i^*} < \alpha'_{i^*}$ where $i^* = \min\{1\le i\le d:\, \alpha_i \neq \alpha_i'\}$. 
Define
$\mathcal{B}_{m, h}(z) 
% ~&:=~ \frac{1}{h^d}\int_{\mathbb{R}^d} \Upsilon_m(u'/h)\Upsilon_m^{\top}(u'/h)\mathbf{1}\{\|u'\|_{\infty} \le h,\, u' + z\in\Omega\}du'\\ 
:= \int_{\mathbb{R}^d} \Upsilon_m(u)\Upsilon_m^{\top}(u)\widetilde{K}(u)du.$
Lemma 6.1 of~\cite{bertin2025new} implies that $\mathcal{B}_{m,h}(z)$ is a positive definite matrix.
Set
\[
\widetilde{K}_{z,h}(u) = \widetilde{K}(u)\left(1 - \int_{\mathbb{R}^d} \widetilde{K}(r)\Upsilon_m^{\top}(r)\mathcal{B}_{m,h}(z)^{-1}\Upsilon_m(u)dr\right).
\]
This is the residual obtained from regressing $\tilde{K}(u)$ on $\Upsilon_m(u)$ in $L_2(V_{z,h})$ and hence, satisfies the second condition of~\eqref{eq:m-order-kernel}. To satisfy the first condition of~\eqref{eq:m-order-kernel}, it suffices to normalize $\widetilde{K}_{z,h}(u)$. Note that the boundedness of $V_{z,h}$ implies that $K_{z,h}(\cdot)$, defined this way, is also bounded.

Define the kernel density estimator
\begin{equation}\label{eq:new-boundary-corrected-estimators}
    \widehat{p}_j(z) = \frac{1}{n}\sum_{i=1}^n \frac{K_{z,h}(X_{ji} - z)}{\mbox{Vol}(V_{z,h})}.
\end{equation}
Recall that $X_{j1}, \ldots, X_{jn}$ represents an IID sample from $\mu_j, j = 0, 1$. For faster rates for the density estimator, we need to strengthen the smoothness assumption on the densities $p_0$ and $p_1$. Following~\cite{bertin2025new}, we consider the following assumption:
\begin{enumerate}[label=(B4)]
    \item The densities $p_j, j = 0, 1$ are $\beta$-smooth, i.e., for all $z\in\Omega$, there exists a polynomial $q_{z,j}:\mathbb{R}^d\to\mathbb{R}$ of degree $\llfloor \beta\rrfloor$\footnote[2]{$\llfloor \beta\rrfloor$ is the greatest integer smaller than $\beta$. In particular, $\llfloor \beta\rrfloor = \beta - 1$ if $\beta$ is an integer.}  such that \label{eq:Holder-smooth-densities}
    \[
    |p_j(z + u) - q_{z,j}(u)| \le L\|u\|^{\beta},\quad\mbox{for all }u\in V_{z,h}.
    \]
\end{enumerate}
It is interesting to note that~\ref{eq:Holder-smooth-densities} is not a traditional smoothness assumption in the sense that~\ref{eq:Holder-smooth-densities} does not require any strong notion of differentiability at the boundary of $\Omega$. (See Remark 2 of~\cite{bertin2025new} for more details.)
\begin{proposition}\label{prop:density-estimator-bounded}
    Suppose assumption~\ref{eq:Holder-smooth-densities} holds and that $K_{z,h}:V_{z,h}\to\mathbb{R}$ is a $m$-th order kernel for some $m \ge \llfloor \beta\rrfloor$. Additionally, assume that $|K_{z,h}(u)| \le \mathfrak{K}$ for all $u\in V_{z,h}, z\in\Omega$. Then the density estimator~\eqref{eq:new-boundary-corrected-estimators} satisfies
    \[
    \max_{j = 0, 1}\sup_{z\in\Omega}\left|\mathbb{E}[\widehat{p}_j(z)] - p_j(z)\right| \le L\mathfrak{K}h^{\beta},
    \]
    and
    \[
    \mathrm{Var}(\widehat{p}_j(z)) ~\le~ \frac{\mathfrak{K}^2\|p_j\|_{\infty}}{n\mathrm{Vol}(V_{z,h})},\quad\mbox{for all}\quad z\in\Omega,\; j = 0, 1.
    \]
    Additionally, under assumption~\ref{eq:compact-support}, $\mathrm{Vol}(V_{z,h})/\mathrm{Vol}(\mathcal{B}(0, h))$ stays bounded away from zero for all $z\in\Omega$, and hence, $\mathrm{Vol}(V_{z,h}) \ge Ch^{d}$ for all $z\in\Omega$ (for small enough $h$).
\end{proposition}
A proof of Proposition~\ref{prop:density-estimator-bounded} can be found in Section~\ref{appsubsec:proof-of-density-estimator-bias}. The final part of Proposition~\ref{prop:density-estimator-bounded} appeared as an assumption in several works; see, for example, Assumption X of~\cite{fan2016multivariate}, Definition 2 of~\cite{cuevas1997plug}, and Definition 2.2 of~\cite{cholaquidis2023standardness}. Proposition~\ref{prop:density-estimator-bounded} implies that the variance converges to zero at the rate of $1/(nh^d)$. Additionally, assuming $\mathcal{K} = \{u\mapsto K_{z,h}(u):\, h\ge0,\, z\in\Omega\}$ satisfies $N(\varepsilon, \mathcal{K}) \le C\varepsilon^{-\alpha}$ for some $C>0$ and $\alpha > 0$\footnote[1]{See~\cite{einmahl2005uniform} for details on the covering number.}, the proof of Theorem 1 (and corollary 1) 
% \textcolor{red}{gonzalo:check that numbering is correct} 
of~\cite{einmahl2005uniform} implies that with probability 1,
\begin{equation}\label{eq:supremum-density-rate}
\sup_{h\in[a_n, b_n]}\sup_{z\in\Omega}|\widehat{p}_j(z) - p_j(z)| = O\left(\sqrt{\frac{\log(1/a_n)+\log\log n}{na_n^d}} + b_n^{\beta}\right),\quad\mbox{as}\quad n\to\infty,
\end{equation}
whenever $b_n\to0,$ and $na_n^d/\log n \to \infty$ as $n\to\infty$. (This is an almost sure convergence statement.)

\begin{theorem}[Asymptotic Normality]\label{thm:asymp-normality-bounded-case}
    Fix any $x\in\Omega^\circ$ such that $\mathfrak{R}(s, x)\in\Omega^\circ$ for all $s\in[0,1]$. Consider the density estimators $\widehat{p}_j(\cdot)$ as in~\eqref{eq:new-boundary-corrected-estimators}. Then for any $h \to 0$, we have
    \begin{equation}\label{eq:rate-of-convergence-bounded-case}
    \sup_{t\in[0,1]}\|\widehat{\mathfrak{R}}(t, x) - \mathfrak{R}(t, x)\| = O_p\left(h^{\beta} + \frac{(\log(1/h))^{\mathbf{1}\{d = 1\}/2}}{\sqrt{nh^{d(d-1)_+/(d+1)}}}\right).
    \end{equation}
    Take any $h = o(n^{-1/(2\beta + d(d-1)_+/(d+1))})$ as $n\to\infty$.
    Further suppose 
    \begin{equation}\label{eq:lower-bound-variance-assumption}
    \|\widehat{E}(1, x) - \mathbb{E}[\widetilde{E}(1, x)]\| ~\gg~ \frac{(\log(1/h))^{\mathbf{1}\{d = 1\}/2}}{\sqrt{nh^{d(d-1)_+/(d+1)}}},\footnote{The notation $W_n \gg a_n$ for a sequence of random variable $W_n$ means that $b_nW_n/a_n \overset{p}{\to} \infty$ for any $b_n\to\infty$ as $n\to\infty.$}
    \end{equation}
    then 
    \[
    \frac{(\log(1/h))^{\mathbf{1}\{d = 1\}/2}}{\sqrt{nh^{d(d-1)_+/(d+1)}}}(\widehat{\mathfrak{R}}(1, x) - \mathfrak{R}(1, x)) ~\overset{d}{\to}~ N(0, \Sigma(x)),
    \]
    for some covariance matrix $\Sigma(x)\in\mathbb{R}^{d\times d}$.
\end{theorem}
A proof of Theorem~\ref{thm:asymp-normality-bounded-case} can be found in Section~\ref{appsec:proof-asym-normality-bounded-case}. The first part~\eqref{eq:rate-of-convergence-bounded-case} proves that for $h \asymp n^{-1/(2\beta + d(d-1)_+/(d+1))}$, we get
\[
\sup_{t\in[0,1]}\|\widehat{\mathfrak{R}}(t, x) - \mathfrak{R}(t, x)\| = O_p\left(n^{-\beta/(2\beta + d(d-1)_+/(d+1))}(\log n)^{\mathbf{1}\{d = 1\}/2}\right).
\]
This is better than the usual non-parametric rate of convergence of $n^{-\beta/(2\beta + d)}$ but not as good as the rate of convergence for the optimal transport map. It is not readily obvious if this is the correct rate for rectified flow.
For asymptotic normality, we use a smaller bandwidth for undersmoothing.
The (cumbersome) assumption~\eqref{eq:lower-bound-variance-assumption} is made to ensure that $\sup_{t\in[0,1]}\|\widetilde{E}(t, x)\| = O_p(\|\widetilde{E}(1, x)\|)$. The proof of Theorem~\ref{thm:asymp-normality-bounded-case} proves that $\sup_{t\in[0,1]}\|\tilde{E}(t, x) - \mathbb{E}[\widetilde{E}(t, x)]\|$ is at most the right hand side of~\eqref{eq:lower-bound-variance-assumption}. 

\subsection{Definition of Velocity field (Strictly Convex Domain)}\label{sec:strictly-convex-set-definition}
A set $A\subseteq \mathbb{R}^d$ is said to be strictly convex at a boundary point $a\in\partial A$ if for all $b\in A$ ($b\neq a$), $(a + b)/2 \in A^{\circ}$. 
A quantitative measure of strict convexity at a point $a\in A$ can be defined as
\begin{equation}\label{eq:modulus-at-a-point}
    \mathfrak{m}_a(\varepsilon; A) := \inf\left\{\mbox{dist}\left(\frac{a + b}{2},\, \partial A\right):\, b\in A,\, \|a - b\| \ge \varepsilon\right\},\quad\mbox{for}\quad \varepsilon \in [0, \mbox{diam}(A)].
\end{equation}
This is closely related to the modulus of uniform convexity of a set~\citep{de2024moduli,balashov2009uniform,de2023extension}. However, unlike the existing definitions, we define the modulus of uniform convexity at a point. (Definition 3.1 of~\cite{de2024moduli} can be recovered by taking the infimum over all $a\in A$.) We make this distinction to allow for sets $A$ that are strictly convex at some points but not at others (cf. Observation 3.3(d) of~\cite{de2024moduli}). For any set $A$, define
\[
\mathrm{SC}(A) := \left\{a\in \partial A:\, \mathfrak{m}_a(\varepsilon; A) > 0\quad\mbox{for all}\quad\varepsilon\in(0, \mbox{diam}(A)]\right\}.
\]
It is easy to see that if $a\in\mathrm{SC}(A)$, then $A$ is strictly convex at $a$. One can also prove the converse for compact sets: $A$ is strictly convex at $a$ implies $a\in\mathrm{SC}(A)$, if $A$ is compact. A proof can be obtained by contradiction.\footnote{Suppose, if possible, for some $\varepsilon > 0$, $\mathfrak{m}_a(\varepsilon; A) = 0$. This implies that for any $k\ge1$, there exists $b_k\in A$ such that $\|a - b_k\| \ge \varepsilon$ and $\mbox{dist}((a + b_k)/2,\, \partial A) \le 1/k$.
    The compactness of $A$ implies that every sequence has a subsequence that converges. Let $b^*$ be one such limit. Then the above conditions imply that $\mbox{dist}((a + b^*)/2, \partial\Omega) = 0$ and $\|a - b^*\| \ge \varepsilon$, which is a contradiction to the hypothesis that $a$ is a point of strict convexity.}  
The set $\mathrm{SC}(A)$ is allowed to be empty; in fact, for polyhedral sets $A$, $\mathrm{SC}(A) = \emptyset$~\citep[Observation 3.3(d)]{de2024moduli}.
A set $A\subseteq\mathbb{R}^d$ is said to be strictly convex if it is strictly convex at all its boundary points (or equivalently, $\mathrm{SC}(A) = \partial A$). 

\begin{proposition}[Well-defined Velocity]\label{prop:well-defined-strictly-convex-set}
    Suppose assumptions~\ref{eq:compact-support} and~\ref{eq:bounded-away-densities} hold. Then the velocity field $v(\cdot, \cdot)$ defined as
    \[
    v(t, z) = \begin{cases}\mathbb{E}[X_1] - z, &\mbox{if }t = 0, z\in\Omega,\\
    \frac{\int_{S_t(z)} \delta p_0(z - t\delta)p_1(z + (1-t)\delta)d\delta}{\int_{S_t(z)} p_0(z - t\delta)p_1(z + (1-t)\delta)d\delta}, &\mbox{if }t\in(0, 1), z\in\Omega^\circ,\\
    z - \mathbb{E}[X_0], &\mbox{if }t = 1, z\in\Omega,
    \end{cases}
    \]
    is continuous in $z$ for each $t\in[0, 1]$ on its domain. Additionally, $v(t, z) = 0$ for $t\in(0, 1)$ and $z\in\mathrm{SC}(\Omega)$ is the unique continuous extension of $z\mapsto v(t, z)$ to $\mathrm{SC}(\Omega)$.
\end{proposition}
See Section~\ref{appsubsec:well-defined-strictly-convex} for a proof. Additionally, from the proof, it follows that for any $z\in\mathrm{SC}(\Omega)$ and $z'\in\Omega$, we have
\[
\|v(t, z) - v(t, z')\| \le \frac{\|z' - z\| + \mathfrak{m}_z^{-1}(\|z - z'\|/2;\,\Omega)}{\min\{t, 1 - t\}}.
\]
In particular, this implies that the velocity field is uniformly continuous in the second argument if $\Omega$ is strictly convex. Unfortunately, this does not suffice to establish a unique solution for~\eqref{eq:rectified-flow-integral-equation} when the starting point is at the boundary, because the Lipschitz constant is not necessarily integrable. Even in the univariate case with the domain of $[0, 1]$, one can construct densities $p_0$ and $p_1$ for which the Lipschitz constant at time $t$ behaves like $C/t$ for $t < 1/2$ and $C > 1$. On the other hand, clearly, $z(t) = \mathfrak{R}(t, x) = x$ for all $t\in[0, 1]$ is a solution to~\eqref{eq:rectified-flow-integral-equation} when $x\in\mathrm{SC}(\Omega)$.  

\section{Numerical experiments}

We supplement our theoretical results with simulations highlighting the main features of rectified transport estimators.  In the first experiment, in Figure \ref{fig:gaussiantrajectories}, we show samples from the kernel-regression estimator of trajectories $z_t(x)$, and hence $R(x)=z_1(x)$, for different starting points $x$. In this case, by the same argument as in the proof of Proposition \ref{prop:flow-from-Gaussian-to-Gaussian}, the ground truth trajectories are given by $z_t(x)=\sqrt{t^2+(1-t)^2}x$, and in particular, $R(x)=x$. As we decrease the bandwidth parameter, the bias decreases, but the variance increases. However, the variance remains bounded even for small values of $h$, consistent with our result that the variance is of constant order in the one-dimensional case, Corollary \ref{cor:1d}. Interestingly, although for larger values of $h$ there are significant values at intermediate values of $t$, this tends to be smaller for $R(x)=z_1(x)$.
\begin{figure}[ht!]
   \centering\includegraphics[width=1.0\textwidth]{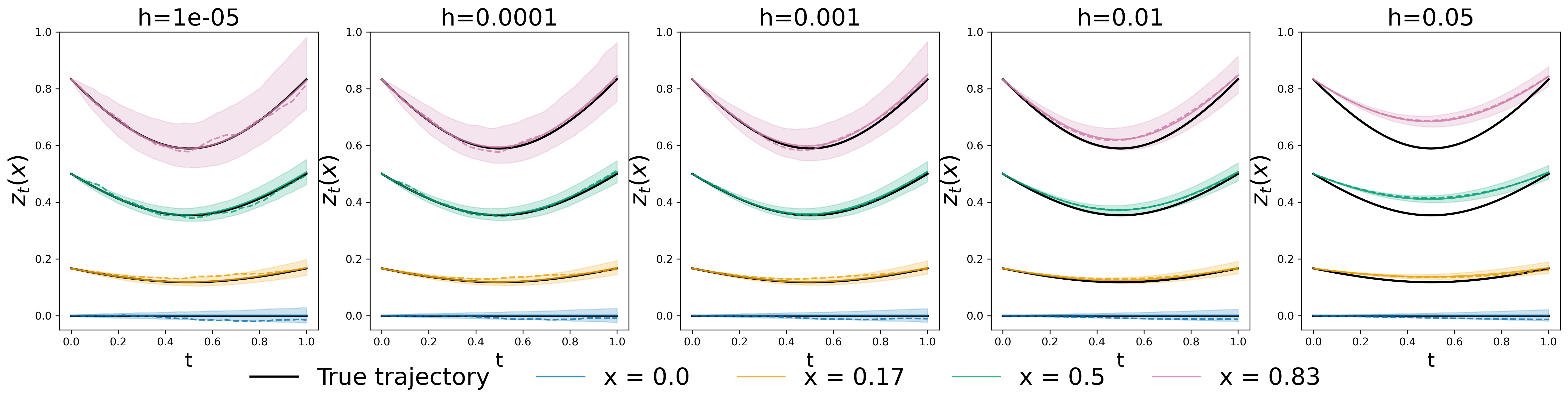}
    
    \caption{Estimating trajectories $z_t(x)$ and $R(x)=z_1(x)$ in the one-dimensional case, estimating the velocity with kernel regression. Each plot shows results for a different choice of bandwidth parameter. We show true and estimated trajectories $z_t(x)=z(t,0,x)$ as a function of different starting points $x$, where $X_0,X_1$ are independent standard Gaussians. Estimators are based on $n=200$ samples. Ground true trajectories are shown in black, and colored solid lines show mean trajectories over 1000 experiment repetitions, for each starting point (different colors). Shades represent 95\% empirical intervals from these repetitions. Additionally, dashed colored lines show one selected sample from the kernel regression-based estimator.} 

    \label{fig:gaussiantrajectories}
\end{figure}

In the second experiment, in Figure \ref{fig:identity}, we show different approaches for estimating the transport map between high-dimensional ($d=50$) standard Gaussians. In this case, the Rectified transport coincides with the optimal transport, the identity function. To visualize our estimates, we plot each coordinate estimated function as a function of a unique variable while making all the other zeros; i.e., in plot $i,j$ we show the function $\hat{R}_i(0,\ldots,x_j,\ldots,0)$. The first is a plug-in estimator for $R(x)$ in the Gaussian case, as described in \eqref{eq:rectgaussian}, where we replace the true means and covariance with their sample versions (purple). In the second estimator, we estimate the velocity by regressing $X_1-X_0$ on $X_t$ using a linear model on all coordinates (orange). The third estimator (blue) is like the second, but we used a cross-validated Lasso instead of linear regression.  Finally, we used kernel regression (green) with bandwidth $h=1.0$ and a Gaussian kernel.

All estimators exhibit reasonable performance, although the linear estimator is typically the worst. Even if the Nadaraya-Watson estimator has some errors, this is reasonable as we did not attempt parameter tuning. The best results are attained for the plug-in and Lasso estimators, illustrating the important point that we should encode in the estimators whatever structure of the underlying function is available to us.  For example, the Lasso regression expresses that each velocity component is a sparse function of its coordinates, and our approach explicitly allows us to encode this in the modeling of the velocity.

\begin{figure}[ht!]
   \centering\includegraphics[width=0.8\textwidth]{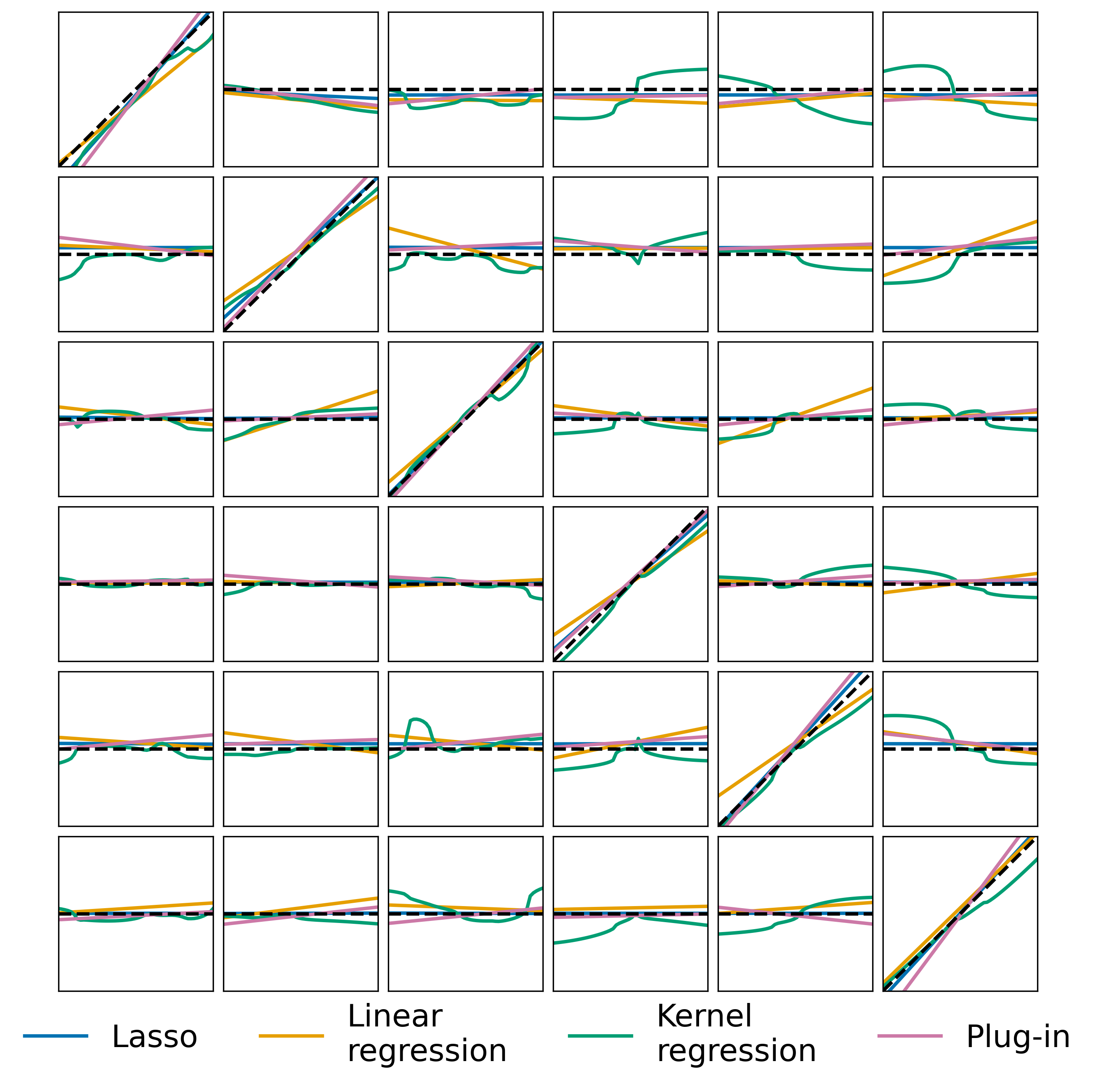}
    
    \caption{Performance of four rectified estimators in the Gaussian case. We drew $n=100$ samples from $X_1,X_0\sim N\left(0,I_d\right)$ with $d=50$.
    In plot $(i,j)$ we show $\hat{R}_i(0,\ldots,x_j,\ldots,0)$ as a function of $x_j\in[-3,3]$. We show only the first 36=6x6 described above. functions. We consider the plug-in estimator (purple), linear regression (orange), cross-validated Lasso (blue), and kernel regression (green). Dashed black lines represent truth $R_i(0,\ldots,x_j,\ldots,0)=x_j \delta_{i=j}$. In all cases, we used a naive ODE discretization by dividing the $[0,1]$ interval into $T=50$ steps.}
    \label{fig:identity}
\end{figure}

\section{Conclusions}\label{sec:conclusions-future-directions}

In this paper, we provide statistical theory and inference for rectified flow for bounded and unbounded random vectors. Although the current work focused on rectified flow that starts with a linear interpolation of random vectors, most of the techniques should carry over to general interpolations discussed in~\cite{albergo2023stochastic}.

Several topics deserve
further investigation.
For example: iterating rectified flow to
approximate optimal transport,
the use of regression and density estimation tools
for estimating rectified flow
and the properties of smoothed rectified flow.
We hope to report on these issues
in the future.

\section*{Acknowledgments.} The authors gratefully thank Dejan Slep{\v c}ev for several insightful discussions, especially for his help with Lemma~\ref{lem:distance-to-boundary-latest} and the proof of Theorem~\ref{thm:rectified-flow-ODE-unique-sol}. AKK also thanks Tomasz Tkocz and Abhimanyu Choudhary for helpful discussions on convex analysis. GM thanks Tudor Manole for discussions on linearization analysis.
We thank Chandramauli Chakraborty who helped with an earlier version of this work. GM was supported by NSF-DMS 2412895.

\bibliography{paper}
\bibliographystyle{apalike}

\newpage
\setcounter{section}{0}
\setcounter{equation}{0}
\setcounter{figure}{0}
\renewcommand{\thesection}{S.\arabic{section}}
\renewcommand{\theequation}{E.\arabic{equation}}
\renewcommand{\thefigure}{A.\arabic{figure}}
\renewcommand{\theHsection}{S.\arabic{section}~}
\renewcommand{\theHequation}{E.\arabic{equation}}
\renewcommand{\theHfigure}{A.\arabic{figure}}
% \tableofcontents
% \titlelabel{\thetitle: }
% \cftsetindents{section}{1em}{2.5em}
% \cftsetindents{subsection}{1.5em}{3em}
% \setcounter{page}{1}
  \begin{center}
  \Large {\bf Supplement to ``Statistical Properties of Rectified Flow''}
  \end{center}
       
\begin{abstract}
This supplement contains the proofs of all the main results in the paper and some additional simulation results.
\end{abstract}

\section{Auxiliary Results from Convex Analysis}
\begin{lemma}\label{lem:concavity-of-distance-to-boundary}
    For a convex set $A\subseteq\mathbb{R}^d$, the function $a\mapsto\mbox{dist}(a,\,\partial A)$ is a concave function.
\end{lemma}
\begin{proof}
    It suffices to prove that for any two points $a_1, a_2\in A$, and $\lambda \in (0, 1)$,
    \begin{equation}\label{eq:concavity-proof}
    \mbox{dist}(a,\,\partial A) \ge \lambda\mbox{dist}(a_1,\,\partial A) + (1-\lambda)\mbox{dist}(a_2,\, \partial A)\quad\mbox{for}\quad a = \lambda a_1 + (1-\lambda) a_2.
    \end{equation}
    If $\varepsilon_j = \mbox{dist}(a_j, \partial A)$, then convexity of $A$ implies $\mathcal{B}(a_j, \varepsilon_j)\subseteq A$. Consider the set
    \[
    \mathcal{C} = \lambda B(a_1, \varepsilon_1) + (1-\lambda)B(a_2, \varepsilon_2) = \left\{\lambda b_1 + (1-\lambda) b_2:\, \|b_j - a_j\| \le \varepsilon_j\mbox{ for }j = 1, 2\right\}. 
    \]
    Clearly, $\mathcal{C}\subseteq A$ by convexity. Additionally, $a\in\mathcal{C}$. To prove~\eqref{eq:concavity-proof}, it now suffices to show that
    \[
    \mathcal{B}(a,\, \lambda\varepsilon_1 + (1-\lambda)\varepsilon_2) \subseteq \mathcal{C}.
    \]
    Take any point $b\in\mathcal{B}(a,\,\lambda\varepsilon_1 + (1-\lambda)\varepsilon_2)$. Then there exists $u\in S^{d-1}$ such that
    \[
    b = a + (\lambda\varepsilon_1 + (1-\lambda)\varepsilon_2)u = \lambda(a_1 + \varepsilon_1u) + (1-\lambda)(a_2 + \varepsilon_2u).
    \]
    By definition, $a_j + \varepsilon_ju\in \mathcal{B}(a_j, \varepsilon_j)$ and hence, the point on the right belongs to $\mathcal{C}$. Therefore, the result follows.
\end{proof}
\begin{lemma}\label{lem:projection-to-boundary}
    For a convex set $A\subseteq\mathbb{R}^d$ and $x\in A$, if $y\in\partial A$ satisfies
    \[
    \|x - y\| = \inf\{\|x - z\|:\, z\in \partial A\},
    \]
    then 
    \[
    (x - y)^{\top}(z - y) \ge 0\quad\mbox{for all}\quad z\in A.
    \]
\end{lemma}
\begin{proof}
    Note for any point $z'\notin A$, convexity of $A$ implies the existence of $z\in \partial A$ such that $z = \lambda x + (1-\lambda)z'$. Therefore, 
    \[
    \|x - z'\| = \|x - z\| + \|z - z'\| \ge \|x - z\| \ge \|x - y\|.
    \]
    This implies that 
    \[
    \|x - y\| = \inf\{\|x - z\|:\, z\in \mbox{closure}(A^c)\}.
    \]
    Proposition 2.2 of~\cite{briec1997minimum} implies the result.
\end{proof}

\begin{lemma}\label{lem:Lipschitz-of-volume}
    Suppose $A\subseteq\mathbb{R}^d$ is a compact convex set such that $\mathcal{B}(0, r) \subseteq A$ for some $r > 0$. Then, for any $\gamma > 0$, 
    \[
    \frac{\mathrm{Vol}(A^{\gamma}\setminus A)}{\mathrm{Vol}(A)} \le \frac{d\gamma}{r}\left(1 + \frac{\gamma}{r}\right)^{d-1}.
    \]
\end{lemma}
\begin{proof}
    Steiner formula (Eq. (4.1) of~\cite{schneider2013convex}) implies that
    \[
    \mbox{Vol}(A^{\gamma}\setminus A) = \sum_{m=1}^d \binom{d}{m}W_m(A)\gamma^{m},
    \]
    where $W_m(A)$ represent the quermassintegrals of $A$. (Here $W_0(A) = \mbox{Vol}(A)$ and $dW_{1}(A) = \mathcal{H}^{d-1}(\partial A)$; see, for example, Eq. (7.7) of~\cite{schneider2013convex}.) This implies that
    \[
    \frac{\mbox{Vol}(A^{\gamma}\setminus A)}{d\gamma W_1(A)} = \sum_{m=1}^{d} \frac{1}{d}\binom{d}{m}\frac{W_m(A)}{W_1(A)}\gamma^{m-1}.
    \]
    From Aleksandrov-Fenchel inequality (Eq. (7.66)of~\cite{schneider2013convex} with $j = m, i= m - 1, k = m+1$), we obtain
    \[
    \frac{W_{m+1}(A)}{W_m(A)} \le \frac{W_m(A)}{W_{m-1}(A)} \le \cdots \le \frac{W_1(A)}{W_0(A)},
    \]
    which, in turn, implies
    \[
    \frac{W_{m}(A)}{W_1(A)} \le \left(\frac{W_1(A)}{W_0(A)}\right)^{m-1}\quad\mbox{for all}\quad m\ge 1. 
    \]
    Therefore, we conclude that
    \[
    \frac{\mbox{Vol}(A^{\gamma}\setminus A)}{d\gamma W_1(A)} \le \sum_{m=0}^{d-1} \binom{d-1}{m}\left(\frac{\gamma W_1(A)}{W_0(A)}\right)^{m-1} \le \left(1 + \frac{\gamma W_1(A)}{W_0(A)}\right)^{d-1}.
    \]
    Equivalently,  
    \[
    \frac{\mbox{Vol}(A^{\gamma}\setminus A)}{\mbox{Vol}(A)} \le \gamma\frac{\mathcal{H}^{d-1}(\partial A)}{\mbox{Vol}(A)}\left(1 + \frac{\gamma}{d}\frac{\mathcal{H}^{d-1}(\partial A)}{\mbox{Vol}(A)}\right)^{d-1}.
    \]
    Lemma 2.1 of~\cite{giannopoulos2018inequalities} (or Remark 13 on page 392 of~\cite{schneider2013convex}) implies
    \[
    \frac{\mathcal{H}^{d-1}(\partial A)}{\mbox{Vol}(A)} \le \frac{d}{r}.
    \]
    Therefore, 
    \[
    \frac{\mbox{Vol}(A^{\gamma}\setminus A)}{\mbox{Vol}(A)} \le \frac{d\gamma}{r}\left(1 + \frac{\gamma}{r}\right)^{d-1},
    \]
    which proves the result.
\end{proof}

\begin{lemma}\label{lem:eventually-strongly-convex}
For a set $A\subseteq\mathbb{R}^d$, suppose $a\in \mathrm{SC}(A)$ and $\{a_k\}_{k\ge1}\subset\partial A$ is a sequence converging to $a$. Then $a_k\in \mathrm{SC}(A)$ for large enough $k$.    
\end{lemma}
\begin{proof}
    We have $\mathfrak{m}_a(\varepsilon; A) > 0$ for any $\varepsilon > 0$. Fix $\varepsilon > 0$. It suffices to show that $\mathfrak{m}_{a_k}(\varepsilon; A) > 0$ for large enough $k$.
    Note that for any $k\ge1$ and $b\in A$,
    \[
    \mbox{dist}\left(\frac{a_k + b}{2},\, \partial A\right) \ge \mbox{dist}\left(\frac{a+b}{2},\,\partial A\right) - (1/2)\|a_k - a\|.
    \]
    This implies that
    \begin{align*}
    \mathfrak{m}_{a_k}(\varepsilon; A) &\ge \inf\left\{\mbox{dist}\left(\frac{a + b}{2},\,\partial A\right):\, b\in A,\,\|a_k - b\| \ge \varepsilon\right\} - (1/2)\|a_k - a\|,\\
    &\ge \inf\left\{\mbox{dist}\left(\frac{a + b}{2},\,\partial A\right):\, b\in A,\,\|a - b\| \ge \varepsilon - \|a_k - a\|\right\} - (1/2)\|a_k - a\|\\
    &= \mathfrak{m}_a(\varepsilon - \|a_k - a\|; A) - (1/2)\|a_k - a\|.
    \end{align*}
    For any $\varepsilon > 0$, there exists $K \ge 1$ such that for all $k \ge K$, $\|a_k - a\| \le \min\{\varepsilon/2,\, \mathfrak{m}_a(\varepsilon/2; A)\}$. This implies that for $k\ge K$,
    \[
    \mathfrak{m}_{a_k}(\varepsilon; A) \ge \mathfrak{m}_a(\varepsilon/2)/2 > 0.
    \]
    This proves the result.
\end{proof}

\section{Additional Results on Ordinary Differential Equations}
\begin{proposition}[Solutions of Specific Volterra Equations of Second Kind]\label{prop:Volterra}
    Suppose $a:[0,1]\to\mathbb{R}^d$ is an integrable function and $b:[0,1]\to\mathbb{R}^{d\times d}$ be a bounded (in operator norm) function, i.e., $\|b(s)\|_{\mathrm{op}} \le M < \infty$. Consider the integral equation
    \begin{equation}\label{eq:Volterra-equation}
    w(t) = \int_0^t a(s)ds + \int_0^t b(s)w(s)ds\in\mathbb{R}^d,\quad\mbox{for all}\quad t\in[0, 1],\quad\mbox{with}\quad w(0) = 0.
    \end{equation}
    Then~\eqref{eq:Volterra-equation} has a unique solution and moreover, 
    \[
    w(t) = \Phi(t)\int_0^t (\Phi(s))^{-1}a(s)ds,
    \]
    where $\Phi:[0,1]\to\mathbb{R}^{d\times d}$ is the unique invertible solution to the 
    \begin{equation}\label{eq:fundamental-matrix-Volterra}
    \Phi(t) = I + \int_0^t b(s)\Phi(s)ds,\quad\mbox{for all}\quad t\in[0, 1],\quad\mbox{with}\quad \Phi(0) = I\in\mathbb{R}^{d\times d}.
    \end{equation}
    Moreover, \begin{equation}\label{eq:derphit} (d/dt)(\Phi(t))^{-1} = -(\Phi(t))^{-1}b(t)\end{equation} for almost all $t\in[0, 1]$ and 
    \[
    \|\Phi(1)(\Phi(t))^{-1}\|_{\mathrm{op}}= \|(\Phi(t))^{-1}\Phi(1)\|_{\mathrm{op}} \le \|\Phi(1)\|_{\mathrm{op}}\exp(Mt)\quad\mbox{ for all }t\in[0, 1].
    \]
\end{proposition}
\begin{proof}
    The uniqueness result follows from Theorem 1.2.3 of~\cite{brunner2017volterra}. (The assumption of continuous kernel (i.e., $K\in C(D)$, in the notation of~\cite{brunner2017volterra} is not needed when $\|b(s)\|_{\mathrm{op}} \le M$ uniformly over $s\in[0,1]$. In this case, we interpret the solution as the Carath{\'e}odory solution. Moreover, the proof extends to the vector-valued case.) Also, see Theorem 3.1 (Chapter 3) of~\cite{Coddington55}. 

    Note that the integral equation~\eqref{eq:fundamental-matrix-Volterra} is of the same type as~\eqref{eq:Volterra-equation} and hence from Theorem 1.2.3 of~\cite{brunner2017volterra}, we get the existence and uniqueness of the solution for~\eqref{eq:fundamental-matrix-Volterra}. Because $\Phi(\cdot)$ solves~\eqref{eq:fundamental-matrix-Volterra}, it is absolutely continuous and almost everywhere differentiable. This implies that for almost all $t\in[0, 1]$, 
    \[
    \frac{d\mbox{det}(\Phi(t))}{dt} = (\mbox{tr}(b(t)))\mbox{det}(\Phi(t)).
    \]
    (See, for example, Theorem 7.3 of~\citet[Chapter 1]{Coddington55}.) Hence, 
    \[
    \mbox{det}(\Phi(t)) = \mbox{det}(\Phi(0))\exp\left(\int_0^t \mbox{tr}(b(s))ds\right) = \exp\left(\int_0^t \mbox{tr}(b(s))ds\right),
    \]
    which implies that $\Phi(t)$ is invertible for all $t\in[0, 1]$. (This equation for the determinant is called the Abel–Jacobi–Liouville identity.)

    To prove the final part, note that from Eq. (1.6) of~\citet[Chapter 3]{Coddington55}, we get $(d/dt)(\Phi(t))^{-1} = -(\Phi(t))^{-1}b(t)$ for almost all $t\in[0, 1]$ with $(\Phi(0))^{-1} = I$. Equivalently, $(\Phi(t))^{-1} = I + \int_0^t b(s)(\Phi(s))^{-1}ds$ for all $t\in[0, 1]$. Therefore, 
    \begin{align*}
    \|(\Phi(t))^{-1}\Phi(1)\|_{\mathrm{op}} &\le \|\Phi(1)\|_{\mathrm{op}} + \int_0^t \|(\Phi(s))^{-1}\Phi(1)\|_{\mathrm{op}}\|b(s)\|_{\mathrm{op}}ds\\ 
    &\le \|\Phi(1)\|_{\mathrm{op}} + C\int_0^t \|(\Phi(s))^{-1}\Phi(1)\|_{\mathrm{op}}ds.
    \end{align*}
    Solving this inequality, we get $\|(\Phi(t))^{-1}\Phi(1)\|_{\mathrm{op}} \le \|\Phi(1)\|_{\mathrm{op}}\exp(Mt)$ for all $t\in[0, 1]$. Following the same strategy with the differential equation defining $\Phi(\cdot)$, we get $\|\Phi(1)\| \le \exp(M)$. Hence, $\max_{s\in[0, 1]}\|\Phi(1)(\Phi(s))^{-1}\|_{\mathrm{op}} \le \exp(2M)$.
\end{proof}

\section{Proofs of Results in Section~\ref{sec:rectified-flow}}
\subsection{Proof of Theorem~\ref{thm:uniquene-ae-implies-transport-map}}\label{appsubsec:proof-unique-almost-surely}
Note that $t\mapsto X_t$ is a differentiable map and hence, for any continuously differentiable function $\varphi:\Omega\to\mathbb{R}$ (with bounded derivative), we have by the dominated convergence theorem that
\[
\frac{d}{dt}\mathbb{E}[\varphi(X_t)] = \mathbb{E}\left[\varphi'(X_t)(X_1 - X_0)\right] = \mathbb{E}[\varphi'(X_t)\mathbb{E}[X_1 - X_0|X_t]] = \mathbb{E}[\varphi'(X_t)v(t, X_t)].
\]
Here, the second equality follows from the law of iterated expectations, and the third equality follows from the definition of the velocity field $v$.
Hence, $\mu_t$, the law of $X_t$, satisfies the continuity equation (Eq. (8.1.1) of~\cite{AmbrosioGigliSavare2008}; see the equivalent formulation (8.1.4) there). Hence, by Theorem 8.2.1(ii) of~\cite{AmbrosioGigliSavare2008}, we conclude that 
\[
\mathbb{E}[\varphi(X_t)] = \int_{\Omega\times\Gamma} \varphi(\gamma_x(t))d\eta(x, \gamma_x),
\]
where $\Gamma$ is the collection of functions from $[0, 1]$ to $\Omega$, and $\eta(\cdot, \cdot)$ is a probability measure concentrated on the set of pairs $(x, \gamma_x)$ such that $\gamma_x:[0,1]\to\Omega$ is absolutely continuous and solves $\gamma_x'(t) = v(t, \gamma_x(t))$ for almost all $t\in(0, 1)$ with $\gamma_x(0) = x$. From the hypothesis of a unique solution to~\eqref{eq:rectified-flow-integral-uniqueness} for almost all $x$, we get that
\[
\int_{\Omega\times\Gamma} \varphi(\gamma_x(t))d\eta(x, \gamma_x) = \int_{\Omega} \varphi(\mathfrak{R}(t, x))d\mu_0(x) = \mathbb{E}[\varphi(\mathfrak{R}(t, X_0))].
\]
Therefore, $\mathbb{E}[\varphi(X_t)] = \mathbb{E}[\varphi(\mathfrak{R}(t, X_0))]$ for all $t\in[0, 1]$. Hence, $\mathfrak{R}(t, X_0)$ has the same distribution as $X_t$ for all $t\in[0, 1].$ This completes the proof.
% The proof is a slightly more detailed argument than the one in the proof of Theorem 3.2 in \cite{liu2022rectified}. The transport map is well defined in $A$ since it is the unique solution to an ODE. To show that it is a valid map, we rely on Theorem 3.1 in \cite{ambrosio2014continuity}: that the uniqueness of solutions for all $x\in A$ is equivalent to the uniqueness of solutions to the continuity equation \begin{equation}\label{eq:continuity} \frac{d}{dt}\mu_t(z)+\nabla\cdot\left(v(t,z)\mu_t(z)\right)=0,\quad (t,z)\in[0,1]\times \mathbb{R}^d,\end{equation}
% for any starting $\mu_0$ such that $\mu_0(A)=1$. Using this fact, it suffices to show that both the laws of $\mathfrak{R}(t, X_0)$ and the interpolation process $X_t=tX_1+(1-t)X_0$ satisfy \eqref{eq:continuity} since this will imply that they are the same. By construction, $\mathfrak{R}(t, x)$ satisfies \eqref{eq:continuity}. The fact that $X_t$ also satisfies \eqref{eq:continuity} is shown in \cite{liu2022flow} (integration by parts), so we skip it. 

\subsection{Proof of Lemma~\ref{lemma::gaussianvel}}
We have
\[
\begin{pmatrix}
        X_1 - X_0 \\
        tX_1 + (1-t)X_0
    \end{pmatrix} = \begin{pmatrix}
        I_d & -I_d\\
        tI_d & (1-t)I_d
    \end{pmatrix} \begin{pmatrix}
        X_1 \\
        X_0
\end{pmatrix}
\]
Here $I_d$ is the identity matrix with $d\times d$ order, and 
\[
    \begin{pmatrix}
        X_1 \\
        X_0
    \end{pmatrix} \sim N\bigg(\begin{pmatrix}
        m_1\\
        m_0
    \end{pmatrix},\begin{pmatrix}
        \Sigma_1 & 0 \\
        0 & \Sigma_0
    \end{pmatrix}\bigg)
\]
Then, 
\[  \begin{pmatrix}
        X_1 - X_0 \\
        tX_1 + (1-t)X_0
    \end{pmatrix} \sim N\bigg(\begin{pmatrix}
        m_1-m_0 \\
        tm_1+(1-t)m_0
    \end{pmatrix},\begin{pmatrix}
            \Sigma_1 + \Sigma_0 & t\Sigma_1 - (1-t)\Sigma_0\\
            t\Sigma_1 - (1-t)\Sigma_0 & t^2\Sigma_1 +(1-t)^2\Sigma_0
        \end{pmatrix}\bigg)
\]
and, therefore,
%\[
%     X_1 - X_0 \mid X_t\sim \mathcal{N}_{d}\bigg(m_1 - m_0 + (t\Sigma_1 - (1-t)\Sigma_0)(t^2\Sigma_1 +(1-t)^2\Sigma_0)^{-1}(X_t - m_t),\bigg)
%\]
%where $m_t = tm_1+(1-t)m_0$, and thereby,
\begin{align*}
    \mathbb{E}(X_1 - X_0\mid X_t) &= m_1 - m_0 + (t\Sigma_1 - (1-t)\Sigma_0)(t^2\Sigma_1 +(1-t)^2\Sigma_0)^{-1}(X_t - m_t),
\end{align*}
and hence
\begin{equation}
\label{eq:vtgaussian}
v_t(z) = m_1 - m_0 + (t\Sigma_1 - (1-t)\Sigma_0)(t^2\Sigma_1 +(1-t)^2\Sigma_0)^{-1}(z - m_t).
\end{equation}

\subsection{Proof of Proposition~\ref{prop:flow-from-Gaussian-to-Gaussian}}\label{appsec:proof-of-proposition-flow-Gaussian-to-Gaussian}
We first need a lemma.
\begin{lemma}\label{lemma::gaussianvel2}
Let 
$X_0 \sim N(m_0,\Sigma_0)$,
$X_1 \sim N(m_1,\Sigma_1)$,
and suppose that $X_1=A X_0+b$ so that $A\Sigma_0A^\top=\Sigma_1$ and $b=Am_0-m_1$. Also, suppose that $\Sigma_0$ and $\Sigma_1$ are invertible and that $A$ is positive semidefinite. Then, 
\begin{equation}
\label{eq:gaussianvelT}
v_t(z) = m_1 - m_0 + (A-I_d)(tA+ (1-t)I_d)^{-1}(z - m_t).
\end{equation}
\end{lemma}
\begin{proof}
The proof is similar to the one of Lemma \ref{lemma::gaussianvel}. In this case, we have
\[
\begin{pmatrix}
        X_1 - X_0 \\
        tX_1 + (1-t)X_t
    \end{pmatrix} = \begin{pmatrix}
        I_d & -I_d\\
        tI_d & (1-t)I_d
    \end{pmatrix} \begin{pmatrix}
        X_1 \\
        X_0
\end{pmatrix}
\]
and 
\[
    \begin{pmatrix}
        X_1 \\
        X_0
    \end{pmatrix} \sim N\bigg(\begin{pmatrix}
        m_1\\
        m_0
    \end{pmatrix},\begin{pmatrix}
        \Sigma_1 & A\Sigma_0 \\
        \Sigma_0A^\top & \Sigma_0
    \end{pmatrix}\bigg).
\]
Then,
\[  \begin{pmatrix}
        X_1 - X_0 \\
        tX_1 + (1-t)X_0
    \end{pmatrix} \sim N\bigg(\begin{pmatrix}
        m_1-m_0 \\
        tm_1+(1-t)m_0
    \end{pmatrix},\tilde{\Sigma}\bigg)
\]
with
\begin{eqnarray*}\tilde{\Sigma}&=&\begin{pmatrix}
            A\Sigma_0A^\top + \Sigma_0 & tA\Sigma_0A^\top - (1-t)\Sigma_0+(1-t)A\Sigma_0-t\Sigma_0A^\top\\
           tA\Sigma_0A^\top - (1-t)\Sigma_0+(1-t)\Sigma_0A^\top-tA\Sigma_0 & t^2A\Sigma_0A^\top +(1-t)^2\Sigma_0+t(1-t)\left(A\Sigma_0+\Sigma_0A^\top\right)
        \end{pmatrix}
        \\ &=& \begin{pmatrix}
            A\Sigma_0A^\top + \Sigma_0 & (A-I_d)\Sigma_0(tA^\top +(1-t)I_d) \\
          (tA+(1-t)I_d)\Sigma_0 (A^\top -I_d) & (tA+(1-t)I_d)\Sigma_0(tA^\top +(1-t)I_d)
        \end{pmatrix} \end{eqnarray*}
Therefore,   
\begin{eqnarray*}
\label{eq:vtgaussian}
v_t(z) &=& m_1 - m_0 + (A-I_d)\Sigma_0(tA^\top +(1-t)I_d)\left((tA +(1-t)I_d)\Sigma_0(tA^\top +(1-t)I_d)\right)^{-1}(z - m_t)\\
&=& m_1 - m_0 + (A-I_d)(tA +(1-t)I_d)^{-1}(z - m_t).
\end{eqnarray*}
In the last line, we have used the fact that $tA+(1-t)I_d$ is an invertible matrix. In turn, this follows from the fact that $A$ is invertible and hence positive definite. To see that $A$ is invertible we argue by contradiction. If it was not, let $m<d$ be the rank of $A$. Then, $d=rank(\Sigma_1)=rank(A\Sigma_0A^T)\leq \min\{d,m\}=m$ which is impossible.
\end{proof}

\begin{proof}[Proof of Proposition~\ref{prop:flow-from-Gaussian-to-Gaussian}]
Let's compute the map $R_1(x)$ resulting from the first iteration. We have, by Lemma \ref{lemma::gaussianvel} that
\begin{equation}
v_t(z) = m_1 - m_0 + (t\Sigma_1 - (1-t)\Sigma_0)(t^2\Sigma_1 +(1-t)^2\Sigma_0)^{-1}(z - m_t).
\end{equation}
Denote $C\equiv \Sigma_0^{-1/2}\Sigma_1\Sigma_0^{-1/2}$. Then, we can write
\begin{equation}
v_t(z) = m_1 - m_0 + \Sigma_0^{1/2}(tC - (1-t)I_d)(t^2C +(1-t)^2I_d)^{-1}\Sigma_0^{-1/2}(z - m_t).
\end{equation}
Write the eigendecomposition of $C$ as $C=P\Lambda P^\top$. Then,
\begin{equation}\label{eq:velg}
v_t(z) = m_1 - m_0 + \Sigma_0^{1/2}P(t\Lambda - (1-t)I_d)(t^2\Lambda +(1-t)^2I_d)^{-1}P^\top\Sigma_0^{-1/2}(z - m_t).
\end{equation}
Denote by $z_t$ the solution to \eqref{eq:rectified-flow-ODE} with velocity given by \eqref{eq:velg} and initial condition $z_0=x$. Define the variable $y_t=P^\top\Sigma_0^{-1/2}z_t$. This variable satisfies the ODE
\begin{eqnarray*}dy_t &=&P^\top \Sigma_0^{-1/2}dz_t \\ &=& \left(P^\top \Sigma_0^{-1/2}(m_1 - m_0) + (t\Lambda - (1-t)I_d)(t^2\Lambda +(1-t)^2I_d)^{-1}P^\top\Sigma_0^{-1/2}(z - m_t)\right)dt\\
&=& \left(P^\top \Sigma_0^{-1/2}(m_1 - m_0) + (t\Lambda - (1-t)I_d)(t^2\Lambda +(1-t)^2I_d)^{-1}(y - P^\top \Sigma_0^{-1/2}m_t)\right)dt.\end{eqnarray*}
With $y_0=P^\top \Sigma_0^{-1/2}z_0=P^\top \Sigma_0^{-1/2}x.$ This ODE is of the form $dy_t=(a_t+p_ty_t)dt$ where $a_t$ is a vector and $p_t$ is a diagonal matrix. We can solve this equation component-wise. Using, for each coordinate, the integrating factor $g_t=\exp\left(\int_0^t -D_s ds\right)$ we have the relation
$\frac{d }{dt}(y_t g_t)=a_tg_t$. By integrating both sides between 0 and t we obtain
$$y_t=g_t^{-1}\left(y_0+\int^t_0 a_sg_sds\right).$$
In particular, we observe that each coordinate of $Y_t$ is an affine function of the initial condition $Y_0$ (and hence, $X$). Therefore, for some vector $b$,
$$y_1=b+g_1^{-1}y_0,$$ where 
\begin{eqnarray*} g_1&=&\exp\left(-\int_0^1 (t\Lambda - (1-t)I_d)(t^2\Lambda +(1-t)^2I_d)^{-1}ds\right) \\ &=& 
\exp\left(-\frac{1}{2}\log (t^2\Lambda +(1-t)^2I_d) \Big \rvert_0^1 \right) \\ &=&
\exp\left(-\frac{1}{2}\Lambda  \right)  = \sqrt{\Lambda}^{-1}. 
\end{eqnarray*}

We have found that for some vector $b$,
$y_1=\sqrt{\Lambda}y_0+b$. Expressing the above in terms of the original variables $z$,
$$P^\top \Sigma^{-1/2}_0 z_1=\sqrt{\Lambda}P^\top \Sigma^{-1/2}_0 z_0 +b.$$
Equivalently, since $R_1(x)=z_1$,
\begin{eqnarray*}R_1(x)&=&\Sigma_0^{1/2}P\sqrt{\Lambda}P^\top \Sigma^{-1/2}_0 x +b'\\&=& 
\Sigma_0^{1/2}C^{1/2} \Sigma^{-1/2}_0 x +b'\\ &=& 
\Sigma_0^{1/2}\left(\Sigma_0^{-1/2}\Sigma_1\Sigma^{-1/2}_0\right)^{1/2} \Sigma^{-1/2}_0 x +b',
\end{eqnarray*}
where $b'$ is another constant vector, whose value we can compute explicitly. Indeed, since $R_1(X_0)\sim N\left(m_1,\Sigma_1\right)$  where $X_0\sim N\left(m_0,\Sigma_0\right)$ we can simply express
$$R_1(x)=m_1 +\Sigma_0^{1/2}\left(\Sigma_0^{-1/2}\Sigma_1\Sigma^{-1/2}_0\right)^{1/2} \Sigma^{-1/2}_0 \left(x-m_0\right).$$

We have found the solution to the first iteration. It remains to show that the second iteration leaves this initial iteration unchanged. The key fact is that after one iteration the initial coupling between $X_0$ and $X_1$ is deterministic (instead of the independence coupling), so we can use Lemma \ref{lemma::gaussianvel2}. We can use a similar argument to solve the resulting ODE with velocity field 
$$v_t(z)=m_1 - m_0 + (A-I_d)(tA +(1-t)I_d)^{-1}(z - m_t),$$
where $A=\Sigma_0^{1/2}C\Sigma_0^{-1/2}$ and $C=\left(\Sigma_0^{-1/2}\Sigma_1\Sigma^{-1/2}_0\right)^{1/2}$ is a symmetric matrix. Therefore,
$$v_t(z)=m_1 - m_0 + \Sigma_0^{1/2}(C-I_d)(tC +(1-t)I_d)^{-1}\Sigma_0^{-1/2}(z - m_t).$$
We can use the same diagonalization argument as before to decouple the system and solve each coordinate separately. Expressing $C=P\Lambda P^t$, this time 
\begin{eqnarray*} g_1&=&
\exp\left(-\int_0^1 (\Lambda  -I_d)(t\Lambda +(1-t)I_d)^{-1}ds\right) \\ 
&=& 
\exp\left(-\log (t\Lambda +(1-t)I_d) \Big \rvert_0^1 \right) =
\exp\left(-\Lambda  \right)  = \Lambda^{-1}. 
\end{eqnarray*}
Consequently, calling $R_2(X)$ the second iteration of the rectified map
\begin{eqnarray*} 
R_2(X)&=&m_1+\Sigma_0^{1/2}P\Lambda P^\top\Sigma_0^{-1/2}\left(X-m_0\right)\\ 
&=& m_1+\Sigma_0^{1/2}C\Sigma_0^{-1/2}\left(X-m_0\right)= R_1(X).
\end{eqnarray*}
\end{proof}

\subsection{Proof of Lemma \ref{lemma::mixture}}\label{appsec:proof-of-lemma-mixture}
When $X_0$ and $X_1$ are mixtures of Gaussians, then we have,
\begin{align*}
    v_t(z)&= \frac{\int y \bigg(\sum_{i=1}^{I_0}\pi_0^i p^i_{0}(z-ty)\bigg)\bigg(\sum_{j=1}^{I_1}\pi^j_{1}p_1^{j}(z+(1-t)y)\bigg)dy}{\int \bigg(\sum_{i=1}^{I_0}\pi^i_{0}p^i_{0}(z-ty)\bigg)\bigg(\sum_{j=1}^{I_1}\pi^j_{1}p_1^{j}(z+(1-t)y)\bigg)dy}\\
    & = \frac{\sum_{i,j}^{I_0,I_1}\pi^i_{0}\pi_1^{j}\int y p^i_{0}(z-ty) p^j_{1}(z+(1-t)y) dy}{\sum_{i,j}^{I_0,I_1}\pi^i_{0}\pi^j_{1}\int p^i_{0}(z-ty)p^{j}_1(z+(1-t)y)dy} \\
    & = \frac{\sum_{i,j}^{I_0,I_1}\pi^i_{0}\pi^j_{1}v^{i,j}_t(z)\tau^{i,j}_t(z)}{\sum_{i,j}^{I_0,I_1}\pi_{0,i}\pi_{1,j}\tau^{i,j}_t(z)},
\end{align*}
where
\begin{align*}
    v^{i,j}_t(z) & = \frac{\int y p^i_{0}(z-ty) p^j_{1}(z+(1-t)y) dy}{\int p^i_{0}(z-ty)p^j_{1}(z+(1-t)y)dy}\\
    \tau^{i,j}_t(z) & = \int p^i_{0}(z-ty)p^j_{1}(z+(1-t)y)dy
\end{align*}
and 
 $$p^i_0 \sim N(\mu^i_0,\Sigma^i_{0}),
    \quad \quad p^j_1 \sim N(\mu^j_1,\Sigma^j_{1}).$$
Each of the $v^{i,j}_t(z)$ is the velocity field between pairs of mixture components $X_0^i\sim N(\mu^i_0,\Sigma^i_{0})$ and $X_1^j\sim N(\mu^j_1,\Sigma^j_{1})$ so that by Equation \eqref{eq:gaussianvel} 
$$v^{i,j}_t(z)=\mu^j_1 - \mu^i_0 + (t\Sigma^j_1 - (1-t)\Sigma^i_0)(t^2\Sigma^j_1 +(1-t)^2\Sigma^i_0)^{-1}(z - (t\mu_1^j-(1-t)\mu^i_0).$$

It remains to compute $\tau^{i,j}(z,t)$. By simplicity remove $i,j$ and focus on the integral $$\tau^{i,j}(z,t)=\int p_{0}(z-ty)p_{1}(z+(1-t)y)dy.$$
Consider the change of variables $x=z+(1-t)y$ (the case $t=1$ is trivial). Then, we write
$$\tau_{i,j}(z,t)=\frac{1}{(1-t)^d}\int p_{0}\left(\frac{1}{1-t}-\frac{t}{1-t}x\right)p_{1}(x)dx.$$

Define the quantities $a=\frac{z}{1-t}$ and $b=\frac{t}{1-t}$ so that
$$\tau^{i,j}_t(z)=\frac{1}{(1-t)^d}\int p_{0}(a-bx)p_{1}(x)dx.$$
We can interpret the above integral as the marginal likelihood of $a$ in the model
$X=bX_1+X_0$ where $X_0\perp X_1$ and $X_{0i}\sim N(\mu_i,\Sigma_i)$. The random variable $X$ then distributes $$X\sim N\left(b\mu_1+\mu_0,b^2\Sigma_1+\Sigma_0\right)=N\left(\frac{1}{1-t}\left(t\mu_1+(1-t)\mu_0\right),\frac{1}{(1-t)^2}\left(t^2\Sigma_1+(1-t)^2\Sigma_0\right)\right).$$
If $p_X$ denotes the density of $X$ and $\mu_t=t\mu_0+(1-t)\mu_1, \Sigma_t=t^2\Sigma_1+(1-t)^2\Sigma_0$, then
\begin{align*} \tau^{i,j}_t(z)&=\frac{1}{(1-t)^d}p_X(a)\\ &=\frac{1}{(1-t)^d}p_X\left(\frac{z}{1-t}\right)\\
&=\frac{1}{(1-t)^d}\frac{1}{\det\left(2\pi\frac{\Sigma_t}{(1-t)^2}\right)^{1/2}}\exp\left(-\frac{1}{2}\frac{1}{1-t}\left(z-\mu_t\right)^\top \left(\frac{\Sigma_t}{ (1-t)^2}\right)^{-1} \frac{1}{1-t}\left(z-\mu_t\right)^\top\right) \\
&=\frac{1}{\det\left(2\pi\Sigma_t\right)^{1/2}}\exp\left(-\frac{1}{2}\left(z-\mu_t\right)^\top \Sigma_t^{-1} \left(z-\mu_t\right)^\top\right)\\
&=N\left(z;\mu_t,\Sigma_t\right).
\end{align*}

\subsection{Proof of Proposition~\ref{prop:consistency-of-density-estimator}}\label{appsec:proof-of-consistency-of-density-estimator}
    Observe that
    \begin{align*}
        \hat{f}_t(z) - f_t(z) &= \int_{\mathbb{R}^d} \delta \{\hat{p}_0(z - t\delta) - p_0(z - t\delta)\}\hat{p}_1(z + (1 - t)\delta)d\delta\\
        &\quad + \int_{\mathbb{R}^d} \delta p_0(z - t\delta)\{\hat{p}_1(z + (1 - t)\delta) - p_1(z + (1 - t)\delta)\}d\delta\\
        &= \mathbf{I} + \mathbf{II}.
    \end{align*}
    We start with $\mathbf{II}$. By H{\"o}lder's inequality, we have
    \begin{align*}
    \|\mathbf{II}\| &\le \left(\int_{\mathbb{R}^d} \|\delta\|^2p_0^2(z - t\delta)d\delta\right)^{1/2}\left(\int_{\mathbb{R}^d} |\hat{p}_1(z + (1 - t)\delta) - p_1(z + (1 - t)\delta)|^2d\delta\right)\\
    &= t^{-d/2}(1 - t)^{-d/2}\left(\int_{\mathbb{R}^d} \|z - u\|^2p_0^2(u)du\right)^{1/2}\left(\int_{\mathbb{R}^d} |\hat{p}_1(u) - p_1(u)|^2du\right)^{1/2}.
    \end{align*}
    On the other hand, 
    \begin{align*}
        \|\mathbf{I}\| &\le \left(\int_{\mathbb{R}^d} \|\delta\|^2\hat{p}_1^2(z + (1-t)\delta)d\delta\right)^{1/2}\left(\int_{\mathbb{R}^d} |\hat{p}_0(z - t\delta) - p_0(z - t\delta)|^2d\delta\right)\\
        &= (1-t)^{-d/2}t^{-d/2}\left(\int_{\mathbb{R}^d} \|z - u\|^2\hat{p}_1^2(u)du\right)^{1/2}\left(\int_{\mathbb{R}^d} |\hat{p}_0(u) - p_0(u)|^2du\right)^{1/2}.
    \end{align*}
    To bound the first integral, we note that
    \[
    \int_{\mathbb{R}^d} \|z - u\|^2\hat{p}_1^2(u)du \le \|\hat{p}_1\|_{\infty}\int_{\mathbb{R}^d} \|z - u\|^2\hat{p}_1(u)du \le \|\hat{p}_1\|_{\infty}\left(\int_{\mathbb{R}^d} \|z - u\|^2p_1(u)du + \bar{\zeta}_2(\hat{p}_1, p_1)\right).
    \]
    Hence, 
    {\small\[
    \|\hat{f}_t(z) - f_t(z)\| \le \frac{\|\hat{p}_1 - p_1\|_2\|p_0\|_{\infty}(\mathbb{E}\|z - X_0\|^2)^{1/2} + \|\hat{p}_0 - p_0\|_2\|\hat{p}_1\|_{\infty}(\mathbb{E}\|z - X_1\|^2 + \bar{\zeta}_2(\hat{p}_1, p_1))^{1/2}}{t^{d/2}(1 - t)^{d/2}}. 
    \]}
    Following the same strategy for $\hat{p}_t(\cdot)$ yields
    \[
    |\hat{p}_t(z) - p_t(z)| \le \frac{\|p_0\|_{\infty}\|\hat{p}_1 - p_1\|_2 + \|\hat{p}_1\|_{\infty}\|\hat{p}_0 - p_0\|_2}{t^{d/2}(1 - t)^{d/2}}.
    \]
    \subsection{Proof of Theorem~\ref{thm:semiparametric-rate}}\label{appsubsec:proof-semiparametric-rate}
By the chain rule,
the efficient influence function is
$$
\varphi = \frac{\varphi_N}{p_t} - v_t(z) \frac{\varphi_D}{p_t} 
$$
where
$\varphi_N$ is the efficient influence function of the numerator and
$\varphi_D$ is the efficient influence function of the denominator.
There are two distributions so
$$
\varphi = \frac{\varphi_{N0} + \varphi_{N1}  }{p_t(z)} - v_t(z) \frac{\varphi_{D0}+\varphi_{D1}}{p_t(z)} =
\frac{\varphi_{N0}}{p_t(z)} - v_t(z) \frac{\varphi_{D0}}{p_t(z)} +
\frac{\varphi_{N1}}{p_t(z)} - v_t(z) \frac{\varphi_{D1}}{p_t(z)} \equiv
\varphi_0 + \varphi_1.
$$
The equations for these functions
follows by computing the Gateaux derivative.
Now consider the limiting distribution.
First consider $\psi_D$.
The von Mises expansion is
\begin{align*}
\psi_D &= \hat\psi_{pi} + \int \varphi_0(x,\hat p_1)d\mu_0 + \int \varphi_1(x,\hat p_0)d\mu_1 + R_n\\
\end{align*}
where $R_n$ is a second order remainder.
Ignoring $R$ we have,
and writing $\psi_D$ as $\psi$ for simplicity,
\begin{align*}
\hat\psi - \psi &=
\psi_{pi}+ \frac{1}{n}\sum_i \varphi(X_i,\hat p_1) + \frac{1}{n}\sum_i \varphi(Y_i,\hat p_0) - \psi(p_0,p_1)\\
&=
\psi_{pi} +
\frac{1}{n}\sum_i \varphi(X_i,\hat p_1) + \frac{1}{n}\sum_i \varphi(Y_i,\hat p_0) - 
\psi_{pi} - \int \varphi_0(x,\hat p_1)d\mu_0 - \int \varphi_1(x,\hat p_0)d\mu_1 \\
&=
\frac{1}{n}\sum_i \varphi(X_i,\hat p_1) + \frac{1}{n}\sum_i \varphi(Y_i,\hat p_0) - 
\int \varphi_0(x,\hat p_1)d\mu_0 - \int \varphi_1(x,\hat p_0)d\mu_1 \\
&=
\frac{1}{n}\sum_i \varphi(X_i,p_1) + \frac{1}{n}\sum_i \varphi(Y_i,p_0) +
\frac{1}{n}\sum_i (\varphi(X_i,\hat p_1)- \varphi(X_i,p_1) ) +
\frac{1}{n}\sum_i (\varphi(Y_i,\hat p_0) - \varphi(Y_i,p_0)) \\
&\ \ \ \ \ \ \ \ -
\int (\varphi_0(x,\hat p_1)-\varphi_0(x,p_1))d\mu_0 -
\int (\varphi_1(x,\hat p_0)-\varphi_1(x,p_0)) d\mu_1\\
&=
\frac{1}{n}\sum_i \varphi(X_i,p_1) + \frac{1}{n}\sum_i \varphi(Y_i,p_0) +
\frac{1}{\sqrt{n}}\mathbb{G}_n (\varphi_0(X,\hat p_1)- \varphi_0(X,p_1) ) + 
\frac{1}{\sqrt{n}}\mathbb{G}_n (\varphi_1(Y,\hat p_0)- \varphi_1(Y,p_0) )
\end{align*}
where $\mathbb{G}_n(f(X_i)) = (1/\sqrt{n})( n^{-1}\sum_i f(X_i) - \int f(x) dP(x))$
is the empirical process.
As the two empirical process terms are $o_p(1)$,
$$
\sqrt{n}(\hat\psi - \psi) =
\sqrt{n}\left(\frac{1}{n}\sum_i \varphi(X_i,p_1) + \frac{1}{n}\sum_i \varphi(Y_i,p_0)\right) + o_p(1).
$$
A similar argument applies to
$\psi_N$ and then
the limit of the ratio followed by the delta method.

\subsection{Proof of Lemma~\ref{lem:semiparametric-bad}}\label{appsubsec:proof-semiparametric-bad}
The asymptotic variance is
$\E[\varphi_0^2(X_0)] + \E[\varphi_1^2(X_1)]$.
Now
\begin{align*}
\E[\varphi_0^2(X_0)] &= \int \varphi_0^2(x,p_1) d\mu_0(x)\\
& =
\frac{1}{p_t^2(z)}\int \varphi_{N0}^2(x,p_1)p_0(x) dx +
\frac{v_t^2(z)}{p_t^2(z)} \int \varphi_{D0}^2(x,p_1)p_0(x) dx - 
2\frac{v_t(z)}{p_t^2(z)} \int \varphi_{D0}(x,p_1) \varphi_{N0}(x,p_1) p_0(x)dx.
\end{align*}
The first term is
\begin{align*}
\frac{1}{p_t^2(z)}\int \frac{1}{t^{2d}} &
\int \left(\frac{z-x}{t}\right)^2 p_1^2\left( \frac{z-x}{t} + x \right) p_0(x) dx + f_t^2(z) -
\frac{2 f_t(z)}{t^d}
\int \left(\frac{z-x}{t}\right) p_1\left( \frac{z-x}{t} + x \right) p_0(x) dx\\
&\equiv A_1 + A_2 + A_3.
\end{align*}
Next,
\begin{align*}
A_1 &=
\frac{\int \left(\frac{z-x}{t}\right)\left(\frac{z-x}{t}\right)^T p_1^2\left( \frac{z-x}{t} + x \right) p_0(x) dx }
{(\int p_0(z-tx) p_1(z+ (1-t)x) dx)^2}\\
&=
\frac{1}{t^d} 
\frac{\int x x^T p_0(z-tx) p_1^2(z + (1-t)x) dx}{(\int  p_0(z-tx) p_1(z + (1-t)x) dx)^2}.
\end{align*}
Hence
$$
\lim_{t\to 0} t^d A_1 = \frac{1}{p_0(z)}
\int \delta^2 p_1^2(z+x) dx
$$
so that
$A_1 \asymp 1/t^d$.
Similar calculations apply to the other terms.

    \section{Proofs of Results in Section~\ref{sec:review-ODE}}
    \subsection{Proof of Theorem~\ref{thm:Peano-existence}}\label{appsubsec:proof-Peano-existence}
    \begin{enumerate}
    \item The first part of the theorem follows from Theorem 1 of~\cite{filippov2013differential}; Take $d = \mbox{dist}(x, \partial\Omega)$ and $m(s) = B$ for all $s\in[0, 1]$. Although Theorem 1 of~\cite{filippov2013differential} is written for $d = 1$, it still applies for $d > 1$. 

    \item The second part of the theorem follows from the first one using the fact that under~\ref{eq:integrably-bounded}, no solution of~\eqref{eq:generic-ODE-integral} can escape to $\infty$. To see this, note that for any solution $y(\cdot)$ of~\eqref{eq:generic-ODE-integral} defined on $[0, T]$, we have
    \[
    \|y(t)\| \le \|x\| + \int_0^t \|F(s, y(s))\|ds \le \|x\| + B\int_0^t (1 + \|y(s)\|)ds.
    \]
    Set $R(t) = \|x\| + B\int_0^t (1 + \|y(s)\|)ds$, so that $\|y(t)\| \le R(t)$. Observe that
    \[
    R'(t) = B(1 + \|y(t)\|) \le B(1 + R(t))\quad\Rightarrow\quad \frac{R'(t)}{1 + R(t)} \le B.
    \]
    Integrating both sides implies that
    \[
    \ln\left(\frac{1 + R(t)}{1 + R(0)}\right) \le Bt\quad\Rightarrow\quad \|y(t)\| \le R(t) \le (1 + \|x\|)e^{Bt} - 1,\mbox{ for all }t\in[0, T].
    \]
    Now applying part 1 of the theorem with $\mathcal{S} = \mathbb{R}^d$ and the upper bound on $\|F(t, x)\|$ as $(1 + \|x\|)Be^{B}$, we obtain a solution defined on $[0, 1]$.
    
    \item The proof of the third part follows the same logic as the first. However, let us first address a trivial case. If $\mbox{dist}(x,\, \partial\mathcal{S}) \ge B$, then one can simply apply part 1 of the theorem to prove the existence of a solution on $[0, 1]$ that lies entirely in $\mathcal{S}$. Under assumption~\ref{eq:bounded-in-x}, any solution to~\eqref{eq:generic-ODE-integral} satisfies $\|y(t) - x\| \le Bt$ for all $t\in[0, 1]$. Hence, in the definition of~\eqref{eq:generic-ODE-integral}, one can without loss of generality define $F$ on $[0,1]\times \mathcal{B}(x, B)\cap\mathcal{S}$. Additionally, for all $x\in \mathcal{B}(x, B)\cap\mathcal{S}$, $T_{\mathcal{S}}(x) = T_{\mathcal{S}\cap\mathcal{B}(x, B)}(x)$. Therefore, for the remaining part of the proof, we take $\mathcal{S} = \mathcal{S}\cap \mathcal{B}(x, B)$, which is a compact subset of $\mathbb{R}^d$.
    
    The proof proceeds as follows. We construct a uniformly equicontinuous sequence of functions $\{x_m:[0, 1]\to\mathcal{S}\}_{m\ge1}$, show that they approximately solve the ODE~\eqref{eq:generic-ODE}, and obtain a subsequential limit which solves the ODE. For any $x\in\mathcal{S}$, define the intermediate cone of $\mathcal{S}$ at $x$ as
    \begin{align*}
    T_{\mathcal{S}}^{\flat}(x) &:= \left\{v\in\mathbb{R}^d:\, \lim_{h\downarrow0}\frac{\mbox{dist}(x + hv, \mathcal{S})}{h} = 0\right\}.
    \end{align*} 
    See Definition 4.1.5 of~\cite{aubin2009set}.
    Proposition 4.2.1 of~\cite{aubin2009set} implies that $T_{\mathcal{S}}(x) = T^{\flat}_{\mathcal{S}}(x)$ under convexity of $\mathcal{S}$ assumed in~\ref{eq:tangent-cone-viability}. From Lemma 3 of~\cite{tallos1991viability}, it follows that for $\Delta:\mathcal{S}\times\mathbb{R}_+\to\mathbb{R}$ defined as
    \[
    \Delta(y, h) := \sup_{v\in T_{\mathcal{S}}(y)\cap \mathcal{B}_0}\, \frac{1}{h}\mbox{dist}(y + hv,\,\mathcal{S}), 
    \]
    satisfies $\lim_{h\to0} \Delta(y, h) = 0$ for all $y\in\mathcal{S}$. (The convergence need not be uniform in $y$.) Hence, for each $y\in\mathcal{S}$, there exists $h_y\in(0, 1/m)$ (for any $m\ge1$) such that
    \begin{equation}\label{eq:Delta-small}
        \Delta(y, h_y) \le \frac{1}{3m}\quad\mbox{for all}\quad y\in\mathcal{S}. 
    \end{equation}
    Define the open set
    \[
    U(y) = \left\{x\in\mathbb{R}^d:\, \mbox{dist}(x + h_yF(t, y), \mathcal{S}) < \mbox{dist}(y + h_yF(t, y)) + \frac{h_y}{3m}\right\},\; y\in\mathcal{S}.
    \]
    Clearly, $y\in U(y)$ and hence, there exists a $\delta_y\in (0, 1/m)$ such that $\mathcal{B}(y, \delta_y)\subseteq U(y)$. Moreover, 
    \[
    \mathcal{S} = \bigcup_{y\in\mathcal{S}} \{y\} \subseteq \bigcup_{y\in\mathcal{S}} \mathcal{B}(y, \delta_y).
    \]
    Compactness of $\mathcal{S}$ (from~\ref{eq:tangent-cone-viability}) implies that we can find a finite set $\mathcal{Y}$ such that 
    \[
    \mathcal{S} \subseteq \bigcup_{y\in\mathcal{Y}} \mathcal{B}(y, \delta_y).
    \]
    (This is because every cover has a finite sub-cover for compact sets in $\mathbb{R}^d$.) Set $h_0 = \min\{h_y:\, y\in\mathcal{Y}\} \in (0, 1/m)$. Note that the cardinality of $\mathcal{Y}$ depends on $m$, in general. 

    Define the sequence of approximate solutions, recursively, as follows. Set $x_m(0) = x$ for all $m \ge 1$. Since $x\in \mathcal{S}$, there exists $y_1\in\mathcal{Y}\subseteq\mathcal{S}$ such that $x\in \mathcal{B}(y_{1}, \delta_{y_1})$. Set
    \[
    z_{1}(s) := \mbox{Proj}_{\mathcal{S}}(x + h_{y_1} F(s, y_1))\quad\mbox{for all}\quad s\in[0, 1],
    \]
    and define $t_1 = h_{y_1}$, $v_0(s) = (z_1(s) - x)/h_y$,
    \[
    x_m(t) = x + \int_0^{t} v_0(s)ds\quad\mbox{for}\quad t\in[0, t_1].
    \]
    By convexity of $\mathcal{S}$,
    \[
    x_m(t) = x\left(1 - \frac{t}{t_1}\right) + \frac{t}{t_1}\frac{1}{t}\int_0^t v_0(s)ds \in\mathcal{S}\quad\mbox{for all}\quad t\in[0, t_1].
    \]
    Having defined $t_0 = 0, t_1, \ldots, t_k$ along with $v_0, v_1, \ldots, v_{k-1}$, we get $x_m(t_k) \in \mathcal{S}$. Find $y_{k+1}\in\mathcal{Y}$ such that $x_m(t_k) \in \mathcal{B}(y_{k+1}, \delta_{y_{k+1}})$. Set
    \[
    z_{k+1}(s) := \mbox{Proj}_{\mathcal{S}}(x_m(t_k) + h_{y_{k+1}} F(s, y_{k+1}))\quad\mbox{for all}\quad s\in[0, 1].
    \]
    Define $t_{k+1} = t_{k} + h_{y_{k+1}}$, 
    \[
    v_{k}(s) = \frac{z_{k+1}(s) - x_m(t_k)}{h_{y_{k+1}}}\quad\mbox{and}\quad x_{m}(t) = x_m(t_{k}) + \int_{t_k}^{t} v_k(s)ds\quad\mbox{for}\quad t\in[t_k, t_{k+1}].
    \]
    Continue this process until $t_{k}$ reaches $1$. There are only finitely many steps (for a fixed $m\ge1$) because $h_0 > 0$. Let $k_m \ge 1$ be such that $t_{k_m} > 1$. Then redefine $t_{k_m} = 1$.

    Setting
    \[
    e_k(t) = \begin{cases}1, &\mbox{if }t\in[t_k, t_{k+1}),\\
    0, &\mbox{otherwise,}\end{cases}
    \]
    we observe
    \begin{equation}\label{eq:representation-x_m}
    x_m(t) = x + \int_0^t \sum_{k=0}^{k_m} e_k(s)v_k(s)ds.
    \end{equation}
    Therefore, $x_m:[0, 1]\to\mathcal{S}$ is absolutely continuous. Moreover, for $t\in(t_k, t_{k+1})$,
    \begin{align*}
        \|x_m'(t) - F(t, y_{k+1})\| &= \left\|v_{k}(t) - F(t, y_{k+1})\right\|\\
        &= \left\|\frac{z_{k+1}(t) - x_m(t_k)}{h_{y_{k+1}}} - F(t, y_{k+1})\right\|\\
        &= \frac{1}{h_{y_{k+1}}}\|z_{k+1}(t) - x_m(t_k) - h_{y_{k+1}}F(t, y_{k+1})\|\\
        &= \frac{1}{h_{y_{k+1}}}\mbox{dist}\left(x_{m}(t_k) + h_{y_{k+1}}F(t, y_{k+1}), \mathcal{S}\right).
    \end{align*}
    Recall now that $x_m(t_k) \in \mathcal{B}(y_{k+1}, \delta_{y_{k+1}}) \subseteq U(y_{k+1})$, and therefore, from the definition of $U(y)$, we conclude
    \[
    \mbox{dist}(x_{m}(t_k) + h_{y_{k+1}}F(t, y_{k+1}),\, \mathcal{S}) < \mbox{dist}(y_{k+1} + h_{y_{k+1}}F(t, y_{k+1}),\,\mathcal{S}) + \frac{h_{y_{k+1}}}{3m}.
    \]
    Hence, for $t\in(t_k, t_{k+1})$, we conclude
    \[
    \|x_m'(t) - F(t, y_{k+1})\| \le \Delta(y_{k+1}, h_{y_{k+1}}) + \frac{1}{3m}.
    \]
    Recall, from~\eqref{eq:Delta-small}, $\Delta(y, h_y) \le 1/(3m)$ for all $y$. Therefore, 
    \begin{equation}\label{eq:x-m'-is-approximately-F}
    \|x_m'(t) - F(t, y_{k+1})\| = \|v_k(t) - F(t, y_{k+1})\| \le \frac{2}{3m} \quad\mbox{ for all }t\in(t_k, t_{k+1}). 
    \end{equation}
    % More generally, if we set
    % \[
    % y_m(t) = y_1\mathbf{1}\{x\in[t_0, t_1)\} + \sum_{k=1}^{k_m} y_{k+1}e_k(t),
    % \]
    % then
    % \[
    % \|x_m'(t) - F(t, y_m(t))\| \le \frac{2}{3m}\quad\mbox{for all}\quad t\in[0, 1].
    % \]
    In particular, $\|x_m'(t)\| \le B + 1$ for $m \ge 1$, which implies uniform equicontinuity and also boundedness of the sequence of functions $\{x_m(\cdot)\}_{m\ge1}$. Hence, by the Arzela-Ascoli theorem, we can choose a uniformly convergent subsequence (denoted again by $\{x_m(\cdot)\}_{m\ge1}$) with its limit denoted by $x^*:[0, 1]\to\mathcal{S}$. Now we need to show that $x^*$ satisfies the ODE~\eqref{eq:generic-ODE} almost everywhere $t\in[0, 1]$. We will show that
    \begin{equation}\label{eq:x_m-is-an-approximate-solution}
    \left\|x_m(t) - x - \int_0^t F(s, x_m(s))ds\right\| \to 0\quad\mbox{as}\quad m\to\infty.
    \end{equation}
    From dominated convergence theorem and uniform convergence of $x_m$ to $x^*$, this yields $\|x^*(t) - x - \int_0^t F(s, x^*(s))ds\| = 0$. 
    For this, fix $t\in[t_{k}, t_{k+1}]$, and note that
    \begin{align*}
        \|x_m(t) - y_{k+1}\| &\le \|x_m(t) - x_m(t_k)\| + \|x_m(t_{k}) - y_{k+1}\|\\
        &\le h_{y_{k+1}}\sup_{s\in[t_k, t_{k+1}]}\|x_m'(s)\| + \delta_{y_{k+1}}\\
        &\le h_{y_{k+1}}\left(B + 1\right) + \frac{1}{m} \le \frac{B + 2}{m}.
    \end{align*}
    This can equivalently be written as
    \begin{equation}\label{eq:x_m-is-approximately-y}
        \sup_{s\in[0, 1]}\left\|x_m(s) - \sum_{k=0}^{k_m} y_{k+1}e_k(s)\right\| \le \frac{B + 2}{m}.
    \end{equation}
    This implies
    \begin{align*}
    &\left\|\int_0^t \sum_{k=0}^{k_m} e_k(s)v_k(s)ds - \int_0^t F(s, x_m(s))ds\right\|\\ &\quad\le \left\|\int_0^t \sum_{k=0}^{k_m} e_k(s)F(s, y_{k+1})ds - \int_0^t F(s, x_m(s))ds\right\|\\ 
    &\qquad+ \left\|\int_0^t \sum_{k=0}^{k_m} e_k(s)\{v_k(s) - y_{k+1}\}\right\|\\
    &\quad\le \left\|\int_0^t \{F\left(s,\sum_{k=0}^{k_m} e_k(s)y_{k+1}\right) - F(s, x_m(s))\}ds\right\| + \frac{2}{3m}
    \end{align*}
    The last inequality follows from~\eqref{eq:x-m'-is-approximately-F}. Because of~\eqref{eq:x_m-is-approximately-y} and the fact that continuity of $F$ on the compact set $\mathcal{S}$ implies uniform continuity, we get that the left-hand side converges to zero. Therefore,~\eqref{eq:x_m-is-an-approximate-solution} holds, which in turn implies $x^*$ is a solution to the ODE and belongs to $\mathcal{S}$.
    \end{enumerate}

    \subsection{Proof of Theorem~\ref{thm:uniqueness-viability}}\label{appsubsec:proof-uniqueness-viability}
    Suppose $y_1$ and $y_2$ are any two solutions to the ODE~\eqref{eq:generic-ODE}. This implies that
    \[
    y_j(t) = x + \int_0^t F(s, y_j(s))ds. 
    \]
    From assumption~\ref{eq:bounded-in-x}, we get $\|y_j(t) - x\| \le Bt$ and hence $y_j(t)\in \mathcal{S}$ for all $t \le \mbox{dist}(x, \partial\mathcal{S})/B$. In other words, all the solutions to the ODE lie in $\mathcal{S}$ up to time $T$. 

    To prove the uniqueness, observe that
    \[
    y_1(t) - y_2(t) = \int_0^t \{F(s, y_1(s)) - F(s, y_2(s))\}ds.
    \]
    From assumption~\ref{eq:bounded-in-x}, we have the simple upper bound
    \[
    \|y_1(t) - y_2(t)\| \le 2Bt\quad\mbox{for all}\quad t\in[0, T].
    \]
    Fix any $\gamma \in (0, t)$. From assumption~\ref{eq:Lipschitz-Osgood}, we get
    \[
    \|y_1(t) - y_2(t)\| \le \int_0^{\gamma} a(s)\kappa(2Bs)ds + \int_{\gamma}^t a(s)\kappa(\|y_1(s) - y_2(s)\|)ds.
    \]
    Hence, $\Delta(s) = \|y_1(s) - y_2(s)\|$ satisfies the inequality
    \[
    \Delta(t) \le \int_0^{\gamma} a(s)\kappa(2B s)ds + \int_{\gamma}^t a(s)\kappa(\Delta(s))ds.
    \]
    Applying Bihari's generalization~\citep[Theorem 4]{Dragomir2003Gronwall} of Gr{\"o}nwall's inequality~\citep[Lemma 1.3.8]{KanschatScheichl2021}, we conclude
    \[
    \Delta(t) \le \Psi^{-1}\left(\Psi\left(\int_0^{\gamma} a(s)\kappa(2B s)\right) + \int_{\gamma}^t a(s)ds\right).
    \]
    Because $\gamma > 0$ is arbitrary, from assumption~\ref{eq:Lipschitz-Osgood}, we get $\Delta(t) = 0$ for all $t\in[0, T]$. The proof of part 2 under~\ref{eq:integrably-bounded} follows the same steps as in Theorem~\ref{thm:Peano-existence}.

    To prove the uniqueness under~\ref{eq:tangent-cone-viability}, define $G:[0, 1]\times\mathbb{R}^d\to\mathbb{R}^d$ as 
    \[
    G(t, y) = G(t, \mbox{Proj}_{\mathcal{S}}(y)).
    \]
    Because $\mathcal{S}$ is a compact convex subset of $\mathbb{R}^d$, the projection of $y$ onto $\mathcal{S}$ is uniquely defined and hence, $G(\cdot, \cdot)$ is a well-defined function. It is also clear that $\|G(t, y)\| \le B$ from~\ref{eq:bounded-in-x}. Additionally, for any $y, y'\in\mathbb{R}^d$,
    \begin{align*}
    \|G(t, y) - G(t, y')\| &= \|F(t, \mbox{Proj}_{\mathcal{S}}(y)) - F(t, \mbox{Proj}_{\mathcal{S}}(y'))\|\\
    &\le a(t)\kappa(\|\mbox{Proj}_{\mathcal{S}}(y) - \mbox{Proj}_{\mathcal{S}}(y')\|)\\
    &\le a(t)\kappa(\|y - y'\|),
    \end{align*}
    because projection is a contraction. Hence, $G(\cdot, \cdot)$ verifies conditions~\ref{eq:bounded-in-x} and~\ref{eq:Lipschitz-Osgood} with the domain for the second argument being $\mathbb{R}^d$ (i.e, with $\mathcal{S}$ in assumptions~\ref{eq:bounded-in-x} and~\ref{eq:Lipschitz-Osgood} being $\mathbb{R}^d$). Hence, applying the first part (of Theorem~\ref{thm:uniqueness-viability}), we find that there exists a unique solution to the ODE $y^*:[0, 1]\to\mathbb{R}^d$ such that
    \begin{equation}\label{eq:dummy-ODE}
    y^*(0) = x\quad\mbox{and}\quad \frac{dy^*(t)}{dt} = F(t, \mbox{Proj}_{\mathcal{S}}(y^*(t)))\mbox{ almost everywhere }t.
    \end{equation}
    But Theorem~\ref{thm:Peano-existence} already proves the existence of a solution $\tilde{y}:[0, 1]\to\mathcal{S}$ such that
    \[
    \tilde{y}(0) = x\quad\mbox{and}\quad \frac{d\tilde{y}(t)}{dt} = F(t, \tilde{y}(t))\mbox{ almost everywhere }t.
    \]
    Because $\tilde{y}(t) \in \mathcal{S}$ and hence, $\tilde{y}(t) = \mbox{Proj}_{\mathcal{S}}(\tilde{y}(t))$. This implies that $\tilde{y}(\cdot)$ satisfies~\eqref{eq:dummy-ODE}, which makes it the unique solution to~\eqref{eq:generic-ODE}.

    \subsection{Proof of Theorem~\ref{thm:stability}}\label{appsubsec:proof-of-stability}
    Theorem~\ref{thm:Peano-existence} implies the existence of $T\in(0, 1]$ such that $w:[0,T]\to\mathcal{S}$ solves~\eqref{eq:generic-ODE2-integral} for $t\in[0,T]$.
    Because $y(\cdot)$ and $w(\cdot)$ are respectively solutions to~\eqref{eq:generic-ODE-integral} and~\eqref{eq:generic-ODE2-integral}, we get
    \begin{align*}
    y(t) - w(t) &= x - x' + \int_0^t \{F(s, y(s)) - G(s, w(s))\}ds\\
    &= x - x' + \int_0^t \{F(s, y(s)) - F(s, w(s))\}ds + \int_0^t \{F(s, w(s)) - G(s, w(s))\}ds.
    \end{align*}
    From assumption~\ref{eq:bounded-in-x} (assumed for both $F$ and $G$), we conclude that
    \[
    \|y(t) - w(t)\| \le \|x - x'\| + 2Bt,\quad\mbox{for all}\quad t\in[0, T].
    \]
    This proves~\eqref{eq:perturbation-bound} for $t\in[0,\delta]$. For $t \in [\delta, T]$,
    \begin{align*}
    \|y(t) - w(t)\| &\le \|x - x'\| + \left\|\int_0^{\delta} \{F(s, y(s)) - G(s, w(s))\}ds\right\|\\
    &\qquad+ \left\|\int_{\delta}^t \{F(s, w(s)) - G(s, w(s))\}ds\right\| + \left\|\int_{\delta}^t \{F(s, y(s)) - F(s, w(s))\}ds\right\|\\
    &\le \mathcal{E}_{\delta}(t) + \int_{\delta}^t \|F(s, y(s)) - F(s, w(s))\|ds\\
    &\le \mathcal{E}_{\delta}(t) + \int_{\delta}^t a(s)\kappa(\|y(s) - w(s)\|)ds.
    \end{align*}
    This is a differential inequality, with a drift of $\mathcal{E}_{\delta}(t)$. To control $\Delta(s) = \|y(s) - w(s)\|$, define
    \[
    V(t) = \mathcal{E}_{\delta}(t) + \int_{\delta}^t a(s)\kappa(\Delta(s))ds.
    \]
    We have $\Delta(s) \le V(s)$, which implies
    \[
    \frac{d}{ds}V(s) = \frac{d}{ds}\mathcal{E}_{\delta}(s) + a(s)\kappa(\Delta(s)) \le \frac{d}{ds}\mathcal{E}_{\delta}(s) + a(s)\kappa(V(s)).
    \]
    Dividing both sides by $\kappa(V(s))$, we get
    \[
    \frac{V'(s)}{\kappa(V(s))} \le \frac{\mathcal{E}_{\delta}'(s)}{\kappa(V(s))} + a(s).
    \]
    Because $\kappa(\cdot)$ is non-decreasing and by definition, $V(s) \ge \mathcal{E}_{\delta}(s)$, we get
    \[
    \frac{V'(s)}{\kappa(V(s))} \le \frac{\mathcal{E}_{\delta}'(s)}{\kappa(\mathcal{E}_{\delta}(s))} + a(s).
    \]
    Integrating both sides over $s\in[\delta, t]$, we get
    \[
    \Psi(V(t)) - \Psi(V(\delta)) \le \Psi(\mathcal{E}_{\delta}(t)) - \Psi(\mathcal{E}_{\delta}(\delta)) + \int_{\delta}^t a(s)ds.
    \]
    Equivalently, (because $V(\delta) = \mathcal{E}_{\delta}(\delta)$), we obtain
    \[
    V(t) \le \Psi^{-1}\left(\Psi(\mathcal{E}_{\delta}(t)) + \int_{\delta}^t a(s)ds\right).
    \]
    Hence, the first part of the result follows. The second part follows from the same proof by noting that under~\ref{eq:tangent-cone-viability} any solution $w(\cdot)$ of~\eqref{eq:generic-ODE2-integral} can be extended on $[0, 1]$ to lie in $\mathcal{S}$.

\subsection{Proof of Theorem~\ref{prop:extension}}\label{appsubsec:proof-of-extension}
Firstly, note that from the Lipschitz condition on $F(\cdot, \cdot)$, we have
\[
\|y(t) - x\| \le \int_0^t \|F(s, y(s)) - F(s, x)\|ds + \int_0^t \|F(s, x)\|ds \le L\int_0^t \|y(s) - x\|ds + \int_0^1 \|F(s, x)\|ds
\]
Hence, $R(t) = \int_0^1 \|F(s, x)\|ds + L\int_0^t \|y(s) - x\|ds$ satisfies
\[
R'(t) = L\|y(t) - x\| \le LR(t)\quad\mbox{for all}\quad t\in[0, 1].
\]
This implies, in particular, that every solution of~\eqref{eq:generic-ODE-integral} (under the given hypothesis) satisfies
\begin{equation}\label{eq:good-solution-bounded}
\sup_{t\in[0, 1]}\|y(t) - x\| \le e^L\int_0^1 \|F(s, x)\|ds.
\end{equation}
Because Lipschitz functions are bounded on bounded sets, Theorem~\ref{thm:Peano-existence}(1) with $\mathcal{S} = \mathbb{R}^d$ and $B = \sup_{t\in[0, 1], y\in \mathcal{B}(x, R)}\|F(t, y)\|$ implies the existence and uniqueness of the solution $y(\cdot)$ to~\eqref{eq:generic-ODE-integral}. Further, from~\eqref{eq:good-solution-bounded} , we get that 
\[
\inf_{t\in[0, 1]}\,\mathrm{dist}(y(t),\, \partial\mathcal{B}(x, R)) \ge 10.
\]

Since $G(\cdot, \cdot)$ is continuous and locally Lipschitz in the second argument, Theorem~\ref{thm:Peano-existence}(1) implies the existence of the maximal solution $w:[0,\tau]\to\mathbb{R}^d$ for some
$\tau\in(0,1]$. Note that local Lipschitzness of $y\mapsto G(\cdot, y)$ implies that $\|w(\tau)\| = \infty$ for the maximal solution to exist only on $[0,\tau]$. (Otherwise (i.e., if $\|w(\tau)\| < \infty$), the solution can be extended at least up to time $\tau + \varepsilon_{w(\tau)} > \tau$.) We shall show that
$\|w(\tau)\| \le R$ which implies $\tau = 1$ and that $w(t)\in \mathcal{B}(x, R)$ for all $0\leq t\leq 1$. Define the \emph{exit time} of $w(\cdot)$ from $ \mathcal{B}(x, R)$ as
\[
  T := \inf\{t<\tau: y(t)\notin  \mathcal{B}(x, R) \,\},
\]
with the convention $T=\tau$ if $y(t)\in  \mathcal{B}(x,R)$ for all $t<\tau$.
Since $w(0)=x\in \mathrm{int}( \mathcal{B}(x,R))$ and $w(\cdot)$ is continuous, we have $T>0$,
and for every $t\leq T$, $w(t) \in  \mathcal{B}(x,R)$. Fix any $t\in[0,T)$. Then, from~\eqref{eq:good-solution-bounded}, we obtain
\begin{eqnarray*}
  \|y(t)-w(t)\|
  &=& \Bigg \|\int_0^t \bigl(F(s,y(s)) - G(s,w(s))\bigr)\,ds \Bigg \|\\
    &\leq & \int_0^t \left\|F(s,y(s)) - G(s,y(s))\right\|ds
     + \int_0^t \left\|F(s,y(s)) - F(s,z(s))\right\|ds\\
  &\le& \int_0^t \Delta ds
     + \int_0^t L \|y(s)-z(s)\|\,ds \\
  &\leq& \Delta t + L \int_0^t \|y(s)-z(s)\|\,ds.
\end{eqnarray*}
Following the similar argument that led to~\eqref{eq:good-solution-bounded} (or Gr{\"o}nwall's inequality), we conclude
\[
\|y(t) - w(t)\| \le \frac{\Delta}{L}(e^{Lt} - 1)\quad\mbox{for all}\quad t\in[0, T).
\]
By passing the limit as $t\uparrow T$ and using the continuity of $y(\cdot), w(\cdot)$, we get that this inequality holds for all $t\in[0, T]$. In particular,
\begin{equation}\label{eq:dist-less-than-10}
\|y(T) - w(T)\| \le \frac{\Delta}{L}(e^{Lt} - 1) \le \frac{\Delta}{L}(e^L - 1) \le 9,
\end{equation}
where the last inequality follows from the assumption that $\Delta(e^{L} - 1)/L \le 9$. 

We shall first prove that $T = \tau$. Suppose by contradiction that $T < \tau$. Then by continuity of $w(\cdot)$,
we must have $w(T)\in\partial\mathcal{B}(x, R)$ and $w(t)\in \mathcal{B}(x, R)$ for $t<T$. Inequalities~\eqref{eq:good-solution-bounded} and~\eqref{eq:dist-less-than-10} now yield
\begin{align*}
\mbox{dist}(w(T),\, \partial\mathcal{B}(x, R)) &\ge \mbox{dist}(y(T),\,\partial\mathcal{B}(x, R)) - \|y(T) - z(T)\|\\ 
&\ge 10 - \Delta(e^L - 1)/L \ge 1.
\end{align*}
This contradicts $w(T)\in\partial\mathcal{B}(x, R)$ (i.e., $T < \tau$) and hence, $T = \tau$. In particular, this implies that $w(\tau) \notin \partial\mathcal{B}(x, R)$ but $w(\tau) \in \mathcal{B}(x, R)$. Therefore, $\tau = 1$ and $w(t)\in\mathcal{B}(x, R)$ for all $t\in[0, 1].$
This completes the proof.

\section{Proofs of results in Section \ref{sec:unbounded}}\label{sec:unboundedproof}

\subsection{Proof of Lemma ~\ref{lemma:holderpj}}\label{sub:holderpap}

\begin{proof}[Proof of Lemma ~\ref{lemma:holderpj}]

Let's first prove bounds on the derivatives of $p_j$ under~\eqref{assump:unbnd-log-Holder}. 
A Taylor series expansion around $\tilde{x}_j$, using that the second derivatives of $\phi$ are bounded, yields that for some constant $\mathfrak{C}$,
\[
|\phi_j(x) - \phi_j(\tilde{x}_j)| = |\phi_j(x) - \phi_j(\tilde{x}_j) - \nabla \phi_j(\tilde{x}_j)^{\top}(x - \tilde{x}_j)| \le \mathfrak{C}\|x - \tilde{x}_j\|^{2}.
\]
This implies $|\phi_j(x)| \le c_0 + c_{2}\|x\|^{2}$ for some constants $c_l<\infty$. Moreover, for some constant $\mathfrak{C}$,
\[
\left|\frac{\partial\phi_j(x)}{\partial x_k} - \frac{\partial\phi_j(\tilde{x}_j)}{\partial x_k}\right| \le \mathfrak{C}\|x - \tilde{x}_j\|.
\]
This implies $|\partial\phi_j(x)/\partial x_k| \le c_0 + c_{\beta}\|x\|$ for some constants $c_0, c_{1} < \infty$. 
We conclude that $\phi_j(\cdot)$ and its derivative grow like quadratic and linear functions, respectively. For higher-order derivatives, we can generalize the above bounds, and under ~\eqref{assump:unbnd-log-Holder} we can state that for any $k = (k_1, \ldots, k_d)$ with $2 \le \|k\|_1 \le \lceil\beta\rceil - 1$, there exist constants $c_{l,k}$ for $0 \le l\le \lceil\beta\rceil - \|k\|_1 - 1$ such that
\[
\left|\frac{\partial^{\|k\|_1}}{\partial x_1^{k_1}\cdots \partial x_d^{k_d}}\phi_j(x)\right|  \le \sum_{l=0}^{\lceil\beta\rceil - \|k\|_1 - 1} c_{l,k}\|x\|^{l} + c_{\beta,k}\|x\|^{\beta - \|k\|_1}. 
\]
This follows by a Taylor series expansion of $\partial^{\|k\|_1}\phi_j(x)/\partial x_1^{k_1}\cdots \partial x_d^{k_d}$ around $x = 0$. This inequality implies that any $l$-th order partial derivative is bounded by a polynomial of degree at most $\max\{2, \beta - l\}$ for any $0 \le l \le \lceil\beta\rceil - 1$. By the multivariate Fa{\'a} di Bruno's formula, we get that any $l$-th order partial derivative of $p_j(x)$ has the following form:
\begin{equation}\label{eq:faa}
\partial_lp_j(x) = p_j(x)\sum_{\substack{w_1, \ldots, w_l\ge0,\\\sum_{u=1}^l uw_u = l}} c_{\omega_1,\ldots,\omega_l}\prod_{u=1}^l (\partial_{u}\phi_j(x))^{w_l},
\end{equation}
where $c_{\omega_1,\ldots,\omega_l}$ are some combinatorial constants and $\partial_u\phi_j(x)$ is a short-hand notation for some $u$-th order derivative of $\phi_j(x)$; $(\partial_u\phi_j(x))^{w_l}$ here can refer to product of $w_l$ many potentially different $u$-th order partial derivatives of $\phi_j(x)$. Because of the polynomial bounds on the derivatives of $\phi_j(\cdot)$, we get~\eqref{eq:bound-on-derivatives}. This also implies the boundedness of the derivatives of $p_j$ up to the order of $\lceil\beta\rceil - 1$. It then only remains to prove that for $2\leq l=\lceil \beta\rceil -1$
$$\lvert \partial_l p_j(x)-\partial_l p_j(y)\rvert \leq \mathfrak{C}|x-y|^{\beta-l}.$$
To do so, we show that each of the terms in \eqref{eq:faa} satisfies this bound. Reducing, by simplicity, to the one-dimensional case, note each of these terms is of the form $p_j(x) \partial_u\phi (x)^\omega_l$ where $1\leq u\leq l$ and $w_l$ is an integer with $w_l=1$ if $u=l=\lceil \beta \rceil-1$. We first analyze this case
\begin{eqnarray*} \lvert p_j(x)\partial_l \phi(x)-p_j(y)\partial_l\phi(y)\rvert &=& \lvert p_j(x)\left(\partial_l \phi(x)-\partial_l \phi(y)\right) + \partial_l \phi(y)\left(p_j(x) -p_j(y)\right)\rvert  \\
&\leq & \mathfrak{C}\lvert x-y\rvert^{\beta-l}+\lvert \partial_l \phi(y)\rvert \lvert p_j(x) -p_j(y) \rvert \\ &\leq &
 \mathfrak{C}\left(\lvert x-y\rvert^{\beta-l}+\lvert x-y\rvert\right).\end{eqnarray*}
In the second line, we used the  H{\"o}lder condition for $\phi$, and in the third line, we used that $\partial_l \phi$ is bounded since $l\geq 2$ and that $p_j(x)$ is satisfies a Lipschitz condition since its gradient is bounded, as a consequence of \eqref{eq:bound-on-derivatives} and log-concavity.  To conclude, we divide the analysis into $\lvert x-y\rvert \leq 1$ and $\lvert x-y\rvert \geq 1$. If $\lvert x-y\rvert \leq 1$ then $\lvert x-y\rvert \leq \lvert x-y\rvert^{\beta-l}$ so
$\lvert p_j(x)\partial_l \phi(x)-p_j(y)\partial_l\phi(y)\rvert\leq \mathfrak{C}\lvert x-y\rvert^{\beta-l}$, by the above bounds. If $\lvert x-y\rvert \geq 1$ we use the fact that $x\to p_j(x)\partial_l \phi(x)$ is a bounded function (by some constant $\mathfrak{C'}$), again, as a consequence of polynomial and log-concavity bounds. This implies that
$ \lvert p_j(x)\partial_l \phi(x)-p_j(y)\partial_l\phi(y)\rvert\leq 2\mathfrak{C}'\leq 2\mathfrak{C}|x-y|^{\beta-l}$.

The lower-order derivatives are easier to analyze. Indeed, the derivatives of $p_j(x) \partial_u\phi (x)^\omega$ are bounded since they are bounded by products of polynomials and $p_j(x)$, by the same arguments. Therefore, the functions are Lipschitz, and since the functions themselves are also bounded, this implies the final bound by the same reasoning as before.
\end{proof}
\subsection{Proof of Lemma ~\ref{lemma:holdersmooth}}\label{sub:proofholdersmooth}

\begin{proof}[Proof of Lemma \ref{lemma:holdersmooth}]
For simplicity, we consider the $d=1$ case, but the argument carries over to several dimensions. Note first that we can write $f_k(t,z)$ as
\begin{eqnarray*} f_k(t,z)&=&  \int \frac{(z-x)^k}{t^{k+d}} p_0(x)p_1\left(x+\frac{z-x}{t}\right)dx \\
&=& \sum_{i=0}^k \kappa_i z^i\int x^{k-i}p_0(x)p_1\left(x+\frac{z-x}{t}\right)dx=\sum_{i=0}^k \kappa_i z^ig_{k-i}(z,t).\end{eqnarray*}
The constants $\kappa_i$ depend on $t$, but we can assume w.l.g. that $t>0.5$ so that they all remain bounded; otherwise, we can write the above as an integral with respect to $p_1(x)$. Above, we have defined $g_k(z,t):=\int x^k p_0(x)p_1(x+(z-x)/t)dx$.
Therefore, it suffices to show that each of the $z_ig_{k-i}(z,t)$ are $\beta$-H{\"o}lder continuous. Indeed, let's first bound their derivatives of order $l\leq \lfloor \beta \rfloor$. We have
$$\lvert g_k^{(l)}(t,z)\rvert =\Big \lvert \left(\frac{t-1}{t}\right)^{l}\int x^k p_0(x)p_1^{(l)}\left(x+\frac{z-x}{t}\right)dx\Big \lvert\lesssim \int x^k P\left(x+\frac{z-x}{t}\right) p_0(x)p_1\left(x+\frac{z-x}{t}\right)dx.$$
where we have used that, by Lemma \ref{lemma:holderpj}, $p_1^{(l)}$ is bounded by a polynomial of degree bounded by $l+1$ (denoted by $P$) times $p_1$. If we are able to show that the right-hand side above is finite, this would imply both the validity of differentiation under the integral sign (by dominated convergence) and a bound to the derivative. By Lemma \ref{lemma:boundp} we have that for $m>0$, \begin{equation} \label{eq:hbound} \int x^mp_0(x)p_1\left(x+\frac{z-x}{t}\right)dx\leq c_1e^{-c_0||z||^2}\left(1+\lVert z\rVert ^m\right).\end{equation}
 and since $t>0.5$, we can bound, for some constants $\tilde{c}_l$,
\begin{equation}\label{eq:gpoly} g_k^{(l)}(t,z)\leq \exp\left(-c_1\lVert z\rVert^2\right)\sum_{l=0}^{l+k+1} \tilde{c}_l \lVert z\rVert^l,\end{equation}
so that $g_k^{(l)}(t,z)$ is bounded over $z$ and $t$. Now, note that derivatives of $f_k(t,z)$ are sums of derivatives of products of the monomials $z^i$ and $g_k(t,z)$. Each derivative of $z^i$ will at most increase the degree of the polynomial above, but this will remain bounded because of the exponential term. Therefore, derivatives of $f_k(t,z)$ up to order $l$ are bounded. 

We now need to bound the differences $\lVert  f^{(l)}_k(t,z_1)- f^{(l)}_k(t,z_2)\rVert$ where $l=\lfloor \beta \rfloor$. Again, by the above argument, this derivative expresses as a product of derivatives of monomials of $z^i$ and of $g_{k-i}(z,t)$. In turn, by the product rule, these derivatives will be sums of monomials of lower degree and of $g^{(l')}_{k-i}(z,t)$, for $l'\leq \lfloor \beta\rfloor $. In particular, note that if $l'<\lfloor \beta\rfloor$ then $l'+1\leq \lfloor \beta\rfloor $ so that by \eqref{eq:gpoly}
$$(z^{i}g^{(l')}_{k-i}(z,t))'\leq \exp(-c_1\lVert z\rVert^2)Q(z),$$
for some polynomial $Q$ of bounded degree, so the right-hand side is uniformly bounded in $t$ and $z$. Therefore, since the derivative of $z^{i}g^{(l')}_{k-i}(z,t)$ is bounded we get the Lipschitz bound
$$\lVert z_1^{i}g^{(l')}_{k-i}(z_1,t)-z_2^{i}g^{(l')}_{k-i}(z_2,t)\rVert \lesssim \lVert z_1-z_2\rVert.$$
Also, since $z^{i}g^{(l')}_{k-i}(z,t)$ is bounded, arguing as in the proof of Lemma \ref{lemma:holderpj} (dividing in the cases $\lVert z_1-z_2\rVert<1$ and $\lVert z_1-z_2\rVert>1$) we conclude that the right-hand side term in the bound above can be replaced by $\lVert z_1-z_2\rVert^{\beta-\lfloor \beta\rfloor}$.
Then, it remains to analyze the case $l'=\lfloor \beta\rfloor$. In this case, if we are able to show that
\begin{equation}\label{eq:holderg}\lVert  z_1^ig^{(l')}_k(t,z_1)- z_2^ig^{(l')}_k(t,z_2)\rVert\lesssim \lVert z_1-z_2\rVert^{\beta-\lfloor \beta\rfloor},\end{equation}
then, by the above argument, collecting terms, we will deduce that
$$\lVert  f^{(l)}_k(t,z_1)- f^{(l)}_k(t,z_2)\rVert \lesssim  \lVert z_1-z_2\rVert^{\beta-\lfloor \beta\rfloor}+ \lVert z_1-z_2\rVert\lesssim\lVert z_1-z_2\rVert^{\beta-\lfloor \beta\rfloor},$$
and the proof would be concluded. To show \eqref{eq:holderg}, we note first that by the Faá di Bruno's formula (as in the proof of Lemma \ref{lemma:holderpj})
\begin{eqnarray*} g_k^{(l')}(t,z)&=& \left(\frac{t-1}{t}\right)^{l}\int x^k p_0(x)p_1^{(l')}\left(x+\frac{z-x}{t}\right)dx \\&=&
\left(\frac{t-1}{t}\right)^{l'}\sum_{\substack{w_1, \ldots, w_{l'}\ge0,\\\sum_{u=1}^{l'} uw_u = l'}} c_{w_1,\ldots,w_{l'}}\int x^k p_0(x)p_1\left(x+\frac{z-x}{t}\right) \prod_{u=1}^{l'}\left(\partial_{u}\phi_1\left(x+\frac{z-x}{t}\right)\right)^{w_{l'}}dx.
\end{eqnarray*} 
Again, we will produce bounds based on the study of each terms. As in the proof of Lemma \ref{lemma:holderpj}, we argue first that for $u<l'=\lfloor \beta \rfloor$, all the terms containing $\partial_{u}\phi_j\left(x+\frac{z-x}{t}\right)^{w_{l'}}$ remain bounded.
Indeed, the function $z\rightarrow z^i \int p_0(x) p_1(x+(z-x)/t)\partial_{u}\phi_1(x+(z-x)/t)^{w_{l'}}dx$ has bounded derivative since, by the product rule, expresses as an integral of terms that can be uniformly bounded, by our previous arguments. Therefore, again, it suffices to analyze the term $u=l'=\lfloor \beta \rfloor$. To analyze this case, consider the function
$$m(z,t):=\int z^i x^k p_0(x)p_1\left(x+\frac{z-x}{t}\right) \phi_1^{(l')}\left(x+\frac{z-x}{t}\right)dx.$$
We can bound
\begin{eqnarray}\nonumber \lVert m(z_1,t)-m(z_2,t)\rVert  \nonumber &\leq &
\Big \lVert  \int z_1^i x^k p_0(x)p_1\left(x+\frac{z_1-x}{t}\right) \left( \phi_1^{(l')}\left(x+\frac{z_1-x}{t}\right)- \phi^{(l')}_1\left(x+\frac{z_2-x}{t}\right)\right) dx\Big \rVert \\ \nonumber &&+\Big \lVert  \int  x^k p_0(x)\phi_1^{(l')}\left(x+\frac{z_2-x}{t}\right)\left(z_1^ip_1\left(x+\frac{z_1-x}{t}\right) -z_2^ip_1\left(x+\frac{z_2-x}{t}\right)\right)  dx\Big \rVert \\
\nonumber &\lesssim & 
\lVert z_1\rVert^i\int  \lVert x\rVert^k  p_0(x)p_1\left(x+\frac{z_1-x}{t}\right)\frac{1}{t} \lVert z_1-z_2\rVert^{\beta-\lfloor \beta\rfloor}  dx\\&&+   \int  \lVert x \rVert^k   p_0(x) H(x)\lVert z_1-z_2\rVert dx.
%\lVert z_1\rVert^i Q(z)e^{-c\lVert z\rVert^2}\lVert z_1-z_2\rVert^{\beta-\lfloor \beta\rfloor}  dx+  \int  \lVert x \rVert^k p_0(x)H(z,x)\lVert z_1-z_2\rVert  dx\\
%&\lesssim& \lVert z_1-z_2\rVert^{\beta-\lfloor \beta\rfloor}+\lVert z_1-z_2\rVert\lesssim \lVert z_1-z_2\rVert^{\beta-\lfloor \beta\rfloor}.
\end{eqnarray}
In the second inequality, we used the H{\"o}lder bound on $\phi_1$ and the fact that $\phi_1^{(l')}$ is bounded. Also, we have defined $H(x):=\sup_{z\in\mathbb{R}^d} \lVert \partial_z G(x,z)\rVert$ where $G(x,z)=z^ip_1(x+(z-x)/t)$. We can bound $H(x)$ as follows. If $i\geq 1$ (otherwise the analysis is simpler), using that $\phi_1(x)$ is minimized at $x_1$ we have \begin{eqnarray*}
H(x)&=&\sup_{z\in\mathbb{R}^d} \lVert \partial_z G(x,z)\rVert \\
&\lesssim & \sup_{z\in\mathbb{R}^d} \lVert iz^{i-1} \rVert  p_1\left(x+\frac{z-x}{t}\right)+  \frac{\lVert z^i \rVert}{t} p_1\left(x+\frac{z-x}{t}\right)\Big \lVert \phi_1^{(1)}\left(x+\frac{z-x}{t}\right)\Big \rVert\\
&\lesssim& \sup_{z\in\mathbb{R}^d} \lVert z \rVert^{i-1} \left(\lVert z \rVert +1\right) \exp(-\frac{\alpha}{2}\Big\lVert x+\frac{z-x}{t}-x_1\Big \rVert^2)\\
&\lesssim& \sup_{\lVert z\rVert \leq 1} \lVert z \rVert^{i-1} \left(\lVert z \rVert +1\right) \exp\left(-\frac{\alpha}{4t^2}\Big\lVert z-tx_1\Big \rVert^2\right)\exp\left(\frac{\alpha(t-1)^2}{2t^2}\lVert x\rVert^2\right)\\
&\lesssim &\exp\left(\frac{\alpha(t-1)^2}{2t^2}\lVert x\rVert^2\right).
\end{eqnarray*}
Above, we used that $\phi^{(1)}$ is bounded, that $\lVert x-y\rVert^2\geq \lVert x\rVert^2/2-\lVert y\rVert^2$, and the fact that the products of terms containing $z$ is bounded.
Also, since 
$$p_0(x)H(x)\lesssim\exp\left(-\frac{\alpha \lVert x\rVert^2}{2}\left(1-\frac{t^2}{(1-t)^2}\right)\right)=\exp\left(-\frac{\alpha}{2}\lVert x\rVert^2\frac{2t-1}{(1-t)^2}\right),$$
we can continue bounding the differences as
\begin{eqnarray*}\lVert m(z_1,t)-m(z_2,t)\rVert  \nonumber &\leq &\lVert z_1\rVert^i \lVert z_1-z_2\rVert^{\beta-\lfloor \beta\rfloor} 
\int \lVert x\rVert^k  p_0(x)p_1\left(x+\frac{z_1-x}{t}\right)dx \\&& + \lVert z_1-z_2\rVert\int \lVert x\rVert^k  \exp\left(-\frac{\alpha}{2}\lVert x\rVert^2\frac{2t-1}{(1-t)^2}\right)dx\\
&\lesssim& \lVert z_1\rVert^i Q(z_1)\exp\left(-c\lVert z_1\rVert ^2\right)\lVert z_1-z_2\rVert^{\beta-\lfloor \beta\rfloor} +\lVert z_1-z_2\rVert \\&\lesssim& \lVert z_1-z_2\rVert^{\beta-\lfloor \beta\rfloor}.
\end{eqnarray*}
In the second-to-last inequality, we used the polynomial bound described earlier and the fact that $t > 0.5$. The last inequality follows by dividing into the cases $\lVert z_1-z_2\rVert\leq 1$ and  $\lVert z_1-z_2\rVert\geq 1$ (to apply this argument, we need to ensure that $m(z_1,t)$ is bounded, which is also true by the same rationale as before).

%In the third inequality we used \eqref{eq:hbound} to bound the first term and the fact that $\phi_1^{(l')}$ is bounded. In the fourth inequality we used that, by strong-log concavity, the moments of $x^k$ under $p_0$ are bounded.
%In the last inequality, we used the fact that the function $z\rightarrow z^i \int p_0(x) p_1(x+(z-x)/t)\phi^{(l)'}_1(x+(z-x)/t)dx$ is bounded uniformly in $t,z$ (indeed, by a polynomial times $\exp(-c\lVert z\rVert^2)$ to replace $\lVert z_1-z_2\rVert$ by $\lVert z_1-z_2\rVert^{\beta-\lfloor \beta\rfloor}$

\end{proof}

\subsection{Proof of Proposition ~\ref{prop:derzbound}}

\begin{proof}[Proof of Proposition \ref{prop:derzbound}]
Without loss of generality, $t>0.5$. Otherwise, can switch the roles of $\phi_0$ and $\phi_1$. Let's call $\gamma>0$ any upper bound on the Hessians of $\phi_0,\phi_1$.
We have that
\begin{eqnarray*}\nabla_z v_t(z)&=&\frac{1}{t^d}\nabla_z \frac{\int (z-x)p_0(x)p_1\left(x+\frac{z-x}{t}\right)dx}{\int p_0(x)p_1\left(x+\frac{z-x}{t^d}\right)dx}\\
&=&\frac{1}{t^d}-\frac{1}{t^d}\nabla_z\frac{\int xp_0(x)p_1\left(x+\frac{z-x}{t}\right)dx}{\int p_0(x)p_1\left(x+\frac{z-x}{t}\right)dx}.
\end{eqnarray*}
The first term $1/t$ is uniformly bounded in $z$. Define the family of measures
$$p_z(x)=\frac{p_0(x)p_1\left(x+\frac{z-x}{t}\right)}{\int p_0(x)p_1\left(x+\frac{z-x}{t}\right)dx} $$
Note that
\begin{eqnarray*}t\nabla_z\frac{\int xp_0(x)p_1\left(x+\frac{z-x}{t}\right)dx}{\int p_0(x)p_1\left(x+\frac{z-x}{t}\right)dx} &=& t\nabla_z \int xp_z(x)dx  \\
&=& \frac{\int xp_0(x)\nabla p_1\left(x+\frac{z-x}{t}\right)^\top dx}{\int p_0(x)p_1\left(x+\frac{z-x}{t}\right)dx}- \int xp_z(x)dx\frac{\int p_0(x)\nabla p_1\left(x+\frac{z-x}{t}\right)^\top dx}{\int p_0(x)p_1\left(x+\frac{z-x}{t}\right)dx}
\\
&=& \int xp_z(x)\frac{\nabla p_1\left(x+\frac{z-x}{t}\right)^\top}{p_1\left(x+\frac{z-x}{t}\right)}dx -\int xp_z(x)dx\int\frac{\nabla p_1\left(x+\frac{z-x}{t}\right)^\top}{p_1\left(x+\frac{z-x}{t}\right) }p_z(x)dx
\\&=&\int xp_z(x)\left(\frac{\nabla p _1\left(x+\frac{z-x}{t}\right)^\top}{p_1\left(x+\frac{z-x}{t}\right)}-\mathbb{E}_{p_z}\left(\frac{\nabla p_1\left(X+\frac{z-X}{t}\right)^\top}{p_1\left(X+\frac{z-X}{t}\right)}\right)\right)dx\\
&=& \int xp_z(x)\left(\nabla \log \left(x+\frac{z-x}{t}\right)^\top-\mathbb{E}_{p_z}\left(\nabla \log \left(X+\frac{z-X}{t}\right)\right)^\top\right)dx\\
&=& \mathbb{E}_{p_z}\left(X\left( \nabla \log \left(X+\frac{z-X}{t}\right)^\top-\mathbb{E}_{p_z}\left(\nabla \log \left(X+\frac{z-X}{t}\right)\right)\right)^\top\right)\\ &=&
\text{Cov}_{p_z}\left(X,\nabla \log \left(X+\frac{z-X}{t}\right)\right). 
\end{eqnarray*}
%Note that above by $\nabla \log f(g(x))$ we mean
%Therefore, calling
%$$S(x,z)=-\phi_1\left(x+\frac{z-x}{t}\right)=\log p_1\left(x+\frac{z-x}{t}\right)$$
%so that
%$$ \nabla_x S(x,z) =-\frac{(1-t)}{t}\frac{\nabla p_1\left(x+\frac{z-x}{t}\right)}{p_1\left(x+\frac{z-x}{t}\right)}=\frac{1-t}{t}\nabla \log p_1\left(x+\frac{z-x}{t}\right).$$
%We may write

%\begin{eqnarray*}-\frac{t^2}{1-t}\nabla_z\frac{\int xp_0(x)p_1\left(x+\frac{z-x}{t}\right)dx}{\int p_0(x)p_1\left(x+\frac{z-x}{t}\right)dx}&=&\int xp_z(x)\left(\nabla_x S^\top(x,z)-E_{p_z}\left( \nabla_x S^\top(X,z)\right)\right)dx\\
%&=&\int \left(x-E_{p_z}(X)\right)\left(\nabla_x S^\top(x,z)-E_{p_z}\left(\nabla_x S^\top(X,z)\right)\right)p_z(x)dx\\
%&=& E_{p_z}\left( \left(X-E_{p_z}(X)\right)\left(\nabla_x S^\top(x,z) -E_{p_z}\left(\nabla_x S^\top(X,z)\right)\right)\right). 
%\end{eqnarray*}
Therefore, it suffices to bound each of the terms in the above covariance matrix uniformly over $z$ and $t\geq 0.5$. Call $C_{i,j}$ each entry, then by Cauchy-Schwarz
$$\lvert C_{i,j}\rvert \leq \mathrm{Var}_{p_z}\left( X_{0i}\right)^{1/2} \mathrm{Var}_{p_z}\left(\frac{\partial}{\partial x_j} \log p_1\left(X+\frac{z-X}{t}\right)\right)^{1/2}. $$

 Note that assumption \eqref{assump:strong-log-concave} implies that $p_z$ is strongly log-concave, uniformly on $z$ and $t$. Indeed, if we write
$$p_z(x)\propto \exp\left(-\phi_z(x)\right),\quad \phi_z(x)=\phi_0(x)+\phi_1\left(\frac{z-(1-t)x}{t}\right),$$
the function $\phi_z(\cdot)$ satisfies
\begin{equation}\label{eq:logpzhess} \alpha\left(1+\frac{(1-t)^2}{t^2}\right)I_d\preceq \nabla^2 \phi_z(x)=\nabla^2 \phi_0(x)+\frac{(1-t)^2}{t^2}\nabla^2 \phi_1\left(x+\frac{z-x}{t}\right)\preceq \gamma\left(1+\frac{(1-t)^2}{t^2}\right)I_d.\end{equation}
Additionally, since $p_1$ is strongly log-concave, 

\begin{equation} \label{eq:logpzhessinv}\left(\nabla^2 \phi_z(x)\right)^{-1}\preceq \frac{t^2}{(1-t)^2}\left(\nabla^2 \phi_1\left(x+\frac{z-x}{t}\right)\right)^{-1}. \end{equation}
We will use the above bounds later on when we 
invoke the Brascamp-Lieb inequality \cite[Theorem 4.2]{brascamp1976extensions}, that for a (strictly) log-concave measure $p(x)\propto \exp \left(-\phi(x)\right)$ and a differentiable function $h$ with finite $\mathrm{Var}_p(h(X))$ we have
\begin{equation}\label{eq:brascamp}
\mathrm{Var}_p\left(h(X)\right)\leq \mathbb{E}_p\left(\nabla^\top h(X) \left(\nabla^2 \phi(X)\right)^{-1} \nabla h(X)\right).
\end{equation}
In particular, if $p$ is strongly log-concave, so that $I_d\preceq \sigma^2 \nabla^2\phi(x) $ for some $\sigma^2>0$ then
it satisfies a Poincaré inequality with constant $\sigma^2$;
\begin{equation}\label{eq:poincare}
\mathrm{Var}_p\left(h(X)\right)\leq \sigma^2 \mathbb{E}_p\left(\lVert \nabla h(X)\rVert^2\right).
\end{equation}
To bound the first variance term, we consider the function $h_i(x)=X_{0i}$ with gradient $\nabla h(x)=e_i$ (the $i$-th canonical vector). By \eqref{eq:logpzhess} and \eqref{eq:brascamp}.
\begin{eqnarray*}
\mathrm{Var}_{p_z}\left(X_{0i}\right)\leq \mathbb{E}_{p_z}\left( e_i^\top \left(\nabla^2 \phi_z(X)\right)^{-1}e_i \right)&\leq& \frac{t^2}{\alpha(t^2+(1-t)^2)}\mathbb{E}_{p_z}\left(e_i^\top e_i\right)\\
&\leq & \frac{dt^2}{\alpha(t^2+(1-t)^2)}.
\end{eqnarray*}
To bound the other variance term, consider the function $$h_j(x)=\frac{\partial }{\partial x_j}\log p_1\left(x+\frac{z-x}{t}\right) =-\frac{\partial }{\partial x_j}\phi_1\left(x+\frac{z-x}{t}\right).$$

%\begin{eqnarray*}
%var_{p_z}\left(X_{0i}\right)&\leq& E_{p_z}\left(\nabla \frac{\partial }{\partial x_j}\phi_1\left(X+\frac{z-X}{t}\right)^\top \left(\nabla^2 \phi_z(X)\right)^{-1}\nabla \frac{\partial }{\partial x_j}\phi_1\left(X+\frac{z-X}{t}\right) \right)\\
%&\leq& \frac{t^2}{\alpha(t^2+(1-t)^2)}E_{p_z}\left(\Big\lVert \frac{\partial }{\partial x_j}\phi_1\left(X+\frac{z-X}{t}\right)\Big \rVert^2 \right).\\
%\end{eqnarray*}

By \eqref{eq:logpzhessinv} and \eqref{eq:brascamp}
\begin{eqnarray*}
\mathrm{Var}_{p_z}\left(\frac{\partial}{\partial x_j} \log p_1\left(X+\frac{z-X}{t} \right)\right)&\leq& \mathbb{E}_{p_z}\left(\nabla h_j(X)^\top \left(\nabla^2 \phi_z(X)\right)^{-1}\nabla h_j(X) \right)\\
&\leq& \frac{t^2}{(1-t)^2} \mathbb{E}_{p_z}\left(\nabla h_j(X)^\top \left(\nabla^2 \phi_1\left(X+\frac{z-X}{t}\right) \right)^{-1}\nabla h_j(X)\right). 
\end{eqnarray*}
Now, note that the vector $\nabla h_j(x)$ is (up to the term $t/(1-t)$) the $j$-th row of the matrix $\nabla^2 \phi_1(x+(z-x)/t)$. Then, the term inside the expectation is of the form $v_jV^{-1}v_j$ where $V$ is a positive definite matrix, and $v_j$ is an arbitrary row of $v_j$. Since
$$v_jV^{-1}v_j\leq \sum_{i=1}^d v_iV^{-1}v_i=\mathrm{Tr}(V),$$
and using the uniform upper bound on $\nabla^2 \phi_1$ we conclude that
\begin{eqnarray*}
\mathrm{Var}_{p_z}\left(\frac{\partial }{\partial x_j}\log p_1\left(X+\frac{z-X}{t}\right) \right)&\leq& \frac{t^2}{(1-t)^2}\frac{(1-t)^2}{t^2}\ \mathbb{E}_{p_z}\left(\mathrm{Tr}\left(\nabla^2 \phi_1\left(X+\frac{z-X}{t}\right) \right) \right)\\
&\leq&  d\gamma.
\end{eqnarray*}
Let's now turn to the second derivatives. We fix one coordinate $v_i(t,z)$ and will bound each coordinate of the Hessian $\nabla^2 v_i(t,z)$. We will use the identity
$$(f/g)''=\frac{f''}{g}-\frac{f}{g}\frac{g''}{g}-2\frac{g'}{g}\left(\frac{f'}{g}-\frac{f}{g}\frac{g'}{g}\right).$$ Taking $g=p_t(z)$, $f=\int X_{0i} p_0(x)p_1(x+(z-x)/t)dx$, and reasoning as with the first derivative we obtain 
\begin{eqnarray*}t^{d+2}\nabla^2_z v_i(t,z)&=&\mathbb{E}_{p_z}\left(X_{0i} \frac{\nabla^2 p_1\left(w(X,z)\right)}{p_1\left(w(X,z)\right)}\right) -\mathbb{E}_{p_z}\left(X_{0i}\right)\mathbb{E}_{p_z} \left( \frac{\nabla^2 p_1\left(w(X,z)\right)}{p_1\left(w(X,z)\right)}\right) \\&& 
-2\mathbb{E}_{p_z}\left(\nabla \log p_1\left(w(X,z)\right) \right)\mathbb{E}_{p_z}\left(X_{0i}\left(\nabla \log p_1\left(w(X,z)\right) -\mathbb{E}_{p_z}\left(\nabla \log\left(w(X,z\right)\right) \right)^\top\right),\end{eqnarray*}
where we have defined $w(x,z):=x+(z-x)/t$.
We will further rearrange the above expressions. Since
$$\left(\log f(x)\right)''=\frac{f''(x)}{f(x)}-\frac{f'(x)^2}{f(x)^2}=\frac{f''(x)}{f(x)}-(\log f(x))'^2,$$
We can express the above as
\begin{eqnarray*}t^3\nabla^2_z v_i(t,z)&=&\mathbb{E}_{p_z}\left(X_{0i} \nabla^2 \log p_1\left(w(X,z)\right)\right)  -\mathbb{E}_{p_z}\left(X_{0i}\right)\mathbb{E}_{p_z} \left( \nabla^2 \log p_1\left(w(X,z)\right)\right) \\&&+\mathbb{E}_{p_z}\left(X_{0i} \nabla\log p_1\left(w(X,z)\right)\nabla\log p_1\left(w(X,z)\right)^\top\right) \\
&&-\mathbb{E}_{p_z}\left(X_{0i}\right)\mathbb{E}_{p_z} \left( \nabla\log p_1\left(w(X,z)\right)\nabla\log p_1\left(w(X,z)\right)^\top\right)\\&& 
-2\mathbb{E}_{p_z}\left(\nabla \log p_1\left(w(X,z)\right) \right)\mathbb{E}_{p_z}\left(X_{0i}\left(\nabla \log p_1\left(w(X,z)\right) -\mathbb{E}_{p_z}\left(\nabla \log p_1\left(w(X,z)\right)\right) \right)^\top\right).\end{eqnarray*}
We can identify the first line above as $\text{Cov}_{p_z}\left(X_{0i},\nabla^2 \log p_1(w(X,z)\right)$. Likewise, since
\begin{eqnarray*}
\mathbb{E}\left(\left(X-\mathbb{E}(X)\right)(Y-\mathbb{E}(Y))^2\right) &=& \mathbb{E}(XY^2)-\mathbb{E}(X)\mathbb{E}(Y^2)-2\mathbb{E}(Y) \mathbb{E}\left(X(Y-\mathbb{E}(Y))\right)\\
&=&\text{Cov}(X,Y^2)-2\mathbb{E}(Y)\text{Cov}(X,Y).
\end{eqnarray*}
Then, if we call $\tilde{w}(x,z)=\nabla \log p_1(w(x,z))$
\begin{eqnarray*}t^{2+d}\nabla^2_z v_i(t,z)&=&\text{Cov}_{p_z}\left(X_{0i},\nabla^2 \log p_1(w(X,z))\right)\\&&
+\mathbb{E}_{p_z}\left(\left(X_{0i}-\mathbb{E}_{p_z}(X_{0i})\right)\left(\tilde{w}(X,z)-\mathbb{E}_{p_z}\left(\tilde{w}(X,z)\right)\right)\left(\tilde{w}(X,z)-\mathbb{E}_{p_z}\left(\tilde{w}(X,z)\right)\right)^\top\right)\\
&
\leq& \mathrm{Var}_{p_z}(X_{0i})^{1/2}\mathrm{Var}_{p_z}\left(\nabla^2 \log p_1\left(w(X,z)\right)\right)^{1/2}\\&&+
\mathrm{Var}_{p_z}\left(X_{0i}\right)^{1/2}\mathbb{E}_{p_z}\left(\left(\tilde{w}(X,z)-\mathbb{E}_{p_z}\left(\tilde{w}(X,z)\right)\right)^2\left(\tilde{w}(X,z)-\mathbb{E}_{p_z}\left(\tilde{w}(X,z)\right)^\top\right)^2\right)^{1/2}.
\end{eqnarray*}
In the above displays, inequalities are interpreted component-wise. It remains to bound each of the variance and expectation terms. The variance of $X_{0i}$ was already bounded. Also, an immediate bound for the variance of the Hessian term follows from the fact that $\nabla^2 \log p_1(w(X,z))=-\nabla^2 \phi_1(w(X,z))$ has bounded eigenvalues, by \eqref{eq:logpzhess}. In turn, this entails entry-wise bounds for the matrix, thanks to elementary properties of matrix norms \citep{golub2013matrix}.

%\begin{eqnarray*} \frac{t^2}{(1-t)^2}\left(E_{p_z}\left(X_{0i} \nabla_x S(X,z) \nabla_x S(X,z)^\top \right) -E_{p_z}\left(X_{0i}\right)E_{p_z} \left( \nabla_x S(X,z)\nabla_x S(X,z)^\top\right)\right)\\
%-\frac{t^2}{(1-t)^2} \left(E_{p_z}\left(X_{0i} \nabla^2_x S(X,z)\right) -E_{p_z}\left(X_{0i}\right)E_{p_z} \left( \nabla^2_x S(X,z)\right)\right).
%\end{eqnarray*}

%Since the Hessian of $S$ is uniformly bounded in $z$, and since $Var_{p_z}(X)$ is bounded, each term in the second line above is bounded. We only need to bound the first term. This can also be expressed as a covariance, and we have the following bound (interpreted coordinate-wise)
%\begin{eqnarray*}  Var_{p_z}\left(X_{0i}, \nabla_x S(X,z)\nabla_x S(X,z)^\top\right) \leq Var_{p_z}(X_{0i})^{1/2}Var_{p_z}\left(\nabla_x S(X,z) \nabla_x S(X,z)^\top\right).
%\end{eqnarray*}

It only remains to bound the quadratic gradient term above. We use an established moment bound for measures satisfying a Poincaré condition (see \cite{gotze2019higher} and Lemma \ref{lemma:lpmomentpoincare}). By applying \eqref{eq:lpmomentpoincare} to the $j$-th coordinate function of the centered version of $\tilde{w}(x,z)$, 
$h_{j}(x)= \tilde{w}_j(x,z)-\mathbb{E}_{p_z}(\tilde{w}_j(x,z))$, and using that $p_z(\cdot)$ satisfies \eqref{eq:poincare} with parameter $\sigma^2=t^2/\alpha((t^2+(1-t)^2))$, we have
$$\mathbb{E}_{p_z}\left(\left(\tilde{w}_j(x,z)-\mathbb{E}(\tilde{w}_j(x,z)\right)^4\right)\leq 4\frac{t^4}{\alpha^2(t^2+(1-t)^2)^2} \mathbb{E}_{p_z}\left(\Big \lVert \nabla \frac{\partial}{\partial x_j}\phi_1\left(X+\frac{z-X}{t} \right)\Big \rVert^4\right).$$
The term inside the expectation is bounded uniformly over $z$ since it is the fourth power of the norm of a vector made up from entries of the matrix $\nabla^2 \phi_1(x+(z-x)/t)$. Therefore, the fourth centered moment above is bounded for each index $j$. We achieve the final conclusion by applying the Cauchy–Schwarz inequality to the matrix of cross moments in the above displays.
\end{proof}

\subsection{Additional results for Section \ref{sub:reguexistanceunbounded}}

\begin{lemma}\label{lemma:momentbound} 
Suppose that $p_0,p_1$ are log-concave. Then,

and for each $\alpha>0$:
$$\sup_{s\in[0,1],z\in \mathbb{R}^d} \mathbb{E}\left(\lVert \Delta\rVert^\alpha |X_s=z\right) p_s(z)< \infty.$$

\end{lemma}
\begin{proof}

By Lemma 1 in \cite{cule2010theoretical} there are constants $a>0$ and $b\in\mathbb{R}$ such that 
$p_1(x)\leq e^{-a\|x\|+b}$. Also, without loss of generality, $t>0.5$. From this,
\begin{eqnarray*}\mathbb{E}\left(\lVert \Delta\rVert^\alpha |X_s=z\right)  p_s(z)&=& \int_{\mathbb{R}^d} \lVert \delta\rVert ^\alpha p_0(z-t\delta)p_1(z+(1-t)\delta)\\ &=&
\frac{1}{t^d} \int_{\mathbb{R}^d} \lVert z-x\rVert^\alpha  p_0(x)p_1\left(x+\frac{z-x}{t}\right)dx\\
&\lesssim &  \lVert z\rVert ^\alpha \int_{\mathbb{R}^d}p_0(x)p_1\left(x+\frac{z-x}{t}\right)dx+\int_{\mathbb{R}^d}\lVert x\rVert^\alpha p_0(x)dx\\
&\lesssim &  \lVert z\rVert^\alpha p_s(z)+\int_{\mathbb{R}^d}\lVert x\rVert^\alpha e^{-a\lVert x \rVert +b} dx\\
&\lesssim &  \lVert z\rVert^\alpha p_s(z)+\int_{0}^\infty r^{\alpha+d-1} e^{-r} dr.
\end{eqnarray*}

In the last line, we used polar coordinates. We can identify the last integral with a moment of the exponential distribution, which is finite if $\alpha+d-1>-1$. The first term above is finite if $\alpha>0$, by Lemma \ref{lemma:boundp}.
\end{proof}

\begin{lemma}\label{lemma:infprob} 
Suppose that $p_0$ and $p_1$ satisfy~\eqref{assump:unbnd-densitiy-positive}. Then, for any bounded set $B$, we have that
$$\inf_{z\in B,t\in[0,1]} p_t(z)>0.$$
\end{lemma}
\begin{proof}[Proof of Lemma~\ref{lemma:infprob}]
Let $B_1$ be the unit ball in $\mathbb{R}^d$ centered at the origin. Recall that
\[
p_t(z) = \int_{\mathbb{R}^d} p_0(z - t\delta)p_1(z + (1 - t)\delta)d\delta \ge \int_{B_1} p_0(z - t\delta)p_1(z + (1 - t)\delta)d\delta.
\]
Hence, for $z\in B$, $p_t(z)$ is bounded below because $\inf_{\delta\in B_1} p_0(z - t\delta) > 0$ and $\inf_{\delta\in B_1}p_1(z + (1 - t)\delta) > 0$. (Recall that assumption~\eqref{assump:unbnd-densitiy-positive} implies that $p_0$ and $p_1$ are bounded away from zero on all bounded subsets of $\mathbb{R}^d$. Moreover, logconcave densities are continuous in the interior of any compact subset of $\mathbb{R}^d$.)
\end{proof}
A first observation, showed as a separate Lemma, is that the above assumptions imply that $p_0,p_1$ also have $\beta$-H{\"o}lder smoothness, so our results can be put in the perspective of the classical Kernel-based estimators~\citep{Tsybakov}.

\begin{lemma}\label{lemma:boundp} Suppose that $X_0\sim \mu_0$ and $X_1\sim \mu_1$ are independent, and that assumption ~\eqref{assump:unbnd-log-Holder}
holds. Set $m_j = \mathbb{E}[X_j]$ and $\Sigma_j = \mathrm{Var}(X_j)$ ($j = 0,1$). If $\Sigma_0,\Sigma_1$ are positive definite with eigenvalues bounded away from $0$ and $\infty$, and if $p_0,p_1$ are log-concave, then there exist constants $c_0 = c_0(\{m_j\}, \{\Sigma_j\}), c_1 = c_1(\{m_j\}, \{\Sigma_j\}) > 0$ depending only on $m_0, m_1, \Sigma_0, \Sigma_1$ such that
\begin{equation}\label{eq:exponential-decay}
p_t(z) \le c_0\exp(-c_1\|z\|)\quad\mbox{for all}\quad z\in\mathbb{R}^d, t\in [0, 1].
\end{equation}
Additionally, if assumption \eqref{assump:strong-log-concave} holds (i.e. $p_0$,$p_1$ are strongly-log concave), for following functions $h_m(z,t)$ defined for $m>0$ 
$$h_m(z,t):=\frac{1}{t^{d+m}} \int_{\mathbb{R}^d}x^m p_0(x)p_1\left(x+\frac{z-x}{t}\right)dx,$$ 
we have that for some $c_0,c_1>0$
\begin{equation}\label{eq:quadratic-decay}
h_m(z,t) \le  c_0(1+\lVert z\rVert^m) \exp(-c_1\|z\|^2)\quad\mbox{for all}\quad x\in\mathbb{R}^d, t\in [0, 1].
\end{equation}
In particular,
\begin{equation}\label{eq:quadratic-decay}
p_t(z) \le c_0\exp(-c_1\|z\|^2)\quad\mbox{for all}\quad z\in\mathbb{R}^d, t\in [0, 1].
\end{equation}

\end{lemma}
% \begin{proof}

\begin{proof}[Proof of Lemma~\ref{lemma:boundp}]
Let's first show the statement under log-concavity. As $X_0$ and $X_1$ are log-concave and independent, this implies that $X_t = (1 - t)X_0 + tX_1$ is also log-concave. Moreover, $\mbox{Var}(X_t) = (1 - t)^2\Sigma_0 + t^2\Sigma_1$, whose minimum eigenvalues are lower bounded by those of $\Sigma_0, \Sigma_1$ and the maximum eigenvalues are upper bounded by those of $\Sigma_0, \Sigma_1$.
Thus, inequality~\eqref{eq:exponential-decay} follows from Assumption~\eqref{assump:strong-log-concave} and Corollary 6(a) of~\cite{kim2016global} and implies the boundedness of the densities. We note that while Corollary 6(a) of~\cite{kim2016global} is proved with an assumption that the largest eigenvalue of the covariance matrix is bounded by $1 + \eta$ for some $\eta \in (0, 1)$, the proof can be extended to the case of an arbitrarily finite largest eigenvalue. 
%By Lemma 5 of~\cite{cule2010theoretical}, we get that $x\mapsto \phi_j(x), j = 0, 1$ attain their minimum on $\mathbb{R}^d$. Let $\tilde{x}_j$ be the point of minimum of $\phi_j(x), x\in\mathbb{R}^d$. Furthermore, under the assumption~\ref{assump:unbnd-log-Holder}, this implies $\nabla \phi_j(\tilde{x}_j) = 0$. 

Now, under \eqref{assump:strong-log-concave} call $x_i$ the unique minimizer of $\phi_i$, and assume w.l.g that $t>0.5$; otherwise we can exchange the roles of $p_0,p_1$. By strong log-concavity we have
\begin{eqnarray*}
h_m(z,t)&=& \frac{1}{t^{d+m}}\int_{\mathbb{R}^d}x^m p_0(x)p_1\left(x+\frac{z-x}{t}\right)dx\\
&\leq& x^m \exp \left(-\left(\phi_0(x_0)+\phi_1(x_1)\right)\right) \int_{\mathbb{R}^d}\exp\left(-\frac{\alpha}{2}\left(\lVert x-x_0\rVert ^2+\lVert x+\frac{z-x}{t}-x_1\rVert^2\right)\right)dx\\
&\lesssim &
 \int_{\mathbb{R}^d}(x+x_0)^m\exp\left(-\frac{\alpha}{2}\left(\lVert x\rVert ^2+\lVert x+x_0+\frac{z-x-x_0}{t}-x_1\rVert^2\right)\right)dx.
\end{eqnarray*}
In the last line, we used the change of variables $y=x-x_0$.
We can identify the above sum of squares as $\|x\|^2+||Ax+\tilde{x}||^2$ where $A=(t-1)/t$ and $\tilde{x}=(z+(t-1)x_0-tx_1)/t$. Further, since
$$\|x\|^2+||Ax+\tilde{x}||^2 = (1+A^2)\left(\lVert x+\frac{A}{1+A^2}\tilde{x}\rVert^2\right)+\frac{||\tilde{x}||^2}{1+A^2}, $$
we conclude that up to the $||\tilde{x}^2||/(1+A^2)$ term and the normalizing constant, the above exponential identifies with a Gaussian with variance $\sigma^2_t$ and mean $\mu_t$ given by
\begin{eqnarray*}
\sigma^2_t&=I_d\dfrac{1}{\alpha(1+A^2)}=&I_d\frac{t^2}{\alpha(t^2+(1-t)^2)}\\
\mu_t&= -\dfrac{A}{1+A^2}\tilde{x}=&(1-t)\frac{z-(1-t)x_0-tx_1}{t^2+(1-t)^2}
\end{eqnarray*}
By replacing, and using that $\sigma_t$ is bounded we get
\begin{eqnarray*} h_m(z,t)&\lesssim& (2\pi \sigma_t^2)^{d/2}\exp\left(-\frac{\lVert \tilde{X}\rVert^2}{2\alpha (1+A^2)}\right) \mathbb{E}_{X\sim N\left(\mu_t,\sigma^2_t\right)} (X+x_0)^m \\
&\lesssim& \exp\left(-\alpha\frac{\lVert z+(t-1)x_0-tx_1\rVert^2}{2(t^2+(1-t)^2)}\right)\mathbb{E}_{X\sim N\left(0,I_d\right)} \left(\sigma_t X+\mu_t+x_0\right)^m.\end{eqnarray*}
Let's first show that the first term is bounded (up to constants independent on $t$) by $\exp(-c||z||^2)$. Indeed, the denominator in the exponential is bounded below and the numerator writes as $\alpha \lVert z-x_t\rVert^2$, where $x_t=tx_1+(1-t)x_0$. Then, $\lVert z-x_t\rVert^2\geq 0.5\lVert z\rVert^2-\lVert x_t\rVert^2\geq 0.5\lVert z\rVert^2 -\max_{t\in[0,1]}\lVert x_t\rVert^2,$ and we conclude.
Then, it only remains to bound each of the expectation terms. We claim that each of these is bounded by a constant times $z^k$. Indeed, since $\mu_t+x_0=a_tz+b_t$ for some $a_t,b_t$ of bounded norm (since $t>0.5$), and since the $(\sigma_t X+\mu_t +x_0)^m$ is a polynomial expression in the above bounded quantities of degree $m$, we conclude that 
$\mathbb{E}_{X\sim N\left(0,I_d\right)}\left(\sigma_t X\mu_t+x_0\right)^m\lesssim \lVert z\rVert^m+1$.
\end{proof}

\subsection{Proof of  Proposition \ref{prop:consistR}, Theorem \ref{teo:CLT} and related results}
 Here we prove Proposition \ref{prop:consistR}, Theorem \ref{teo:CLT} and other intermediate results regarding the linearization and their bounds (Proposition \ref{prop:linearized}, Theorem \ref{teo:CLTlinear},  Lemma \ref{lemma:concentrationprob})
 
 \begin{proof}[Proof of Proposition \ref{prop:consistR}]
Let $L=O_p(1)$ such that $\sup_{s\in[0,t]}\lVert \mathfrak{R}(s,x)\rVert \leq L$  and ~$\sup_{s\in[0,1]}\lVert \hat{\mathfrak{R}}(s,x)\rVert \leq L$. In this setup, the assumptions of Theorem \ref{thm:stability} hold. Indeed, we can set $\mathcal{S}$ as the ball or radius $B=\sup_{\lVert z\rVert \leq L, s\in[0,1]} \lVert v(s,z)\rVert$ and choosing $\delta=0$, $a(s):=C$ where $C$ is a uniform (in $t$) bound on the Lipschitz constant of $v$, $\kappa(s)=s$ and $\Psi(u)=\log u$. By \eqref{eq:perturbation-bound} we obtain, for each $0\leq t\leq 1$
\begin{eqnarray*} \lVert \hat{\mathfrak{R}}(t,x)-\mathfrak{R}(t,x)\rVert &\leq& \exp\left(\log\left(\int_0^t \lVert \hat{v}(s,\mathfrak{R}(s,x))-v\left(\mathfrak{R}(s,x)\right)\rVert ds\right)+C\right)\\
&\leq& \exp(Ct)\int_0^t \lVert \hat{v}(s,\mathfrak{R}(s,x))-v\left(\mathfrak{R}(s,x)\right)\rVert ds.
\end{eqnarray*}
We can now take the supremum on both sides above with respect to $x$. As $x$ moves in a compact $C$, so will the entire paths $z_{s,x}=\mathfrak{R}(s,x)$ do over a new compact $B$. This shows the first equality in the statement of Proposition \ref{prop:consistR}. The second equality follows from Lemma \ref{lemma:kernelderbound}(c), which is a consistency statement uniform in $t$ and $z\in B$.
 \end{proof}
 
 \begin{proof}[Proof of Theorem \ref{teo:CLT}]
 
Denote $\hat{\mathfrak{R}}(s,x)$ and $\mathfrak{R}(s,x)$ the solutions to an ODE with $t\geq 0$, initial condition $\hat{\mathfrak{R}}(s,x)=\mathfrak{R}(0,x)=x$ and right hand sides $\hat{v}(s,z)$ and $v(s,z)$ and respectively. Note that while in principle only $\mathfrak{R}(s,x)$ is well defined, our conditions guarantee the existence and uniqueness of the empirical ODE as well: Indeed, by Theorem \ref{prop:extension}, to guarantee almost sure uniqueness and existence of the population ODE it suffices to show that we can approximate the velocity with arbitrary resolution on a fixed compact (where the trajectories of the population ODE lie). Because of our weak consistency results (Lemma \ref{lemma:kernelderbound}), the empirical velocity can be made arbitrary close to the population quantity on an event of probability converging to one.
In this event $\hat{\mathfrak{R}}(s,x)$ is well defined for all $s\in[0,1]$.
Now, since we have
\begin{align*}
    \widehat{\mathfrak{R}}(t, x) &= x + \int_0^t \widehat{v}(s,\,\widehat{\mathfrak{R}}(s, x))ds,\\
    \mathfrak{R}(t, x) &= x + \int_0^t v(s,\,\mathfrak{R}(s, x))ds.
\end{align*}
If we define $\widehat{E}(t, x) := \widehat{\mathfrak{R}}(t, x) - \mathfrak{R}(t, x)$, this quantity satisfies
\begin{align*}
\widehat{E}(t, x) &= \int_0^t \{\widehat{v}(s, \widehat{\mathfrak{R}}(s, x)) - v(s, \mathfrak{R}(s, x))\}ds \\
&= \int_0^t\underbrace{ \{\widehat{v}(s, \widehat{\mathfrak{R}}(s, x)) - v(s, \widehat{\mathfrak{R}}(s, x)) - \widehat{v}(s, \mathfrak{R}(s, x)) + v(s, \mathfrak{R}(s, x))\}dx}_{R_1(s)}\\
&\quad+ \int_0^t \underbrace{\{v(s, \widehat{\mathfrak{R}}(s, x)) - v(s, \mathfrak{R}(s, x))\}ds}_{R_2(s)}\\
&\quad+  \int_0^t \{\widehat{v}(s, \mathfrak{R}(s, x)) - v(s, \mathfrak{R}(s, x))\}ds
\end{align*}
Let's first analyze $R_1(t)$. By a second-order Taylor expansion on $v$ and $\hat{v}$ on their second arguments,

\begin{eqnarray*} R_1(s)&=&\frac{\partial }{\partial z}\widehat{v}\left(s,\mathfrak{R}(s,x)\right)^\top\widehat{E}(s, x)+\widehat{E}(s, x)^\top\frac{\partial^2 }{\partial z^2}\widehat{v}\left(s,\widehat{\psi}(s)\right) \widehat{E}(s, x)\\&& -\left(\frac{\partial }{\partial z}v\left(s,\mathfrak{R}(s,x)\right)^\top\widehat{E}(t, x)+\widehat{E}(s, x)^\top\frac{\partial^2 }{\partial z^2}v\left(s,\psi(s)\right) \widehat{E}(s, x)\right),\end{eqnarray*}

where $\widehat{\psi}(s)$ and $\psi(s)$ are some vectors in the line joining $\mathfrak{R}(s,x)$ and $\widehat{\mathfrak{R}}(s,x)$. Using Lemma \ref{lemma:concentrationprob}(b), we can bound the second derivative of the empirical velocity. Therefore,

$$  R_1(s)= O_p\left(\Big\lVert \frac{\partial }{\partial z}\widehat{v}\left(s,\mathfrak{R}(s,x)\right)-\frac{\partial }{\partial z}v\left(s,\mathfrak{R}(s,x)\right)\Big\rVert  \lVert \widehat{E}(s, x) \rVert\right) +O_p\left(\lVert \widehat{E}(s, x)^2 \rVert\right)$$

Additionally, by the same argument,

\[
R_2(s)= \frac{\partial}{\partial z}v(s, \mathfrak{R}(s, x))^\top\widehat{E}(s, x)^\top + O_p(\|\widehat{E}(s, x)\|^2).
\]
Therefore, collecting terms,
$$\widehat{E}(t, x) -  \left(\int_0^t \{\widehat{v}(s, \mathfrak{R}(s, x)) - v(s, \mathfrak{R}(s, x))\}ds +  \frac{\partial}{\partial z}v(s, \mathfrak{R}(s, x))^\top\widehat{E}(s, x)\right)ds = \int_0^t R_3(s)ds,$$
where \begin{equation} \label{eq:R3} R_3(s)= O_p\left(\Big\lVert \frac{\partial }{\partial z}\widehat{v}\left(s,\mathfrak{R}(s,x)\right)-\frac{\partial }{\partial z}v\left(s,\mathfrak{R}(s,x)\right)\Big\rVert  \lVert \widehat{E}(s, x) \rVert\right)+O_p\left(\lVert \widehat{E}(s, x)^2 \rVert\right).\end{equation}
In turn, by the definition of $\tilde{E}(s,x)$ the above implies that

$$\widehat{E}(t, x) - \widetilde{E}(t,x)-   \int_0^t \frac{\partial}{\partial z}v(s, \mathfrak{R}(s, x))^\top\left(\widehat{E}(s, x)-\widetilde{E}(s, x)\right) ds = \int_0^t R_3(s)ds.$$

Call $F(t,x):=\lVert \widehat{E}(t, x) -\widetilde{E}(t, x) \rVert$. Then, by boundedness of the derivative, for some constant, 
$$F(t,x)\leq \int_0^t \lVert R_3(s)\rVert ds + \int_0^t C  F(s,x) ds,$$
so that, by Gronwall's inequality, 
$$F(t,x)\leq \exp(Ct) \int_0^t \lVert R_3(s)\rVert ds.$$
Now, let's bound the right-hand side above.  By \eqref{eq:R3}, proposition \ref{prop:consistR} and Lemma \ref{lemma:kernelderbound}, we have that

$$\sup_{s\in[0,1]}\lVert R_3(s)\rVert =O_p\left(\left[\sqrt{\frac{\log n}{nh_n^{d+2}}}+h^{\beta-1}\right]\left[\sqrt{\frac{\log n}{nh_n^{d}}}+h^\beta\right]\right) +O_p\left(\left[\sqrt{\frac{\log n}{nh_n^d}}+h_n^\beta\right]^2\right)=o_p\left(\sqrt{\frac{1}{nh^{d-1}_n}}\right),$$
whenever $h_n\gg n^{-\frac{1}{d+2+\varepsilon}}$, for any $\varepsilon>0.$ 

This implies that
$F(t,x)=o_p\left((nh_n^{d-1})^{-1/2}\right)$.
Then, for any sequence $a_n$
\begin{eqnarray}
\label{eq:linearexplicit} \nonumber a_n\left(\hat{R}(x)-R(x)\right)&=& a_n \tilde{E}(1,x)+ a_n\left(\hat{E}(1,x)- \tilde{E}(1,x)\right)=a_n \tilde{E}(1,x) +o_p\left(a_n(nh_n^{d-1})^{-1/2}\right).
\end{eqnarray}
In particular, for $a_n=\sqrt{nh_n^{d-1}}$ we obtain 
\begin{equation}
\label{eq:linearR}
 \sqrt{nh_n^{d-1}}\left(\hat{R}(x)-R(x)\right)= \sqrt{nh_n^{d-1}} \tilde{E}(1,x)+ o_p\left(1\right),
\end{equation}
So it suffices to establish a CLT for $\sqrt{nh_n^{d-1}} \tilde{E}(1,x)$. To derive such result we will linearize $\tilde{E}(1,x)$ by the usual analysis of ratio estimators:  we will use the fact that for any arbitrary $\hat{f},f,\hat{p}>0,p>0$ we have
 \begin{equation}\label{eq:ratio}
 \frac{\hat{f}}{\hat{p}}-\frac{f}{p}=\frac{1}{p}\left(\hat{f}-\frac{f}{p}\hat{p}\right)\left(1-\frac{\hat{p}-p}{\hat{p}}\right).
 \end{equation}
 %Define the quantities  \begin{equation}\label{eq:lnz}\hat{L}(s,z):=\Phi\left(1,s,z\right) \tilde{\hat{L}}(s,z),\quad \tilde{\hat{L}}(s,z):=\frac{\hat{f}(s,z)-v(s,z)\hat{p}(s,z)}{p_s(z)},\quad \hat{r}(s,z):=\hat{p}(s,z)-p(s,z).\end{equation}
 from \eqref{eq:ratio} it follows that, if $L=\sup_{s\in[0,1]}\lVert z_s\rVert$
 \begin{eqnarray*}
 \hat{v}_s(z_s)-v_s(z_s)&=&\frac{\hat{f}_s(z_s)-v_s(z_s)\hat{p}_s(z_s)}{p_s(z_s)} \\&&-\left(\frac{\hat{f}_s(z_s)-v_s(z_s)\hat{p}_s(z_s)}{p_s(z_s)}\right)\frac{\hat{p}_s(z_s)-p_s(z_s)}{\hat{p}_s(z_s)}\\ &=&
\frac{\hat{f}_s(z_s)-v_s(z_s)\hat{p}_(z_s)}{p_s(z_s)} \\&&+O_p\left(\sup_{s\in[0,1],\lVert z\rVert \leq L}\lVert \hat{f}_s(z_s)-v_s(z_s)\hat{p}_s(z_s)\rVert\lVert \hat{p}_s(z_s)-p_s(z_s)\rVert \right)\\&=&
\frac{\hat{f}_s(z_s)-v_s(z_s)\hat{p}_s(z_s)}{p_s(z_s)} +O_p\left(\left[\sqrt{\frac{\log n}{nh_n^d}}+h_n^\beta\right]^2\right)\\
&=&\frac{\hat{f}_s(z_s)-v_s(z_s)\hat{p}_s(z_s)}{p_s(z_s)} +o_p\left((nh_n^{d-1})^{-1/2}\right)
.\end{eqnarray*}
Above, we have used Lemma \ref{lemma:kernelderbound}(b) and the fact that $\sup_{t\in[0,1]}\frac{1}{\hat{p}_t(z_t)}=O_p(1)$ (see the proof of \ref{lemma:kernelderbound}). Therefore, by virtue of \eqref{eq:linearR}
 \begin{equation*}
\nonumber \sqrt{nh_n^{d-1}}\left(\hat{R}(x)-R(x)\right)= \sqrt{nh_n^{d-1}} \left(\int_0^1  \tilde{\Phi}(s) \frac{\hat{f}_s(z_s)-v_s(z_s)\hat{p}_s(z_s)}{p_s(z_s)}ds\right)+ o_p\left(1\right).
\end{equation*}
 
 The conclusion is now a direct consequence of Theorem \ref{teo:CLTlinear}.

 \end{proof}

\begin{proposition}\label{prop:linearized}
 Consider the linearized quantity
\begin{equation}\label{eq:ln} \hat{L}(x):=\int_0^1  \tilde{\Phi}(s) \left(\frac{\hat{f}_s(z_s)-v_s(z_s)\hat{p}_s(z_s)}{p_s(z_s)}\right)ds.\end{equation}
Suppose that assumptions \eqref{assump:K}, \eqref{assump:unbnd-densitiy-positive}, \eqref{assump:unbnd-log-Holder} and \eqref{assump:strong-log-concave} hold. Then, for any convex and compact set $B$

\begin{itemize}

\item[(a)] Bias: \begin{equation}\label{eq:biasbound} \sup_{x\in B}||\mathbb{E}(\hat{L}(x))|| \leq C h_n^{\beta},\end{equation}
\item[(b)]  Uniform weak consistency: \begin{equation}\label{eq:fluctuations} \sup_{x\in B} ||\hat{L}(x)||=O_p\left(\sqrt{\frac{\log n}{nh_n^d}}+h_n^\beta\right). \end{equation}
\item[(c)] Variance: \begin{equation} \label{eq:varbound} \sup_{x\in B} \lVert nh_n^{d-1}\text{Cov}\left(\hat{L}(x)\right) -\Sigma_h(x)\rVert =o(1) ,\end{equation}
for some matrix $\Sigma_h$ detailed in the proof. 
If \eqref{assump:morse} also holds, we have that $\Sigma_{h_n}(x)\to \Sigma(x)$ where $\Sigma(x)$ is defined in \eqref{eq:sigma}, and
\begin{equation} \label{eq:varbound}  \lVert \text{Cov}\left(\hat{L}(x)\right)\rVert \lesssim \frac{1}{nh_n^{d-1}}.\end{equation}
\end{itemize}
 \end{proposition}
  Note that (c) improves over the ``naive'' variance rate given in (b).
 We defer the more involved proof of Proposition \ref{prop:linearized} to the next Section \ref{sub:proplinearizedproof}. There, we also include additional lemmas that this proof relies on. With all of the above, we can state a CLT for $\hat{L}(x)$. 
 \begin{theorem}\label{teo:CLTlinear}
Suppose that assumptions Suppose that assumptions \eqref{assump:K}, \eqref{assump:unbnd-densitiy-positive}, \eqref{assump:unbnd-log-Holder}, and \eqref{assump:strong-log-concave} hold.  In the setup of Theorem \ref{teo:CLT} then, for $\Sigma_h$ defined in \eqref{eq:sigmah} we have
 %\sqrt{nh_n^{d-1}}
$$  \sqrt{nh_n^{d-1}}\Sigma_h^{-1/2}(x)\hat{L}(x)~\overset{d}{\to}~ N\left(0,I_d\right).$$
If \eqref{assump:morse} also holds, we have that 
$$  \sqrt{nh_n^{d-1}}\hat{L}(x)~\overset{d}{\to}~ N\left(0,\Sigma(x)\right),$$

 \end{theorem}

 \begin{proof}[Proof of Theorem \ref{teo:CLTlinear}]
 $$ \sqrt{nh_n^{d-1}} \hat{L}(x)=\sqrt{nh_n^{d-1}} \left(\hat{L}(x)-\mathbb{E}(\hat{L}(x))\right)+\sqrt{nh_n^{d-1}}\mathbb{E}(\hat{L}(x)).$$
 
 By the bias property in Proposition \ref{prop:linearized}a, \eqref{eq:biasbound}, and the undersmoothing condition $h_n\ll n^{-1/(d-1+2\beta)}$, the second term in the right-hand side above is $o(1)$. We will now apply the Lyapunov CLT \eqref{lemma:lyapunov} to the triangular array
 $$U_{n,i}:=L_{n,i}(x)-\mathbb{E}(L_{n,i}(x)),\quad L_{n,i}(x)=\int_0^1  \frac{ \tilde{\Phi}(s)}{p_s(z_s)}(\Delta_i-v(s,z_s))K_{h_n}(X_i(s)-z(s))ds.$$
 We can then express
 \begin{eqnarray*}U_{n,i}&=&U^1_{n,i}+U^2_{n,i},\\
 U^1_{n,i}&:=&\int_0^1\frac{ \tilde{\Phi}(s)}{p_s(z_s)}\left(\Delta_iK_{h_n}(X_i(s)-z_s)-\mathbb{E}\left(\Delta_i K_{h_n}(X_i(s)-z_s)\right)\right) ds, \\
 U^2_{n,i}&:=&\int_0^1\frac{ \tilde{\Phi}(s)}{p_s(z_s)}v(s,z_s)\left(K_{h_n}(X_i(s)-z_s)-\mathbb{E}\left( K_{h_n}(X_i(s)-z_s)\right)\right) ds.
 \end{eqnarray*}
  We will bound $\mathbb{E}\left(||U_{n,i}||^{2+\kappa}\right)$ for any $\kappa>0$. To do this, first recall that $||x+y||^{2+\kappa}\lesssim \|x\|^{2+\kappa}+||y||^{2+\kappa}$, so we can analyze $U_{n,i}^1$ and $U_{n,i}^2$ separately. Also, recall that for a vector-valued function $f:\mathbb{R}^{d_1}\rightarrow \mathbb{R}^{d_2}$ on a probability space with measure $\mu$ we have
  \begin{eqnarray*}
  \Big \lVert \int f(x)d\mu(x) \Big \rVert^{2+\kappa} &\leq&   \left(\int \lVert f(x)\rVert d\mu(x)\right) ^{2+\kappa}\\
   &\leq& \int \lVert f(x)\rVert^{2+\kappa} d\mu(x).
  \end{eqnarray*}
   The second inequality follows from H{\"o}lder's inequality. Note also that for any vector $x$
$$\Big \lVert \frac{ \tilde{\Phi}(s)}{p_s(z_s)}x\Big \rVert\leq \Big \lVert \frac{ \tilde{\Phi}(s)}{p_s(z_s)}\Big \rVert \lVert x\rVert \lesssim \|x\|,$$
where the last inequality follows from uniform upper and lower bounds on $\tilde{\Phi}(s)$ and $p_s(z_s)$ over $s\in [0,1]$, respectively. Also, since $v(s,z_s)$ is also upper bounded on this interval, we conclude that it suffices to bound uniformly in $s$ the quantities
$$U_{n,i}^1(s):=\mathbb{E}\left[\left(K_{h_n}(X_i(s)-z_s)-\mathbb{E}\left( K_{h_n}(X_i(s)-z_s)\right)\right)^{2+\kappa}\right]$$
and
$$U_{n,i}^2(s):=\mathbb{E}\left[\lVert\Delta_i K_{h_n}(X_i(s)-z_s)-\mathbb{E}\left( \Delta_i K_{h_n}(X_i(s)-z_s)\right)\rVert^{2+\kappa}\right]).
$$
For an arbitrary $a>0$, we have
\begin{eqnarray*} \mathbb{E}\left(K_h(X_i(s)-z_s)^{a}\right)&=&h^{-ad}\int_{\mathbb{R}^d}K^a\left(\frac{z-z_s}{h}\right) p_s(z)dz\\&=& h^{-ad+d}\int_{\mathbb{R}^d} K(u)^ap_s(z_s+hu)du\\&=& h^{-ad+d}p_s(z_s)\int K(u)^a du+h^{-ad+d+1}\int_{\mathbb{R}^d}K(u)^a u\cdot \nabla p_s(z_s+\tau hu)du\\&=& 
O(h^{-ad+d}),\end{eqnarray*}
since $p_s(z)$ and the gradient of $p_s(z)$ are uniformly bounded in $s$ and $z$ (Lemma \ref{lemma:holdersmooth}). Then, for an arbitrary $a>0$,
Taking $a=1$, we obtain that $\mathbb{E}(K_{h}(X_i(s)-z_s))\lesssim 1$. Taking $a=\kappa+2$, using that $K$ has bounded support we obtain $\mathbb{E}(K_{h}(X_i(s)-z_s)^{2+\kappa})\lesssim h^{-d(1+\kappa)}$  Then, since $\mathbb{E}|X+Y|^p\lesssim \mathbb{E}|X|^p+\mathbb{E}|Y|^p$ we conclude that $U_{n,i}(s)\lesssim h^{-d(1+\kappa)}$, uniformly in $s$.
To bound $U_{n,i}^2(s)$ we follow a conditioning argument. Call $f(z)=\mathbb{E}\left(\lVert\Delta_i\rVert^{2+\kappa}|X_i(s)=z\right)$. Then,
\begin{eqnarray*} \mathbb{E}\left[\lVert\Delta_iK_h\left(X_i(s)-z_s\right)\rVert^{2+\kappa}\right]&=&\mathbb{E}\left[\mathbb{E}\left(\lVert\Delta_iK_h\left(X_i(s)-z_s\right)\rVert^{2+\kappa}|X_{i}(s)\right)\right]
\\&=& \mathbb{E}\left[\mathbb{E}\left(\lVert\Delta_i\rVert^{2+\kappa}|X_i(s)\right)|K_h\left(X_i(s)-z_s\right)|^{2+\kappa}\right]\\ &=&
\int_{\mathbb{R}^d} f(z)K^{2+\kappa}_h(z-z_s)p_s(z)dz \\&=&h^{-d(1+\kappa) }\int_{\mathbb{R}^d} f(z_s+hu)p_s(z_s+hu)K^{2+\kappa}(u)du\\&\lesssim & 
h^{-d(1+\kappa)}\int_{\mathbb{R}^d} K^{2+\kappa}(u)du\\
&\lesssim & 
h^{-d(1+\kappa)}.
\end{eqnarray*}
In the second-to-last line, we used the moment bound in Lemma \ref{lemma:momentbound}; $f(x)p_s(x)$ is bounded uniformly in $x\in \mathbb{R}^d$ and $s\in[0,1]$. In the last line, we used that $K$ is bounded. By a similar argument as before, we conclude that $U^2_{n,i}(s)\lesssim h^{-d(1+\kappa)}$ and so $U_{n,i}(s)$ is also bounded by $h^{-d(1+\kappa)}$.

We now turn to lower-bound the smallest eigenvalue of $U_{n,i}$. Using that (Proposition \ref{prop:linearized}c)
$$\lVert nh^{d-1} \text{Cov}(\hat{L}(x))-\Sigma_h(x)\rVert_{\infty}=o(1),$$
we can deduce that $\lambda_{min}(\text{Cov}(\hat{L}(x)))\gtrsim n^{-1}h^{1-d}\left(1+o(1)\right)$: indeed, suppose first that $\Sigma_h(x)$ is the identity matrix. In that case, since $||A-B||\lesssim ||A-B||_{\infty}$, it suffices to show that if $\lVert A-I_d\lVert_\infty <a$ then $\lambda_{min}(A)> 1-a$, which easily follows from Gershgorin circle theorem. In the general case, since $\Sigma(x)$ is positive definite, we can pre-multiply to obtain a lower bound on the eigenvalues of $\Sigma(x)^{-1/2}\hat{L}(x)\Sigma(x)^{-1/2}$. In turn, since $\Sigma(x)$ is positive definite, this leads to a bound on the eigenvalues of $\hat{L}(x)=\Sigma(x)^{1/2}\Sigma(x)^{-1/2}\hat{L}(x)\Sigma(x)^{-1/2}\Sigma(x)^{1/2}$. Then, using that
$v_n^2=n^2\lambda_{min}\left(\text{Cov}(\hat{L}(x))\right)$, we conclude that
\begin{eqnarray*}
\frac{1}{v_n^{2+\kappa}}\sum_{i=1}^n \lVert U_{n,i}\rVert^{2+\kappa} \lesssim \frac{nh^{-d(1+\kappa)}}{n^{1+\kappa/2}h^{(1-d)(1+\kappa/2)}}\lesssim \left(nh^{d+1+2/\kappa}\right)^{-\kappa/2}. 
\end{eqnarray*}
Therefore, for any $\varepsilon>0$ we can take $\kappa=2/(1+\varepsilon)\leq 2$ such that $\left(nh^{d+1+2/\kappa}\right)^{-\kappa/2}=\left(nh^{d+2+\varepsilon}\right)^{-\kappa/2}=o(1)$ under the condition $h=h_n\gg n^{-\frac{1}{d+2+\varepsilon}}$
From this, we conclude that
$$\text{Cov}\left(\hat{L}(x)\right)^{-1/2}\left(\hat{L}(x)-\mathbb{E}(\hat{L}(x))\right)~\overset{d}{\to}~ N\left(0,I_d\right),$$
and since $nh^{d-1}\text{Cov}\left(\hat{L}(x)\right)~\overset{d}{\to}~ \Sigma(x)$, by Slutsky's theorem, that
$$\sqrt{nh^{d-1}}\left(\hat{L}(x)-\mathbb{E}(\hat{L}(x))\right)~\overset{d}{\to}~ N\left(0,\Sigma(x)^{1/2}\Sigma(x)^{1/2}\right).$$
\end{proof}
 %Note that this bandwidth range is more aggressive than the one for optimal transport. In \citep{manole2023clt}, for their CLT scaled at the rate $\sqrt{nh^{d-2}}$ they require that
 %$$n^{-\frac{1}{d+1+2/\delta}}\lesssim h_n\lesssim n^{-\frac{1}{d-1+2\beta}}.$$
 
\begin{lemma}\label{lemma:concentrationprob}
Suppose that $p_0,p_1$ satisfy \eqref{assump:unbnd-densitiy-positive},\eqref{assump:unbnd-log-Holder}, and \eqref{assump:strong-log-concave}. Then, 
for any radius $L=O(1)$
\begin{itemize}
\item[(a)] $$\sup_{s\in[0,1],||z||\leq L}\frac{\partial}{\partial z}\hat{v}(t,z)=O_p(1).$$
\item[(b)] 
$$\sup_{s\in[0,1],||z||\leq L }\frac{\partial^2}{\partial z^2}\hat{v}(t,z)=O_p(1).$$
%\item[(e)] $$\sup_{t\in[0,1],x\in B}z_n(t,0,x)\leq o_p(1).$$

\end{itemize}

\end{lemma}

\begin{proof} To show (a) and (b), we use that for $l=1,2$ and any radius $L>0$
$$\sup_{t\in[0,1],||z||\leq L}\Big\lVert\frac{\partial^l}{\partial z^l}v(t,z)\Big \rVert \leq \sup_{t\in[0,1],||z||\leq L} \Big\lVert\frac{\partial^l}{\partial z^l}v(t,z)\Big \rVert +\sup_{t\in[0,1],||z||\leq L}\Big\lVert\frac{\partial^l}{\partial z^l}\hat{v}(t,z)-\frac{\partial^l}{\partial z^l}v(t,z)\Big \rVert.$$ By Lemma \ref{lemma:kernelderbound}, the difference term above is bounded in probability, provided that the radius is bounded. \end{proof}

\subsection{Proof of Proposition \ref{prop:linearized} and additional lemmas}\label{sub:proplinearizedproof}

 \begin{proof}[Proof of  Proposition \ref{prop:linearized}]
 
For the bias (a) and consistency analysis (b) we note that
\begin{eqnarray*}\lVert \hat{L}(x)\rVert &=& \Big \lVert \int_0^1\frac{\tilde{\Phi}(s)}{p_s(z_s)} \left(\hat{f}_s(z_s)-v_s(z_s)\hat{p}_s(z_s) \right) ds\Big \rVert \\
&\lesssim& \int_0^1\frac{\tilde{\Phi}(s)}{p_s(z_s)} \lVert \rVert  \lVert \hat{f}_s(z_s)-v_s(z_s)\hat{p}_s(z_s) \rVert ds\\
&\lesssim& \int_0^1\frac{\lVert\tilde{\Phi}(s)\rVert}{p_s(z_s)}  ds \sup_{s\in[0,1]}  \lVert \hat{f}_s(z_s)-v_s(z_s)\hat{p}_s(z_s) \rVert ds
\\
&=&O_p\left(\sqrt{\frac{\log n}{nh_n^{d}}}+h_n^\beta\right) .\end{eqnarray*}
The above term is made from a bias and fluctuations term as shown in the proof of
we used Lemma \ref{lemma:kernelderbound}, the upper bounds on  $\tilde{\Phi}(s)$ (Proposition \ref{prop:Volterra}) the lower bounds on $p_s(z_s)$ (
$\inf_{s\in[0,1],z\in\tilde{B}} p_s(z)>0$ by Lemma \ref{lemma:infprob}). 
%respectively, as we argued before. The bound is uniform on $x$ as long as $x$ lies on a compact set since all trajectories can be bounded by the same constant, by virtue of Gronwall's inequality and the uniform bound on $\partial_z v(s,z)$.
Thus, we have established \eqref{eq:biasbound} and \eqref{eq:fluctuations}.

It remains to show (c). Recall that $\hat{f}(s,z)$ and $\hat{p}(s,z)$ are sample averages of independent copies of the variables $\Delta_iK_h(I_i(s)-z)$ and $K_h(I_i(s)-z)$, respectively. Therefore, using that $\text{Cov}(\bar{X},\bar{Y})=\text{Cov}(X,Y)/n$ and upon calling
$$\gamma(t):=\frac{ \tilde{\Phi}(t)}{p_t(z_t)},$$
we have
\begin{eqnarray*}
\text{Cov}\left(\hat{L}(x)\right) &=& \int_0^1 \int_0^1 \text{Cov}\left(\gamma(s)\left(\hat{f}(s,z_s)-v(s,z_s)\hat{p}(s,z_s)\right),\gamma(t)\left(\hat{f}(t,z_t)-v(t,z_t)\hat{p}(t,z_t)\right)\right)dsdt\\ &=&
\frac{1}{n}\int_0^1 \int_0^1 \text{Cov}\left(\gamma(s)\left(\Delta-v(s,z_s)\right)K_h\left(X_s-z_s\right) ,\gamma(t)\left(\Delta-v(t,z_t)\right) K_h\left(X_t-z_t\right) \right)dsdt \\
&&\underbrace{\frac{1}{n} \int_0^1 \int_0^1 \gamma(s) \mathbb{E}\left(K_h\left(X_s-z_s\right)K_h\left(X_t-z_t\right)\left(\Delta-v(s,z_s)\right)\left(\Delta-v(t,z_t)\right)^\top \right)\gamma(t)^\top dsdt}_{\Sigma^0_h(x)}\\
 &&-
\underbrace{\frac{1}{n}\int_0^1 \int_0^1 \mathbb{E}\left(\gamma(s) K_h\left(X_s-z_s\right)\left(\Delta-v(s,z_s)\right)\right)\mathbb{E}\left(\gamma(t) K_h\left(X_t-z_t\right)\left(\Delta-v(t,z_t)\right)\right)^\top dsdt}_{\Omega_h(x)}.
\end{eqnarray*}

Let's analyze $\Sigma^0_h(x),\Omega_2(x)$ separately. Regarding $\Omega_h(x)$, since $\lVert uv^t\rVert \leq \lVert u\rVert \lVert v \rVert $:
\begin{eqnarray*}n\lVert \Omega_h(x)\rVert&\lesssim&\int_0^1  \int_0^1 \lVert \mathbb{E}\left(K_h\left(X(s)-z_s\right)\gamma(s) \left(\Delta-v(s,z_s)\right)\right)\rVert \lVert \mathbb{E}\left( K_h\left(X_t-z_t\right)\gamma(t)\left(\Delta-v(t,z_t)\right)\right)\rVert dsdt\\
&\lesssim& h^2.\end{eqnarray*}
In the last line, we used the fact that
\begin{eqnarray*}
\mathbb{E}\left(K_h\left(X_s-z_s\right)\lVert \Delta-v(s,z_s)\rVert \right)&=&\int_{\mathbb{R}^d} K(u)\lVert \delta-v(s,z_s) \rVert p_0(z_s+hu-s\delta)p_1(z_s+hu+(1-s)\delta)dud\delta\\
&=& \int_{\mathbb{R}^d} K(u)\lVert \delta-v(s,z_s) \rVert p_0(z_s-s\delta)p_1(z_s+(1-s)\delta)dud\delta\\&&+ h\int_{\mathbb{R}^d} K(u)\lVert \delta-v(s,z_s))\rVert \nabla  p_0(z_s-s\delta)p_1(z_s+(1-s)\delta)d\delta du\\&&+h\int_{\mathbb{R}^d} K(u)\lVert\delta-v(s,z_s))\rVert p_0(z_s-s\delta)\nabla p_1(z_s+(1-s)\delta) d\delta du \\
&\lesssim& h.
\end{eqnarray*}
In the last line, we used that the first term in the sum of the second to last line is zero by the definition of $v(s,z_s)$, and that the sum of the integrals is bounded (uniformly over $s$) by Lemma \ref{lemma:holdersmooth}.
Then, if we define \begin{equation}\label{eq:sigmah} \Sigma_h(x):=nh^{d-1}\Sigma^0_h(x)\end{equation} 
we have 
$$\lim_{n\rightarrow \infty}\lVert nh^{d-1}\text{Cov}(\hat{L}(x))-\Sigma_h(x)\lVert = \sup_{x\in B} h^{d-1}\lVert \Omega_h(x)\rVert =O(h^{d+1})=o(1).$$
The convergence analysis of $\Sigma_h(x)$ is more delicate, and its bulk is contained in the separate Lemma \ref{lemma:variancebound0} that we only invoke here. By the identity $X_s=X_t+(s-t)\Delta$, we can express the integral in $\Sigma^0_h(x)$ in terms of $X_t$ and $\Delta$ only. Specifically, 
\begin{eqnarray*}
n\Sigma^0_h(x)&=& \int_0^1\int_0^1\gamma(s) \mathbb{E}\left(K_h\left(X_t+(s-t)\Delta-z_s\right)K_h\left(X_t-z_t\right)\left(\Delta-v(s,z_s\right)\left(\Delta-v(t,z_t\right)^\top \right)\gamma(t)^\top dsdt\\
&=& \int_0^1 \gamma(s) \int_{\mathbb{R}^d} \int_{\mathbb{R}^d} K_h\left(hu+(s-t)\delta+z_t-z_s\right)K\left(u\right)\left(\delta-v_s(z_s)\right)\left(\delta-v_t(z_t)\right)^\top \gamma(t)^\top \\
&& \times p_0(z_t+hu-t\delta)p_1(z_t+hu+(1-t)\delta)du d\delta dsdt.
    \end{eqnarray*}
The above expression follows from i) the fact that the joint density of $\Delta,X_t$ at $(\delta,z)$ is $p_0(z-t\delta)p_1(z+(1-t)\delta)$ and ii) the change of variables $u=(X_t-z_t)/h$.  Considering the additional change of variable $\omega=(s-t)/h$, we obtain
\begin{eqnarray*}
n\Sigma_h^0(x)
&=& h\int_0^1 \int_{-t/h}^{(1-t)/h} \gamma(\omega h+t) \int_{\mathbb{R}^d} \int_{\mathbb{R}^d} K_h\left(hu+h\omega\delta+z_t-z_{\omega h+t}\right)K\left(u\right)\left(\delta-v_{\omega h+t}(z_{\omega h+t})\right) \\
&& \times \left(\delta-v_t(z_t)\right)^\top \gamma(t)^\top p_0(z_t+hu-t\delta)p_1(z_t+hu+(1-t)\delta)du d\delta d\omega dt\\
&=& h^{1-d} \int_0^1 \int_{\mathbb{R}^d} \int_{\mathbb{R}^d} K(u) g(t,u,\delta,h)\left(\delta-v_t(z_t)\right)^\top  \\&&\times p_0(z_t+hu-t\delta)p_1(z_t+hu+(1-t)\delta)\gamma(t)^\top du d\delta dt,
    \end{eqnarray*}
    where \begin{eqnarray*} g(t,u,\delta,h)&:=& \int_{-t/h}^{(1-t)/h} \gamma(wh+t) K\left(u+\omega\delta-\frac{z_{\omega h+t}-z_{t}}{h}\right)\left(\delta-v_{\omega h+t}(z_{\omega h+t})\right)d\omega.
    \\&=&
    \int_{-t/h}^{(1-t)/h} \gamma(t) K\left(u+\omega\delta-\frac{z_{\omega h+t}-z_{t}}{h}\right)\left(\delta-v_{t}(z_{t})\right)d\omega\\&&+ \int_{-t/h}^{(1-t)/h} wh K\left(u+\omega\delta-\frac{z_{\omega h+t}-z_{t}}{h}\right) \frac{d}{dy}
    \left[\gamma(y) \left(\delta-v_{y}(z_{y})\right)\right] \Big \lvert_{y\in[t,t+\omega h]}d\omega.
    \end{eqnarray*} Using the fact that the derivatives of $\gamma(y)$ and $v_y(z_y)$ are bounded in the compact $[0,1]$ (they are continuous), and that $|\omega h|\leq 1$, we derive the following component-wise bound for the entries of $g(t,u,\delta,h)$, that holds whenever the denominator is well-defined:
    \begin{eqnarray*}
    \lvert g_i(t,u,\delta,h)\rvert&\lesssim& (1+\lVert \delta\rVert)\int_{-t/h}^{(1-t)/h} K\left(u+\omega\delta-\frac{z_{\omega h+t}-z_{t}}{h}\right)d\omega  \\
    &\lesssim &  \frac{(1+\lVert \delta\rVert)(1+\lVert u\rVert) }{\lvert \delta_j- v_j(t,z_t)\rvert}\sum_{m=1}^M \frac{1}{|\delta_j-F^j_t(x_j(t,m))|^{1/2}},
    \end{eqnarray*}
    where, in the last line, we used Lemma \ref{lemma:variancebound0}, and the functions $F^j_t$ and the points $x_j(t,m)$ are the ones defined in this Lemma. Note that the above bound is valid for each $1\leq i,j\leq d$. We have found a bound on $g(t,u,\delta,h)$ that is independent of $h$. We will use it to invoke dominated convergence on each coordinate of the matrix $\Sigma_h^{i,j}(x)$ of $\Sigma_h(x)$. Indeed, we have (since $\Sigma^{i,j}_h(x):= \Sigma_h^{0,i,j}(x)$
\begin{eqnarray*} \Sigma^{i,j}_h(x) &\lesssim& \int_0^1\int_{\mathbb{R}^d}\int_{\mathbb{R}^d} K(u) \lvert g_i(t,u,\delta,h)\rvert \lvert \delta_j-v_j(t,z_t)\rvert p_0(z_t+hu-t\delta)p_1(z_t+hu+(1-t)\delta) du d\delta dt\\
&\lesssim &\int_0^1\int_{\mathbb{R}^d}\int_{\mathbb{R}^d} K(u) (1+\lVert \delta\rVert)(1+\lVert u\rVert)  p_0(z_t+hu-t\delta)p_1(z_t+hu+(1-t)\delta) du d\delta dt\\ 
&\lesssim &\int_0^1\int_{\mathbb{R}^d}\int_{\mathbb{R}^d} K(u)  \sum_{m=1}^M \frac{(1+\lVert u\rVert)(1+\lVert \delta\rVert)}{|\delta_j-F^j_t(x_j(t,m))|^{1/2}} p_0(z_t+hu-t\delta)p_1(z_t+hu+(1-t)\delta) du d\delta dt\\ 
&\lesssim &\int_0^1 \int_{\mathbb{R}^d} K(u)(1+\lVert u\rVert)dt\\
&<& \infty
\end{eqnarray*}
The second-to-last step above is not fully obvious, and we show it as a separate Lemma \ref{Lemma:boundroot}. To invoke this Lemma, we used that $z_t,u,h,F_t^j(x_j(t,m))$ are uniformly bounded over $t$.

As a conclusion of the above discussion, $\Sigma^{i,j}_h(x)$ converges to the integral of the integrand with respect to $t,\delta,u,\omega$ when $h\rightarrow 0$. It only remains to evaluate this limit. We start with $g(t,u,\delta,h)$. Note that the integrals with respect to $\omega$ become integrals in $[-\infty,\infty]$, except from the measure-zero set $\{0,1\}$. Also, by \eqref{eq:rectified-flow-ODE}, the definition of derivative and using the continuity of $K$ we have
$$\lim_{h\rightarrow 0} K\left(u+\omega -\frac{z_{\omega h+t}-z_t}{h}\right)=K\left(u+\omega\left(\delta-v_t(z_t)\right)\right).$$
The rest of the expressions containing $h$ are all of the form $\omega h+t$, which converge pointwise to $t$. Therefore, $\Sigma(x):=\lim_{h\to 0}\Sigma_h(x)$ satisfies
\begin{eqnarray*}
\Sigma(x)&=&\int_0^1 \int_{\mathbb{R}^d}\int_{\mathbb{R}^d} K(u) \lim_{h\rightarrow 0} g(t,u,\delta,h)(\delta-v_t(z_t))^\top p_0(z_t+hu-t\delta)p_1(z_t+hu+(1-t)\delta) \gamma(t)^\top du d\delta dt\\
&=&
\int_0^1 \int_{\mathbb{R}^d}\int_{\mathbb{R}^d} K(u) \int_{-\infty}^\infty \gamma(t) K(u+\omega(\delta-v_t(z_t)))(\delta-v_t(z_t))(\delta-v_t(z_t))^\top\\ &&\times  p_0(z_t-t\delta)p_1(z_t+(1-t)\delta) \gamma(t)^\top du d\delta dt
\end{eqnarray*}
This last expression coincides with \eqref{eq:sigma}.
\end{proof}

\begin{lemma}\label{Lemma:boundroot}
In the setup of Proposition \ref{prop:linearized}, if $u,h,z_t,F_t(x_j(t,m))$ are bounded uniformly in $t\in [0,1]$, then
\begin{equation} \label{eq:boundquad} \int_{\mathbb{R}^d}  \sum_{m=1}^M \frac{1+\lVert \delta\rVert}{|\delta_j-F^j_t(x_j(t,m))|^{1/2}} p_0(z_t+hu-t\delta)p_1(z_t+hu+(1-t)\delta) d\delta \lesssim 1,\end{equation}
\end{lemma}

\begin{proof}
Using that both $p_0$,$p_1$ are $\alpha$-strongly log-concave, by completing squares we deduce that for each $z=z_t+hu$ 
\begin{eqnarray*} p_0(z-\delta t)p_1(z+(1-t)\delta)\lesssim p_0(x_0)p_1(x_1) \exp\left(-a\lVert \delta - \delta_0\rVert^2+K\right),\end{eqnarray*}
where $a=\alpha(t^2+(1-t)^2)/2$, $x_0,x_1$ are the minimizers of $p_0,p_1$, and
\begin{eqnarray*} 
\delta_0&=&\frac{(1-t)x_1-tx_0-(1-2t)z}{t^2+(1-t)^2}\\
K &=& \frac{\alpha}{2}\left(\lVert z-x_0\lVert^2+\lVert z-z_1\rVert^2-\frac{\lVert(1-2t)z+tx_0-(1-t)x_1\rVert^2}{t^2+(1-t)^2}\right).
\end{eqnarray*}
Note that as $z$ is bounded, then both $\delta_0$ and $K$ are bounded uniformly in $t$. Likewise, $a$ is bounded above and below by a constant greater than zero. This implies that in the bounds below, we can treat them as constants.
 Likewise, $\delta_0$ also depends on these quantities, but it is free from $\delta$. The above indicates that as long as we can produce the following bound
$$\int_{\mathbb{R}^d}\frac{\lVert \delta \rVert}{|\delta_i-y_i|^{1/2}} \exp\left(-a\lVert \delta- \delta_0\rVert^2\right)d\delta<L_{y,\delta_0},$$ for $L_{y,\delta_0}$ that is bounded if $y$ and $\delta_0$ are bounded, we can conclude \eqref{eq:boundquad}. We finally show that this is the case. We will rely on the two following bounds: first, that 
\begin{eqnarray*} \int_{\mathbb{R}}\frac{\exp\left(-a\lVert x-\mu\rVert^2\right)}{|x-x_0|^{1/2}}dx &\leq &\int_{|x-x_0|\leq 1} \frac{\exp\left(-a\lVert x-\mu\rVert^2\right)}{|x-x_0|^{1/2}}dx+\int_{|x-x_0|\geq 1} \frac{\exp\left(-a\lVert x-\mu\rVert^2\right)}{|x-x_0|^{1/2}}dx  \\ &\leq &
\int_{|x-x_0|\leq 1} \frac{1}{|x-x_0|^{1/2}}dx+\int_{|x-x_0|\geq 1} \frac{\exp\left(-a\lVert x-\mu\rVert^2\right)}{|x-x_0|^{1/2}}dx\\ &\leq & 4+\sqrt{\frac{\pi}{a}},\end{eqnarray*}
which is a constant independent of $\mu, x_0$. Second, we use that
$$\int |x| \exp(-a||x-\mu||^2)dx\lesssim \mu +1.$$ To use the above bounds, note first that $\lVert \delta \rVert\lesssim \sum_{i=1}^d |\delta_i|$. Therefore, we can reduce the analysis to examining each integral with a $|\delta_j|$ term in the numerator. For $j\neq i$, we find
\begin{eqnarray*}
\int_{\mathbb{R}^d}\frac{|\delta_j| \exp\left(-a\lVert \delta- \delta_0\rVert^2\right)}{|\delta_i-y_i|^{1/2}}d\delta&=&  \sqrt{\frac{\pi}{a}}^{d-2}\int |\delta_j|\exp\left(-a\lvert \delta_j- \delta^j_0\rvert^2\right)d\delta_j\int  \frac{\exp\left(-a\lvert \delta_i- \delta^i_0\rvert^2\right)}{\lvert \delta_i-y_i\rvert^{1/2} }d\delta_i\\
&\lesssim & |\delta^j_0| +1,\end{eqnarray*} 
 The case of $j=i$ is similar, we express $$\frac{|\delta_i|}{|\delta_i-y_i|}\leq \frac{|\delta_i-y_i|+|y_i|}{|\delta_i-y_i|}=1+ \frac{|y_i|}{|\delta_i-y_i|},$$
 implying that 
 \begin{eqnarray*}
\int_{\mathbb{R}^d}\frac{|\delta_i| \exp\left(-a\lVert \delta- \delta_0\rVert^2\right)}{|\delta_i-y_i|^{1/2}}d\delta
\lesssim  |y_i| +1,\end{eqnarray*} 
\end{proof}
\begin{lemma}\label{lemma:variancebound0}
Suppose that $K$ is a kernel with bounded support, and let $f:[0,1]\to \mathbb{R}^d$ be a twice continuously differentiable function with coordinates $f_i$. For each $1\leq i\leq d$ define the functions 
$$F^i_t(x):=\frac{f_i(t+x)-f_i(t)}{x},\quad x\in[-t,1-t].$$
Note that these are all uniformly bounded over $t,x$ since $f_i'$ is bounded. Suppose that they satisfy the following conditions
\begin{itemize}
    \item Are twice differentiable, and their second derivatives are equicontinuous in $t$.
    \item For each $t$, $F^i_t$ has at most $M$ critical points, where $M$ is uniform in $d$ and $t\in[0,1]$.
    \item  The critical points are uniformly non-denegerate, i.e. if $$C(t,i)=\Big\{x\in[-t,1-t], (F^i)'_t(x)=0\Big\}.$$
is the set of such critical points, then, there is $L>0$ such that
$$\sup_{x\in \bigcup_{t\in[0,1],1\leq i\leq d}C(t,i)}\Big \lvert (F_t^i)^{''}(x)\Big \rvert \geq L.$$
The above is equivalent to 
$$\sup_{x\in \bigcup_{t\in[0,1],1\leq i\leq d}C(t,i)} \Big \lvert \frac{1}{x} f_i''(t+x)\Big \rvert>L.$$

\end{itemize}
 If we denote by $x_i(t,m)$ any enumeration of the critical points of $F^i_t$, for each $1\leq i\leq d,\delta,u\in\mathbb{R}^d$ such that the right-hand side below is well-defined, we have that
\begin{eqnarray} \label{eq:kernelbound} \int_{-t/h}^{(1-t)/h} K\left(u+\omega\delta-\frac{f(t+h\omega)-f(t)}{h}\right)d\omega \lesssim \frac{1+||u||}{|\delta_i-f_i'(t))|}\sum_{m=1}^M \frac{1}{\sqrt{\lvert \delta_i - F^i_t(x_i(t,m))\rvert}},
\end{eqnarray}
In the above bound, there are only hidden dependencies the support of the kernel and bounds on the first and second derivatives of $F_i(x)$.
\end{lemma}

\begin{proof}
Since $K$ is a bounded kernel, it suffices to bound the Lebesgue measure $\Lambda(S)$ of the set $S\subseteq [-t/h,(1-t)/h]$ where the evaluation of the kernel may be positive, i.e., 
\begin{eqnarray*} \int_{-t/h}^{(1-t)/h} K\left(u+\omega\delta-\frac{f(t+h\omega)-f(t)}{h}\right)d\omega \lesssim \lambda(S).
\end{eqnarray*}
Specifically, let's $R$ be the size of a ball containing the support of the kernel. The condition $\omega \in S$ implies, in addition to $-t\leq h\omega\leq 1-t$, that
\begin{equation}
    \label{eq:kerzer0}
K\left(u+\omega\delta-\frac{f(t+h\omega)-f(t)}{h}\right)>0,
\end{equation}
which in turn implies that 
$$\Big\lVert \omega\delta-\frac{f(t+h\omega)-f(t)}{h}\Big\rVert\leq \Big\lVert u+\omega\delta-\frac{f(t+h\omega)-f(t)}{h}\Big\rVert+\lVert u\rVert \leq R+\lVert u\rVert.$$
Then, for each $i$, $S$ is contained in the set
\begin{equation} \label{eq:si}S_i=\Big\{\omega\in[-t,1-t],\Big\lvert \omega\delta_i-\frac{f_i(t+h\omega)-f_i(t)}{h}\Big\rvert \leq R+\lVert u\rVert\Big\},\end{equation}
and so $\lambda(S)\leq \lambda(S_i)$.
The conclusion follows from bounding the Lebesgue measure of each of these sets, which we address in the following lemma

\end{proof}

\begin{lemma}\label{lemma:variancebound1}
Under the setup of \ref{lemma:variancebound0}. We have that
$$\lambda(S_i)\lesssim \frac{1+||u||}{|\delta_i-f_i'(t,z_t)|}\sum_{j=1}^M \frac{1}{\sqrt{\lvert \delta_i - F^i_t(x_i(t,j))\rvert}}$$
\end{lemma}

\begin{proof}
In this proof we drop the $i$ indexing and denote $\delta,f,f',f'',F_t$ any of the $\delta_i,f'_i,f''_i,F^i_t$. Likewise, we denote by $S$ any of the $S_i$ in \eqref{eq:si}. We define the more generic sets, for $-t\leq a\leq b\leq 1-t$ and $\tilde{R}>0$,
$$S(\tilde{R},a,b)=\Big\{\omega\in\left[\frac{a}{h},\frac{b}{h}\right]: \Big\lvert \omega\delta-\frac{f(t+h\omega)-f(t)}{h}\Big\rvert \leq \tilde{R} \Big\}.$$
By the definition of $F_t(x)$ the condition $\omega\in S(\tilde{R},a,b)$ implies
$$ \Big\lvert \omega\left(\delta-F_t(\omega \delta)\right)\Big \rvert =\Big\lvert \omega\delta-\frac{f(t+h\omega)-f(t)}{h}\Big\rvert \leq \tilde{R}.$$
We want to control $S=S(R+\lVert u\rVert,-t,1-t)$. This is not obvious since the above inequality is nonlinear in $\omega$. We will partition the interval of $[-t,1-t]$ in a way so that on each of these sub-intervals the above measure is well controlled. Specifically, if $a_0=-t,\ldots a_m,\ldots, a_{\tilde{M}}=1-t$ is a partition of the interval $[-t,1-t]$, then
\begin{equation}\label{eq:partition}S(R+\lVert u\rVert ,t,1-t)=\bigcup_{m=0}^{\tilde{M}-1} S(R+\lVert u\rVert ,a_m,a_{m+1}),\end{equation}
so that as long as the number of pieces is finite, we can bound $S$ if we have individual bounds on the subintervals. We choose the partition consisting on the division of the $[-t,1-t]$ interval induced by the critical points of $F_t$. This partition contains at most $M+1$ subintervals. We bound the measure of each of these subintervals using Lemma \ref{lemma:variancebound}, from which we easily conclude.

\end{proof}

%we obtain that for any $x_1$ such that %$F''_t(x_1)>F''_t(x(t))$, we have
%$$ F''_t(x_1)  \leq 3\frac{F''_t(x(t)}{2}.$$
%Likewise, for any $x_1$ such that $F''_t(x_1)\leq F''_t(x(t))$ we have
%$$\frac{F''_t(x(t))}{2} \leq  F''_t(x_1).$$
%The above inequalities must hold for the maximum and minimum, respectively. We conclude that
%$$\max_{|x-x(t)|<\delta} F''_t(x)-\min_{|x-x(t)|<\delta} F''_t(x)$$
%\begin{lemma}\label{lemma:variancebound} Under the assumptions of Lemma \ref{lemma:variancebound0}, consider the restriction of $F_t(x)$ to a subinterval $[a,b]\subseteq[-t,1-t]$ such that $F_t(x)$ is strictly monotonic on this interval, and for some $L>0$,
%$$\inf_{x\in[a,b]} \Big \lvert \frac{d}{dx} F_t(x)\Big \rvert \geq L.$$
%We define one additional function
%$$G_t(x)=\frac{F_t(x)-f'(t)}{x}.$$
%Note that $G_t(x)$ is uniformly bounded in $t$ and $x\in[-t,1-t]$ since $|G_t(x)|=|f''(\upsilon)|$ for some $\upsilon\in[t,t+x]\subseteq[0,1]$, and $f''$ is continuous.
%Then, whenever
% $\delta\neq f'(t)$,
%$$\lambda\left(S(\tilde{R},a,b)\right)\lesssim  \frac{1}{L}\frac{\tilde{R}}{|\delta-f'(t)|}.$$
%\end{lemma}

\begin{lemma}\label{lemma:variancebound} Under the assumptions of Lemma \ref{lemma:variancebound0}, consider the restriction of $f(x)$ and $F_t(x)=(f(x+t)-f(t))/x$ to an interval $[a,b]\subseteq[-t,1-t]$ such that $F_t(x)$ is strictly monotonic on this interval, and that if there is a critical point of $F_t(x)$ it must be either $a$ or $b$, and such that if $a$ and $b$ are both critical points then $F''_t(a)F''_t(b)<0$.
%$F_t(x)$ are the only critical points of $F_t(x)$ on $[a,b]$.

Then, whenever $\delta\neq f'(t),\delta\neq F_t(a), \delta\neq F_t(b)$,
$$\lambda\left(S(\tilde{R},a,b)\right)\lesssim  \frac{\tilde{R}}{|\delta-f'(t)|}\left(\frac{1}{\sqrt{|\delta-F_t(a)|}}+\frac{1}{\sqrt{|\delta-F_t(b)|}}\right).$$

\end{lemma}

\begin{proof}

Let us define one additional function
$$G_t(x)=\frac{F_t(x)-f'(t)}{x}.$$
Note that $G_t(x)$ is uniformly bounded in $t$ and $x\in[-t,1-t]$ since $f''(t)$ is bounded in $[0,1]$ and $$|G_t(x)|=\Big\lvert\frac{f(t+x)-f(t)-xf'(t)}{x}\Big\rvert=|f''(\upsilon)|.$$ for some $\upsilon\in[t,t+x]\subseteq[0,1]$.
%In what follows we assume that $f''(t)$ is not identically zero; otherwise, since the quotient above equals $f'(t)$ we directly obtain that $|\omega|\leq (R+\lVert u\rVert)\lvert \delta-  f'(t)\rvert^{-1}$.
Since
$G_t(x)=xF_t(x)+f'(t)$, the condition $\omega\in S(\tilde{R},a,b)$ implies
$$ \Big\lvert \omega^2 hG_t(h\omega)-\omega\left(\delta-  f'(t)\right) \Big\rvert = \Big\lvert \omega\left(\delta-F_t(\omega \delta)\right)\Big \rvert  \leq \tilde{R}.$$
The above displays imply that for each index (w.l.g., $\omega\neq 0$)
$$\Big\lvert \omega^2 h G_t(\omega h)-\omega\left(\delta-  f'(t)\right) \Big \rvert \leq \tilde{R}.$$
%We will later use the fact that, by a Taylor expansion of order 2 over $f_i$ around $t$, 
%\begin{equation}\label{eq:gti} \lvert G^t_i(x)\rvert =\frac{\lvert f_i(x+t)-f_i(t)-xf_i'(t)\rvert}{x^2}\leq \frac{1}{2}\sup_{s\in [0,1]} \lvert f_i''(s)\rvert \lesssim 1.\end{equation}
Write $A=A(\omega,\delta,t,h):=h G_t(\omega h)/(\delta-  f'(t))$. Note that the case $A=0$ directly implies that $\lvert \omega\rvert\lesssim \tilde{R}|\delta-f'(t)|^{-1}$. So, assuming that $A\neq 0$, we can express the above as
\begin{equation}\label{eq:abc}\Big\lvert \left(\omega-\frac{1}{2A}\right)^2  -\frac{1}{4A^2}\Big \rvert\leq \frac{\tilde{R}}{\lvert A\rvert \lvert \delta-  f'(t)\rvert}.\end{equation}

We will carefully study this region. Consider first the set  $$D(\tilde{R},a,b) =\{\omega \in S(\tilde{R},a,b): \lvert \delta-f'(t)\rvert <4\lvert A\rvert \tilde{R}\}.$$ In this set, the condition \eqref{eq:abc}
is equivalently stated as
\begin{equation}\label{eq:Cbound} \frac{1}{2A}-\sqrt{\frac{1}{4A^2}+\frac{\tilde{R}}{\lvert A\rvert\lvert \delta-  f'(t)\rvert}}\leq \omega\leq  \frac{1}{2A}+\sqrt{\frac{1}{4A^2}+\frac{\tilde{R}}{\lvert A\rvert\lvert \delta-  f'(t)\rvert}}.\end{equation}
Suppose first that $A>0$. We will use repeatedly the fact that for any $x>0$ and $y$ such that $x+y\geq 0$, \begin{equation}\label{eq:sqrt}\lvert\sqrt{x+y}-\sqrt{x}\rvert \leq \frac{\lvert y\rvert }{\sqrt{x}}.\end{equation}
By \eqref{eq:sqrt}, the leftmost inequality in \eqref{eq:Cbound} implies that
$$-\frac{\tilde{R}}{\lvert \delta-f(t)\rvert}\lesssim \omega.$$
On the other hand, by \eqref{eq:sqrt} and owing to the fact that $\lvert \delta-f'(t)\rvert <4\lvert A\rvert \tilde{R}$, the rightmost inequality in \eqref{eq:Cbound} implies
$$\omega\leq  \frac{1}{A}+\sqrt{\frac{1}{4A^2}+\frac{\tilde{R}}{\lvert A\rvert\lvert \delta-  f'(t)\rvert}}-\frac{1}{2A}\lesssim \frac{1}{A}+ \frac{\tilde{R}}{\lvert \delta-f(t)\rvert}\lesssim \frac{4\tilde{R}}{\lvert \delta-f(t)\rvert}+ \frac{\tilde{R}}{\lvert \delta-f(t)\rvert}\lesssim  \frac{\tilde{R} }{\lvert \delta-f(t)\rvert}.$$
The case $A<0$ is completely analogous. Therefore,  \begin{equation}\label{eq:dlambda}\Lambda(D(\tilde{R},a,b))\lesssim \Lambda\left(|\omega|\lesssim \frac{\tilde{R}}{|\delta-f'(t)|}\right)\lesssim \tilde{R}{|\delta-f'(t)|},\end{equation}
which is independent of $h$.
The above implies that we can now focus only on the complement $D^c(\tilde{R},a,b)$ of this set. In this complement, the set of $\omega$ satisfying \eqref{eq:abc} is the set belonging to the union of the two sets defined below
$$S_1(\tilde{R},a,b)=\Bigg\{\omega\in D^c(\tilde{R},a,b): \sqrt{\frac{1}{4A^2}-\frac{\tilde{R}}{\lvert A\rvert\lvert \delta-  f'(t)\rvert}}+\frac{1}{2A}\leq \omega\leq  \sqrt{\frac{1}{4A^2}+\frac{\tilde{R}}{\lvert A\rvert\lvert \delta-  f'(t)\rvert}}+\frac{1}{2A}\Bigg\} ,$$
and
$$S_2(\tilde{R},a,b)=\Bigg\{\omega\in D^c(\tilde{R},a,b):  \frac{1}{2A}-\sqrt{\frac{1}{4A^2}+\frac{\tilde{R}}{\lvert A\rvert\lvert \delta-  f'(t)\rvert}}\leq \omega\leq  \frac{1}{2A}-\sqrt{\frac{1}{4A^2}-\frac{\tilde{R}}{\lvert A\rvert\lvert \delta-  f'(t)\rvert}}\Bigg\},$$
Suppose $A>0$ so $A=|A|$. In this case, the analysis of $S_2(\tilde{R},a,b)$ is simpler. 
From \eqref{eq:sqrt} we deduce that if $\omega\in S_2$ then
$$\lvert \omega\rvert  \leq 2\lvert A\rvert \frac{\tilde{R}}{\lvert A\rvert \lvert \delta-  f'(t)\rvert}=2\frac{\tilde{R}}{\lvert \delta-  f'(t)\rvert }\lesssim \frac{\tilde{R}}{\lvert \delta-  f'(t)\rvert}.$$
Therefore, in this set $\omega$ has an amplitude at most proportional to  $\tilde{R}\lvert\delta-  f'(t)\rvert^{-1}$ and so
\begin{equation}\label{eq:s2lambda}\Lambda(S_2(\tilde{R},a,b))\lesssim \frac{\tilde{R}}{\lvert \delta-f'(t)\rvert}.\end{equation}
The analysis of $S_1(\tilde{R},a,b)$ is much more delicate. We will show that $\omega$ must be contained on an interval of small length, although the center can be large. By subtracting $1/A=(\delta-f'(t))/hG_t(\omega h)$ to the inequalities defining $S_1(\tilde{R},a,b)$ and using \eqref{eq:sqrt} we have
$$\frac{-\tilde{R}}{|\delta-f'(t)|}\lesssim \sqrt{\frac{1}{4|A|^2}-\frac{\tilde{R}}{\lvert A\rvert\lvert \delta-  f'(t)\rvert}}-\frac{1}{2|A|}\leq \omega-\frac{\delta-f'(t)}{hG_t(\omega h)} \leq  \sqrt{\frac{1}{4A^2}+\frac{\tilde{R}}{\lvert A\rvert\lvert \delta-  f'(t)\rvert}}-\frac{1}{2|A|}\leq \frac{\tilde{R}}{|\delta-f'(t)|}.$$
The above implies that
$$\frac{-\tilde{R}}{|\delta-f'(t)|}\lesssim \frac{ h\omega G_t(\omega h)-(\delta-f'(t))}{hG_t(\omega h)} \lesssim  \frac{\tilde{R}}{|\delta-f'(t)|},$$
Since $xG(x)+f'(t)=F_t(x),$ and since $G_t(x)$ is bounded, the above, in turn, implies that whenever $\omega \in S_1(\tilde{R},a,b)$ then
\begin{equation}\label{eq:ftw} \lvert F_t(\omega h)-\delta\rvert \lesssim  h\frac{|G_t(\omega h)|\tilde{R}}{|\delta-f'(t)|}\lesssim \frac{h\tilde{R}}{|\delta-f'(t)|}:=\tilde{R}_h.\end{equation}
In what follows, we will show that the above inequality implies that $\omega$ belongs to a set whose Lebesgue measure by a quantity independent on $h$ (but depending on $\delta,\tilde{R},a,b, f$ and $t$). We will analyze several scenarios, depending on whether $\delta$ is in the range of $F_t$, whether $F_t$ is increasing or decreasing, and whether $a$ or $b$ are critical points. 
We will first assume that $F_t$ is increasing, that $b$ is a critical point with $F''_t(b)<0$, and that $a$ is not a critical point. Other scenarios will reduce to this one.

We will analyze different ranges for $\delta$. On each of them, will also separate $S_1(\tilde{R},a,b)$ into two sub-intervals, $S_1(\tilde{R},a,b)=\underline{S}_1(\tilde{R},a,b)\cap \bar{S}_1(\tilde{R},a,b)$ with $$\underline{S}_1(\tilde{R},a,b)=\{\omega \in S_1(\tilde{R},a,b); a\leq h\omega \leq h\omega_\kappa \} \quad \text{and} \quad \bar{S}_1(\tilde{R},a,b)=\{\omega \in S_1(\tilde{R},a,b); h\omega_{\kappa} <h\omega \leq b\},$$  where $\kappa$ is the length of the interval anticipated in Lemma \ref{lemma:variancebound3} and $\omega_{\kappa}:=(b-\kappa)/h$, so that $\lvert h\omega -b\rvert >\kappa$ on $\underline{S}_1(\tilde{R},a,b)$ and $ \lvert h\omega -b\rvert <\kappa $ on $\bar{S}_1(\tilde{R},a,b)$. By Lemma \ref{lemma:variancebound3}, on $\underline{S}_1(\tilde{R},a,b)$, $h\omega$ is far from the critical point $b$ and so $F'_t(h\omega)>L'>0$ for some constant $L'>0$ that depends only on the family $F_t$. 
Therefore, if $\omega_1,\omega_2 \in \underline{S}_1(\tilde{R},a,b)$, we have that whenever $\lvert F_t(h\omega_1)-F_t(h\omega_2)\rvert\lesssim Q$
\begin{equation}\label{eq:ftw1}
\lvert  h\omega_1-h\omega_2 \lvert \leq \frac{1}{L'} \lvert F_t(h\omega_1)-F_t(h\omega_2)\rvert \lesssim Q.
\end{equation}
Contrarily, on  $\bar{S}_1(\tilde{R},a,b)$, $h\omega$ is near the critical point $b$ and so by Lemma \ref{lemma:variancebound3} $F''_t(h\omega)<-L/2$, where $L$ is the uniform lower bound on the absolute second derivatives at critical points. Then, if $\omega_1,\omega_2\in \bar{S}_1(\tilde{R},a,b)$, by Lemma \ref{lemma:variancebound2} we have the relations
\begin{equation*}
\frac{1}{2}\sqrt{\frac{L}{2}} \lvert h\omega_1-h\omega_2 \rvert \lvert  F_t(h\omega_1) -F_t(b)\rvert^{1/2}  \leq \frac{L}{2} |h\omega_1 -b\rvert\lvert h\omega_1-h\omega_2\rvert \leq \lvert F_t(h\omega_1)-F_t(h\omega_2)\rvert,
\end{equation*}
that imply that whenever $\lvert F_t(h\omega_1)-F_t(h\omega_2)\rvert\lesssim Q$
\begin{equation}\label{eq:ftw2A}
\lvert h\omega_1-h\omega_2 \lvert   \lesssim \frac{Q}{\lvert F_t(h\omega_1)-F_t(b)\rvert^{1/2}}.
\end{equation}
and
\begin{equation}\label{eq:ftw2B}
\lvert h\omega_1-h\omega_2 \rvert    \lesssim \frac{Q}{\lvert h\omega_1-b\rvert}.
\end{equation}

Equipped with this, we will show that regardless of the value of $\delta$, we will have
\begin{equation}\label{eq:s1lambda}\Lambda(S_1(\tilde{R},a,b))=\Lambda(\underline{S}_1(\tilde{R},a,b))+\Lambda(\bar{S}_1(\tilde{R},a,b))\lesssim \frac{\tilde{R}}{|\delta-f'(t)|}+\frac{\tilde{R}}{|\delta-f'(t)|}\frac{1}{\lvert \delta -F_t(b)\rvert^{1/2} }.\end{equation}

 \textbf{Case 1: $\delta<F_t(a)$}

if $\omega\in \underline{S}_1(\tilde{R},a,b)$, since $h\omega >a$, so $\delta<F_t(a)<F_t(h\omega)$, and by \eqref{eq:ftw}$$ \lvert F_t(h\omega )-F_t(a)\lvert <\lvert F_t(\omega h)-\delta\rvert  \lesssim \tilde{R}_h.$$
Therefore, since $a\in \underline{S}_1(\tilde{R},a,b)$, by \eqref{eq:ftw1} the above implies that
$\lvert h\omega -a\rvert\lesssim \tilde{R}_h.$ 

If $\omega\in \bar{S}_1(\tilde{R},a,b)$, $h\omega \geq h\omega_\kappa> a$ and $\delta< F_t(a)<F_t(h\omega_{\kappa}) <F_t(h\omega)$, so by \eqref{eq:ftw},
$$ \lvert F_t(h\omega )-F_t(h\omega_\kappa)\lvert <\lvert F_t(h\omega)-\delta\rvert  \lesssim \tilde{R}_h.$$
and since both $\omega,\omega_{\kappa}\in \bar{S}_1(\tilde{R},a,b)$, and since $h\omega_{\kappa}-b = \kappa>0$, by \eqref{eq:ftw2B} the above implies that 
$$\lvert h\omega -h\omega_\kappa\rvert\lesssim \frac{\tilde{R}_h}{\lvert h\omega_\kappa-b\rvert }\lesssim \tilde{R}_h.$$
Let's compute the Lebesgue measure of the sets implied by the above conditions. If  $\omega\in \underline{S}_1(\tilde{R},a,b)$, $\vert h\omega-a\rvert \lesssim \tilde{R}_h$, so $\omega$ lies on an interval centered at a moving $a/h$ but with constant length. Therefore, $$\Lambda(\omega\in \bar{S}_1(\tilde{R},a,b))\leq\Lambda \Big( \omega\in \Big[\frac{a}{h}-\frac{\tilde{R}}{\lvert \delta-f'(t)\rvert},\frac{a}{h}+\frac{\tilde{R}}{\lvert \delta-f'(t)\rvert}\Big]\Big)\lesssim \frac{\tilde{R}}{\lvert \delta-f'(t)\rvert},$$
If $\omega\in \bar{S}_1(\tilde{R},a,b)$ then by the same argument (the interval is centered now at $\omega_\kappa$) we have the same bound for the size of the set. Then,
$$\Lambda(S_1(\tilde{R},a,b))=\Lambda(\underline{S}_1(\tilde{R},a,b))+\Lambda(\bar{S}_1(\tilde{R},a,b))\lesssim \frac{\tilde{R}}{|\delta-f'(t)|}.$$
and so \eqref{eq:s1lambda} holds.

\textbf{Case 2: $F_t(a)\leq \delta<F_t(h\omega_\kappa)$}

Suppose that $\omega \in \underline{S}_1(\tilde{R},a,b)$. Let $\omega_\delta$ be the unique such that $F_t(h\omega_\delta)=\delta$. We have $a\leq h\omega_\delta\leq h\omega_\kappa$, i.e., $\omega_\delta \in \underline{S}_1(\tilde{R},a,b)$.
 Then, by \eqref{eq:ftw},\eqref{eq:ftw1} we have
$$\lvert h\omega -h\omega_\delta\rvert\leq \frac{1}{L'}\lvert F_t(h\omega) -F_t(h\omega_\delta)\rvert = \frac{1}{L'}\lvert F_t(h\omega) -\delta \rvert  \leq  \tilde{R}_h.$$ 
If $\omega \in \bar{S}_1(\tilde{R},a,b)$, then, since $\delta<F_t(h\omega_\kappa)$, and by similar arguments as in Case 1, we pivot on $h\omega_\kappa$ to obtain
$$ \lvert F_t(h\omega )-F_t(h\omega_\kappa)\lvert <\lvert F_t(h\omega)-\delta\rvert  \lesssim \tilde{R}_h,$$
so that by \eqref{eq:ftw1},
$\lvert h\omega -h\omega_\kappa\rvert \lesssim \tilde{R}_h.$

In this case, the total measure is bounded by
$$\Lambda(S_1(\tilde{R},a,b))=\Lambda(\underline{S}_1(\tilde{R},a,b))+\Lambda(\bar{S}_1(\tilde{R},a,b))\lesssim \frac{\tilde{R}}{|\delta-f'(t)|},$$
so \eqref{eq:s1lambda} holds.

\textbf{Case 3: $F_t(h\omega_\kappa)\leq \delta < F_t(b)$}

If $\omega \in \underline{S}_1(\tilde{R},a,b)$, we pivot on $\omega_\kappa$ to use \eqref{eq:ftw} and \eqref{eq:ftw1}. Specifically, since $F_t(h\omega)<F_t(h\omega_\kappa)<\delta$, in this case
$$ \lvert F_t(h\omega )-F_t(h\omega_\kappa)\lvert <\lvert F_t(h\omega)-\delta\rvert  \lesssim \tilde{R}_h,$$
which implies that $\lvert h\omega  -h\omega_\kappa\rvert \lesssim \tilde{R}_h$.

Now, if $\omega \in \bar{S}_1(\tilde{R},a,b)$, define $\omega_\delta$ as before, since $h\omega_\kappa<h\omega_\delta<b$, we have that $\omega,\omega_{\delta}\in\bar{S}_1(\tilde{R},a,b)$. Then, 
$$ \lvert F_t(h\omega )-F_t(h\omega_\delta)\rvert = \lvert F_t(h\omega)-\delta\rvert  \lesssim \tilde{R}_h,$$
Therefore, by \eqref{eq:ftw} and \eqref{eq:ftw2A}
\begin{equation*}
\lvert h\omega-h\omega_{\delta} \lvert   \lesssim \frac{\tilde{R}_h}{\lvert F_t(h\omega_{\delta})-F_t(b)\rvert^{1/2}}=\frac{\tilde{R}_h}{\lvert \delta-F_t(b)\rvert^{1/2}}=\frac{h\tilde{R}}{|\delta-f'(t)||\delta-F_t(b)|^{1/2}}.
\end{equation*}

Combining the bounds, we control the Lebesgue measure as before, and \eqref{eq:s1lambda} holds.

\textbf{Case 4: $\delta>F_t(b)$}
In this case, if $\omega \in \underline{S}_1(\tilde{R},a,b)$ it is implied that $\lvert h\omega  -h\omega_\kappa\rvert \lesssim \tilde{R}_h$, by the same argument as in Case 3. If $\omega \in \bar{S}_1(\tilde{R},a,b)$ then, by \eqref{eq:ftw} and since $b$ is a local maximum
$$\frac{L}{2}\lvert h\omega -b\rvert^2 \lesssim F_t(b)-F_t(h\omega)\lesssim F_t(b)-(\delta-\tilde{R}_h).$$
In particular, for the set defined by the above inequality to be non-empty, $\delta-F_t(b) \leq \tilde{R}_h$.
Up to constants, we conclude that $\omega$ lies on an interval of length $$l(h)=\frac{1}{h}\sqrt{\tilde{R}_h-(\delta-F_t(b))}=\sqrt{\frac{\tilde{R}}{h\lvert \delta-f'(t)\rvert}-\frac{\delta-F_t(b)}{h^2}}$$  if $\delta-F_t(b) \leq \tilde{R}_h$, $l(h)=0$ otherwise. This length function admits a bound independent of $h$. Indeed, the function $g(x)=A/x-B/x^2$ for $x>0$ is increasing if $x\leq 2B/A$ and decreasing if $2B/A$. Therefore, $x^\ast=2B/A$ is the unique maximum, and $g(x^\ast)=A^2/4B$. Taking $A=\tilde{R}/|\delta-f'(t)|$ and $B=\delta-F_t(b)$ we obtain that
$$l(h)\lesssim\sqrt{\left(\frac{\tilde{R}}{|\delta-f'(t)}\right)^2\frac{1}{4(\delta-F_t(b))}}\lesssim \frac{\tilde{R}}{\lvert \delta-f'(t)\rvert } \frac{1}{\lvert\delta-F_t(b)\rvert^{1/2}}.$$
We have concluded that in all cases \eqref{eq:s1lambda} holds. In summary, by \eqref{eq:s1lambda},\eqref{eq:s2lambda} and \eqref{eq:dlambda} we have provided bounds for the Lebesgue measure for of the sets $S_1(\tilde{R},a,b)$ and $S_2(\tilde{R},a,b)$, and $D(\tilde{R},a,b)$ so that we can bound 
\begin{eqnarray*}\Lambda\left(S(\tilde{R},a,b)\right)&\leq& \Lambda\left(S_1(\tilde{R},a,b)\right)+\Lambda\left(S_2\left(\Lambda(\tilde{R},a,b\right)\right)+\Lambda\left(D(\tilde{R},a,b)\right)\\
&&\lesssim \frac{\tilde{R}}{|\delta-f'(t)|}+\frac{\tilde{R}}{|\delta-f'(t)|}\frac{1}{\lvert \delta -F_t(b)\rvert^{1/2} }.\end{eqnarray*}

Now, it remains to analyze possible scenarios. Note first that we have assumed that $A>0$. But if
 if, contrarily, $A<0$, then $A=-|A|$, and we can sketch the same argument switching the roles of $S_1(\tilde{R},a,b)$ and $S_2(\tilde{R},a,b)$ (for the analysis of $D(\tilde{R},a,b)$ we already considered these two scenarios). Also, so far we assumed that $F$ is increasing and $b$ is the only critical point of $F$. If $F$ was decreasing, the argument is essentially the same, as we can do the same case analysis in inverted ordering. Additionally, if $a$ is also a critical point (if $F$ is increasing, it must be a local minimum), then, again, we can replicate this case analysis: we now divide the $[a,b]$ interval into $[a,\omega_\kappa^a],[\omega_\kappa^a,\omega_\kappa^b]$ and $[\omega_\kappa^b,\delta]$ where $\omega_\kappa^a=a+\delta, \omega_\kappa^b=b-\delta$ and bound the measures of each of these sets for different values of $\delta$.  These can all be controlled by the same arguments as we did before, yielding the same final bound, but this time we must include a new term including the contribution of $a$. Therefore,
$$\Lambda\left(S(\tilde{R},a,b)\right)\lesssim \frac{\tilde{R}}{|\delta-f'(t)|}+\frac{\tilde{R}}{|\delta-f'(t)|}\frac{1}{\lvert \delta -F_t(b)\rvert^{1/2} }+\frac{\tilde{R}}{|\delta-f'(t)|}\frac{1}{\lvert \delta -F_t(a)\rvert^{1/2} }.$$

\end{proof}
\begin{lemma}\label{lemma:variancebound2}
Suppose that on the interval $0<b-x<\kappa$, $F(x)$ is strictly increasing, $F'(b)=0$, and that $F''(x)<-M$ for some constant $M>0$ (therefore, $b$ is a local minimum). Then, for any $x_1,x_2$ such that $|x_1-b|<\delta, |x_2-b|\leq \delta$,
$$ \frac{\sqrt{M}}{2}\lvert x_2-x_1\rvert \lvert F(b)-F(x_1)\rvert ^{1/2} \leq \frac{M}{2} \lvert x_2-x_1\rvert \lvert b-x_1\rvert  \leq \lvert F(x_2)-F(x_1)\rvert.$$
\end{lemma}
The same conclusion applies if the function is striclty decreasing, $F'(b)=0$ and $F''(x)>M$ for some $M>0$.
\begin{proof}
Let's assume first that $x_1<x_2$. By a second-order Taylor expansion on $F$ around $x_2$ and a first order expansion for $F$ around $b$, we have that for some $\psi_1\in [x_1,x_2]$ and $\psi_2\in[x_2,b]$
\begin{eqnarray*}
F(x_1)-F(x_2)&=& (x_1-x_2)\left(F'(x_2)+\frac{1}{2}F''(\psi_1)\left(x_1-x_2\right)\right) \\ &=&
(x_1-x_2)\left(F''(\psi_2)(x_2-b)+\frac{1}{2}F''(\psi_1)\left(x_1-x_2\right)\right)\end{eqnarray*} 
Then, inverting signs in the above, and using that $-F(x)>M$ on that interval and that $x_1<x_2<b$, the above implies
\begin{eqnarray*}
F(x_2)-F(x_1)&=& (x_2-x_1)  \left(F''(\psi_2)(x_2-b)+\frac{1}{2}F''(\psi_1)\left(x_1-x_2\right)\right) \\
&\geq &\lvert x_1-x_2\rvert\left(\frac{1}{2} F''(\psi_2)(x_2-b)+\frac{1}{2}F''(\psi_1)\left(x_1-x_2\right)\right) \\
&\geq &\lvert x_1-x_2\rvert\left(-\frac{1}{2} F''(\psi_2)(b-x_2)-\frac{1}{2}F''(\psi_1)\left(x_2-x_1\right)\right) \\
&\geq& \lvert x_1-x_2\rvert  \frac{M}{2} \left(b-x_2 +x_2-x_1\right)\\
&\geq& \frac{M}{2} \lvert x_1-x_2\rvert  \lvert b-x_1\rvert.
\end{eqnarray*}
We obtain the final bound using that, by second-order approximation for $F$ around $b$
$$|F(x_1)-F(b)|=|F''(\psi)||x_1-b|^2>M|x_1-b|^2$$
If now $x_1>x_2$, we can exchange the roles to obtain
$$F(x_1)-F(x_2)\geq  \frac{M}{2}\lvert x_1-x_2\rvert \lvert b-x_2\rvert  \geq  \frac{M}{2}\lvert x_1-x_2\rvert\lvert b-x_1\rvert\geq   \frac{\sqrt{M}}{2}\lvert x_1-x_2\rvert|F(x_1)-F(b)|^{1/2}, $$

where we used that $x_1>x_2$ so that $b-x_2>b-x_1>0$.
The conclusion follows since $x_1<x_2$ is equivalent to $F(x_1)<F(x_2)$. The proof for decreasing $F$ follows by applying the same argument to $-F$.
\end{proof}

\begin{lemma}\label{lemma:variancebound3}
Suppose that $F_t(x)$ is a coordinate of a function satisfying the conditions of Lemma \ref{lemma:variancebound0}. Then, 
There is $\kappa>0$ such that for each $t$ and critical point $x(t)$, 
$$0<\frac{L}{2}\leq \max_{|x-x(t)|<\delta}\Big \lvert F''_t(x)\Big \rvert.$$
Also, there is a constant $L'>0$ such that
$$\inf_{|x-x(t)|>\delta}\Big \lvert F'_t(x) \Big \rvert>L'.$$

% We create such partition in the following manner: if $M$ is the number of critical points of $F_t$ in $[-t,1-t]$, the partition will be made of $2M+1$ subintervals. 
 %For each critical point $x(t,j)$ we consider two values $\underline{x}(t,j)\leq x(t,j)\leq \bar{x}(t,j)$ such that 
%$$\frac{L_2}{4}\max_{x\in[\underline{x}(t,j),\bar{x}(t,j)]}\Big \lvert\frac{d^2}{dx^2}F_t(x)\Big \rvert\leq 2\min_{x\in[\underline{x}(t,j),\bar{x}(t,j)]}\Big \lvert\frac{d^2}{dx^2}F_t(x)\Big \rvert$$
%We have $a_0=-t$ and the

\end{lemma}
\begin{proof}

We first analyze the behavior around a critical point. Suppose that $F_t(x(t))>0$ (the other case is analogous). 
By the uniform equicontinuity of the second derivative, we have that for some $\delta>0$, whenever $|x-y|\leq \delta$, we have
$$\Big \lvert F''_t(x)-F''_t(y)\Big \rvert \leq \frac{L}{2}.$$
Then, choosing $y=x(t)$ we obtain that if $|x-x(t)|\leq \delta$ then 
$$\frac{L}{2} \leq -\frac{L}{2} +F''_t(x(t)) \leq F''_t(x).$$
Now, regarding the first derivative, for each $t$, define $L_t$ as the infimum of $|F'_t(x)|$ in the region $|x-x(t)|\geq \delta.$ We need to show that $\inf_{t\in[0,1]} L_t> 0$. Then, for some sequence $t_n\rightarrow t^\ast$ we have $L_{t_n}\to 0$. By continuity of $F_t$ and compactness, the infimum is realized on a certain $x_{t_n}$ satisfying $|x_{t_n}-x(t_n)|>\delta$ ($x_{t_n}$ is away from any critical point of $F_{t_n}$). The sequence $x_{t_n}$ must have an accumulation point $x^\ast=x_{t^\ast}$. By equicontinuity, this implies $L_{t_n}=F_{t_n}(x_{t_n})\to F(x_{t^\ast})$, so that $F(x_{t^\ast})=0$. Then, $x_{t^\ast}$ must be a critical point of $F_{t^\ast}$, contradicting that $x_{t^\ast}$ is away from the critical points of $F_{t^\ast}$.
\end{proof}

\subsection{Proof of Corollary \ref{cor:1d}}
 \begin{proof}[Proof of Corollary \ref{cor:1d}]
 We first show that in this case, $R(x)$ coincides with optimal transport. To see this, first we use that, by \eqref{eq:derphit} in Proposition \ref{prop:Volterra}
 \begin{eqnarray*}
\frac{d}{ds}\tilde{\Phi}(s) &=& -\tilde{\Phi}(s)\frac{\partial }{\partial z}v(s,z_s).
\end{eqnarray*}
%Indeed, this follows from the fact that $\tilde{\Phi}(s)=\Phi(1,s,z(s,0,x))=\Phi(1,s,z_s)$. Then, differentiating both sides with respect to $s$, and using Lemma \ref{lemma:fundamentalprop}:
%\begin{eqnarray*}
%\frac{d}{ds}\Phi(1,s,z_s) &=& \frac{\partial}{\partial t_0}\Phi(1,s,z_s)+\frac{\partial }{\partial z_0}\Phi(1,s,z_s)\frac{d}{ds}z_s\\&=&
%-\frac{\partial }{\partial z_0}\Phi(1,s,z_s)v(s,z_s)-\Phi(1,s,z_s) \frac{\partial v}{\partial z} (s,z_s)+\frac{\partial }{\partial z_0}\Phi(1,s,z_s)v(s,z_s)\\
%&=&-\tilde{\Phi}(s)\frac{\partial }{\partial z}v(s,z_s)
%.
%\end{eqnarray*}

 %by Lemma \ref{lemma:fundamentalprop}
%\begin{eqnarray*}
%\frac{d}{ds}\Phi(1,s,z_s)&=&-\frac{\partial }{\partial z} \Phi(1,s,z_s) v(s,z_s) -\frac{\partial }{\partial z}v(s,z_s)\Phi(1,s,z_s)+\frac{\partial}{\partial z}\Phi(1,s,z_s)\frac{d}{ds}z_s\\
%&=& -\frac{\partial }{\partial z}v(s,z_s)\Phi(1,s,z_s).
%\end{eqnarray*}
Therefore, we can write, for $0\leq s\leq 1$
$$\tilde{\Phi}(s)=\exp\left(\int_s^1 \frac{\partial v}{\partial z}(t,z_t)dt\right).$$
In particular, for $s=0$ we obtain
$R'(x)=\tilde{\Phi}(1)>0$. Therefore, $R$ is an increasing transport map, so it must be the optimal transport.

 Now, let's prove the remaining claims. Note that if we can establish the improved bias rate $h_n^{\beta+1}$ then the CLT follows easily, as all the arguments for the variance analysis don't depend on the dimension. We only need to establish this rate and show that the asymptotic variance doesn't depend on the Kernel.
 
To establish the rate, recall first that, as in the proof of Proposition \ref{prop:linearized}, we need to bound the bias
\begin{eqnarray*} b(h):=\mathbb{E}(\hat{L}(x))=\int_0^1\int_{\mathbb{R}} \frac{\tilde{\Phi}(s)}{p_s(z)} K\left(u\right)\left(v(s,z+hu)-v(s,z)\right)p_s(z+hu) duds.
\end{eqnarray*}
We will use the following result, as stated in Lemma \ref{lemma:intdm} (note that we are allowed to compute first derivatives since $\beta>2$).
\begin{equation}
u\frac{d}{ds}\left(\tilde{\Phi}(s)\frac{p_s(z_s+hu)}{p_s(z_s)}\right)=-\frac{\tilde{\Phi}(s)}{p_s(z_s)}\frac{d}{d h} m(u,h,s),
\end{equation}
 where $$m(u,h,s)=\left(v(s,z_s+hu)-v(s,z_s)\right)p_s(z_s+hu).$$
 This implies that the function 
 $$b_0(h,u) = \int_0^1 \frac{\tilde{\Phi}(s)}{p_s(z)}K(u) m(u,h,s)ds$$
 is differentiable with respect to $h$ at all $u$, with derivative 
\begin{eqnarray*}
\frac{d}{dh} b_0(h,u)&=&-uK(u)\int_0^1 \frac{d}{ds}\left(\tilde{\Phi}(s)\frac{p_s(z_s+hu)}{p_s(z_s)}\right)ds\\&=&uK(u)\left(\tilde{\Phi}(0)\frac{p_0(x+hu)}{p_0(x)}-\tilde{\Phi}(1)\frac{p_1(R(x)+hu)}{p_1(R(x))}\right)\\
&=& uK(u)\left(\tilde{\Phi}(1)\frac{p_0(x+hu)}{p_0(x)}-\frac{p_1(R(x)+hu)}{p_1(R(x))}\right)\\
\end{eqnarray*}
%where the function $$m_2(u,h):=$
At this point, the argument is the same as in usual kernel density estimation; by using the H{\"o}lder regularity of $p_0,p_1$ now we can demonstrate that $b_0(h)$ has $\beta+1$ H{\"o}lder regularity with respect to $h$ and use it on the Taylor expansions as in the proof of Proposition \ref{prop:linearized}. Exchanging derivatives and integration with respect to $u$ is justified by continuity of $\Phi,p_0,p_1$ and boundedness of the kernel.

 It only remains to show \eqref{eq:sigmax1d}. To do so, we consider the change of variables $\omega_1=u+\omega(\Delta-v_t(z_t))$ in the integral with respect to $\omega$ in \eqref{eq:sigma}, whenever $(\Delta-v_t(z_t))\neq 0$. Then, we obtain
 \begin{eqnarray*}\nonumber 
\Sigma(x) &=& \int_0^1  \frac{\tilde{\Phi}(t) }{p_t(z_t)}  \mathbb{E}\left[ \frac{\left(\Delta-v_{t}(z_{t})\right)^2}{\lvert \Delta-v_{t}(z_{t})\rvert }\int_{\mathbb{R}} \int_{-\infty}^\infty K\left(\omega'_1\right) K\left(u\right)dud\omega_1   \Big |X_t=z_t\right]  \tilde{\Phi}(t) dt.\\
&=& \int_0^1  \frac{\tilde{\Phi}^2(t) }{p_t(z_t)}  \mathbb{E}\left[ \lvert\Delta-v_{t}(z_{t})\rvert   \Big |X_t=z_t\right]   dt,
 \end{eqnarray*}
 a quantity that doesn't depend on $K$.
 \end{proof}

\subsection{Proof of Example \ref{example:1dgaussian}}
\begin{proof}[Proof of Example \ref{example:1dgaussian}]
We can easily extend the argument in the proof of Proposition \ref{prop:flow-from-Gaussian-to-Gaussian} to show that if $X_i\sim N\left(m_i,\Sigma_i\right)$ are independent, then $z_t=z_t(x)=z(t,0,x)$ satisfies
$$z_t(x)= m_t+\Sigma_0^{1/2}\left(\Sigma_0^{-1/2}(t^2\Sigma_1+(1-t)^2\Sigma_0)\Sigma_0^{-1/2}\right)^{1/2}\Sigma_0^{-1/2}\left(x-m_0\right).$$
Also, $$v(t,z_t)=m_1-m_0 +(t\Sigma_1-(1-t)\Sigma_0)\left(t^2\Sigma_1+(1-t)^2\Sigma_0\right)^{-1}(z_t-m_t),$$
and since $X_t\sim N\left(m_t,t^2\Sigma_1+(1-t)^2\Sigma_0\right)$, in the one dimensional case we get
\begin{eqnarray*}
\tilde{\Phi}(t)=\exp\left(\int_t^1\frac{\partial }{\partial z}v(s,z_s) ds\right)&=&\exp\left(\int_t^1(s\Sigma_1-(1-s)\Sigma_0)\left(s^2\Sigma_1+(1-s)^2\Sigma_0\right)^{-1}ds\right)\\
&=& \frac{\exp\left(\frac{1}{2}\log\left(\Sigma_1\right)\right)}{\exp\left(\frac{1}{2}\log\left(t^2\Sigma_1+(1-t)^2\Sigma_0\right)\right)}\\&=&\Sigma_1^{1/2}\left(t^2\Sigma_1+(1-t)^2\Sigma_0\right)^{-1/2}.
\end{eqnarray*}
Additionally, it is easy to check that the distribution of $\Delta=X_1-X_0$ conditional on $X_t=z_t$ is Gaussian with mean $m$ and variance $\Sigma$ given by
$$m:=m_0-m_1+\frac{s\Sigma_1-(1-s)\Sigma_0}{s^2\Sigma_1+(1-s)^2\Sigma_0}, \quad \text{ and } \quad\Sigma:=\frac{1}{t^2\Sigma_1+(1-t)^2\Sigma_0}.$$

Therefore, $$\mathbb{E}(\lvert \Delta-v(t,z_t)\lvert \Big |X_t=z_t)=\sqrt{\frac{2\Sigma}{\pi}}=\frac{1}{(t^2\Sigma_1+(1-t)^2\Sigma_0)^{1/2}}\sqrt{\frac{2}{\pi}},$$
and 
\begin{eqnarray*}
\Sigma(x)&=&\int_0^1 \frac{\tilde{\Phi}^2(t) }{p_t(z_t)}  \mathbb{E}\left[ \lvert\Delta-v_{t}(z_{t})\rvert   \Big |X_t=z_t\right]ds\\
&=&\int_0^1 \frac{\Sigma_1 \sqrt{2\pi}(t^2\Sigma_1+(1-t)^2\Sigma_0)^{1/2}}{t^2\Sigma_1+(1-t)^2\Sigma_0} \exp\left( \frac{(z_t-m_t)^2}{2(t^2\Sigma_1+(1-t)^2\Sigma_0)}\right)\frac{1}{(t^2\Sigma_1+(1-t)^2\Sigma_0)^{1/2}}\sqrt{\frac{2}{\pi}}dt\\
&=& 2\Sigma_1\int_0^1 \frac{1}{t^2\Sigma_1+(1-t)^2\Sigma_0}\exp\left(\frac{1}{2\Sigma_0}(x-m_0)^2\right)dt\\
&=& 2\frac{\Sigma_1^{1/2}}{\Sigma_0^{1/2}}\left(\arctan\left(\frac{\Sigma_1^{1/2}}{\Sigma_0^{1/2}}\right)+\arctan\left(\frac{\Sigma_0^{1/2}}{\Sigma_1^{1/2}}\right)\right)\exp\left(\frac{1}{2\Sigma_0}(x-m_0)^2\right)
\end{eqnarray*}
\end{proof}

\subsection{Additional results for Section \ref{sec:unbounded}}\label{sub:addunbounded}
\begin{lemma}[Lemma 2.2 in \cite{gotze2019higher}] \label{lemma:lpmomentpoincare} Let $\mu$ be a probability measure on $\mathbb{R}^d$ satisfying the Poincaré inequality \eqref{eq:poincare} with constant $\sigma^2>0$. If $h$ is a locally Lipschitz function with $\mathbb{E}_\mu(h)=0$, then, for any $p\geq 2$:
\begin{equation} \label{eq:lpmomentpoincare} \mathbb{E}_\mu\left(\lvert h(X)\rvert ^p\right)\leq \left(\frac{\sigma p}{\sqrt{2}}\right)^p \mathbb{E}_\mu \left(\lVert \nabla h(X)\rVert^p \right).\end{equation}

\end{lemma}

\begin{lemma}[Lyapunov CLT]
\label{lemma:lyapunov} Let $U_{n,i}$ by a triangular array in $\mathbb{R}^d$ with mean zero and finite second moment. Define $V_n=\sum_{i=1}^n\text{Cov}(U_{n,i})$ and let $v_n^2=\lambda_{min}(V_m)$ (the least eigenvalue of $V_n$) If for some $\delta>0$ we have that $$\frac{1}{v_n^{2+\delta}}\sum_{i=1}^n\mathbb{E}\left(|U_{n,i}^{2+\delta}|\right)=o(1),$$
then, $$V_n^{-1/2}\sum_{i=1}^nU_{n,i}\rightarrow N\left(0,I_d\right).$$
\end{lemma}
\begin{lemma}[Gronwall's inequality]\label{lemma:gronwall} If $u$ and $\beta$ are differentiable and continuous real-valued functions defined on an interval $I$ such that $u'(t)\leq \beta(t)u(t)$ in the interior of $I$, then 
$$u(t)\leq u(t_0)\exp\left(\int_{t_0}^t\beta(s)ds\right).$$
\end{lemma}

\begin{lemma}
\label{lemma:nonreg}
Let $Y,X$ be random variables in $\mathbb{R}^d$. Consider the problem of estimating the conditional mean $\mu(x)=\mathbb{E}(Y|X=x)$ from i.i.d pairs $(X_{0i},X_{1i})$. Suppose that $X$ has a bounded density $p$ which is of H{\"o}lder class $\beta$ and that $K$ is a kernel of order $l=\left \lfloor{\beta}\right \rfloor$. Assume that $h=o(1)$, $nh^d\rightarrow \infty$. Assume also that for some $\delta>0$ and every $x$
$$\mathbb{E}(|Y-\mu(X)|^{2+\delta}|X=x)<M,\quad \int |K(\psi)|^{2+\delta}d\psi <\infty. $$
Then, the classical Nadaraya-Watson kernel density estimator $\hat{\mu}_h(x)$ of $\mu(x)$
$$\hat{\mu}_h(x)\equiv\frac{\sum_{i=1}^n K_h(X_{0i}-x)X_{1i}}{\sum_{i=1}^n K_h(X_{0i}-x)},$$ satisfies a central limit theorem
$$\sqrt{nh^d}(\hat{\mu}_h(x)- \mu_h(x)) ~\overset{d}{\to} N\left(0,\frac{\big\Vert K \big \Vert_{L_2}^2}{p(x)}\Sigma(x)\right),$$
where $\mu_h=\mu(x)+O(h^\beta)$ and $\Sigma(x)\equiv \mathrm{Var}(Y|X=x)=\mathbb{E}(Y^2|X=x)-\mathbb{E}(Y|X=x)^2$.
The proof is standard  (e.g. see \cite{li2023nonparametric, ullah1999nonparametric}) and we skip it.
\end{lemma}

\begin{lemma}\label{lemma:unifcons}
Let $(X_{i},Y_{i})\in \mathbb{R}^{d_1}\times \mathbb{R}^{d_2}$ be an i.i.d sequence of random variables. Let $K$ be a bounded continuous kernel. For fixed continuous functions $c$ and $d$ define
\begin{equation}\label{eq:limempiricalw} W_n(z)=\frac{1}{n}\sum_{i=1}^n \left(c(z)Y_{i}+d(z)\right) K_h(X_{i}-z)-\mathbb{E}\left(\left(c(z)Y_{i}+d(z)\right) K_h(X_{i}-z)\right).\end{equation}
Suppose that the density $p$ of $X$ is continuous and strictly positive over a compact $B\subseteq\mathbb{R}^{d_1}$, and that the joint density of $X,Y$ is continuous in $B\times \mathbb{R}^{d_2}$. Then, with probability one,
\begin{equation}\label{eq:limempirical}\lim_{n\rightarrow \infty} \sqrt{\frac{nh_n^d}{2\log h_n^{-d}}}\sup_{z\in B}||W_n(z)||=\sigma^2(B),\end{equation}
where \begin{equation}\label{eq:limempiricalsigma} \sigma^2(B)=\sup_{z\in B}\mathbb{E}\left(\Vert(c(z)Y+d(z)\rVert^2|X=z\right)p(z)||K||_{L_2}^2.\end{equation}
\end{lemma}

\begin{proof}
This is a simple multivariate extension of Theorem 1 in \cite{einmahl2000empirical}. Extension of closely related results to the multivariate case has already been pursued, e.g., in \cite[Proposition 3.1]{gine2002rates}, so we skip the details.
 \end{proof}
 \begin{lemma}\label{lemma:einimproved}
 Let $X_t=tX_1+(1-t)X_0$ for some random variables $X_0,X_1$ with bounded second moments and $t\in[0,1]$.  Suppose that the the hypothesis of Lemma \ref{lemma:unifcons}, hold for $(X_t,Y)$ for each $t\in[0,1]$, defining $W_n(t,z)$ and $\sigma_t(B)$ the time-dependent counterparts of \eqref{eq:limempiricalw} and \eqref{eq:limempiricalsigma}. Suppose further that $K$ has a bounded derivative. If $\tilde{\sigma}^2(B):=\sup_{t\in[0,1]}\sigma^2_t(B)<\infty$, and $nh_n^d/\log(1/h_n) \to \infty$ as $n\to\infty$ then, 
 \begin{equation}\label{eq:limempirical}\sqrt{\frac{nh_n^d}{\log h_n^{-1}}}\sup_{z\in B,t\in [0,1]}||W_n(t,z)|| = O_p(1),\end{equation}
 \end{lemma}
 \begin{proof}

 This result is a uniform in $t$ version of Lemma \ref{lemma:unifcons}, but weakened in the sense that the limit is replaced by a bounded in probability statement. These weaker versions are usually established as intermediate results, e.g., in \cite{einmahl2000empirical,gine2002rates,einmahl2005uniform,hansen2008uniform}, and are sufficient for our applications. The proof of this extension follows from the standard arguments in these papers; here, he will only mention the main idea and argue that uniformity over $t$ doesn't degrade the rates.

 Define, for a class $\mathcal{G}$ the following quantity
 $$\Big\lVert  \mathbb{P}_n-P \Big \rVert_{\mathcal{G}}:=\sup_{g\in\mathcal{G}}|\mathbb{P}_ng-Pg|=\sup_{g\in\mathcal{G}}\Big \lvert \frac{1}{n}\sum_{i=1}^n g(X_i)-\mathbb{E}(g(X))\Big \rvert$$
  We will use the following chaining bound \cite[Remark 3.4.2]{gine2021mathematical}
 \begin{equation} \label{eq:chaing} \mathbb{E} \left(\sqrt{n} \Big\lVert  \mathbb{P}_n-P \Big \rVert_{\mathcal{G}} \right) \leq 8\sqrt{2} \mathbb{E}\left[\int_0^{\sqrt{\|P_n g^2\|_{\mathcal{G}}}}\sqrt{\log 2N(\tau, \mathcal{F},L_2(P_n))}d\tau \right]\end{equation}
Consider the case $c(z)=0,d(z)=1$. We study the function class
 $$\mathcal{G}_h:=\Big\{g:\mathbb{R}^d\times \mathbb{R}^d\to \mathbb{R}, g(x,y)=\frac{1}{h^d} K\left(\frac{tx+(1-t)y -z}{h}\right),z\in B,  t\in[0,1]\Big\}.$$
 Note that an envelope for $\mathcal{G}_h$ is $\lVert K\rVert_\infty/h^d$. We can link this family to the more elementary classes
 $$\mathcal{F}=\bigcup_{h>0}h^d \mathcal {F}_h \quad \mathcal{F}_h=\Big\{f:\mathbb{R}^d\to\mathbb{R}, f(x)= \frac{1}{h^d} K\left(\frac{x -z}{h}\right),z\in B\Big\}$$
We will bound the covering number of $\mathcal{G}_h$ using the well-known fact, by assumption \ref{assump:K}, that K is a regular Kernel in the sense of \cite{einmahl2005uniform}, it follows that $\mathcal{F}$ is a bounded VC class \citep{einmahl2005uniform}, so it satisfies following bound, for each measure $\tilde{P}$ \cite[Equation 2.1]{gine2002rates}, \cite{van1996bracketing} $$N\left(\varepsilon\|F\|_{L_2(\tilde{P})} ,\mathcal{F},L_2(\tilde{P})\right)\leq \left(\frac{A}{\varepsilon}\right)^{\nu},$$  for some $A>3e,\nu\geq 1$. Note that, by scaling, and since $N\left(\varepsilon,\mathcal{F}^1,L_2(\tilde{P})\right)\leq N\left(\varepsilon ,\mathcal{F}^2,L_2(\tilde{P})\right)$ if $\mathcal{F}^1\subseteq \mathcal{F}^2$, this immediately implies a bound on the covering number of the $\mathcal{F}_h$'s:
\begin{equation} \label{eq:covfh} N\left(\varepsilon ,\mathcal{F}_h,L_2(\tilde{P})\right)=N\left(h^d\varepsilon, h^d\mathcal{F}_h,L_2(\tilde{P})\right)  \leq N\left(h^d\varepsilon, \mathcal{F},L_2(\tilde{P})\right)\leq \left(\frac{A\|F\|_{L_2(\tilde{P})}}{h^d\varepsilon}\right)^{\nu}\leq \left(\frac{A\lVert K\rVert_\infty}{h^d\varepsilon}\right)^{\nu},\end{equation}
Where the envelope $F$ can be taken as $\lVert K\rVert_\infty$.
Now we will bound the covering number of $\mathcal{G}_h$ in terms of the one for $\mathcal{F}_h$.
More precisely, we now show that for some $C_2>0$, for any $h>0$, and any $P$ in the product space, if we define $\Lambda^2(P)=\int \Vert x-y\Vert^2dP(x,y)$, and if $N\left(\varepsilon,\mathcal{F}\right):=\sup_{Q}N\left(\varepsilon,\mathcal{F},L_2(Q)\right)$ then
\begin{equation}\label{eq:nfg} N\left(\varepsilon+\frac{\lVert K'\rVert_\infty\eta}{h^{1+d}}\Lambda(P), \mathcal{G}_h,L_2(P)\right)\leq \frac{1}{\eta} N\left(\varepsilon,\mathcal{F}_h\right)\leq \frac{1}{\eta} N\left(\varepsilon,\mathcal{F}\right)\leq \frac{1}{\eta} \left(\frac{A\lVert K\rVert_\infty}{h^d\varepsilon}\right)^{\nu}.\end{equation}
To show this, let $\varepsilon>0$, $\eta>0$, and a probability measure $P$ in the product space. For each $k=1,\ldots, \lfloor 1/\eta\rfloor$ consider an $\varepsilon$-covering $f^k_i(x)$ of the space $L_2(P^k)$ of size $N\left(\varepsilon,\mathcal{F}_h,L_2(P^k)\right)$, where $P^k$ is the law of $k\eta X_1 +(1-k\eta) X_0$, and $(X_0,X_1)\sim P$. By \eqref{eq:covfh}, $N\left(\varepsilon,\mathcal{F}_h,L_2(P^k)\right)\leq \left(\frac{h^dA\lVert K\rVert_\infty}{\varepsilon}\right)^{\nu}.$ We now show how to cover the class $\mathcal{G}_h$ with the family $g^k_i(x,y):=f^k_i(k\eta x+(1-k\eta)y)$ whose size is bounded by $\sup_{\tilde{P}} N\left(\varepsilon,\mathcal{F}_h,L_2(\tilde{P})\right)/\eta$. Let $g\in\mathcal{G}_h$, then $g(x,y)=f(tx+(1-t)y)$ for some $t\in[0,1]$ and $f\in\mathcal{F}_h$. Let $f^k_i$ be such that $|\eta k-t|\leq \eta$ and such that $d_{L_2(Q^k)}(f,f^k_i)\leq \varepsilon$. Then, if we define $g^k(x,y)=f(\eta k x+(1-\eta k)y)$ we have $d_{L_2(P)}(g,g^k_i)\leq d_{L_2(P)}(g,g^k)+d_{L_2(P)}(g^k,g^k_i)$. We can bound each individual term: first, note that
\begin{eqnarray*} d^2_{L_2(P)}(g^k,g_i^k)&=&\int (g^k(x,y)-g^k_i(x,y))^2dP(x,y)\\
&=& \int \left(f(k\eta x+(1-k\eta)y)-f_i^k(k\eta x+(1-k\eta)y)\right)^2dP(x,y)\\
&=&\int \left(f(x)-f_i^k(x)\right)^2dP^k(x)\\
&=&d^2_{L_2(P^k)}(f,f^i_k)\\
&\leq &\varepsilon^2
\end{eqnarray*}
Above, we used the definition of $P^k$ in terms of $P$, and the fact that $d_{L_2(P^k)}(f,f^k_i)\leq \varepsilon$. Additionally, 
\begin{eqnarray*}
d^2_{L_2(P)}(g,g^k)&\leq&
\int \left(f(tx+(1-t)y)-f(k\eta x+(1-k\eta)y)\right)^2 dP(x,y) \\
&\leq & \sup_{z}\lVert K'(z)\rVert^2 |t-\eta k|^2 \int  \lVert x- y\rVert^2 dP(x,y)\\
&\leq & \lVert K'\rVert_\infty^2\Lambda^2(P)\frac{\eta^2}{h^{2+2d}}. 
\end{eqnarray*}
In the last inequality, we used the fact that $K'$ has a bounded derivative, so \eqref{eq:nfg} is proven. Since \eqref{eq:nfg} is valid for all $\varepsilon,\eta, h$ we can choose $\eta=h^{1+d}\varepsilon/(\lVert K'\rVert_\infty\Lambda(P))$.
to obtain that for any measure $P$ in the product space
\begin{equation}\label{eq:nfg2} N\left(2\varepsilon, \mathcal{G}_h,L_2(P)\right)\leq  \frac{\lVert K'\rVert_\infty\Lambda(P)}{h^{1+d}\varepsilon} \left(\frac{A\lVert K\rVert_\infty}{h^d\varepsilon}\right)^{\nu}\leq \frac{\Lambda(P)}{h^{1+d+\nu d}}\left(\frac{A'}{\varepsilon}\right)^{1+\nu},\end{equation}
for some big enough $A'=f(A,\lVert K\rVert_\infty,\lVert K'\rVert_\infty,\nu)$. 
We will now apply \eqref{eq:chaing} to the family $\mathcal{G}_h$. Call $\delta_n^2 = \lVert P_n g^2\rVert_{\mathcal{G}_h}$ and $\delta^2 = \lVert P g^2\rVert_{\mathcal{G}_h}$. First, using \eqref{eq:nfg2} and that $\sqrt{a+b}\leq \sqrt{a}+\sqrt{b}$ we have
\begin{eqnarray}
\nonumber \int_0^{\delta_n} \sqrt{\log 2N(\tau,\mathcal{G}_h,L_2\left(P_n\right))}d\tau
&\lesssim& \int_0^{\delta_n} \sqrt{\log \Lambda(P_n)}dt+\int_0^{\delta_n}\sqrt{\log h^{-(1+d+\nu d)}}dt +\int_0^{\delta_n}\sqrt{\log (2A't)^{-2(\nu+1)} }dt\\
\label{eq:chainint} &\lesssim & \delta_n\sqrt{\log \Lambda(P_n)}+ \delta_n \sqrt{\log h^{-(1+d+\nu d)})}+A''.
\end{eqnarray}
In the last inequality, we used the fact that the last integral can be bounded uniformly, regardless of the limit of integration (this follows from an elementary change of variables). To apply \eqref{eq:chaing}, we need to take expectations. To do so, we will and the fact that (see the proof of  Theorem 3.5.4 in \cite{gine2021mathematical})
$$n\mathbb{E}\left(\delta^2_n\right)\leq n\delta^2+2^53^{3/2} \sqrt{\mathbb{E}(U^2)}\left(2^3\sqrt{n}\mathbb{E} \left(\sqrt{n} \Big\lVert  \mathbb{P}_n-P \Big \rVert_{\mathcal{G}_h} \right)+\sqrt{\mathbb{E}(U^2)}\right),$$
where $U=\max_{1\leq i \leq n} G_h(X_i)$ and $G_h$ is an envelope for $\mathcal{G}_h$. In our case, we can bound $U\leq \lVert K\rVert_\infty h^{-d}$. Therefore, 
\begin{eqnarray}
\nonumber \mathbb{E}\left(\delta^2_n\right)&\lesssim& \delta^2+ h^{-d} \left(\frac{1}{\sqrt{n}} \mathbb{E} \left(\sqrt{n} \Big\lVert  \mathbb{P}_n-P \Big \rVert_{\mathcal{G}_h} \right)+ \frac{h^{-d}}{n}\right)\\
\label{eq:deltan}&\lesssim& h^{-d}\left(1+  \frac{1}{\sqrt{n}} \mathbb{E} \left(\sqrt{n} \Big\lVert  \mathbb{P}_n-P \Big \rVert_{\mathcal{G}_h}\right) + \frac{h^{-d}}{n}\right).
\end{eqnarray}

In the last line, we used that that for each $g\in \mathcal{G}_h$ we have the bound 
$$\mathbb{E}\left(g(X)^2\right)= \int \frac{1}{h^{2d}} K^2\left(\frac{X-z}{h}\right)p(x)dx=\frac{1}{h^d}\int  K^2\left(u\right)p(z+hu)du\leq h^{-d}\lVert p\rVert_\infty\lVert K\rVert_2^2,$$
so that $\delta^2\lesssim h^{-d}$.

Denoting $Y=\mathbb{E}\left(\sqrt{n} \Big\lVert  \mathbb{P}_n-P \Big \rVert_{\mathcal{G}_h} \right)$, by \eqref{eq:chaing}, \eqref{eq:chainint}, \eqref{eq:deltan} we have

\begin{eqnarray*}
 Y &\leq& \sqrt{\mathbb{E}\left(\delta^2_n\right)} \frac{1}{\sqrt{2}}\sqrt{\mathbb{E}\left(\log \Lambda^2(P_n)\right)}+\sqrt{\mathbb{E}\left(\delta^2_n\right)}\sqrt{\log h^{-(1+d+\nu d)})}+A''\\
 &\lesssim&
\frac{1}{\sqrt{2}}\sqrt{\mathbb{E}\left(\delta^2_n\right)} \sqrt{\log \mathbb{E}\left( \Lambda^2(P_n)\right)}+\sqrt{\mathbb{E}\left(\delta^2_n\right)}\sqrt{\log h^{-(1+d+\nu d)})}+A''\\
&\lesssim&
\frac{1}{\sqrt{2}} \sqrt{\mathbb{E}\left(\delta^2_n\right)} \sqrt{\log \left( \Lambda^2(P)\right)}+\sqrt{\mathbb{E}\left(\delta^2_n\right)}\sqrt{\log h^{-(1+d+\nu d)})}+A''\\
&\lesssim& 
h^{-d/2}\left(1 +\sqrt{\frac{Y}{\sqrt{n}}}+ \frac{h^{-d/2}}{\sqrt{n}} \right)\left(1+\sqrt{\log h^{-1})}\right) +1
\end{eqnarray*}
In the first line, we used Cauchy-Schwarz. In the second, Jensen's inequality. In the third, the fact that $\Lambda^2(P)$ is linear in $P$. We can express the right-hand side above as $a\sqrt{Y}+b$ where $a=h^{-d/2}n^{-1/4}(1+\sqrt{\log h^{-1}})$ and $b=(1+\sqrt{h^{-1}})(h^{-d/2}+h^{-d}n^{-1/2})+1$. Since the relation $Y\lesssim a\sqrt{Y}+b$ implies $Y\lesssim a^2+b$ we obtain
$$Y\lesssim h^{-d}n^{-1/2}(1+\log h^{-1})+ (1+\sqrt{\log h^{-1}})(h^{-d/2}+h^{-d}n^{-1/2})+1.$$
Therefore,
\begin{eqnarray*}
\mathbb{E}\left(\frac{\sqrt{nh_n^d}}{\sqrt{\log h^{-1}_n}} \sup_{z\in B,t\in [0,1]}||W_n(t,z)||\right) &=&
\mathbb{E}\left(\frac{\sqrt{nh_n^d}}{\sqrt{\log h^{-1}_n}} \Big\lVert  \mathbb{P}_n-P \Big \rVert_{\mathcal{G}_h} \right) \\
&=&
\mathbb{E}\left(\frac{\sqrt{h_n^d}}{\sqrt{\log h^{-1}_n}} Y \right)\\
&\lesssim& \frac{(1+\log h_n^{-1})}{\sqrt{\log h^{-1}_n nh_n^d}}+ \frac{(1+\sqrt{\log h_n^{-1}})}{\sqrt{\log h_n^{-1}}}\left(1+\frac{1}{\sqrt{nh_n^d}}\right)+\frac{\sqrt{h_n^{d}}}{\sqrt{\log h_n^{-1}}}\\
&\lesssim &
1.
\end{eqnarray*}
In the last line, we used that by our hypothesis on the bandwidth, $\sqrt{\log h_n^{-1}}^{-1}=o(1)$, $\sqrt{nh_n^d}^{-1}=o(1)$, $\sqrt{h^d_n}=o(1)$, and $\sqrt{\log h^{-1}_n}\sqrt{ nh_n^d}^{-1}=o(1)$.
Since the expectation is bounded, we concluded that its argument is $O_p(1)$.

We have shown the desired statement in the case where $c(z)=0,d(z)=1$. For the general case, we define the class of functions
$$\mathcal{W}_h=\Big\{g:\mathbb{R}^d\times \mathbb{R}^d\to \mathbb{R}, l(x,y,v)=\frac{1}{h^d} \left(c(z)y +d(z)\right) K\left(\frac{tx+(1-t)y -z}{h}\right),z\in B,  t\in[0,1]\Big\},$$
and proceed as in \cite{einmahl2000empirical,einmahl2005uniform,hansen2008uniform}. The core of the argument was already extended, but a tedious truncation and conditioning argument is still necessary to control fluctuations when $X_1-X_0$ is unbounded. We skip these calculations as the argument remains unchanged in this case, since we have already provided appropriate uniform control over the time-dependent classes.

 \end{proof}

\begin{lemma}\label{lemma:kernelderbound}
Let $\hat{v}_t(z)$ be the kernel-based estimator of the velocity, and $\hat{f}_t(z),\hat{p}_t(z)$ be the corresponding numerator and denominator. Then, under the assumptions of Theorem \ref{teo:CLT} we have that if $L=O_p(1)$,

\begin{itemize}
\item[(a)]

\begin{equation}\label{eq:biasker} \sup_{t\in[0,1],\lVert z\rVert \leq L} \lVert \mathbb{E}\left(\widehat{p}_t(z)\right)-p_t(z)\rVert =O\left(h_n^\beta \right)
    \end{equation}
    and 
    \begin{equation}\label{eq:biasdif} \sup_{t\in[0,1],\lVert z\rVert \leq L} \lVert \mathbb{E}\left(\hat{f}_t(z)-v_t(z)\hat{p}_t(z)\right)\rVert =O\left(h_n^\beta \right)\end{equation}
    
    \item[(b)] 
    \begin{equation}\label{eq:kerbound} \sup_{t\in[0,1],\lVert z\rVert \leq L} \lVert \hat{p}_t(z)-p_t(z)\rVert =O_p\left(\sqrt{\frac{\log n}{nh_n^d}}+h_n^\beta \right)
    \end{equation}
    and 
    \begin{equation}\label{eq:difbound} \sup_{t\in[0,1],\lVert z\rVert \leq L} \lVert \hat{f}_t(z)-v_t(z)\hat{p}_t(z)\rVert =O_p\left(\sqrt{\frac{\log n}{nh_n^d}}+h_n^\beta \right)\end{equation}

\item[(c)] for $l=0,1,2$, 
\begin{equation}
\label{eq:dervbound} \sup_{t\in[0,1],\lVert z\rVert \leq L}\Big \lVert \frac{\partial^l}{\partial z^l}\hat{v}(t,z)-\frac{\partial^l}{\partial z^l}v(t,z)\Big \rVert =O_p\left(h_n^{\beta-l}+\sqrt{\frac{\log n}{nh_n^{d+2l}}}\right)\end{equation}
\end{itemize} 
\end{lemma}
\begin{proof}

For the bias rates  (a), let's focus on \eqref{eq:biasdif}; the analysis of \eqref{eq:biasdif} is similar. Note first that the population quantity equals zero since $f_t(z)=v_t(z)p_t(z)$ for all $t,z$, by definition. We have
\begin{eqnarray*}\mathbb{E}\left(\hat{f}_t(z)-v_t(z)\hat{p}_t(z)\right)&=& \mathbb{E}\left(\frac{1}{n}\sum_{i=1}^n \left(\Delta_i -v(s,z)\right) K_h\left(X_i(s)-z\right)\right)\\
&=& \int \int  \left(\delta -v(s,z)\right) K_h\left(u-z\right)p_0(u-s\delta)p_1(u+(1-s)\delta)d\delta du\\
&=&  \int \int  \left(\delta -v(s,z)\right) K\left(u\right)p_0(z+hu-s\delta)p_1(z+hu+(1-s)\delta) d\delta du\\
&=&\int K\left(u\right)\left(v(s,z+hu)-v(s,z)\right)p_s(z+hu) du.\end{eqnarray*}
%\cite[Theorem 2]{einmahl2005uniform}

We now follow the classical analysis of the bias of kernel estimators \cite[Chapter 1]{Tsybakov} of prescribed order. We perform a Taylor expansion of order $l=\lfloor \beta \rfloor$ for the function \begin{equation}\label{eq:gholder} g_{s,z}(w):=\left(v(s,z+w)-v(s,z)\right)p_s(z+w)= f(s,z+w)-v(s,z)p_s(z+w),\end{equation}
around $w=0$, to control the difference $g_{s,z}(hu)-g_{s,z}(0)=g_{s,z}(hu)$. To avoid unnecessary notation, and since the argument extends easily, we will proceed as if the derivatives were unidimensional. We have that
$$g_{s,z}(hu)=\sum_{j=1}^{l-1} \frac{(uh)^j}{j!}g_{s,z}^{(j)}(0) +(uh)^{l}g_{s,z}^{(l)}(\tau uh),$$
where $\tau\in[0,1]$ depends on $z,s,u,h$. Since the first $l-1$ terms in the sum depend on $u$ only through the term $(uh)^j$, and since the kernel is of order $l$, these terms vanish after integration with respect to $u$. While the last term does not vanish, we can still bound $g_{s,z}^{(l)}(\tau uh)$ uniformly over $\lVert z\rVert \leq L$ and $s\in[0,1]$: indeed, by Lemma \ref{lemma:holdersmooth} and \eqref{eq:gholder}, $g_{s,z}(\cdot)$ is the difference of two $\beta$-H{\"o}lder functions, $f(s,z+\cdot)$ and $v(s,z)p_s(z+\cdot)$. Although the factor $v(s,z)$ may be unbounded over $z\in\mathbb{R}^d$, it remains bounded under the constraint $\lVert z\rVert \leq L$. Therefore
$$M:=\sup_{\lVert z\rVert \leq L, s\in[0,1]} ||g_{s,z}||_\beta <\infty.$$
Then, since $\int K(u)u^ldu =0$ we have, for $\lVert z\rVert \leq L, t\in[0,1]$
\begin{eqnarray*}\|\mathbb{E}\left(\hat{f}_t(z)-v_t(z)\hat{p}_t(z)\right)\|
&\leq & \sup_{\lVert z\rVert \leq L, s\in[0,1]}    \Big \lVert \int K\left(u\right) g_{s,z}(hu)du \Big \rVert \\
&= &    \Big \lVert  \int K\left(u\right) \frac{(uh)^l}{l!} \left(g^{(l)}_{s,z}(\tau hu)-g^{(l)}_{s,z}(0)\right)du  \Big \rVert\\
&\leq &   \int \lvert  K\left(u\right)\rvert  \lvert  \frac{(uh)^l}{l!} M (\tau uh)^{\beta-l}\rvert du \\
&\lesssim & h^{\beta}
,\end{eqnarray*}
so we have established \eqref{eq:biasbound}. In the last line, we used that $p_s(z_s)$ is bounded below (Lemma \ref{lemma:infprob}) and that $\tilde{\Phi}(s)$ is bounded over $s$ (Proposition \ref{prop:Volterra}).

Statement (b) is a uniform-in-time version of the usual uniform consistency rates for the kernel density estimator and the linearization term appearing in kernel regression. For \eqref{eq:difbound}, we write
$$\hat{f}_t(z)-v_t(z)\hat{p}_t(z)=\left[\hat{f}_t(z)-v_t(z)\hat{p}_t(z)-\mathbb{E}\left(\hat{f}_t(z)-v_t(z)\hat{p}_t(z)\right)\right]+\mathbb{E}\left(\hat{f}_t(z)-v_t(z)\hat{p}_t(z)\right)$$
To control the first term (deviations), we use Lemma \ref{lemma:einimproved}:
take $Y=\Delta$, $X=X_t$, $c(z)=1$ and $d(z)=v_t(z)$. 
 We obtain the term $O_p(\sqrt{\log n/(nh_n^d)})$ term by noting that
 
$$\sup_{t\in [0,1],\lVert z\rVert L}  \mathbb{E}\left(\Big \lVert\Delta-v_t(z)\Big \rVert^2 \Big \lvert X_t=z\right)p_t(z)<\infty,$$

a claim that follows since $\inf_{s\in[0,1],\lVert z\rVert \leq L} p_s(z)>0$ (Lemma \ref{lemma:infprob}), and Lemma \ref{lemma:momentbound}. The $h^\beta$ term comes from the bias analysis (a).  The analysis of \eqref{eq:kerbound} is analogous and also follows from Lemma \ref{lemma:momentbound}. The analysis of \eqref{eq:kerbound} is even simpler.

Regarding (b), in the case $l=0$, the results follow from a linearization of the kernel regression estimator:
%\begin{equation}\label{eq:ratio}
 %\frac{\hat{f}}{\hat{p}}-\frac{f}{p}=\frac{1}{p}\left(\hat{f}-\frac{f}{p}\hat{p}\right)\left(1-\frac{\hat{p}-p}{\hat{p}}\right).
 %\end{equation}
 %Define the quantities  \begin{equation}\label{eq:lnz}\hat{L}(s,z):=\Phi\left(1,s,z\right) \tilde{\hat{L}}(s,z),\quad \tilde{\hat{L}}(s,z):=\frac{\hat{f}(s,z)-v(s,z)\hat{p}(s,z)}{p_s(z)},\quad \hat{r}(s,z):=\hat{p}(s,z)-p(s,z).\end{equation}
 
%\begin{eqnarray*} \hat{v}_s(z_s)-v_s(z_s)&=&\frac{\hat{f}_s(z_s)-v_s(z_s)\hat{p}_(z_s)}{p_s(z_s)} \\&&+O_p\left(\lVert \hat{f}_s(z_s)-v_s(z_s)\hat{p}_s(z_s)\rVert\lVert \hat{p}_s(z_s)-p_s(z_s)\rVert \right)\\&=&
%\frac{\hat{f}_s(z_s)-v_s(z_s)\hat{p}_s(z_s)}{p_s(z_s)} +O_p\left(\left[\sqrt{\frac{\log n}{nh_n^d}}+h_n^\beta\right]^2\right)\\
%&=&O_p\left(\sqrt{\frac{\log n}{nh_n^d}}+h_n^\beta \right)
%.\end{eqnarray*}

 from \eqref{eq:ratio} it follows that
 \begin{eqnarray*}
 \hat{v}_t(z)-v_t(z)&=&\frac{\hat{f}_t(z)-v_t(z)\hat{p}_t(z)}{p_t(z)} \\&&-\left(\frac{\hat{f}_t(z)-v_s(z)\hat{p}_t(z)}{p_t(z)}\right)\frac{\hat{p}_t(z)-p_t(z)}{\hat{p}_t(z)}.
\\
&=&
\frac{\hat{f}_t(z)-v_t(z)\hat{p}_t(z)}{p_t(z)} + O_p\left(\sup_{t\in[0,1],\lVert z\rVert \leq L}\lVert \hat{f}_t(z)-v_t(z)\hat{p}_t(z)\rVert\lVert \hat{p}_t(z)-p_t(z)\rVert \right)\\
&=&O_p\left(\sup_{t\in[0,1],\lVert z\rVert \leq L}\lVert \hat{f}_t(z)-v_t(z)\hat{p}_t(z)\rVert\right)\\
&=& O_p\left(\sqrt{\frac{\log n}{nh_n^d}}+h_n^\beta \right)
\end{eqnarray*}
In the above bounds, we used several facts. First, that since $\inf_{t\in [0,1],\lVert z\rVert\leq L} p_s(z)>0$ (Lemma \ref{lemma:infprob}), and that by  \eqref{eq:kerbound} 
\begin{equation}\label{eq:kerbound2}\sup_{t\in[0,1],\lVert z\rVert \leq L} \lVert\hat{p}_t(z)-p_t(z)\rVert =o_p(1),\end{equation}
we have that
$$\sup_{t\in[0,1],\lVert z\rVert\leq L}\frac{1}{\hat{p}_t(z)}=O_p(1).$$
Also, we used that \eqref{eq:kerbound2} implies that we can ignore the product term, and that the final rate follows from \eqref{eq:kerbound}.

%Theorem \ref{teo:CLT} to reduce the analysis of the ratio to a linear quantity as the one appearing in Lemma \ref{lemma:unifcons}. The deviations of this linearized term can be controlled using \eqref{eq:limempirical}, giving the logarithmic term. The term $h^\beta$ appears from the usual bias analysis.
The analysis of derivatives is standard; each kernel differentiation contributes a $h^{-1}$ term to the variance and bias analysis. The key observation is that (for $l=1$)
$$\frac{\partial}{\partial z}\hat{v}(t,z)-\frac{\partial}{\partial z}v(t,z) = \frac{1}{\hat{p}(t,z)}\frac{\partial}{\partial z} \hat{f}(t,z)-\frac{1}{p(t,z)}\frac{\partial}{\partial z} f(t,z)-\left(\frac{\hat{v}(t,z) }{\hat{p}(t,z)}\frac{\partial}{\partial z}\hat{p}(t,z)-\frac{v(t,z) }{p(t,z)}\frac{\partial}{\partial z}p(t,z)\right).$$
The analysis of the differences above is similar to the one for the original kernel regression estimator, but this one contains derivatives of the kernel. See e.g. the proof of Theorem 6 in \cite{hansen2008uniform} for details in the heuristics. 
\end{proof}

\begin{lemma}\label{lemma:intdm}
Suppose that $p_0,p_1$ are differentiable. Under the assumptions of Theorem \ref{teo:regularity}, in the one-dimensional case, 
\begin{equation*}
u\frac{d}{ds}\left(\tilde{\Phi}(s)\frac{p_s(z_s+hu)}{p_s(z_s)}\right)=-\frac{\tilde{\Phi}(s)}{p_s(z_s)}\frac{d}{d h} m(u,h,s),
\end{equation*}
where $m(u,h,s)=\left(v(s,z_s+hu)-v(s,z_s)\right)p_s(z_s+hu)$.
\end{lemma}

\begin{proof}[Proof of Lemma \ref{lemma:intdm}]

By Proposition \ref{prop:Volterra}

\begin{eqnarray*}
\frac{d}{ds}\tilde{\Phi}(s) &=& -\tilde{\Phi}(s)\frac{\partial }{\partial z}v(s,z_s).
\end{eqnarray*}

Also, by Lemma \ref{lemma:ders}
$$\frac{d}{ds} p_s(z_s) = - \mathrm{Tr}\left(\frac{\partial }{\partial z}v(s,z_s)\right)p_s(z_s).$$
Therefore,
$$\frac{d}{ds}\left(\frac{p_s(z_s+hu)}{p_s(z_s)}\right)=\frac{\frac{d}{ds}p_s(z_s+hu)}{p_s(z_s)}+\frac{\partial}{\partial z}v(s,z_s)\frac{p_s(z_s+hu)}{p_s(z_s)},$$
and consequently
\begin{eqnarray*}
u\frac{d}{ds}\left(\tilde{\Phi}(s)\frac{p_s(z_s+hu)}{p_s(z_s)}\right)&=&-u\frac{\partial }{\partial z}v(s,z_s)\tilde{\Phi}(s)\frac{p_s(z_s+hu)}{p_s(z_s)}\\
&+&u\tilde{\Phi}(s)\left(\frac{\frac{d}{ds}p_s(z_s+hu)}{p_s(z_s)}+\frac{\partial}{\partial z}v(s,z_s)\frac{p_s(z_s+hu)}{p_s(z_s)}\right)\\
&=& \frac{\tilde{\Phi}(s)}{p_s(z_s)}u \frac{d}{ds}p_s(z_s+hu).
\end{eqnarray*}
It only remains to identify the above derivative as a derivative of $m$ with respect to $h$. This is true since
\begin{eqnarray*}
\frac{d}{ds} p_s(z_s+hu)&=&\int_{\mathbb{R}}  \frac{d}{ds}\left( p_0(z_s+hu-s\delta)p_1(z_s+hu+(1-s)\delta)d\delta\right)\\
&=& 
\int_{\mathbb{R}} \left(\frac{d}{ds}z_s-\delta\right)p'_0(z_s+hu-s\delta)p_1(z_s+hu+(1-s)\delta)d\delta\\&& +\int_{\mathbb{R}} \left(v(s,z_s)-\delta\right)p_0(z_s+hu-s\delta)p'_1(z_s+hu+(1-s)\delta)d\delta
\\
&=& \int_{\mathbb{R}} \left(\frac{d}{ds}z_s-\delta\right)p'_0(z_s+hu-s\delta)p_1(z_s+hu+(1-s)\delta)d\delta\\&& +\int_{\mathbb{R}} \left(v(s,z_s)-\delta\right)p_0(z_s+hu-s\delta)p'_1(z_s+hu+(1-s)\delta)d\delta.
\end{eqnarray*}

and, on the other hand, 
\begin{eqnarray*}
\frac{d}{dh} m(u,h,s)&=&\int_{\mathbb{R}} \left(\delta- v(s,z_s)\right)\frac{d}{dh}  p_0(z_s+hu-s\delta)p_1(z_s+hu+(1-s)\delta)d\delta\\
&=& u\int_{\mathbb{R}} \left(\delta- v(s,z_s)\right) p'_0(z_s+hu-s\delta)p_1(z_s+hu+(1-s)\delta)d\delta\\
&&+u\int_{\mathbb{R}} \left(\delta- v(s,z_s)\right)p_0(z_s+hu-s\delta)p'_1(z_s+hu+(1-s)\delta)d\delta\\
&=&-u\frac{d}{ds}p_s(z_s+hu).
\end{eqnarray*}
\end{proof}
\begin{lemma}\label{lemma:ders}
Under the assumptions of Theorem \ref{teo:regularity}
\begin{equation} \label{eq:derv} \frac{d}{ds} p_s(z_s) = - \mathrm{Tr}\left(\frac{\partial }{\partial z}v(s,z_s)\right)p_s(z_s)\end{equation}
\end{lemma}

\begin{proof}
\begin{eqnarray*}\frac{d}{ds} p_s(z_s) &=& \int \left(\frac{d}{ds}z_s-\delta\right)\cdot \left(\nabla p_0(z_s-s\delta)p_1(z_s+(1-s)\delta)+p_0(z_s-s\delta)\nabla p_1(z_s+(1-s)\delta)\right)d\delta\\
&=& \int \left(v_s-\delta\right)\cdot \left(\nabla p_0(z_s-s\delta)p_1(z_s+(1-s)\delta)+p_0(z_s-s\delta)\nabla p_1(z_s+(1-s)\delta)\right)d\delta \\ &=& 
\int \left(v_s-\delta\right)\cdot f(\delta,s,z_s) d\delta
\end{eqnarray*}
where we have defined the vector-valued function $$f(\delta,s,z_s)=\nabla p_0(z_s-s\delta)p_1(z_s+(1-s)\delta)+p_0(z_s-s\delta)\nabla p_1(z_s+(1-s)\delta)$$

Note that the derivative of the velocity with respect to the $z$ coordinate is the following matrix

\begin{eqnarray}\nonumber \frac{\partial }{\partial z}v(s,z_s) &=& \frac{\partial }{\partial z} \frac{ \int \delta p_0(z_s-s\delta)p_1(z_s+(1-s)\delta)d\delta}{p_s(z_s)}\\
\nonumber &=& \frac{\int f(\delta,s,z_s) \delta^Td\delta }{p_s(z_s)}  -\frac{\int f(\delta,s,z_s) d\delta}{p_s(z_s)} \frac{\int \delta^T p_0(z_s-s\delta)p_1(z_s+(1-s)\delta)d\delta}{p_s(z_s)} \\
&=& \frac{\int f(\delta,s,z_s) \delta^Td\delta }{p_s(z_s)}  -\frac{\int f(\delta,s,z_s) d\delta}{p_s(z_s)} v_s(z_s)^T 
\end{eqnarray}
Then,

\begin{eqnarray*} p_s(z_s)\mathrm{Tr}\left(\frac{\partial }{\partial z}v(s,z_s)\right)&=&\mathrm{Tr}\left(\int f(\delta,s,z_s)\left(\delta- v_s(z_s)\right)^T d\delta\right)\\
&=& \mathrm{Tr}\left(\int f(\delta,s,z_s)\cdot \left(\delta- v_s(z_s)\right) d\delta\right)\\ &=& -
\frac{d}{ds} p_s(z_s)
\end{eqnarray*}

\end{proof}

\section{Proofs of Results in Section~\ref{sec:bounded}}\label{appsec:proofs-exist-uni-rates-bounded}
For any set-valued mapping $t\mapsto A(t)\subseteq\mathbb{R}^d$, define limsup and liminf in the Painlev{\'e}-Kuratowski sense as
\begin{align*}
    \liminf_{t\to t_0} A(t) &:= \left\{v:\, \limsup_{t\to t_0}\,\mbox{dist}(v, A(t)) = 0\right\},\\
    \limsup_{t\to t_0} A(t) &:= \left\{v:\, \liminf_{t\to t_0}\,\mbox{dist}(v,\, A(t)) = 0\right\}.
\end{align*}
(See Definition 1.1.1 of~\cite{aubin2009set}.) The (contingent) tangent cone of a closed set $\Omega$ at a point $x\in\Omega$ is defined as
\[
T_{\Omega}(x) := \limsup_{h\to0+}\, \frac{\Omega - x}{h}.
\]
(See Definition 4.1.1 of~\cite{aubin2009set}.)
\begin{proposition}[Properties of $S_t(z)$]\label{prop:properties-of-S_t}
Suppose assumption~\ref{eq:compact-support} holds. Fix $z\in\Omega$. Define
\begin{align*}
    L_t(z) &:= \{\delta\in\mathbb{R}^d:\, z - t\delta\in\Omega\},\\
    U_t(z) &:= \{\delta\in\mathbb{R}^d:\, z + (1 - t)\delta \in \Omega\},
\end{align*}
and
\[
S_t(z) := L_t(z)\cap U_t(z).
\]
Then the following statements hold:
\begin{enumerate}
    \item for all $t\in(0, 1)$,
\[
L_t(z) ~\subseteq~ -T_{\Omega}(z),\quad\mbox{and}\quad U_t(z)~\subseteq~ T_{\Omega}(z).
\]
Moreover, for $t \le \mathrm{dist}(z, \partial\Omega)/\mathrm{diam}(\Omega)$, then
\[
L_t(z) \subseteq U_t(z) \quad\Rightarrow\quad S_t(z) = U_t(z).
\]
In general,
\[
S_t(z) \subseteq \mathcal{B}\left(0,\,\frac{\mathrm{diam}(\Omega) - \mathrm{dist}(z,\,\partial\Omega)}{\max\{t, 1 - t\}}\right).
\]
\item The set-valued map $t\mapsto L_t(z) := \{\delta\in\mathbb{R}^d:\, z - t\delta\in\Omega\}$ is non-decreasing in $t\in[0, 1]$ (i.e., $L_t(z) \subseteq L_{t'}(z)$ if $t \le t'$) but not continuous at $t = 0$, unless $z\in\Omega^\circ$. (Here continuity is in the Painlev{\'e}-Kuratowski sense.) Similarly, the map $t\mapsto U_t(z)$ is non-increasing in $t\in[0, 1]$, but not continuous at $t = 1$, unless $z\in\Omega^\circ$. 
\item For any non-negative function $(t, \delta)\to f(t, \delta)$ continuous in $t\in[0, 1]$ satisfying
\[
\lim_{t\to0} \int_{-T_{\Omega}(z)} f(t, \delta)d\delta = \int_{-T_{\Omega}(z)} f(0, \delta)d\delta,
\]
we have
\[
\lim_{t\to0}\int_{S_t(z)} f(t, \delta)d\delta = \int_{(-T_{\Omega}(z))\cap(\Omega - z)} f(0, \delta)d\delta.
\]
Similarly, if
\[
\lim_{t\to 1}\int_{T_{\Omega}(z)} f(t, \delta)d\delta = \int_{T_{\Omega}(z)} f(1, \delta)d\delta,
\]
then
\[
\lim_{t\to 1}\int_{S_t(z)} f(t, \delta)d\delta = \int_{T_{\Omega}(z)\cap(z-\Omega)} f(1, \delta)d\delta. 
\]
\item For any $t\in(0, 1)$ and $z, z'\in\Omega$ with 
\begin{equation}\label{eq:dist-interior-cond}
\|z - z'\| ~\le~ \gamma ~\le~ \bar{t}\min\{\mathrm{dist}(z, \partial\Omega),\,\mathrm{dist}(z', \partial\Omega)\},
\end{equation}
we have
\[
(S_t(z))^{-2\gamma/(t(1-t))} ~\subseteq~ S_t(z') ~\subseteq~ (S_t(z))^{2\gamma/(t(1-t))}.
\]
Additionally, for any $t\in(0, 1), h > 0$ such that $t + h\in(0, 1)$ and $z\in\Omega$,
\begin{equation}\label{eq:inclusions-in-time}
      (S_t(z))^{-2\mathrm{diam}^2(\Omega)h/(\varepsilon\bar{t})} \subseteq S_{t+h}(z)\subseteq (S_t(z))^{2\mathrm{diam}^2(\Omega)h/(\varepsilon\bar{t})},\quad\mbox{for}\quad \varepsilon = \mathrm{dist}(z,\,\partial\Omega).
\end{equation}
Here $\bar{t} = \min\{t, 1- t\}$.
\item If $z\in\Omega^\circ$ and $\tilde{z} \in \mathrm{Proj}_{\partial\Omega}(z)$, then
\[
-\frac{\mathrm{dist}^2(z,\,\partial\Omega)}{1-t} ~\le~ \inf_{\delta\in S_t(z)}\, \delta^{\top}(z - \tilde{z}) ~<~ \sup_{\delta\in S_t(z)}\, \delta^{\top}(z - \tilde{z}) ~\le~ \frac{\mathrm{dist}^2(z,\, \partial\Omega)}{t}.
\]
\item If $z\in\mathrm{SC}(\Omega)$, then $S_t(z) = \{0\}$ for all $t\in(0, 1)$. 
% Moreover, if for a $z\in\partial\Omega$, there exists a sequence $\{z_k\}_{k\ge1}\subseteq\mathrm{SC}(\Omega)$ such that $z_k\to z$ as $k\to\infty$, then $S_t(z) = \{0\}$. 
In particular, if $\Omega$ is a strictly convex set, then $S_t(z) = \{0\}$ whenever $t\in(0, 1)$ and $z\in\partial\Omega$.
\end{enumerate}
\end{proposition}
\begin{proof}
By Proposition 4.2.1 of~\cite{aubin2009set}, we have 
\[
\mbox{closure}(\mbox{cone}(\Omega - z)) = T_{\Omega}(z).
\]
In words, this says that the closure of the cone spanned by $\Omega - z$ is the tangent cone.
This implies that
\begin{equation}\label{eq:inclusion-into-tangent-cones}
\frac{\Omega - z}{t} \subseteq T_{\Omega}(z)\quad\mbox{and}\quad \frac{\Omega - z}{1 - t}\subseteq T_{\Omega}(z),\quad\mbox{for any}\quad t\in(0, 1).
\end{equation}

Moreover, 
\begin{equation}\label{eq:definition-of-tangent-cones}
\lim_{t\downarrow 0}\, \frac{\Omega - z}{t} ~=~ T_{\Omega}(z),\quad\mbox{and}\quad \lim_{t\uparrow1}\,\frac{\Omega - z}{1-t} ~=~ T_{\Omega}(z).
\end{equation}
    \begin{enumerate}
        \item From~\eqref{eq:inclusion-into-tangent-cones}, we get that $\delta\in L_t(z)$ implies $\delta\in -T_{\Omega}(z)$ and $\delta\in U_t(z)$ implies $\delta\in T_{\Omega}(z)$, which proves the first statement. We now prove that for any $z\in\Omega$,
        \begin{equation}\label{eq:max-distance-to-boundary}
            \sup_{y\in\Omega}\|z - y\| \le \mbox{diam}(\Omega) - \mbox{dist}(z, \partial\Omega).
        \end{equation}

  Assuming~\eqref{eq:max-distance-to-boundary}, we get that
        \begin{align*}
        U_t(z) &= \frac{\Omega - z}{1 - t}\\ 
        &\overset{\eqref{eq:max-distance-to-boundary}}{\subseteq} \mathcal{B}\left(0,\,\frac{\mbox{diam}(\Omega) - \mbox{dist}(z, \partial\Omega)}{1 - t}\right)\\ 
        &{\subseteq} \mathcal{B}\left(0,\,\frac{\mbox{dist}(z, \partial\Omega)}{t}\right),\quad\mbox{for}\quad t\le \frac{\mbox{dist}(z,\partial\Omega)}{\mbox{diam}(\Omega)},\\ 
        &\subseteq \frac{z - \Omega}{t} = L_t(z).
        \end{align*}
         To prove~\eqref{eq:max-distance-to-boundary}, take any two points $z_1, z_2\in\partial\Omega$ such that $z = \lambda z_1 + (1 - \lambda)z_2$ for some $\lambda\in(0, 1)$. Then
        \[
        \|z_1 - z_2\| = \|z - z_1\| + \|z - z_2\|.
        \]
        Take $z_2$ so that $\|z - z_2\| = \sup_{y\in\Omega}\|z - y\|$; if there is no such point, consider a limiting sequence on the boundary. The corresponding $z_1$ is a point on $\partial\Omega$ and the definition of $\mbox{dist}(z, \partial\Omega)$ implies $\|z - z_1\| \ge \mbox{dist}(z,\,\partial\Omega)$. Therefore, we get
        \[
        \|z_1 - z_2\| \ge \mbox{dist}(z,\,\partial\Omega) + \sup_{y\in\Omega}\|z - y\|.
        \]

        The left hand side is bounded above by the diameter of $\Omega$ by definition. Hence, we obtain~\eqref{eq:max-distance-to-boundary}. The second part of the result follows from the observation that
        \[
        L_t(z) \subseteq \mathcal{B}\left(0,\,\frac{\mbox{diam}(\Omega) - \mbox{dist}(z,\partial\Omega)}{t}\right),\quad\mbox{and}\quad U_t(z) \subseteq\mathcal{B}\left(0,\,\frac{\mbox{diam}(\Omega) - \mbox{dist}(z, \partial\Omega)}{1 - t}\right).
        \]
        \item It is easy to see that 
        \[
        L_0(z) = \mathbb{R}^d\quad\mbox{and}\quad U_1(z) = \mathbb{R}^d.
        \]
        As $t\downarrow 0$, note that $L_t(z)$ increases to $-T_{\Omega}(z)$ and $U_t(z)$ decreases to $\Omega - z$. This is because of the convexity of $\Omega$: $z - t\delta\in \Omega$ (along with $z\in\Omega$) implies that $z - s\delta\in\Omega$ for all $s\in[0, t]$ (or equivalently, $\delta\in (z - \Omega)/s$. (Similar reasoning also holds for $(\Omega - z)/(1-t)$). If $z\in\Omega^\circ$, then $-T_{\Omega}(z) = T_{\Omega}(z) = L_0(z) = U_1(z) = \mathbb{R}^d$.
        \item Because $S_t(z) = L_t(z)\cap U_t(z)$, by non-negativity of $f$, we get
        \begin{align*}
        0 \le f(t, \delta)\mathbf{1}\{\delta\in S_t(z)\} &= f(t, \delta)\mathbf{1}\{\delta\in L_t(z)\}\mathbf{1}\{\delta\in U_t(z)\}\\ 
        &\le f(t, \delta)\mathbf{1}\{\delta\in L_t(z)\}\\ 
        &\le f(t, \delta)\mathbf{1}\{\delta\in -T_{\Omega}(z)\}.
        \end{align*}
        Hence,
        \[
        0 \le f(t, \delta)\mathbf{1}\{\delta\in S_t(z)\} \le f(t, \delta)\mathbf{1}\{\delta\in -T_{\Omega}(z)\},
        \]
        and each of these functions converges to $0$, $f(0, \delta)\mathbf{1}\{\delta\in(-T_{\Omega}(z))\cap(\Omega - z)\}$, and $f(0, \delta)\mathbf{1}\{\delta\in -T_{\Omega}(z)\}$, as $t\to0$. Therefore, by our assumption and Pratt's \citep[Theorem 5.5]{gut2006probability} lemma, we get the result.

        The proof for convergence as $t\to1$ is identical.
        \item We prove the result for $t \le 1/2$. Clearly, 
        \begin{align*}
        S_t(z') &= (S_t(z)\cap S_t(z'))\cup(S_t(z')\setminus S_t(z))\\ 
        &\subseteq S_t(z)\cup(S_t(z')\setminus S_t(z)).
        \end{align*}
        Hence, it suffices to prove that 
        \begin{equation}\label{eq:set-subset-inflation-implosion}
            S_t(z')\setminus S_t(z) ~\subseteq~ (S_t(z))^{2\gamma/(t(1-t))}.
        \end{equation} 
        Consider a vector $\delta'\in S_t(z')\setminus S_t(z)$. To prove~\eqref{eq:set-subset-inflation-implosion}, it suffices to produce $\delta \in S_t(z)$ such that $\|\delta' - \delta\| \le 2\gamma/(t(1-t))$. Because $\delta'\in S_t(z')\setminus S_t(z)$, we must have either
        \begin{equation}
            \begin{split}
                &z' - t\delta'\in\Omega,\; z' + (1-t)\delta'\in\Omega,\; z - t\delta' \notin \Omega,\\
                &z' - t\delta'\in\Omega,\; z' + (1-t)\delta'\in\Omega,\; z + (1-t)\delta' \notin \Omega.
            \end{split}
        \end{equation}
        Let us consider the first case. Because $S_t(z) \neq \emptyset$ (note that $0\in S_t(z)$ for all $z\in\Omega$), the set
        \[
        \Upsilon := \Omega\cap\left(\frac{z - t\Omega}{1-t}\right),
        \]
        is non-empty. In fact, for any $v\in S_t(z)$, we have $z - tv = \xi\in\Omega$ and 
        \[
        z + (1-t)v = z + (1-t)(z - \xi)/t = (z - (1-t)\xi)/t\in\Omega\quad\Rightarrow\quad \xi\in \frac{z - t\Omega}{1-t}.
        \]
        Therefore, $\xi\in\Upsilon$. Define
        \[
        \vartheta = \mbox{Proj}_{\Upsilon}(z - t\delta')\quad\Rightarrow\quad \delta = \frac{z - \vartheta}{t}\in S_t(z).
        \]
        It suffices to prove that $\|\delta - \delta'\| \le 2\gamma/(t(1-t))$. Note that
        \[
        \|\delta - \delta'\| = \left\|\frac{z - t\delta}{t} - \frac{z - t\delta'}{t}\right\| = \frac{1}{t}\left\|\vartheta - (z - t\delta')\right\| = \frac{\mbox{dist}(z - t\delta',\,\Upsilon)}{t}.
        \]
        To control the right hand side, note that the distance to a set is a decreasing function of the set. Hence, if we can show 
        \begin{equation}\label{eq:inclusions-away-from-boundary}
            \Upsilon \supseteq \left(\Omega\cap\frac{z' - t\Omega}{1 - t}\right)^{-\|z - z'\|/(1-t)} \neq \emptyset,
        \end{equation}
        then 
        \begin{align*}
        \mbox{dist}(z - t\delta',\, \Upsilon) &\le \mbox{dist}\left(z - t\delta',\, \left(\Omega\cap\frac{z' - t\Omega}{1 - t}\right)^{-\|z - z'\|/(1-t)}\right)\\ 
        &\le \mbox{dist}\left(z - t\delta',\, \Omega\cap\frac{z' - t\Omega}{1 - t}\right) + \frac{\|z - z'\|}{(1-t)}\\
        &\le \mbox{dist}(z - t\delta', z' - t\delta') + \frac{\|z - z'\|}{1-t} = \frac{2-t}{1-t}\|z - z'\|,
        \end{align*}
        which implies the result. In this calculation, it is necessary to prove that the set in~\eqref{eq:inclusions-away-from-boundary} is non-empty, because the distance of any point to an empty set is infinity.
        To prove~\eqref{eq:inclusions-away-from-boundary}, note that
        \[
        \mathcal{B}\left(z',\, \frac{t\mbox{dist}(z', \partial\Omega)}{1-t}\right) ~\subseteq~ \Omega\cap\left(\frac{z' - t\Omega}{1-t}\right).
        \]
        (By the definition of the distance, we get $\mathcal{B}(z', \mbox{dist}(z', \partial\Omega))\subseteq\Omega$ and so, the first inclusion holds for $t \le 1/2$. For the second one, observe that for any $\mathcal{B}(z', \varepsilon)\subseteq \Omega$, $(z' - t\Omega)/(1-t) \supseteq \mathcal{B}(z', t\varepsilon/(1-t))$. Applying this with $\varepsilon = \mbox{dist}(z', \partial\Omega)$ yields this inclusion result.) This implies that
        \[
        \left(\Omega\cap\frac{z' - t\Omega}{1 - t}\right)^{-\|z - z'\|/(1-t)} \neq \emptyset\quad\mbox{whenever}\quad \|z - z'\| \le t\mbox{dist}(z', \partial\Omega).
        \]
        Using the fact that
        \[
        \frac{z - t\Omega}{1 - t}\supseteq \left(\frac{z' - t\Omega}{1 - t}\right)^{-\|z - z'\|/(1-t)},
        \]
        we conclude
        \[
        \Upsilon \supseteq \Omega\cap\left(\frac{z' - t\Omega}{1 - t}\right)^{-\|z - z'\|/(1-t)} \supseteq \left(\Omega\cap\frac{z' - t\Omega}{1 - t}\right)^{-\|z - z'\|/(1-t)} \neq \emptyset.
        \]
        This proves~\eqref{eq:inclusions-away-from-boundary}. The proof when $z + (1-t)\delta'\notin\Omega$ is similar. To prove the lower inclusion in the result, swap the roles of $z, z'$.

        To prove~\eqref{eq:inclusions-in-time}, consider any $\delta\in S_{t+h}(z)\setminus S_t(z)$. By monotonicity of $t\mapsto L_t(z)$ and $t\mapsto U_t(z)$, we get
    \[
    \delta\in L_{t+h}(z)\cap L_t(z)\cap U_{t+h}(z)\quad\mbox{and}\quad \delta\notin U_t(z).
    \]
    By Lemma 1 of~\cite{hoffmann1992distance}, we get
    \begin{align*}
    \mbox{dist}(\delta,\, L_t(z)\cap U_t(z)) &\le \frac{\mbox{diam}(L_t(z))}{\sup_{y\in S_t(z)}\mbox{dist}(y, \partial L_t(z))}\mbox{dist}(\delta,\, U_t(z))\\ 
    &\le \frac{\mbox{diam}(\Omega)}{t\mbox{dist}(0, \partial L_t(z))}\left\|\delta - \left(\frac{1-t-h}{1-t}\right)\delta\right\|\\
    &\le \frac{2\mbox{diam}^2(\Omega)}{(1-t)\varepsilon}h.
    \end{align*}
    The second inequality follows from the fact that $\delta\in S_{t+h}(z) \subseteq U_{t+h}(z)$ implies that $\delta\in (\Omega - z)/(1-t-h)$ or equivalently, $(1-t-h)\delta/(1-t)\in (\Omega - z)/(1-t) = U_t(z)$.
    This proves the second inclusion of~\eqref{eq:inclusions-in-time}.

    For the first inclusion of~\eqref{eq:inclusions-in-time}, consider any $\delta\in S_t(z)$ such that 
    \[
    \mathcal{B}\left(\delta,\, \frac{2\mbox{diam}^2(\Omega)h}{\varepsilon\bar{t}}\right) \subseteq S_t(z).
    \]
    To prove the first inclusion, it suffices to show that $\delta\in S_{t+h}(z)$. Because $\delta\in S_t(z) = L_t(z)\cap U_t(z)$, we get $\delta\in U_{t+h}(z)$. Consider the vector $(1 + h/t)\delta$. Clearly,
    \[
    \left\|\delta - (1+h/t)\delta\right\| \le \frac{h}{t}\|\delta\| \le \frac{h}{\bar{t}}(2\mbox{diam}(\Omega)) \le \frac{2h\mbox{diam}^2(\Omega)h}{\varepsilon\bar{t}}.
    \]
    This implies $(1+h/t)\in\mathcal{B}(\delta,\,2\mbox{diam}^2(\Omega)h/(\varepsilon\bar{t}))\subseteq S_t(z)$. Therefore, 
    \[
    (1 + h/t)\delta\subseteq L_t(z)\quad\Rightarrow\quad \delta\in L_{t+h}(z). 
    \]
    Hence, $\delta\in S_{t+h}(z)$.
        \item For $z\in\Omega$, $\delta\in S_t(z)$ implies that $z - t\delta, z + (1-t)\delta\in\Omega$. Hence, from Lemma~\ref{lem:projection-to-boundary}, we get
        \begin{align*}
        (z - \tilde{z})^{\top}(z - t\delta - \tilde{z}) \ge 0,\quad &\Rightarrow\quad \|z - \tilde{z}\|^2 \ge t\delta^{\top}(z - \tilde{z}),\\
        (z - \tilde{z})^{\top}(z + (1-t)\delta - \tilde{z}) \ge 0,\quad &\Rightarrow\quad \|z - \tilde{z}\|^2 \ge -(1-t)\delta^{\top}(z - \tilde{z})^2.
        \end{align*}
        This implies the result. The strict inequality between the infimum and the supremum follows from the fact that $\delta\in S_t(z)$ implies $-t\delta/(1-t) \in S_t(z)$ (for $t < 1/2$), and hence, the infimum is negative and the supremum is positive. 
        \item Because $z\in\partial\Omega$ and $\Omega$ is a closed convex set, there exists a supporting hyperplane, i.e., there exists a vector $n\neq 0$ such that $(y - z)^{\top}n \le 0$ for all $y\in\Omega$. Strict convexity of $\Omega$ at $z$ implies that $(y - z)^{\top}n < 0$ for all $y\in\Omega\setminus\{z\}$. (A proof is as follows: suppose there exists $z'\in\Omega, z' \neq z$ and $(z' - z)^{\top}n = 0$. Then by strict convexity at $z$, $(z' + z)/2 \in \Omega^\circ$ and hence, $(z' + z)/2 + \varepsilon u \in \Omega$ for all $u\in S^{d-1}$ and $\varepsilon > 0$ small enough. This implies that 
        \[
        ((z' + z)/2 + \varepsilon u -z)^{\top}n = \varepsilon u^{\top}n \le 0\quad\mbox{for all}\quad u\in S^{d-1},
        \]
        which is a contradiction.) 

        To prove that $S_t(z) = \{0\}$, we note that $S_t(z) \subseteq (-T_{\Omega}(z))\cap U_t(z)$ for all $t\in(0, 1)$. It is clear that $\{0\} \subset (-T_{\Omega}(z))\cap U_t(z)$. Suppose, if possible, there exists $v\neq 0$ such that $-v\in T_{\Omega}(z)$ and $v\in U_t(z)$. Formally, this means that there exists $\tau_k\downarrow0$ and $z_k\in\Omega$ such that $z_k\to z$ and $(z_k - z)/\tau_k \to -v$, and $z + (1 - t)v\in \Omega$. From the supporting hyperplane and that $v \neq 0$, we get 
        \[
        (z_k - z)^{\top}n < 0\quad\mbox{for all}\quad k \ge 1\quad\Rightarrow\quad \left(\frac{z_k - z}{\tau_k}\right)^{\top}n < 0\quad\Rightarrow\quad -v^{\top}n \le 0.
        \]
        On the other hand, we have
        \[
        (z + (1-t)v - z)^{\top}n < 0\quad\Rightarrow\quad v^{\top}n < 0,
        \]
        which is a contradiction. One can also prove this result using the modulus of strict convexity; see, for example, the proof of Proposition~\ref{prop:well-defined-strictly-convex-set}.
    \end{enumerate}
\end{proof}
\begin{proposition}[Examples of $\Omega$ and contingent cones]
\begin{enumerate}

    \item If $\Omega = \{x\in\mathbb{R}^d:\, a_i^{\top}x \le b_i\mbox{ for }1\le i\le m\}$, then setting
\[
I(z) = \{j\in\{1, 2, \ldots, m\}| a_j^{\top}z = b_j\}\quad\mbox{for all}\quad z\in\partial\Omega,
\]
we have
\begin{align*}
T_{\Omega}(z) &= \{w\in\mathbb{R}^d:\, a_j^{\top}w \le 0\quad\mbox{for all}\quad j\in I(z)\},\\
(-T_{\Omega}(z))\cap(\Omega - z) &= \{w\in\mathbb{R}^d:\, a_j^{\top}w = 0\;\mbox{for}\; j\in I(z),\;  a_j^{\top}w \le b_j - a_j^{\top}z\;\mbox{for}\;j\notin I(z)\}.
\end{align*}
\item Suppose $\Omega = \{x\in\mathbb{R}^d:\, c_i(x) = 0\;\mbox{for}\;i\in\mathcal{E}, c_i(x) \le 0\;\mbox{for}\;i\in\mathcal{I}\}$\footnote{For this part, we do not need $c_i$'s to be convex.} for some finite sets $\mathcal{E},\mathcal{I}$. Set
\[
I(z) = \{i\in\mathcal{E}\cup\mathcal{I}:\, c_i(z) = 0\}.
\]
If $\{\nabla c_i(z):\, i\in I(z)\}$ are linearly independent, then
\begin{align*}
    T_{\Omega}(z) &= \{w\in\mathbb{R}^d:\, w^{\top}\nabla c_i(z) = 0\;\mbox{for}\;i\in\mathcal{E}\;\mbox{and}\; w^{\top}\nabla c_i(z) \le 0\;\mbox{for}\;i\in \mathcal{I}\cap I(z)\},\\
    (-T_{\Omega}(z))\cap(\Omega - z) &= \{w\in\mathbb{R}^d:\, w^{\top}\nabla c_i(z) = 0,\,c_i(z + w) = 0,\;\mbox{for}\;i\in\mathcal{E}\}\\
    &\quad\cup\{w\in\mathbb{R}^d:\,c_i(z + w) \le 0\mbox{ for }i\in\mathcal{I},\; w^{\top}\nabla c_i(z) \ge 0\mbox{ for }i\in\mathcal{I}\cap I(z)\}.
\end{align*}
\end{enumerate}
\end{proposition}
\begin{proof}
\begin{enumerate}
        \item The proof is straightforward. Let $\omega\in T_{\Omega}(z)$. Then by definition, there exists $\tau_k\downarrow0, z_k\in\Omega$ such that $z_k\to z$ and $(z_k - z)/\tau_k\to \omega$. Because $z_k\in\Omega$,
        \[
        a_i^{\top}z_k \le b_i\quad\mbox{for all}\quad i\in I(z)\quad\Rightarrow\quad a_i^{\top}(z_k - z)/\tau_k \le 0\quad\mbox{for all}\quad i\in I(z).
        \]
        Taking the limit as $k\to\infty$, we get $T_{\Omega}(z)\in\{w\in\mathbb{R}^d:\, a_j^{\top}w \le 0\;\mbox{for}\; j\in I(z)\}$. To prove the converse, set
        \[
        \varepsilon = \inf_{j\notin I(z)} (b_j - a_j^{\top}z).
        \]
        By definition, $\varepsilon > 0$ (because $m < \infty$). Therefore, $z_k = z + \tau_k w$ for any $w$ satisfying $a_j^{\top}w \le 0$ for $j\in I(z)$ satisfies   
        \[
        a_j^{\top}z_k = a_j^{\top}z + \tau_k a_j^{\top}w\begin{cases} \le b_j + 0, &\mbox{if }j\in I(z),\\
        \le b_j + \varepsilon + \tau_k a_j^{\top}w, &\mbox{otherwise.}\end{cases}
        \]
        Hence, as $\tau_k\to 0$, $\varepsilon + \tau_k a_j^{\top}w > 0$ and hence, $z_k\in\Omega$ for $k$ large enough. Recalling that 
        \begin{align*}
        \Omega - z &= \{w\in\mathbb{R}^d:\, a_j^{\top}w \le b_j - a_j^{\top}z\mbox{ for }1\le j\le m\}\\
        &= \{w\in\mathbb{R}^d:\, a_j^{\top}w \le 0\mbox{ for }j\in I(z),\; a_j^{\top}w \le b_j - a_j^{\top}z\mbox{ for }j\notin I(z)\},
        \end{align*}
        we obtain the result.
        \item Follows from Lemma 12.2 of~\cite{nocedal2006numerical}. See Proposition 4.3.7 of~\cite{aubin2009set} for an alternative set of assumptions.
        \end{enumerate}
\end{proof}
\begin{lemma}\label{lem:uniqueness-final}
    Suppose assumptions~\ref{eq:compact-support},~\ref{eq:bounded-away-densities}, and~\ref{eq:Osgood-condition-densities} hold. Then the integral equation~\eqref{eq:rectified-flow-integral-equation} has a unique solution $z^*:[0, 1]\to\Omega$ with $z^*(t)\in\Omega^\circ$ for all $t\in[0, 1)$, whenever $x\in\Omega^\circ$.
\end{lemma}
\begin{proof}
    By Lemma~\ref{lem:distance-to-boundary-latest}, every solution $y(\cdot)$ belongs to the interior on $[0, 1 - \delta]$, for any $\delta > 0$. Furthermore, Lemma~\ref{lem:boundedness-and-local-Lipschitz} along with assumption~\ref{eq:Osgood-condition-densities} implies that assumption~\ref{eq:Lipschitz-Osgood} holds true. Hence, from Theorem~\ref{thm:uniqueness-viability}, it follows that there exists a unique solution $z^*:[0, 1 - \delta]\to\Omega^\circ$ for any $\delta > 0$. To prove the existence of a unique solution in $[0, 1]$, consider any sequence $\{\delta_k\}_{k\ge1}\subset[0, 1]$ such that $\delta_k\to0$ as $k\to\infty$. By the uniqueness on $[0, 1 - \delta]$ for any $\delta > 0$, we get a unique $z^*(1 - \delta_k)\in \Omega$ such that $z^*(1-\delta_k) = x + \int_0^{1 - \delta_k} v(s, z^*(s))ds$. The boundedness of $v(\cdot, \cdot)$ implies $\{z^*(1 - \delta_k)\}_{k\ge1}$ is a Cauchy sequence and hence, is convergent. Because $\Omega$ is compact (from assumption~\ref{eq:compact-support}), the limit $z^*(1)$ of $z^*(1 - \delta_k)$ belongs to $\Omega$ and must satisfy $z^*(1) = x + \int_0^1 v(s, z^*(s))ds$. 
\end{proof}
\subsection{Proof of Lemma~\ref{lem:boundedness-and-local-Lipschitz}}\label{appsubsec:proof-boundedness-and-local-Lipschitz}    
Recall that $v(t, z) = \mathbb{E}[X_1 - X_0|tX_1 + (1-t)X_0 = z]$. Hence, by Jensen's inequality, $\|v(t, z)\| \le \mathbb{E}[\|X_1 - X_0\||tX_1 + (1-t)X_0 = z]$. Because $X_0, X_1\in\Omega$, we have $\|X_1 - X_0\| \le \mbox{diam}(\Omega)$ almost surely. Therefore, $\|v(t, z)\| \le \mbox{diam}(\Omega)$ for all $z\in\Omega^\circ$. 

    Fix $z, z'\in\Omega^{-\varepsilon}$ with $\|z - z'\| = \eta$. Observe that
\begin{align*}
    v(t, z) - v(t, z') &= \frac{f_t(z)}{p_t(z)} - \frac{f_t(z')}{p_t(z')}\\
    &= \frac{f_t(z) - f_t(z')}{p_t(z)} - v(t, z')\frac{p_t(z) - p_t(z')}{p_t(z)}.
\end{align*}
Note that
\[
p_t(z) \ge \underline{\mathfrak{p}}^2\mbox{Vol}(S_t(z))\quad\mbox{for all}\quad z\in\Omega^{-\varepsilon}.
\]
For any function $h:\mathbb{R}^d\to\mathbb{R}$, consider
\[
T_h(z) = \int_{S_t(z)} h(\delta)p_0(z - t\delta)p_1(z + (1-t)\delta)d\delta.
\]
Then 
\begin{align*}
    |T_h(z) - T_h(z')| &\le \left|\int_{S_t(z)\cap S_t(z')} h(\delta)\left(\frac{p_0(z' - t\delta)p_1(z' + (1-t)\delta)}{p_0(z - t\delta)p_1(z + (1-t)\delta)} - 1\right)p_0(z - t\delta)p_1(z + (1 - t)\delta)d\delta\right|\\
    &\quad+ \left|\int_{S_t(z)\setminus S_t(z')} h(\delta)p_0(z - t\delta)p_1(z + (1-t)\delta)d\delta\right|\\
    &\quad+ \left|\int_{S_t(z')\setminus S_t(z)} h(\delta)p_0(z' - t\delta)p_1(z' + (1-t)\delta)d\delta\right|\\
    &\le p_t(z)\sup_{\delta\in S_t(z)\cap S_t(z')} |h(\delta)|\left|\frac{p_0(z' - t\delta)p_1(z' + (1-t)\delta)}{p_0(z - t\delta)p_1(z + (1-t)\delta)} - 1\right|\\
    &\quad+ p_t(z)\sup_{\delta\in S_t(z)} |h(\delta)|p_0(z - t\delta)p_1(z + (1-t)\delta)\frac{\mbox{Vol}(S_t(z)\setminus S_t(z'))}{\underline{\mathfrak{p}}^2\mbox{Vol}(S_t(z))}\\
    &\quad+ p_t(z)\sup_{\delta\in S_t(z')}|h(\delta)|p_0(z' - t\delta)p_1(z' + (1-t)\delta)\frac{\mbox{Vol}(S_t(z)\setminus S_t(z'))}{\underline{\mathfrak{p}}^2\mbox{Vol}(S_t(z))}.
\end{align*}
From assumption~\ref{eq:bounded-away-densities} (and~\eqref{eq:density-lower-bound-notation},~\eqref{eq:modulus-density}), we get
\[
\sup_{\delta\in S_t(z)\cap S_t(z')}
\left|\frac{p_0(z' - t\delta)p_1(z' + (1-t)\delta)}{p_0(z - t\delta)p_1(z + (1-t)\delta)} - 1\right| \le \left(1 + \omega(\eta)\right)^2 - 1 \le 2\omega(\eta) + \omega^2(\eta),
\]
and
\[
\sup_{\delta\in S_t(z)}
p_0(z - t\delta)p_1(z + (1-t)\delta) \le \overline{\mathfrak{p}}^2.
\]
Therefore, 
\begin{equation}\label{eq:ratio-z-z'-inequality}
\frac{|T_h(z) - T_h(z')|}{p_t(z)} \le \sup_{\delta\in S_t(z)\cup S_t(z')}|h(\delta)|\left[2\omega(\eta) + \omega^2(\eta) + \frac{\overline{\mathfrak{p}}^2}{\underline{\mathfrak{p}}^2}\frac{\mbox{Vol}(S_t(z)\Delta S_t(z'))}{\mbox{Vol}(S_t(z))}\right],
\end{equation}
where $A\Delta B$ represents the symmetric difference of $A$ and $B$.
Taking $h(\delta) = e_j^{\top}\delta$ and $h(\delta)\equiv v(t, z')$, we get
\begin{equation}\label{eq:almost-final-Lipschitz-inequality}
\|v(t, z) - v(t, z')\| \le 3\mbox{diam}(\Omega)\left[2\omega(\eta) + \omega^2(\eta) + \frac{\overline{\mathfrak{p}}^2}{\underline{\mathfrak{p}}^2}\frac{\mbox{Vol}(S_t(z)\Delta S_t(z'))}{\mbox{Vol}(S_t(z))}\right],
\end{equation}
because
\[
\sup_{\delta\in S_t(z)}\|\delta\| \le \frac{\mbox{diam}(\Omega) - \varepsilon}{\max\{t, 1 - t\}} \le 2(\mbox{diam}(\Omega) - \varepsilon),\quad\mbox{and}\quad \sup_{z\in\Omega^\circ}\|v(t, z)\| \le \mbox{diam}(\Omega).
\]
The term $2\omega(\eta) + \omega^2(\eta)$ can be simplified as follows. If $\eta \le \omega^{-1}(1)$, then 
$$2\omega(\eta) + \omega^2(\eta) \le 3\omega(\eta).$$ If $\eta > \omega^{-1}(1)$, then using~\eqref{eq:bounded-velocity} $$\|v(t, z) - v(t, z')\| \le 2\mbox{diam}(\Omega)\eta/\omega^{-1}(1).$$ Therefore, in~\eqref{eq:almost-final-Lipschitz-inequality}, we can replace $2\omega(\eta) + \omega^2(\eta)$ with $3\omega(\eta) + \eta/\omega^{-1}(1)$. This yields
\begin{equation}\label{eq:almost-final-Lipschitz-inequality-2}
\|v(t, z) - v(t, z')\| \le 3\mbox{diam}(\Omega)\left[3\omega(\eta) + \frac{\eta}{\omega^{-1}(1)} + \frac{\overline{\mathfrak{p}}^2}{\underline{\mathfrak{p}}^2}\frac{\mbox{Vol}(S_t(z)\Delta S_t(z'))}{\mbox{Vol}(S_t(z))}\right].
\end{equation}
It suffices to now prove an upper bound on $\mbox{Vol}(S_t(z)\Delta S_t(z'))/\mbox{Vol}(S_t(z))$. 

We shall first prove that bound under the simpler setting of $\min\{t, 1 - t\} \le \varepsilon/\mbox{diam}(\Omega)$.
For $t \le \varepsilon/\mbox{diam}(\Omega)$, Proposition~\ref{prop:properties-of-S_t}(1) implies that
\[
S_t(z) = \frac{\Omega - z}{1 - t}\quad\mbox{and}\quad S_t(z') = \frac{\Omega - z'}{1 - t}.
\]
Therefore, 
\begin{align*}
\frac{\mbox{Vol}(S_t(z)\Delta S_t(z'))}{\mbox{Vol}(S_t(z))} &= \frac{\mbox{Vol}((\Omega - z)\Delta(\Omega - z'))}{\mbox{Vol}(\Omega - z)}\\ 
&= \frac{\mbox{Vol}((\Omega - z + z - z')\Delta(\Omega - z))}{\mbox{Vol}(\Omega - z)}.
\end{align*}
Then Theorem 3 of~\cite{Schymura2014} yields
\[
\frac{\mbox{Vol}(S_t(z)\Delta S_t(z'))}{\mbox{Vol}(S_t(z))} \le \eta\frac{\mathcal{H}^{d-1}(\partial(\Omega - z))}{\mbox{Vol}(\Omega - z)}.
\]
(Here on the right hand side, we could have written $\mathcal{H}^{d-1}(\partial\Omega)$ and $\mbox{Vol}(\Omega)$, but notationally, it would be easier to apply Lemma 2.1 of~\cite{giannopoulos2018inequalities} in the way presented.) Because $\Omega - z \supseteq \mathcal{B}(0, \varepsilon)$, Lemma 2.1 of~\cite{giannopoulos2018inequalities} yields
\begin{equation}\label{eq:low-values-of-t}
\frac{\mbox{Vol}(S_t(z)\Delta S_t(z'))}{\mbox{Vol}(S_t(z))} \le \frac{d\eta}{\varepsilon}.
\end{equation}
Similarly, if $t \ge 1 - \varepsilon/\mbox{diam}(\Omega)$, we get
\[
S_t(z) = \frac{z - \Omega}{t}\quad\mbox{and}\quad S_t(z') = \frac{z' - \Omega}{t}.
\]
Therefore,
\begin{equation}\label{eq:high-value-of-t}
    \frac{\mbox{Vol}(S_t(z)\Delta S_t(z'))}{\mbox{Vol}(S_t(z))} \le \frac{d\eta}{\varepsilon}.
\end{equation}
Hence, for $\min\{t, 1 - t\} \le \varepsilon/\mbox{diam}(\Omega)$, the result is proved because 
\[
\frac{\overline{\mathfrak{p}}^2}{\underline{\mathfrak{p}}^2}\frac{d\eta}{\varepsilon} ~\le~ \frac{\overline{\mathfrak{p}}^2}{\underline{\mathfrak{p}}^2}\frac{d\eta}{\varepsilon}\times \frac{5^{d+1}\mbox{diam}(\Omega)}{\varepsilon}\quad\mbox{for all}\quad \varepsilon \in [0,\mbox{diam}(\Omega)]\quad\mbox{and}\quad d \ge 1.
\]

To prove the result when $\min\{t, 1 - t\} \ge \varepsilon/\mbox{diam}(\Omega)$, we can assume that $\eta \le \varepsilon\min\{t, 1 - t\}/2$.
Otherwise, using the boundedness of $v(t, z)$ (i.e., inequality~\eqref{eq:bounded-velocity}) and that ${\overline{\mathfrak{p}}^2}/{\underline{\mathfrak{p}}^2} \ge 1$, we conclude that
\begin{equation}\label{eq:unrestrictied-Lipschitz}
\begin{split}
\|v(t, z) - v(t, z')\| &\le 2\mbox{diam}(\Omega)\\ 
&\le 2\mbox{diam}(\Omega)\frac{2\eta}{\varepsilon\min\{t, 1 - t\}}\\ 
&= \left(\frac{4\mbox{diam}(\Omega)}{\varepsilon\min\{t, 1 - t\}}\right)\eta\\
&\le \frac{4\mbox{diam}^2(\Omega)}{\varepsilon^2}\eta\\
&\le 3\mbox{diam}^2(\Omega)\frac{\eta}{\varepsilon^2}\frac{\overline{\mathfrak{p}}^2}{\underline{\mathfrak{p}}^2}d5^{d+1},
\end{split}
\end{equation}
which proves the result. Hence, it suffices to prove the validity of the result under
\begin{equation}\label{eq:restricted-Lipschitz}
    \eta \le \varepsilon\min\{t, 1 - t\}/2,\quad\mbox{and}\quad \min\{t, 1 - t\} \ge \varepsilon/\mbox{diam}(\Omega).
\end{equation}
From Proposition~\ref{prop:properties-of-S_t}(4), it follows that, under~\eqref{eq:restricted-Lipschitz},
\[
S_t(z)\cup S_t(z') \subseteq (S_t(z))^{2\eta/(t(1-t))},\quad\mbox{and}\quad S_t(z')\supseteq (S_t(z))^{-2\eta/(t(1-t))}.
\]
These relations imply that
\begin{align*}
    \mbox{Vol}(S_t(z')\setminus S_t(z)) &= \mbox{Vol}(S_t(z')\cup S_t(z)) - \mbox{Vol}(S_t(z))\\
    &\le \mbox{Vol}((S_t(z))^{2\eta/(t(1-t))}) - \mbox{Vol}(S_t(z)),\\
    \mbox{Vol}(S_t(z)\setminus S_t(z')) &= \mbox{Vol}(S_t(z)) - \mbox{Vol}(S_t(z)\cap S_t(z'))\\
    &\le \mbox{Vol}(S_t(z)) - \mbox{Vol}((S_t(z))^{-2\eta/(t(1-t))}). 
\end{align*}
Therefore, 
\[
\mbox{Vol}(S_t(z')\Delta S_t(z)) \le \mbox{Vol}(A^{\gamma}\setminus A)\quad\mbox{where}\quad A = (S_t(z))^{-2\eta/(t(1-t))}\quad\mbox{and}\quad \gamma = \frac{4\eta}{t(1-t)}.
\]
(This follows from Eq. (3.15) of~\cite{schneider2013convex}.)
Because $z\in\Omega^{-\varepsilon}$, we have (under~\eqref{eq:restricted-Lipschitz}) 
\[
\mathcal{B}\left(0,\,\frac{\varepsilon}{\max\{t, 1 - t\}}\right) \subseteq S_t(z)\quad\Rightarrow\quad \mathcal{B}\left(0,\,\frac{\varepsilon}{2\max\{t, 1 - t\}}\right) \subseteq A.
\]
This implies that $A$ satisfies the assumptions of Lemma~\ref{lem:Lipschitz-of-volume} (with $r = \varepsilon/(2\max\{t, 1 - t\})$) and hence,
\[
\frac{\mbox{Vol}(S_t(z')\Delta S_t(z))}{\mbox{Vol}(S_t(z))} \le \frac{\mbox{Vol}(A)}{\mbox{Vol}(S_t(z))}\frac{8d\eta\max\{t, 1 - t\}}{t(1-t)\varepsilon}\left(1 + \frac{8\eta\max\{t, 1 - t\}}{t(1-t)\varepsilon}\right)^{d-1}.
\]
Because $A\subseteq S_t(z)$, we conclude (under~\eqref{eq:restricted-Lipschitz}) that
\begin{equation}\label{eq:correct-restricted-Lipschitz-constant}
\begin{split}
\frac{\mbox{Vol}(S_t(z')\Delta S_t(z))}{\mbox{Vol}(S_t(z))} &\le \eta\times\frac{8d}{\varepsilon\min\{t, 1 - t\}}\left(1 + \frac{4\varepsilon\min\{t, 1 - t\}\max\{t, 1 - t\}}{t(1-t)\varepsilon}\right)^{d-1}\\
&\le \eta\times \frac{d5^{d+1}}{\varepsilon \min\{t, 1 - t\}} \le \eta\times\frac{d5^{d+1}\mbox{diam}(\Omega)}{\varepsilon^2}.
\end{split}
\end{equation}
Substituting in~\eqref{eq:almost-final-Lipschitz-inequality-2}, we obtain
\begin{equation}\label{eq:final-Lipschitz-no-rest-t}
\|v(t, z) - v(t, z')\| \le 3\mbox{diam}(\Omega)\left[{3\omega(\eta)} + \frac{\eta}{\omega^{-1}(1)} + \eta\times\frac{\overline{\mathfrak{p}}^2\mbox{diam}(\Omega)d5^{d+1}}{\underline{\mathfrak{p}}^2\varepsilon^2}\right].
\end{equation}

Combining inequalities~\eqref{eq:almost-final-Lipschitz-inequality},~\eqref{eq:final-Lipschitz-no-rest-t}, and~\eqref{eq:low-values-of-t},~\eqref{eq:high-value-of-t}, we obtain the result.
%%%%%%%%%%%%%%%%%%%%%%%%%%%%%%%%%%%%%%%%%%%%%%
%%%%%%%%%%%%%%%%%%%%%%%%%%%%%%%%%%%%%%%%%%%%%%
\subsection{Proof of Lemma~\ref{lem:distance-to-boundary-latest}}\label{appsubsec:proof-of-dist-to-bdry}
    The existence of a solution follows from Theorem~\ref{thm:Peano-existence}(1) with $\mathcal{S} = \Omega$ and the boundedness of $v(\cdot, \cdot)$ as proved in~\eqref{eq:bounded-velocity}.

    To prove a differential inequality for the distance to the boundary of $\Omega$. Set
    \[
    \mathfrak{D}(t) = \mbox{dist}(y(t), \partial\Omega). 
    \]
    Suppose, if possible, that $T < 1$.
    By definition of $T$, $y(t)\in\Omega^\circ$ for $t\in[0, T)$ and hence, $\mathfrak{D}(t) > 0$ for all $t\in[0, T)$. Moreover, $\mathfrak{D}(T) = 0$ (i.e., $y(T)\in\partial\Omega$), because otherwise an application of Theorem~\ref{thm:Peano-existence}(1) for the ODE starting at time $T$ and the initial value $y(T)$ would have a solution that lies in $\Omega^\circ$ on $[T, T + \eta]$ (for some small $\eta > 0$). This contradicts the definition of $T$. 
    
    We shall now prove that $\mathfrak{D}(\cdot)$ is almost everywhere differentiable on $[0, T]$ and that
    \begin{equation}\label{eq:differential-inequality-distance-to-boundary}
    \mathfrak{D}'(t) \ge -\frac{1}{1 - t}\mathfrak{D}(t)\quad\mbox{almost everywhere}\quad t\in[0, T).
    \end{equation}
    Assuming~\eqref{eq:differential-inequality-distance-to-boundary} holds, we get that 
    \[
    \int_0^T \frac{\mathfrak{D}'(t)}{\mathfrak{D}(t)}dt \ge -\int_0^T \frac{1}{(1-t)}dt\quad\Leftrightarrow\quad \log\left(\frac{\mathfrak{D}(T)}{\mathfrak{D}(0)}\right) \ge \log(1-T)\quad\Leftrightarrow\quad \mathfrak{D}(T) \ge (1 - T)\mathfrak{D}(0). 
    \]
    This implies that~$\mathfrak{D}(T) > 0$ if $T < 1$, contradicting the definition of $T$. Hence, $T = 1$. Inequality~\eqref{eq:distance-from-boundary} follows from this calculation.
    
    To prove~\eqref{eq:differential-inequality-distance-to-boundary}, note that for $t, s\in[0, T]$,
    \begin{align*}
        |\mathfrak{D}(t) - \mathfrak{D}(s)| &\le |\mbox{dist}(y(t), \partial\Omega) - \mbox{dist}(y(s), \partial\Omega)|\\
        &\le \|y(t) - y(s)\|\\
        &\le 2\mbox{diam}(\Omega)|t - s|,
    \end{align*}
    where the last inequality follows from~\eqref{eq:bounded-velocity}. This implies that $t\mapsto \mathfrak{D}(t)$ is absolutely continuous on $[0, T]$, and hence, almost everywhere differentiable. Since $y(\cdot)$ is a solution to the integral equation~\eqref{eq:sub-rectified-flow-equation}, it is also absolutely continuous on $[0, T]$. Hence, there exists a zero measure set $\mathcal{N}\subset[0, 1]$ such that $\mathfrak{D}(\cdot)$ and $y(\cdot)$ are both differentiable for all $s\in[0,1]\setminus\mathcal{N}$.

    We shall prove the inequality~\eqref{eq:differential-inequality-distance-to-boundary} for any $t\in[0, T)\setminus\mathcal{N}$. 
    For any $s\in[0, 1]$, let $\tilde{y}(s)$ be any projection of $y(s)$ on to $\partial\Omega$:
    \[
    \tilde{y}(s) \in \mbox{Proj}_{\partial\Omega}(y(s)).        
    \]
    (Pick any element in the projection. Our bound below does not depend on this choice.)
    By Lipschitz continuity of $\mathfrak{D}(\cdot)$, $\mathfrak{D}(t) > 0$ for $t < T$ implies the existence of $\bar{\eta} > 0$ such that $\mathfrak{D}(t + \eta) > 0$ for all $\eta \in [0, \bar{\eta}]$. 
    Fix any $\eta \in [0, \bar{\eta}]$. 
    It is clear that
    \[
    \mathfrak{D}(t + \eta) = \|y(t + \eta) - \tilde{y}(t + \eta)\|,\quad\mbox{and}\quad \mathfrak{D}(t) \le \|y(t) - \tilde{y}(t + \eta)\|.
    \]
    Therefore,
    \begin{align*}
        \frac{\mathfrak{D}(t + \eta) - \mathfrak{D}(t)}{\mathfrak{D}(t + \eta)} &\ge 1 - \left(\frac{\|y(t) - \tilde{y}(t + \eta)\|^2}{\|y(t + \eta) - \tilde{y}(t + \eta)\|^2}\right)^{1/2}\\
        &= 1 - \left(1 + \frac{\|y(t) - y(t + \eta)\|^2 + 2(y(t) - y(t + \eta))^{\top}(y(t + \eta) - \tilde{y}(t + \eta))}{\mathfrak{D}^2(t + \eta)}\right)^{1/2}. 
    \end{align*}
    Using the inequality $\sqrt{1 + x} \le 1 + x/2$ for all $x \ge -1$, we get
    \begin{equation}\label{eq:intermediate-dist-to-bnd}
        \frac{\mathfrak{D}(t + \eta) - \mathfrak{D}(t)}{\mathfrak{D}(t + \eta)} \ge -\frac{\|y(t) - y(t + \eta)\|^2}{2\mathfrak{D}^2(t + \eta)} - \frac{(y(t) - y(t + \eta))^{\top}(y(t + \eta) - \tilde{y}(t + \eta))}{\mathfrak{D}^2(t + \eta)}. 
    \end{equation}
    Because $y(\cdot)$ solves the integral equation~\eqref{eq:rectified-flow-integral-equation}, inequality~\eqref{eq:bounded-velocity} implies 
    \begin{equation}\label{eq:first-term-dist-to-bnd}
    \|y(t) - y(t + \eta)\|^2 \le 4\mbox{diam}^2(\Omega)\eta^2,
    \end{equation}
    Furthermore, because $t\in[0,1]\setminus\mathcal{N}$, we get 
    \begin{equation}\label{eq:second-term-dist-to-bnd}
    \liminf_{\eta\downarrow0}\,\frac{\|y(t + \eta) - y(t) - \eta v(t, y(t))\|}{\eta} \to 0. 
    \end{equation}
    Hence, inequality~\eqref{eq:intermediate-dist-to-bnd} becomes
    \begin{align*}
    \frac{\mathfrak{D}(t + \eta) - \mathfrak{D}(t)}{\mathfrak{D}(t + \eta)} &\ge -\frac{4\eta^2\mbox{diam}^2(\Omega)}{2\mathfrak{D}^2(t + \eta)} \\
    &\quad + \frac{(y(t + \eta) - y(t) - \eta v(t, y(t)))^{\top}(y(t + \eta) - \tilde{y}(t + \eta))}{\mathfrak{D}^2(t + \eta)}\\
    &\quad + \eta\frac{v(t, y(t))^{\top}(y(t + \eta) - \tilde{y}(t + \eta))}{\mathfrak{D}^2(t + \eta)}.
    \end{align*}
    % The first and second terms converge to zero when scaled by $\eta$, as $\eta\to0$. 
    To control the third term, note that
    \[
    v(t, y(t)) = \int_{S_t(y(t))} \delta q(\delta; t, y(t))dt,\quad\mbox{where}\quad q(\delta; t, z) := \frac{p_0(z - t\delta)p_1(z + (1-t)\delta)}{\int_{S_t(z)} p_0(z - t\delta')p_1(z + (1-t)\delta')d\delta'}.
    \]
    Because $y(t)\in\Omega^\circ,$ $q(\cdot; t, y(t))$ is a valid probability density and, hence,
    \[
    v(t, y(t))^{\top}(y(t + \eta) - \tilde{y}(t + \eta)) ~\ge~ \inf_{\delta\in S_t(y(t))}\, \delta^{\top}(y(t + \eta) - \tilde{y}(t + \eta)).
    \]
    Because $\tilde{y}(t + \eta) \in \mbox{Proj}_{\partial\Omega}(y(t + \eta))$, Lemma~\ref{lem:projection-to-boundary} implies
    \[
    (y(t + \eta) - \tilde{y}(t + \eta))^{\top}(z - \tilde{y}(t + \eta)) \ge 0\quad\mbox{for all}\quad z\in \Omega.
    \]
    Because $\delta\in S_t(y(t))$ implies $y(t) + (1 - t)\delta\in\Omega$, we conclude 
    \[
    (y(t + \eta) - \tilde{y}(t + \eta))^{\top}(y(t) + (1 - t)\delta - \tilde{y}(t + \eta)) \ge 0.
    \]
    Equivalently, for any $\delta\in S_t(y(t))$,
    \begin{align*}
    \delta^{\top}(y(t + \eta) - \tilde{y}(t + \eta)) &\ge -\frac{1}{1 - t}(y(t + \eta) - \tilde{y}(t + \eta))^{\top}(y(t) - \tilde{y}(t + \eta))\\
    &= -\frac{1}{1 - t}(y(t + \eta) - \tilde{y}(t + \eta))^{\top}(y(t+\eta) - \tilde{y}(t + \eta))\\
    &\quad -\frac{1}{1 - t}(y(t + \eta) - \tilde{y}(t + \eta))^{\top}(y(t) - y(t+\eta))\\
    &\ge -\frac{\mathfrak{D}^2(t + \eta)}{1 - t} - \frac{1}{1 - t}\mathfrak{D}(t + \eta)\|y(t) - y(t + \eta)\|\\
    &\ge -\frac{\mathfrak{D}^2(t + \eta)}{1 - t} - \frac{2\eta}{1 - t}\mathfrak{D}(t + \eta)\mbox{diam}(\Omega),
    \end{align*}
    where the last inequality follows from~\eqref{eq:first-term-dist-to-bnd}.
    Combining all the inequalities, we get that for all $\eta\in[0,\bar{\eta}]$,
    \begin{align*}
    \frac{\mathfrak{D}(t + \eta) - \mathfrak{D}(\eta)}{\mathfrak{D}(t + \eta)} &\ge - \frac{4\eta^2\mbox{diam}^2(\Omega)}{\mathfrak{D}^2(t + \eta)}
    % \\
    % &\quad 
    - \frac{\|y(t + \eta) - y(t) - \eta v(t, y(t))\|}{\mathfrak{D}(t + \eta)}\\
    &\quad - \frac{\eta\mathfrak{D}^2(t + \eta)}{(1-t)\mathfrak{D}^2(t + \eta)} - \frac{2\eta^2\mbox{diam}(\Omega)}{(1 - t)\mathfrak{D}(t + \eta)}.
    \end{align*}
    Note that this inequality does not depend on the choice of projection $\tilde{y}(t + \eta)$. Now, dividing both sides by $\eta$ and letting $\eta\downarrow0$ implies (using~\eqref{eq:second-term-dist-to-bnd})
    \[
    \liminf_{\eta\downarrow0}\,\frac{\mathfrak{D}(t + \eta) - \mathfrak{D}(\eta)}{\eta\mathfrak{D}(t + \eta)} \ge -\frac{1}{1 - t}.
    \]
    Furthermore, the continuity of $\mathfrak{D}(\cdot)$ implies $\mathfrak{D}(t + \eta)/\mathfrak{D}(t) \to 1$ as $\eta\to0$ and therefore we get~\eqref{eq:differential-inequality-distance-to-boundary}.
%%%%%%%%%%%%%%%%%%%%%%%%%%%%%%%%%%%%%%%%%%%%%%%%%%%
%%%%%%%%%%%%%%%%%%%%%%%%%%%%%%%%%%%%%%%%%%%%%%%%%%%
\subsection{Proof of Lemma~\ref{lem:distance-to-boundary-at-1}}\label{appsubsec:proof-of-dist-to-bnry-at-1}
    Because $\mu_0$ has a Lebesgue density and $\Omega$ is a convex body implies $\mu_0(\mathcal{E}_1) = 1$. From Theorem~\ref{thm:uniquene-ae-implies-transport-map}, we get $\mathfrak{R}(1, X) \sim \mu_1$ whenever $X\sim \mu_0$ and because $\mu_1$ also has a Lebesgue density, we conclude $\mu_0(\mathcal{E}_2) = 1$. Finally, if $\omega(\eta) \le C\eta$, then Lemma~\ref{lem:boundedness-and-local-Lipschitz} implies that for any $t\in[0, 1]$, $z\mapsto v(t, z)$ is differentiable almost everywhere $z\in\Omega$. (This follows because for any $z\in\Omega^\circ$, $\mbox{dist}(z, \partial\Omega) > 0$ and for all $h$ with small enough Euclidean norm, $z + h \in \Omega^\circ$. This yields $z, z + h \in \Omega^{-\varepsilon}$ for some $\varepsilon > 0$ and hence, $z\mapsto v(t, z)$ is locally Lipschitz on $\Omega^\circ$. Rademacher theorem, now, implies almost everywhere differentiability.) By Theorem~\ref{thm:uniquene-ae-implies-transport-map}, $\mathfrak{R}(s, X) \overset{d}{=} (1-s)X_0 + sX_1$ and hence, is absolutely continuous with respect to the Lebesgue measure. Therefore, $\mu_0(B(t)) = 0$ for all $t\in[0, 1]$. Observe now that, by Fubini's theorem,
    \[
    (\mbox{Leb}\times\mu_0)(\mathcal{E}_3^c) = \int_0^1 \mu_0(B(t))dt = 0.
    \]
    On the other hand, note that
    \[
    0 = (\mbox{Leb}\times\mu_0)(\mathcal{E}_3^c) = \int_{\Omega} \mbox{Leb}(\{t\in[0,1]:\, \mathfrak{R}(t, z)\in B(t)\})\mu_0(dz).
    \]
    Because the integrand is non-negative, it must be zero almost everywhere, which implies that $\mu_0(\mathcal{E}_4) = 0$.
    
    To prove the bounds on distance to boundary, note from~\eqref{eq:whole-path-integral} (and~\eqref{eq:bounded-velocity}) that 
    \[
    \|\mathfrak{R}(t, x) - \mathfrak{R}(t', x)\| \le \mbox{diam}(\Omega)|t - t'|\quad\mbox{for all}\quad t, t'\in[0, 1].
    \]
    Hence, Lipschitzness of the distance to a compact set implies that 
    \[
    \mbox{dist}(\mathfrak{R}(t, x), \partial\Omega) \ge \mbox{dist}(\mathfrak{R}(1, x), \partial\Omega) - \mbox{diam}(\Omega)(1 - t).
    \]
    Additionally, from Lemma~\ref{lem:distance-to-boundary-latest}, we know
    \[
    \mbox{dist}(\mathfrak{R}(t, x), \partial\Omega) \ge (1-t)\mbox{dist}(\mathfrak{R}(0, x),\,\partial\Omega).
    \]
    Therefore, 
    \begin{align*}
    &\inf_{t\in[0, 1]}\,\mbox{dist}(\mathfrak{R}(t, x), \partial\Omega)\\ 
    &\quad\ge \inf_{t\in[0, 1]}\max\left\{\mbox{dist}(\mathfrak{R}(1, x), \partial\Omega) - \mbox{diam}(\Omega)(1 - t),\, (1-t)\mbox{dist}(\mathfrak{R}(0, x),\,\partial\Omega)\right\}\\
    &\quad= \frac{\mbox{dist}(\mathfrak{R}(1, x),\,\partial\Omega)\mbox{dist}(\mathfrak{R}(0, x),\partial\Omega)}{\mbox{diam}(\Omega) + \mbox{dist}(\mathfrak{R}(0, x),\,\partial\Omega)}. 
    \end{align*}
    This proves~\eqref{eq:dist-to-boundary-path}. 
    
    From Theorem~\ref{thm:rectified-flow-ODE-unique-sol}, we know that~\eqref{eq:rectified-flow-integral-equation} has a unique solution, which by Theorem~\ref{thm:uniquene-ae-implies-transport-map} implies that $\mathfrak{R}(1, X) \sim \mu_1$ whenever $X\sim \mu_0$. Therefore, 
    \begin{align*}
    \mathbb{P}(\mbox{dist}(\mathfrak{R}(1, X),\,\partial\Omega) \le \gamma) &= \frac{\mathbb{P}(\mbox{dist}(\mathfrak{R}(1, X),\,\partial\Omega) \le \gamma)}{\mathbb{P}(\mathfrak{R}(1, X)\in\Omega)}\\ 
    &= \frac{\mathbb{P}(\mathfrak{R}(1, X)\in \Omega\setminus \Omega^{-\gamma})}{\mathbb{P}(\mathfrak{R}(1, X)\in\Omega)}\\
    &\le \frac{\overline{\mathfrak{p}}}{\underline{\mathfrak{p}}}\frac{\mbox{Vol}(\Omega\setminus\Omega^{-\gamma})}{\mbox{Vol}(\Omega)}.
    \end{align*}
    Because $\Omega\supseteq \mathcal{B}(z^*, r_{\mathrm{in}})$, we get $\Omega^{-\gamma} \supseteq \mathcal{B}(z^*, r_{\mathrm{in}} - \gamma)$ and hence, Lemma~\ref{lem:Lipschitz-of-volume} yields
    \[
    \frac{\mbox{Vol}(\Omega\setminus\Omega^{-\gamma})}{\mbox{Vol}(\Omega)} \le \frac{d\gamma}{(r_{\mathrm{in}} - \gamma)_+}\left(1 + \frac{\gamma}{(r_{\mathrm{in}} - \gamma)_+}\right)^{d-1}. 
    \]
    If $\gamma \le r_{\mathrm{in}}/2$, then $\max\{\gamma, r_{\mathrm{in}}/2\} \le r_{\mathrm{in}} - \gamma$ and hence, 
    \[
    \mathbb{P}(\mbox{dist}(\mathfrak{R}(1, X),\,\partial\Omega) \le \gamma) \le \frac{\overline{\mathfrak{p}}}{\underline{\mathfrak{p}}}\frac{2^{d}d\gamma}{r_{\mathrm{in}}}.
    \]
    On the other hand, if $\gamma > r_{\mathrm{in}}/2$, then
    \[
    \mathbb{P}(\mbox{dist}(\mathfrak{R}(1, X),\,\partial\Omega) \le \gamma) \le 1 \le \frac{2\gamma}{r_{\mathrm{in}}}.
    \]
    Hence, for all $\gamma > 0$,
    \[
    \mathbb{P}(\mbox{dist}(\mathfrak{R}(1, X),\,\partial\Omega) \le \gamma) \le \frac{\overline{\mathfrak{p}}}{\underline{\mathfrak{p}}}\frac{2^{d}d\gamma}{r_{\mathrm{in}}}.
    \]
    % Setting $\gamma' = \gamma r_{\mathrm{in}}({\underline{\mathfrak{p}}}/{\overline{\mathfrak{p}}})/(d2^d)$ and $\mathcal{A}_{\gamma} = \{x:\, \mbox{dist}(\mathfrak{R}(1, x),\,\partial\Omega) \le \gamma'\}$, we get 
    % \[
    % \mu_0(\mathcal{A}_{\gamma}) = \mathbb{P}(\mbox{dist}(\mathfrak{R}(1, X),\,\partial\Omega)) \le \gamma.
    % \]
    For all $x\in\mathcal{A}_{\gamma}^c\cap\Omega^{-\varepsilon}$, inequality~\eqref{eq:dist-to-boundary-path} implies the result.
%%%%%%%%%%%%%%%%%%%%%%%%%%%%%%%%%%%
%%%%%%%%%%%%%%%%%%%%%%%%%%%%%%%%%%%
\subsection{Proof of Theorem~\ref{thm:continuity-of-paths}}\label{appsubsec:proof-continuity-of-paths}
Fix any $\varepsilon > 0$ and $x, x'\in\Omega^{-\varepsilon}$ with $\|x - x'\| = \eta$. From Theorem~\ref{thm:rectified-flow-ODE-unique-sol}, we know that $\mathfrak{R}(t, x)$ and $\mathfrak{R}(t, x')$ are uniquely defined. From~\eqref{eq:rectified-flow-integral-equation}, we know that
\begin{align*}
\mathfrak{R}(t, x) &= x + \int_0^t v(s, \mathfrak{R}(s, x))ds,\\
\mathfrak{R}(t, x') &= x' + \int_0^t v(s, \mathfrak{R}(s, x'))ds. 
\end{align*} 
This implies that
\begin{equation}\label{eq:basic-inequality}
\|\mathfrak{R}(t, x) - \mathfrak{R}(t, x')\| \le \|x - x'\| + \int_0^t \|v(s, \mathfrak{R}(s, x)) - v(s, \mathfrak{R}(s, x'))\|ds.
\end{equation}
From Lemma~\ref{lem:distance-to-boundary-latest}, we know that
\[
\mathfrak{R}(s, x), \mathfrak{R}(s, x') \in \Omega^{-(1-s)\varepsilon}\quad\mbox{for all}\quad s\in[0,1].
\]
This implies (using~\eqref{eq:locally-Lipschitz}) that
\[
\|v(s, \mathfrak{R}(s, x)) - v(s, \mathfrak{R}(s, x'))\| \le \mathfrak{L}_1\omega(\|\mathfrak{R}(s, x) - \mathfrak{R}(s, x')\|) + \mathfrak{L}_2((1-s)\varepsilon; s)\|\mathfrak{R}(s, x) - \mathfrak{R}(s, x')\|.
\]
Therefore, substituting in~\eqref{eq:basic-inequality}, we conclude
\begin{equation}\label{eq:differential-inequality-bnd-path}
\begin{split}
\|\mathfrak{R}(t, x) - \mathfrak{R}(t, x')\| &\le \|x - x'\| + \int_0^t \mathfrak{L}_1\omega(\|\mathfrak{R}(s, x) - \mathfrak{R}(s, x')\|)ds\\ 
&\qquad+ \int_0^t \mathfrak{L}_2((1-s)\varepsilon; s)\|\mathfrak{R}(s, x) - \mathfrak{R}(s, x')\|ds.
\end{split}
\end{equation}
Set
\begin{align*}
V(t) &:= \|x - x'\| + \int_0^t \mathfrak{L}_1\omega(\|\mathfrak{R}(s, x) - \mathfrak{R}(s, x')\|)ds\\ 
&\qquad+ \int_0^t \mathfrak{L}_2((1-s)\varepsilon; s)\|\mathfrak{R}(s, x) - \mathfrak{R}(s, x')\|ds\quad\mbox{for all}\quad t \in [0, 1].
\end{align*}
Clearly, for all $t\in[0, 1]$,
\begin{align*}
V'(t) &= \mathfrak{L}_1\omega(\|\mathfrak{R}(t, x) - \mathfrak{R}(t, x')\|) + \mathfrak{L}_2((1-t)\varepsilon)\|\mathfrak{R}(t, x) - \mathfrak{R}(t, x')\|\\
&\le \mathfrak{L}_1\omega(V(t)) + \mathfrak{L}_2((1-t)\varepsilon; t)V(t)\\
&\le (\mathfrak{L}_1 + \mathfrak{L}_2((1 - t)\varepsilon; t))(\omega(V(t)) + V(t)).
\end{align*}
This implies that
\begin{equation}\label{eq:main-differential-inequality-continuity-of-paths}
\frac{V'(t)}{V(t) + \omega(V(t))} \le \mathfrak{L}_1 + \mathfrak{L}_2((1-t)\varepsilon; t)\quad\mbox{for all}\quad t\in[0, 1].
\end{equation}
Equivalently, 
\begin{equation}\label{eq:main-integral-inequality}
\Psi(V(t)) - \Psi(V(0)) \le \mathfrak{L}_1t + \int_0^t \mathfrak{L}_2((1 - s)\varepsilon; s)ds,\quad\mbox{for all}\quad t\in [0, 1].
\end{equation}
%Note that if $\mathfrak{w}(u) = \int_0^u \mathfrak{L}_2(a)da$, then 
%\[
%\int_0^t \mathfrak{L}_2((1 - s)\varepsilon)ds = \frac{1}{\varepsilon}\int_{(1-t)\varepsilon}^{\varepsilon} \mathfrak{L}_2(a)da = \frac{1}{\varepsilon}(\mathfrak{w}(\varepsilon) - \mathfrak{w}((1-t)\varepsilon)).
%\]
%Therefore, 
%\[
%V(t) \le \Psi^{-1}\left(\Psi(\|x - x'\|) + \mathfrak{L}_1 + \frac{\mathfrak{w}(\varepsilon) - \mathfrak{w}((1-t)\varepsilon)}{\varepsilon}\right).
%\]
From~\eqref{eq:time-dependent-Lipschitz-constant}, setting $\mathfrak{C} = 3\overline{\mathfrak{p}}^2\mbox{diam}(\Omega)5^{d+1}/\underline{\mathfrak{p}}^2$, we get
\[
\mathfrak{L}_2((1-s)\varepsilon; s) = \frac{3\mbox{diam}(\Omega)}{\omega^{-1}(1)} + \frac{\mathfrak{C}}{(1-s)\varepsilon}\times\begin{cases}
1, &\mbox{if } s \le \varepsilon/(\varepsilon + \mbox{diam}(\Omega)),\\
1/s, &\mbox{if }\varepsilon/(\varepsilon + \mbox{diam}(\Omega)) \le s \le 1/2,\\
1/(1-s), &\mbox{if }s \ge 1/2.
\end{cases}
\]
This yields (for $t \ge 1/2$)
\begin{align*}
\int_0^t \mathfrak{L}_2((1-s)\varepsilon; s)ds &= \frac{3t\mbox{diam}(\Omega)}{\omega^{-1}(1)} + \frac{\mathfrak{C}}{\varepsilon}\left[\int_0^{1/2} (1-s)^{-1}ds + \int_{\varepsilon/(\varepsilon + \mbox{diam}(\Omega))}^{1/2} s^{-1}ds + \int_{1/2}^t (1-s)^{-2}ds\right]\\
&= \frac{3t\mbox{diam}(\Omega)}{\omega^{-1}(1)} + \frac{\mathfrak{C}}{\varepsilon}\left[\ln\left(\frac{\mbox{diam}(\Omega) + \varepsilon}{\varepsilon}\right) + \frac{2t-1}{1-t}\right].
\end{align*}
Noting $V(0) = \|x - x'\|$, we conclude that, for all $t\in [1/2, 1]$,
\begin{equation}\label{eq:partial-result-uniform-cont}
\begin{split}
&\sup_{s\in[0,t]}\|\mathfrak{R}(s, x) - \mathfrak{R}(s, x')\|\\ 
&\quad\le V(t) \le~ \Psi^{-1}\left(\Psi(\|x - x'\|) + \mathfrak{L}_1 + \frac{3\mbox{diam}(\Omega)}{\omega^{-1}(1)} + \frac{\mathfrak{C}\ln(1 + \mbox{diam}(\Omega)/\varepsilon)}{\varepsilon} + \frac{\mathfrak{C}t}{\varepsilon(1-t)}\right)\\
&\quad\le \Psi^{-1}\left(\Psi(\|x - x'\|) + \mathfrak{C}_1 + \frac{\mathfrak{C}\ln(2\mbox{diam}(\Omega)/\varepsilon)}{\varepsilon} + \frac{\mathfrak{C}}{\varepsilon}\frac{t}{1-t}\right),
\end{split}
\end{equation}
for some constant $\mathfrak{C}_1$.
Suppose 
\begin{equation}\label{eq:restriction-x-x'}
\|x - x'\| \le \Psi^{-1}(-2\mathfrak{C}/\varepsilon),
\end{equation}
so that $\Psi(\|x - x'\|) \le -2\mathfrak{C}/{\varepsilon}.$
Set $t^*\in[1/2, 1)$ such that
\begin{equation}\label{eq:some-simple-equation}
(2\mathfrak{C}/{\varepsilon})\frac{t^*}{1 - t^*} = -\Psi(\|x - x'\|)\quad\equiv\quad t^* = \frac{-\Psi(\|x - x'\|)}{2\mathfrak{C}/{\varepsilon} - \Psi(\|x - x'\|)} \ge \frac{1}{2}.
\end{equation}
Inequality~\eqref{eq:partial-result-uniform-cont} with $t = t^*$, now, implies
\[
\sup_{s\in[0,t^*]}\|\mathfrak{R}(s, x) - \mathfrak{R}(s, x')\| ~\le~ \Psi^{-1}(\Psi(\|x - x'\|)/2 + \mathfrak{C}_1 + \mathfrak{C}\ln(2\mbox{diam}(\Omega)/\varepsilon)/\varepsilon).
\]
To bound the difference for $t > t^*$, we use the fact that $\|v(s, z)\| \le \mbox{diam}(\Omega)$ to conclude
\begin{align*}
\sup_{t\in[t^*, 1]}\,\|\mathfrak{R}(t, x) - \mathfrak{R}(t, x')\| &\le \|\mathfrak{R}(t^*, x) - \mathfrak{R}(t^*, x')\| + 2\mbox{diam}(\Omega)|1 - t^*|\\
&\le \Psi^{-1}(\Psi(\|x - x'\|)/2) + 2\mbox{diam}(\Omega)\frac{2\mathfrak{C}}{2\mathfrak{C} - \varepsilon\Psi(\|x - x'\|)}.
\end{align*}
Finally, if~\eqref{eq:restriction-x-x'} is not satisfied, then using the fact that $\mathfrak{R}(t, x), \mathfrak{R}(t, x')\in\Omega$, we obtain that
\[
\sup_{t\in[0, 1]}\|\mathfrak{R}(t, x) - \mathfrak{R}(t, x')\| \le \mbox{diam}(\Omega) \le \mbox{diam}(\Omega)\frac{\|x - x'\|}{\Psi^{-1}(-2\mathfrak{C}/{\varepsilon})}.
\]
Combining all these inequalities, we conclude the first part of the result.

To prove the result for $x, x'\in \mathcal{A}_{\gamma}^c\cap\Omega^{-\varepsilon}$, we note from the proof of Lemma~\ref{lem:distance-to-boundary-at-1} that
\[
\mathfrak{R}(s, x), \mathfrak{R}(s, x') \in \Omega^{-(1-s)\varepsilon}\quad\mbox{for}\quad 0 \le s\le 1 - \frac{\gamma}{\varepsilon + \mbox{diam}(\Omega)},
\]
and
\[
\mathfrak{R}(s, x), \mathfrak{R}(s, x') \in \Omega^{-(\gamma - \mathrm{diam}(\Omega)(1-s))}\quad\mbox{for}\quad 1 - \frac{\gamma}{\varepsilon + \mbox{diam}(\Omega)} \le s \le 1,
\]
Note that $\gamma \le \mbox{diam}(\Omega)/2$ and hence $1 - \gamma/(\varepsilon + \mbox{diam}(\Omega)) \ge 1/2$.
This implies (using~\eqref{eq:locally-Lipschitz}) that
\begin{equation}\label{eq:better-lipschitz-constant-continuity-of-paths}
\|v(s, \mathfrak{R}(s, x)) - v(s, \mathfrak{R}(s, x'))\| \le \mathfrak{L}_1\omega(\|\mathfrak{R}(s, x) - \mathfrak{R}(s, x')\|) + \overline{\mathfrak{L}}_2(s)\|\mathfrak{R}(s, x) - \mathfrak{R}(s, x')\|,
\end{equation}
where
\begin{equation*}
\begin{split}
    \overline{\mathfrak{L}}_2(s) &:= \frac{3\mbox{diam}(\Omega)}{\omega^{-1}(1)}\\ 
    &+ \mathfrak{C}\times
    \begin{cases}
    ((1-s)\varepsilon)^{-1}, &\mbox{if }s \le \varepsilon/(\varepsilon + \mbox{diam}(\Omega)),\\
    ((1-s)s\varepsilon)^{-1}, &\mbox{if }\varepsilon/(\varepsilon + \mbox{diam}(\Omega)) \le s \le 1/2,\\
    ((1-s)^2\varepsilon)^{-1}, &\mbox{if }1/2 \le s \le 1 - \gamma/(\varepsilon + \mbox{diam}(\Omega)),\\
    ((\gamma - \mathrm{diam}(\Omega)(1-s))(1-s))^{-1}, &\mbox{if }1 - \gamma/(\varepsilon + \mbox{diam}(\Omega)) \le s \le 1 - \gamma/(2\mbox{diam}(\Omega)),\\
    (\gamma - \mathrm{diam}(\Omega)(1-s))^{-1}, &\mbox{if }s \ge 1 - \gamma/(2\mbox{diam}(\Omega)).
    \end{cases}
\end{split}
\end{equation*}
From this expression, we also derive
\[
\int_0^1 \overline{\mathfrak{L}}_2(s)ds = \frac{3\mbox{diam}(\Omega)}{\omega^{-1}(1)} + \frac{\mathfrak{C}}{\varepsilon}\left[\ln\left(\frac{\mbox{diam}(\Omega) + \varepsilon}{\varepsilon}\right) + \frac{\mbox{diam}(\Omega) + \varepsilon - 2\gamma}{\gamma}\right] + \frac{\mathfrak{C}}{\gamma}\ln\left(\frac{\mbox{diam}(\Omega)}{\varepsilon}\right) + \frac{\mathfrak{C}\ln 2}{\mbox{diam}(\Omega)}.
\]
Hence, (following the proof of)~\eqref{eq:main-integral-inequality} implies
\[
\Psi(V(1)) - \Psi(V(0)) \le \mathfrak{L}_1 + \frac{3\mbox{diam}(\Omega)}{\omega^{-1}(1)} + \frac{\mathfrak{C}\ln 2}{\mbox{diam}(\Omega)} + \frac{2\mathfrak{C}}{\min\{\varepsilon, \gamma\}}\ln\left(\frac{2\mbox{diam}(\Omega)}{\varepsilon}\right) + \frac{2\mathfrak{C}\mbox{diam}(\Omega)}{\varepsilon\gamma}.
\]
Therefore,
\begin{align*}
    &\sup_{s\in[0,1]}\,\|\mathfrak{R}(s, x) - \mathfrak{R}(s, x')\|\\ 
    &\le~ \Psi^{-1}\left(\Psi(\|x - x'\|) + \mathfrak{L}_1 + \frac{3\mbox{diam}(\Omega)}{\omega^{-1}(1)} + \frac{\mathfrak{C}\ln 2}{\mbox{diam}(\Omega)} + \frac{2\mathfrak{C}}{\min\{\varepsilon, \gamma\}}\ln\left(\frac{2\mbox{diam}(\Omega)}{\varepsilon}\right) + \frac{2\mathfrak{C}\mbox{diam}(\Omega)}{\varepsilon\gamma}\right).
\end{align*}
%%%%%%%%%%%%%%%%%%%%%%%%%%%%%%%%%%
%%%%%%%%%%%%%%%%%%%%%%%%%%%%%%%%%%
\subsection{Proof of Theorem~\ref{thm:equi-continuity-in-velocity}}\label{appsubsec:proof-equicontinuity-in-velocity}
    Fix any $\nu_1, \nu_2\in\mathcal{V}$ such that $\|\nu_1 - \nu_2\|_{\infty} \le \Delta$.
    Applying~\eqref{eq:new-integral-equation-equicontinuity} with $\nu_1$ and $\nu_2$, we obtain
    \begin{align*}
        \mathfrak{R}_{\nu_1}(t, x) - \mathfrak{R}_{\nu_2}(t, x) &= \int_0^t \left\{\nu_1(s, \mathfrak{R}_{\nu_1}(s, x)) - \nu_2(s, \mathfrak{R}_{\nu_2}(s, x))\right\}ds\\
        &= \int_0^t \{\nu_1(s, \mathfrak{R}_{\nu_1}(s, x)) - \nu_2(s, \mathfrak{R}_{\nu_1}(s, x))\}ds\\
        &\quad+ \int_0^t \{\nu_2(s, \mathfrak{R}_{\nu_1}(s, x)) - \nu_2(s, \mathfrak{R}_{\nu_2}(s, x))\}ds.
    \end{align*}
    This implies that for $t\in[0, 1]$
    \begin{align*}
    \|\mathfrak{R}_{\nu_1}(t, x) - \mathfrak{R}_{\nu_2}(t, x)\| \le \Delta + \int_0^t \|\nu_2(s, \mathfrak{R}_{\nu_1}(s, x)) - \nu_2(s, \mathfrak{R}_{\nu_2}(s, x))\|ds.
    \end{align*}
    To control the second term on the right hand side, observe that Lemma~\ref{lem:distance-to-boundary-latest} implies that for all $\nu\in\mathcal{V}$, $\mathfrak{R}_{\nu}(s, x) \in \Omega^{-(1-s)\varepsilon}$ for all $s\in[0, 1]$. Therefore, property~\ref{eq:Lipschitz-v-class-prop} implies
    \begin{equation}\label{eq:equicontinuity-weak-lipschitz}
    \begin{split}
    &\|\nu_2(s, \mathfrak{R}_{\nu_1}(s, x)) - \nu_2(s, \mathfrak{R}_{\nu_2}(s, x))\|\\ 
    &\le \mathfrak{C}\omega(\|\mathfrak{R}_{\nu_1}(s, x) - \mathfrak{R}_{\nu_2}(s, x)\|)\\ 
    &\quad+ \frac{\mathfrak{C}}{(1-s)\varepsilon}\|\mathfrak{R}_{\nu_1}(s, x) - \mathfrak{R}_{\nu_2}(s, x)\|\times
    \begin{cases}
    1, &\mbox{if }s \le \varepsilon/(\varepsilon + \mbox{diam}(\Omega)),\\
    1/s, &\mbox{if }\varepsilon/(\varepsilon + \mbox{diam}(\Omega)) \le s \le 1/2,\\
    1/(1-s), &\mbox{if }s \ge 1/2.
    \end{cases}
    \end{split}
    \end{equation}
    Set 
    \begin{align*}
    V(t) &= \Delta + \int_0^t \mathfrak{C}\omega(\|\mathfrak{R}_{\nu_1}(s, x) - \mathfrak{R}_{\nu_2}(s, x)\|)ds\\
    &\quad+ \int_0^t \mathfrak{L}(s; \varepsilon)\|\mathfrak{R}_{\nu_1}(s, x) - \mathfrak{R}_{\nu_2}(s, x)\|ds,
    \end{align*}
    where
    \begin{align*}
        \mathfrak{L}(\varepsilon; s) :=  \frac{\mathfrak{C}}{(1-s)\varepsilon}
\times
    \begin{cases}
    1, &\mbox{if }s \le \varepsilon/(\varepsilon + \mbox{diam}(\Omega)),\\
    1/s, &\mbox{if }\varepsilon/(\varepsilon + \mbox{diam}(\Omega)) \le s \le 1/2,\\
    1/(1-s), &\mbox{if }s \ge 1/2.
    \end{cases}
    \end{align*}
    This implies
    \begin{align*}
        V'(t) \le \mathfrak{L}(\varepsilon, t)(\omega(V(t)) + V(t)).
    \end{align*}
    Now following the proof of Theorem~\ref{thm:continuity-of-paths} (after~\eqref{eq:main-differential-inequality-continuity-of-paths}), we conclude that
    \begin{align*}
    &\sup_{\substack{\nu_1, \nu_2\in\mathcal{V},\,x\in\Omega^{-\varepsilon}\\\|\nu_1 - \nu_2\|_{\infty} \le \Delta}}\, \sup_{t\in[0, 1]}\,\|\mathfrak{R}_{\nu_1}(t, x) - \mathfrak{R}_{\nu_2}(t, x)\|\\ 
    &\le \Psi^{-1}(\Psi(\Delta)/2 + C\ln(\mbox{diam}(\Omega)/\varepsilon)/\varepsilon) + \mbox{diam}(\Omega)\max\left\{\frac{4C}{2C - \varepsilon\Psi(\Delta)},\,\frac{\Delta}{\Psi^{-1}(-C/\varepsilon)}\right\},
    \end{align*}
    for some constant $C$ depending only on $\mathfrak{C}$.

    To prove the result under~\eqref{eq:dist-to-boundary-nu_2}, note that the above inequality yields
    \[
    \mbox{dist}(\mathfrak{R}_{\nu_1}(1, x),\, \partial\Omega) \ge \mbox{dist}(\mathfrak{R}_{\nu_2}(1, x),\,\partial\Omega) - \mbox{RHS of}~\eqref{eq:equicontinuity-in-velocity-weak}.
    \]
    Hence, when $\Delta$ is small enough so that the right hand side of~\eqref{eq:equicontinuity-in-velocity-weak} is smaller than $\gamma$, we get
    \[
    \mbox{dist}(\mathfrak{R}_{\nu_1}(1, x),\, \partial\Omega) \ge \gamma.
    \]
    This implies that inequality~\eqref{eq:equicontinuity-weak-lipschitz} can be improved in the second term exactly as in the proof of Theorem~\ref{thm:continuity-of-paths} (see inequality~\eqref{eq:better-lipschitz-constant-continuity-of-paths}). Hence, following that proof, the second result follows. 
%%%%%%%%%%%%%%%%%%%%%%%%%%%%%%%%%%
%%%%%%%%%%%%%%%%%%%%%%%%%%%%%%%%%%
\subsection{Proof of Theorem~\ref{thm:rate-of-conv-velocity}}\label{appsubsec:proof-rate-of-conv-velocity}
    For a non-negative function $h:\mathbb{R}^d\to\mathbb{R}_+$, set
    \begin{align*}
    \widehat{T}_h(z) &= \int_{S_t(z)} h(\delta)\widehat{p}_0(z - t\delta)\widehat{p}_1(z + (1-t)\delta)d\delta,\\
    {T}_h(z) &= \int_{S_t(z)} h(\delta){p}_0(z - t\delta){p}_1(z + (1-t)\delta)d\delta
    \end{align*}
    Using~\eqref{eq:estimted-denisty-ratio-implication}, we get
    \begin{equation}\label{eq:T_h-density-estimated-ratio}
    e^{-2r_n}T_h(z) \le \widehat{T}_h(z) \le e^{2r_n}T_h(z)\quad\mbox{for all}\quad z\in\Omega.
    \end{equation}
    For any vector $a\in S^{d-1}$ (the unit sphere in $\mathbb{R}^d$), we have
    \begin{align*}
    &a^{\top}\widehat{v}^{\mathrm{den}}(t, z)\\
    &= \frac{\int_{S_t(z)} a^{\top}\delta\widehat{p}_0(z - t\delta)\widehat{p}_1(z + (1-t)\delta)d\delta}{\int_{S_t(z)} \widehat{p}_0(z - t\delta)\widehat{p}_1(z + (1-t)\delta)d\delta}\\
    &= \frac{\int_{S_t(z)} (a^{\top}\delta)_+\widehat{p}_0(z - t\delta)\widehat{p}_1(z + (1-t)\delta)d\delta}{\int_{S_t(z)} \widehat{p}_0(z - t\delta)\widehat{p}_1(z + (1-t)\delta)d\delta} - \frac{\int_{S_t(z)} (a^{\top}\delta)_-\widehat{p}_0(z - t\delta)\widehat{p}_1(z + (1-t)\delta)d\delta}{\int_{S_t(z)} \widehat{p}_0(z - t\delta)\widehat{p}_1(z + (1-t)\delta)d\delta}\\
    &\le e^{4r_n}\frac{\int_{S_t(z)} (a^{\top}\delta)_+{p}_0(z - t\delta){p}_1(z + (1-t)\delta)d\delta}{\int_{S_t(z)} {p}_0(z - t\delta){p}_1(z + (1-t)\delta)d\delta} - e^{-4r_n}\frac{\int_{S_t(z)} (a^{\top}\delta)_-{p}_0(z - t\delta){p}_1(z + (1-t)\delta)d\delta}{\int_{S_t(z)} {p}_0(z - t\delta){p}_1(z + (1-t)\delta)d\delta}\\
    &\le a^{\top}v(t, z) + \sup_{\delta\in S_t(z)}\|\delta\|\max\{|e^{4r_n} - 1|,\,|1 - e^{-4r_n}|\}.
    \end{align*}
    The first inequality follows from~\eqref{eq:T_h-density-estimated-ratio} with $h(\delta) = (a^{\top}\delta)_+$, $h(\delta) = (a^{\top}\delta)_-$, and $h(\delta) \equiv 1$. The second inequality follows from the fact that, for all $z\in\Omega^\circ$,
    \[
    \frac{T_h(z)}{p_t(z)} \le \sup_{\delta\in S_t(z)}|h(\delta)|.
    \]
    Swapping the roles of $\widehat{v}^{\mathrm{den}}$ and $v$, we get
    \[
    \sup_{t\in[0, 1],z\in\Omega^\circ}\|\widehat{v}^{\mathrm{den}}(t, z) - v(t, z)\| \le \sup_{\delta\in S_t(z)}\|\delta\|\max\{|e^{4r_n} - 1|,\,|1 - e^{-4r_n}|\} \le 2\mbox{diam}(\Omega)|e^{4r_n} - 1|.
    \]
    This completes the proof by noting that $\sup_{\delta\in S_t(z)}\|\delta\| \le 2\mbox{diam}(\Omega)$.
%%%%%%%%%%%%%%%%%%%%%%%%%%%%%%%%%%
%%%%%%%%%%%%%%%%%%%%%%%%%%%%%%%%%%
\subsection{Proof of Lemma~\ref{lem:equicontinuity-vhat-v}}\label{appsubsec:proof-equicontinuity-vhat-v}
Fix any $z, z'\in\Omega^{-\varepsilon}$ with $\eta = \|z - z'\|.$ Set
\begin{align*}
    \widehat{f}_t(z) &= \int_{S_t(z)} \delta \widehat{p}_0(z - t\delta)\widehat{p}_1(z + (1-t)\delta)d\delta\\
    {f}_t(z) &= \int_{S_t(z)} \delta {p}_0(z - t\delta){p}_1(z + (1-t)\delta)d\delta\\
    \widehat{p}_t(z) &= \int_{S_t(z)} \widehat{p}_0(z - t\delta)\widehat{p}_1(z + (1-t)\delta)d\delta\\
    {p}_t(z) &= \int_{S_t(z)} {p}_0(z - t\delta){p}_1(z + (1-t)\delta)d\delta
\end{align*}
Note that
\begin{align*}
m_t(z, z') &:= \left\{\widehat{v}^{\mathrm{den}}(t, z) - v(t, z)\right\} - \left\{\widehat{v}^{\mathrm{den}}(t, z') - v(t, z')\right\}\\
&= \frac{\widehat{f}_t(z) - \widehat{f}_t(z')}{\widehat{p}_t(z)} - \frac{f_t(z) - f_t(z')}{p_t(z)} + \widehat{v}^\mathrm{den}(t, z')\left(\frac{\widehat{p}_t(z')}{\widehat{p}_t(z)} - 1\right) - v(t, z')\left(\frac{p_t(z')}{p_t(z)} - 1\right)\\
&= \left(\frac{p_t(z)}{\widehat{p}_t(z)} - 1\right)\left(\frac{\widehat{f}_t(z) - \widehat{f}_t(z')}{p_t(z)}\right) + \frac{\{\widehat{f}_t(z) - \widehat{f}_t(z')\} - \{f_t(z) - f_t(z')\}}{p_t(z)}\\ 
&\quad+ (\widehat{v}^{\mathrm{den}}(t, z') - v(t, z'))\left(\frac{\widehat{p}_t(z')}{\widehat{p}_t(z)} - 1\right)\\ 
&\quad+ v(t, z')\frac{\widehat{p}_t(z')}{\widehat{p}_t(z)}\frac{\{p_t(z) - p_t(z')\} - \{\widehat{p}_t(z) - \widehat{p}_t(z')\}}{p_t(z)} + v(t, z')\frac{p_t(z')}{p_t(z)}\left(1 - \frac{\widehat{p}_t(z')}{\widehat{p}_t(z)}\right)\left(\frac{\widehat{p}_t(z')}{p_t(z')} - 1\right).
\end{align*}
(For a more detailed proof, see~\url{https://drive.google.com/file/d/1gXQsVFYonEB9eHqq52vlUHnpgffNojdi/view?usp=sharing}.)
Note that the first term can be further decomposed as 
\begin{align*}
\left(\frac{p_t(z)}{\widehat{p}_t(z)} - 1\right)\left(\frac{\widehat{f}_t(z) - \widehat{f}_t(z')}{p_t(z)}\right) &= \left(1 - \frac{\widehat{p}_t(z)}{p_t(z)}\right)\left(\widehat{v}^{\mathrm{den}}(t, z) - \widehat{v}^{\mathrm{den}}(t, z')\frac{\widehat{p}_t(z')}{\widehat{p}_t(z)}\right)\\
&= \left(1 - \frac{\widehat{p}_t(z)}{p_t(z)}\right)\left(\widehat{v}^{\mathrm{den}}(t, z) - \widehat{v}^{\mathrm{den}}(t, z')\right)\\
&\quad + \widehat{v}^{\mathrm{den}}(t, z')\left(1 - \frac{\widehat{p}_t(z)}{p_t(z)}\right)\left(1 - \frac{\widehat{p}_t(z')}{\widehat{p}_t(z)}\right).
\end{align*}
Every term in the decomposition has one factor related to the closeness of $z, z'$, and the other related to the closeness of $(\widehat{p}_0, \widehat{p}_1)$ and $(p_0, p_1)$.

We first note a few simple inequalities that control most of these terms. Under the assumption that $\omega(\kappa) \le L\kappa$ for all $\kappa > 0$ and $\eta \le \min\{1/L,\,\varepsilon^2/\mbox{diam}(\Omega)\}$, we get
\begin{equation}\label{eq:inequalities-ratio-Lipschitz}
    \begin{split}
    \max\{\|v(t, z)\|,\, \|v(t, z')\|,\, \|\widehat{v}^{\mathrm{den}}(t, z)\|,\,\|\widehat{v}^{\mathrm{den}}(t, z')\|\} ~&\le~ \mbox{diam}(\Omega),\\
    \max\left\{\left|\frac{p_t(z)}{\widehat{p}_t(z)} - 1\right|,\, \left|\frac{p_t(z')}{\widehat{p}_t(z')} - 1\right|,\, \left|\frac{\widehat{p}_t(z)}{p_t(z)} - 1\right|,\, \left|\frac{\widehat{p}_t(z')}{p_t(z')} - 1\right|\right\} ~&\le~ e^{2r_n} - 1,\\
    \|\widehat{v}^{\mathrm{den}}(t, z) - v(t, z)\| ~&\le~ 2\mbox{diam}(\Omega)(e^{4r_n} - 1),\\
    \max\left\{\left|\frac{p_t(z)}{p_t(z')} - 1\right|,\, \left|\frac{\widehat{p}_t(z)}{\widehat{p}_t(z')} - 1\right|,\,\left|\frac{p_t(z')}{p_t(z)} - 1\right|,\, \left|\frac{\widehat{p}_t(z')}{\widehat{p}_t(z)} - 1\right|\right\} ~&\le~ \eta\mbox{Lip}_t(\eta,\varepsilon),\\
    \|\widehat{v}^{\mathrm{den}}(t, z) - \widehat{v}^{\mathrm{den}}(t, z')\| ~&\le~ 3\eta\mbox{diam}(\Omega)\mbox{Lip}_t(\eta, \varepsilon),
    \end{split}
\end{equation}
where
\begin{align*}
\mbox{Lip}_t(\eta, \varepsilon) &:= 3L + \frac{\mathfrak{C}}{\varepsilon}\times
\begin{cases}
1, &\mbox{if }\min\{t, 1 - t\} \le \varepsilon/\mbox{diam}(\Omega),\\
1/\min\{t, 1 - t\}, &\mbox{otherwise.}
\end{cases}\\
&\le \frac{3L\varepsilon^2 + \mathfrak{C}\mbox{diam}(\Omega)}{\varepsilon^2} \le \frac{3L\mbox{diam}^2(\Omega) + \mathfrak{C}\mbox{diam}(\Omega)}{\varepsilon^2} =: \frac{C_1}{\varepsilon^2},
\end{align*}
with the constant $C_1$ depending only on $d, L, \mbox{diam}(\Omega),$ and $(\overline{\mathfrak{p}},\,\underline{\mathfrak{p}})$.
These inequalities follow from~\eqref{eq:bounded-velocity},~\eqref{eq:T_h-density-estimated-ratio} (with $h\equiv 1$), Theorem~\ref{thm:rate-of-conv-velocity},~\eqref{eq:ratio-z-z'-inequality} (with $h \equiv 1$), and~\eqref{eq:locally-Lipschitz} (along with~\eqref{eq:time-dependent-Lipschitz-constant}). These inequalities imply that 
\begin{align*}
    \|m_t(z, z')\| &\le 3(e^{2r_n} - 1)\eta\mbox{diam}(\Omega)\mbox{Lip}_t(\eta, \varepsilon) + \mbox{diam}(\Omega)(e^{2r_n} - 1)\eta\mbox{Lip}_t(\eta, \varepsilon)\\ 
    &\quad+ \frac{\|\{\widehat{f}_t(z) - \widehat{f}_t(z')\} - \{f_t(z) - f_t(z')\}\|}{p_t(z)}\\
    &\quad+ 2\mbox{diam}(\Omega)\eta(e^{4r_n} - 1)\mbox{Lip}_t(\eta, \varepsilon)\\ 
    &\quad+ \mbox{diam}(\Omega)(1 + \eta\mbox{Lip}_t(\eta, \varepsilon))\frac{|\{p_t(z) - p_t(z')\} - \{\widehat{p}_t(z) - \widehat{p}_t(z')\}|}{p_t(z)}\\
    &\quad+ \mbox{diam}(\Omega)(1 + \eta\mbox{Lip}_t(\eta,\varepsilon))\eta\mbox{Lip}_t(\eta, \varepsilon)(e^{2r_n} - 1)\\
    &\le C_2\eta(e^{4r_n} - 1)\mbox{Lip}_t(\eta, \varepsilon)(1 + \eta\mbox{Lip}_t(\eta, \varepsilon))\\
    &\quad+ \frac{\|\{\widehat{f}_t(z) - \widehat{f}_t(z')\} - \{f_t(z) - f_t(z')\}\|}{p_t(z)}\\
    &\quad+ \mbox{diam}(\Omega)(1 + \eta\mbox{Lip}_t(\eta, \varepsilon))\frac{|\{p_t(z) - p_t(z')\} - \{\widehat{p}_t(z) - \widehat{p}_t(z')\}|}{p_t(z)}.
\end{align*}
for some constant $C_2$ depending only on $\mbox{diam}(\Omega).$ To bound the last two terms, consider 
\begin{align*}
\Delta(h) &:= \frac{\int_{S_t(z)} h(\delta)\{p_0(z - t\delta)p_1(z + (1-t)\delta) - \widehat{p}_0(z - t\delta)\widehat{p}_1(z + (1-t)\delta)\}d\delta}{p_t(z)}\\ 
&\quad- \frac{\int_{S_t(z')} h(\delta)\{p_0(z' - t\delta)p_1(z' + (1-t)\delta) - \widehat{p}_0(z' - t\delta)\widehat{p}_1(z' + (1-t)\delta)\}d\delta}{p_t(z)}.
\end{align*}
Bounding $\Delta(h)$ for $h(\delta) = \delta$ and $h(\delta) \equiv 1$ provides bounds for the last two terms of $\|m_t(z, z')\|$.
We shall split $\Delta(h)$ into integrals on $S_t(z)\cap S_t(z')$, $S_t(z)\setminus S_t(z')$, and $S_t(z')\setminus S_t(z)$, as in the proof of Lemma~\ref{lem:boundedness-and-local-Lipschitz}. 
The integrand on $S_t(z)\cap S_t(z')$ without the factor of $h(\delta)$ can be decomposed as 
\begin{equation}\label{eq:intersection-Modulus-of-continuity}
    \begin{split}
    &(p_0(z - t\delta) - p_0(z' - t\delta) - \widehat{p}_0(z - t\delta) + \widehat{p}_0(z' - t\delta))p_1(z + (1-t)\delta)\\
    &\quad+ (\widehat{p}_0(z - t\delta) - \widehat{p}_0(z' - t\delta))(p_1(z + (1-t)\delta) - \widehat{p}_1(z + (1-t)\delta))\\
    &\quad+ p_0(z' - t\delta)(p_1(z + (1-t)\delta) - p_1(z' + (1-t)\delta) - \widehat{p}_1(z + (1-t)\delta) + \widehat{p}_1(z' + (1-t)\delta))\\
    &\quad+ (p_0(z' - t\delta) - \widehat{p}_0(z' - t\delta))(\widehat{p}_1(z + (1-t)\delta) - \widehat{p}_1(z'
    + (1-t)\delta)).
    \end{split}
\end{equation}
The second and the fourth terms of~\eqref{eq:intersection-Modulus-of-continuity} can be upper-bounded by
\[
L\|z - z'\|(e^{r_n} - 1)\left[\widehat{p}_0(z - t\delta)\widehat{p}_1(z + (1-t)\delta) + \widehat{p}_0(z' - t\delta)\widehat{p}_1(z' + (1-t)\delta)\right].
\]
The first and the third terms of~\eqref{eq:intersection-Modulus-of-continuity} can be upper-bounded using the log-derivatives of the densities. 
Note that these are well-defined only for $x\in\Omega^\circ$. We write
\begin{align*}
&p_0(z - t\delta) - p_0(z' - t\delta) - \widehat{p}_0(z - t\delta) + \widehat{p}_0(z' - t\delta)\\
&= p_0(z - t\delta)\left(1 - \frac{p_0(z' - t\delta)}{p_0(z - t\delta)}\right) - \widehat{p}_0(z - t\delta)\left(1 - \frac{\widehat{p}_0(z' - t\delta)}{\widehat{p}_0(z - t\delta)}\right)\\
&= p_0(z - t\delta)\left(\frac{\widehat{p}_0(z' - t\delta)}{\widehat{p}_0(z - t\delta)} - \frac{p_0(z' - t\delta)}{p_0(z - t\delta)}\right) + p_0(z - t\delta)\left(1 - \frac{\widehat{p}_0(z - t\delta)}{p_0(z - t\delta)}\right)\left(1 - \frac{\widehat{p}_0(z' - t\delta)}{\widehat{p}_0(z - t\delta)}\right).
\end{align*}
Hence, the first term can be upper-bounded by
\begin{align*}
&p_0(z - t\delta)p_1(z + (1-t)\delta)\left[(1 + L\eta)(\exp(s_n\eta) - 1) + (e^{r_n} - 1)L\eta\right]\\
&\le p_0(z - t\delta)p_1(z + (1-t)\delta)\left[2(\exp(s_n\eta) - 1) + (e^{r_n} - 1)L\eta\right].
\end{align*}
This follows from the fact that
\begin{align*}
&\frac{\widehat{p}_0(z' - t\delta)}{\widehat{p}_0(z - t\delta)} - \frac{p_0(z' - t\delta)}{p_0(z - t\delta)}\\ 
&= \frac{p_0(z' - t\delta)}{p_0(z - t\delta)}\left(\exp\left(\log\widehat{p}_0(z' - t\delta) - \log\widehat{p}_0(z - t\delta) - \log p_0(z' - t\delta) + \log p_0(z - t\delta)\right) - 1\right).
\end{align*}
Similarly, the third term of~\eqref{eq:intersection-Modulus-of-continuity} can be bounded by
\begin{align*}
&p_0(z' - t\delta)p_1(z' + (1-t)\delta)\left[(1 + L\eta)(\exp(s_n\eta) - 1) + (e^{r_n} - 1)L\eta\right]\\
&\le p_0(z' - t\delta)p_1(z' + (1-t)\delta)\left[2(\exp(s_n\eta) - 1) + (e^{r_n} - 1)L\eta\right].
\end{align*}
Combining these bounds, we conclude that the contribution to $\Delta(h)$ from the integration on $S_t(z)\cap S_t(z')$ is bounded by
\begin{equation}\label{eq:bound-on-intersection-part}
    \begin{split}
        &\frac{\sup_{\delta\in S_t(z)\cap S_t(z')}\|\delta\|}{p_t(z)}L\eta(e^{r_n} - 1)\left(\widehat{p}_t(z) + \widehat{p}_t(z')\right)\\
        &\quad+ \frac{\sup_{\delta\in S_t(z)\cap S_t(z')}\|\delta\|}{p_t(z)}\left[2(\exp(s_n\eta) - 1) + (e^{r_n} - 1)L\eta\right]\left(p_t(z) + p_t(z')\right).
    \end{split}
\end{equation}
Because $\sup_{\delta\in S_t(z)\cap S_t(z')}\|\delta\| \le 2\mbox{diam}(\Omega)$ and $\widehat{p}_t(z)/p_t(z), \widehat{p}_t(z')/p_t(z), p_t(z')/p_t(z)$ are all bounded using~\eqref{eq:inequalities-ratio-Lipschitz}, we get the simplified bound on the intersection integral as
\begin{equation}\label{eq:bound-on-intersection-part-final}
    \begin{split}
        &\mathfrak{C}\left(2 + \eta\mbox{Lip}_t(\eta, \varepsilon)\right)\left[2(\exp(s_n\eta) - 1) + (e^{r_n} - 1)L\eta\right].
    \end{split}
\end{equation}

To control the integral over $S_t(z)\setminus S_t(z')$, note that
\begin{align*}
    &|p_0(z - t\delta)p_1(z + (1-t)\delta) - \widehat{p}_0(z - t\delta)\widehat{p}_1(z + (1-t)\delta)|\\
    &\le (e^{2r_n} - 1)\overline{\mathfrak{p}}^2,
\end{align*}
and hence, the second term can be bounded by
\begin{equation}\label{eq:bound-on-second-term}
    \sup_{\delta\in S_t(z)}\,|h(\delta)|\frac{\overline{\mathfrak{p}}^2}{\underline{\mathfrak{p}}^2}\frac{\mbox{Vol}(S_t(z)\setminus S_t(z'))}{\mbox{Vol}(S_t(z))}(e^{2r_n} - 1).
\end{equation}
Similarly, the third term can be bounded by
\begin{equation}\label{eq:bound-on-third-term}
    \sup_{\delta\in S_t(z)}\,|h(\delta)|\frac{\overline{\mathfrak{p}}^2}{\underline{\mathfrak{p}}^2}\frac{\mbox{Vol}(S_t(z')\setminus S_t(z))}{\mbox{Vol}(S_t(z))}(e^{2r_n} - 1).
\end{equation}
Together, the second and third terms can be bounded by
\begin{equation}\label{eq:combined-second-third}
\begin{split}
    &\sup_{\delta\in S_t(z)}\,|h(\delta)|\frac{\overline{\mathfrak{p}}^2}{\underline{\mathfrak{p}}^2}\frac{\mbox{Vol}(S_t(z)\Delta S_t(z'))}{\mbox{Vol}(S_t(z))}(e^{2r_n} - 1)\\ 
    &\le~ \mathfrak{C}(e^{2r_n} - 1)\frac{\eta}{\varepsilon}\times\begin{cases}
        1, &\mbox{if }\min\{t, 1 - t\} \le \varepsilon/\mbox{diam}(\Omega),\\
        1/\min\{t, 1 - t\}, &\mbox{otherwise.}
    \end{cases}\\
    &\le~ \mathfrak{C}(e^{2r_n} - 1)\frac{\eta\mbox{diam}(\Omega)}{\varepsilon^2}.
\end{split}
\end{equation}
Combining bounds~\eqref{eq:bound-on-intersection-part-final},~\eqref{eq:combined-second-third} and the bound on $\|m_t(z, z')\|$, we conclude that
\begin{align*}
&\|m_t(z, z')\|\\ &\le C_2\eta(e^{4r_n} - 1)\mbox{Lip}_t(\eta, \varepsilon)(1 + \eta\mbox{Lip}_t(\eta, \varepsilon))\\
    &\quad+ C_3\mbox{diam}(\Omega)(1 + \eta\mbox{Lip}_t(\eta, \varepsilon))^2\left[2(\exp(s_n\eta) - 1) + (e^{r_n} - 1)L\eta\right]\\
    &\quad+ C_4\mbox{diam}(\Omega)(1 + \eta\mbox{Lip}_t(\eta, \varepsilon))(e^{2r_n} - 1)\frac{\eta}{\varepsilon}\times\begin{cases}
        1, &\mbox{if }\min\{t, 1 - t\} \le \varepsilon/\mbox{diam}(\Omega),\\
        1/\min\{t, 1 - t\}, &\mbox{otherwise.}
        \end{cases}\\
    &\le C_5\eta(e^{4r_n} - 1)(\varepsilon^{-2} + \eta\varepsilon^{-4})\\
    &\quad+ C_6(1 + \eta\varepsilon^{-2})^2\left[\exp(s_n\eta) - 1 + (e^{r_n} - 1)L\eta\right],
\end{align*}
for some constants $C_5$ and $C_6$ depending only on $d, L, \mbox{diam}(\Omega),$ and $(\overline{\mathfrak{p}},\,\underline{\mathfrak{p}})$.
%%%%%%%%%%%%%%%%%%%%%%%%%%%%%%%%%%
%%%%%%%%%%%%%%%%%%%%%%%%%%%%%%%%%%
\subsection{Proof of Proposition~\ref{prop:Lipschitz-cont-velocity-in-time}}\label{appsubsec:proof-Lipschitz-velocity-in-time}
    For $g:\mathbb{R}^d\to\mathbb{R}$, define
    \[
    T_t(z) = \frac{\int_{S_t(z)} g(\delta)p_0(z - t\delta)p_1(z + (1-t)\delta)d\delta}{\int_{S_t(z)} p_0(z - t\delta)p_1(z + (1-t)\delta)d\delta}.
    \]
    Note that
    \begin{align*}
        &T_t(z) - T_{t+h}(z)\\ 
        &= \int_{S_t(z)} g(\delta)p_0(z - t\delta)p_1(z + (1-t)\delta)d\delta - \int_{S_{t+h}(z)} g(\delta)p_0(z - (t+h)\delta)p_1(z + (1-t-h)\delta)d\delta\\
        &= \int_{S_t(z)\cap S_{t+h}(z)} g(\delta)p_0(z - t\delta)p_1(z + (1-t)\delta)\left[1 - \frac{p_0(z - (t+h)\delta)p_1(z + (1-t-h)\delta)}{p_0(z - t\delta)p_1(z + (1-t)\delta)}\right]d\delta\\
        &\quad+ \int_{S_t(z)\setminus S_{t+h}(z)} g(\delta)p_0(z - t\delta)p_1(z + (1-t)\delta)d\delta\\
        &\quad- \int_{S_{t+h}(z)\setminus S_t(z)} g(\delta)p_0(z - (t+h)\delta)p_1(z + (1-t-h)\delta)d\delta.
    \end{align*}
    From the modulus of continuity of the densities, we get
    \begin{align*}
    \left|1 - \frac{p_0(z - (t+h)\delta)p_1(z + (1-t-h)\delta)}{p_0(z - t\delta)p_1(z + (1-t)\delta)}\right| ~&\le~ (1 + \omega^2(h\|\delta\|)) - 1 \le 2Lh\|\delta\| + L^2h^2\|\delta\|^2\\
    ~&\le~ 4Lh\mbox{diam}(\Omega) + 4L^2h^2\mbox{diam}^2(\Omega)\\
    ~&\le~ 8L\mbox{diam}(\Omega)h,\quad\mbox{if }\quad h \le 1/(L\mbox{diam}(\Omega)).
    \end{align*}
    This implies that $t\mapsto 1 - {p_0(z - (t+h)\delta)p_1(z + (1-t-h)\delta)}/{p_0(z - t\delta)p_1(z + (1-t)\delta)}$ is Lipschitz continuous with a Lipschitz constant of $8L\mbox{diam}(\Omega)$ for all $t\in(0, 1)$. Therefore, the bound holds for all $h > 0$ such that $t + h\in[0, 1]$. 
    Hence, the first term in the decomposition of $T_t(z) - T_{t+h}(z)$ is bounded by
    \[
    p_t(z)\left(\sup_{\delta\in S_t(z)}|g(\delta)|\right)8L\mbox{diam}(\Omega)h.
    \]
    The second and third terms in the decompositions are together bounded by
    \[
    \left(\sup_{\delta\in S_t(z)}|g(\delta)|\right)\overline{\mathfrak{p}}^2\mbox{Vol}(S_t(z)\Delta S_{t+h}(z)).
    \]
    Because $p_t(z) \ge \underline{\mathfrak{p}}^2\mbox{Vol}(S_t(z))$, we get
    \[
    \frac{|T_t(z) - T_{t+h}(z)|}{p_t(z)} \le \left(\sup_{\delta\in S_t(z)}|g(\delta)|\right)\left(8L\mbox{diam}(\Omega)h + \frac{\overline{\mathfrak{p}}^2}{\underline{\mathfrak{p}}^2}\frac{\mbox{Vol}(S_t(z)\Delta S_{t+h}(z))}{\mbox{Vol}(S_t(z))}\right).
    \]
    To control the volume of the symmetric difference, we consider two cases.
    \paragraph{Case 1: $\min\{t, 1 - t\} \le \varepsilon/(2\mbox{diam}(\Omega))$.} In this case, using the assumption that $h \le \varepsilon/(2\mbox{diam}(\Omega))$, we get $\min\{t+h, 1 - t - h\} \le \varepsilon/\mbox{diam}(\Omega)$. Hence, Proposition~\ref{prop:properties-of-S_t}(1) implies
    \begin{align*}
    S_t(z) &= \begin{cases}
        (x - \Omega)/t, &\mbox{if }t \ge  \varepsilon/(2\mbox{diam}(\Omega)),\\
        (\Omega - x)/(1-t), &\mbox{otherwise.}
    \end{cases}\;\mbox{and}\\
    S_{t+h}(z) &= \begin{cases}
        (x - \Omega)/(t+h), &\mbox{if }t \ge  \varepsilon/(2\mbox{diam}(\Omega)),\\
        (\Omega - x)/(1-t - h), &\mbox{otherwise.}
    \end{cases}
    \end{align*}
    Therefore, 
    \begin{align*}
    \frac{\mbox{Vol}(S_t(z)\Delta S_{t+h}(z))}{\mbox{Vol}(S_t(z))} &\le \max\left\{\left|1 - \frac{t}{t + h}\right|,\, \left|1 - \frac{1-t}{1 - t - h}\right|\right\}\\
    &\le \frac{h}{\min\{t + h, 1 - t - h\}} \le \frac{\mbox{diam}(\Omega)h}{\varepsilon}.
    \end{align*}
    \paragraph{Case 2: $\min\{t, 1 - t\} > \varepsilon/(2\mbox{diam}(\Omega))$.} In this case, Proposition~\ref{prop:properties-of-S_t}(4) implies that
    \[
    \frac{\mbox{Vol}(S_{t+h}(z)\Delta S_t(z))}{\mbox{Vol}(S_t(z))} \le \frac{\mbox{Vol}((S_t(z))^{2\mathrm{diam}^2(\Omega)h/(\varepsilon\bar{t})}\setminus (S_t(z))^{-2\mathrm{diam}^2(\Omega)h/(\varepsilon\bar{t})})}{\mbox{Vol}(S_t(z))}.
    \]
    If $2\mbox{diam}^2(\Omega)h/(\varepsilon \bar{t}) \le \varepsilon/(2\max\{t, 1 - t\})$ (which is satisfied if $h \le \varepsilon^3/(4\mbox{diam}^3(\Omega))$), then Lemma~\ref{lem:Lipschitz-of-volume} implies
    \begin{align*}
    &\frac{\mbox{Vol}((S_t(z))^{2\mathrm{diam}^2(\Omega)h/(\varepsilon\bar{t})}\setminus (S_t(z))^{-2\mathrm{diam}^2(\Omega)h/(\varepsilon\bar{t})})}{\mbox{Vol}((S_t(z))^{-2\mathrm{diam}^2(\Omega)h/(\varepsilon\bar{t})})}\\ 
    &\le \frac{8d\mbox{diam}^2(\Omega)h\max\{t, 1 - t\}}{\varepsilon^2\bar{t}}\left(1 + \frac{8\mbox{diam}^2(\Omega)h\max\{t, 1 - t\}}{\varepsilon^2\bar{t}}\right)^{d-1}\\
    &\le \frac{4^{d+1}d\mbox{diam}^2(\Omega)h\max\{t, 1 - t\}}{\varepsilon^2\bar{t}}.
    \end{align*}
    Therefore,
    \[
    \frac{\mbox{Vol}(S_{t+h}(z)\Delta S_t(z))}{\mbox{Vol}(S_t(z))} \le \frac{4^{d+1}d\mbox{diam}^2(\Omega)h}{\varepsilon^2\bar{t}},\quad\mbox{if }h \le \varepsilon^3/(4\mbox{diam}^3(\Omega)).
    \]
    Combining both cases, we conclude
    \[
    \frac{|T_t(z) - T_{t+h}(z)|}{p_t(z)} \le \left(\sup_{\delta\in S_t(z)}|g(\delta)|\right)\left(8L\mbox{diam}(\Omega) + \frac{\overline{\mathfrak{p}}^2}{\underline{\mathfrak{p}}^2}\frac{4^{d+1}d\mbox{diam}^2(\Omega)}{\varepsilon^2\bar{t}}\right)h,
    \]
    whenever $h \le \varepsilon^3/(4\mbox{diam}^3(\Omega))$. Taking $g(\delta) = e_j^{\top}\delta$ and $g(\delta)\equiv v(t, z)$, we get 
    \begin{equation}\label{eq:lipschitz-velocity-in-t}
    \begin{split}
    \|v(t, z) - v(t + h, z)\| &\le \frac{\|f_t(z) - f_{t+h}(z)\|}{p_t(z)} + \|v(t + h, z)\|\left|\frac{p_{t+h}(z)}{p_t(z)} - 1\right|\\
    &\le 2\mbox{diam}(\Omega)\left(8L\mbox{diam}(\Omega) + \frac{\overline{\mathfrak{p}}^2}{\underline{\mathfrak{p}}^2}\frac{4^{d+1}d\mbox{diam}^2(\Omega)}{\varepsilon^2\bar{t}}\right)h,
    \end{split}
    \end{equation}
    whenever $h \le \min\{1/(L\mbox{diam}(\Omega)),\, \varepsilon^3/(4\mbox{diam}^3(\Omega))\}$. This implies that $t\mapsto v(t, z)$ is almost everywhere differentiable on $(0, 1)$, with derivative bounded by 
    \[
    2\mbox{diam}(\Omega)\left(8L\mbox{diam}(\Omega) + \frac{\overline{\mathfrak{p}}^2}{\underline{\mathfrak{p}}^2}\frac{4^{d+1}d\mbox{diam}^2(\Omega)}{\varepsilon^2\bar{t}}\right).
    \]
    Hence, inequality~\eqref{eq:lipschitz-velocity-in-t} holds for all $h > 0$ such that $t + h\in[0, 1]$. This proves the result. 
    %%%%%%%%%%%%%%%%%%%%%%%%%%%%%
    %%%%%%%%%%%%%%%%%%%%%%%%%%%%%
\subsection{Proof of Theorem~\ref{thm:linearization-bounded-case}}\label{appsubsec:proof-of-linearization}
Recall that
\begin{align*}
    \widehat{\mathfrak{R}}(t, x) &= x + \int_0^t \widehat{v}^{\mathrm{den}}(s,\,\widehat{\mathfrak{R}}(s, x))ds,\\
    \mathfrak{R}(t, x) &= x + \int_0^t v(s,\,\mathfrak{R}(s, x))ds.
\end{align*}
This implies $\widehat{E}(t, x) = \widehat{\mathfrak{R}}(t, x) - \mathfrak{R}(t, x)$ satisfies
\begin{align*}
\widehat{E}(t, x) &= \int_0^t \{\widehat{v}^{\mathrm{den}}(s, \widehat{\mathfrak{R}}(s, x)) - v(s, \mathfrak{R}(s, x))\}ds \\
&= \int_0^t \{\widehat{v}^{\mathrm{den}}(s, \widehat{\mathfrak{R}}(s, x)) - v(s, \widehat{\mathfrak{R}}(s, x)) - \widehat{v}^{\mathrm{den}}(s, \mathfrak{R}(s, x)) + v(s, \mathfrak{R}(s, x))\}dx\\
&\quad+ \int_0^t \{\widehat{v}^{\mathrm{den}}(s, \mathfrak{R}(s, x)) - v(s, \mathfrak{R}(s, x))\}ds\\
&\quad+ \int_0^t \{v(s, \widehat{\mathfrak{R}}(s, x)) - v(s, \mathfrak{R}(s, x))\}ds.
\end{align*}
From Lemma~\ref{lem:equicontinuity-vhat-v} (along with Corollary~\ref{cor:rate-of-conv-rectified-lipschitz}), the first term can be bounded (for large enough $n$) by
\[
C'_{\varepsilon,\gamma}\sup_{0 \le s \le 1}\|\widehat{E}(s, x)\|(r_n + s_n).
\]
From Lemma~\ref{lem:distance-to-boundary-at-1}, $\mathfrak{R}(s, x)$ belongs to the set of differentiability points of $z\mapsto v(s, z)$ for almost all $s\in[0, 1]$ and hence, 
\[
v(s, \widehat{\mathfrak{R}}(s, x)) - v(s, \mathfrak{R}(s, x)) = \frac{\partial}{\partial z}v(s, z)^\top\bigg|_{z = \mathfrak{R}(s, x)}(\widehat{\mathfrak{R}}(s, x) - \mathfrak{R}(s, x)) + o_p(\|\widehat{E}(s, x)\|).
\]
Moreover, the Lipschitz continuity of $z\mapsto v(s, z)$ on $\Omega^{-\kappa}$ (for any $\kappa > 0$) implies that
\[
\sup_{s\in[0, 1]}\left\|v(s, \widehat{\mathfrak{R}}(s, x)) - v(s, \mathfrak{R}(s, x)) - \frac{\partial}{\partial z}v(s, z)^\top\bigg|_{z = \mathfrak{R}(s, x)}(\widehat{\mathfrak{R}}(s, x) - \mathfrak{R}(s, x))\right\| \le C\sup_{s\in[0,1]}\|\widehat{E}(s, x)\|,
\]
for some constant $C$ depending on the Lipschitz constant of $z\mapsto v(s, z)$. 
Therefore, by the dominated convergence theorem, we get
\begin{equation}\label{eq:bound-differential-expansion}
\begin{split}
\mathcal{E}_n &:= \sup_{t\in[0, 1]}\left\|\widehat{E}(t, x) - \int_0^t \{\widehat{v}^{\mathrm{den}}(s, \mathfrak{R}(s, x)) - v(s, \mathfrak{R}(s, x))\}ds - \int_0^t \partial_zv(s, \mathfrak{R}(s, x))^\top\widehat{E}(s, x)ds\right\|\\
&= o_p\left(\sup_{s\in[0,1]}\|\widehat{E}(s, x)\|\right).
\end{split}
\end{equation}
Then, we get
\begin{align*}
    \sup_{t\in[0,1]}\, \left\|\widehat{E}(t, x) - \widetilde{E}(t, x) - \int_0^t \partial_zv(s, \mathfrak{R}(s, x))^\top(\widehat{E}(s, x) - \widetilde{E}(s, x))ds\right\| = \mathcal{E}_n.
\end{align*}
This implies
\[
\left\|\widehat{E}(t, x) - \widetilde{E}(t, x)\right\| \le \mathcal{E}_n + \int_0^t \|\partial_zv(s, \mathfrak{R}(s, x))\|_{\mathrm{op}}\|\widehat{E}(s, x) - \widetilde{E}(s, x)\|ds. 
\]
From the Lipschitz continuity of $z\mapsto v(s, z)$, we get $\|\partial_zv(s, \mathfrak{R}(s, x))\|_{\mathrm{op}} \le C$. Hence, setting $V(t) = \mathcal{E}_n + C\int_0^t \|\widehat{E}(s, x) - \widetilde{E}(s, x)\|ds$, we have
\[
V'(t) = C\left\|\widehat{E}(t, x) - \widetilde{E}(t, x)\right\| \le CV(t)\quad\mbox{for all}\quad t\in[0, 1].
\]
This implies
\[
\frac{V(t)}{V(0)} \le C\quad\Rightarrow\quad \left\|\widehat{E}(t, x) - \widetilde{E}(t, x)\right\| \le \mathcal{E}_ne^{Ct}\quad\mbox{for all}\quad t\in[0, 1].
\]
In particular, from~\eqref{eq:bound-differential-expansion}, we conclude
\[
\sup_{t\in[0, 1]}\|\widehat{E}(t, x) - \widetilde{E}(t, x)\| = o_p\left(\sup_{s\in[0,1]}\|\widehat{E}(s, x)\|\right).
\]
This is equivalent to
\[
\sup_{t\in[0, 1]}\|\widehat{E}(t, x) - \widetilde{E}(t, x)\| = o_p\left(\sup_{s\in[0,1]}\|\widetilde{E}(s, x)\|\right),\quad\mbox{as}\quad n\to\infty.
\]
This proves the result.
%%%%%%%%%%%%%%%%%%%%%%%%%%%%%%%%%%%%%%%
%%%%%%%%%%%%%%%%%%%%%%%%%%%%%%%%%%%%%%%
\subsection{Proof of Proposition~\ref{prop:density-estimator-bounded}}\label{appsubsec:proof-of-density-estimator-bias}
    Note that
    \begin{align*}
        \mathbb{E}[\widehat{p}_j(z)] &= \int_{\Omega} \frac{K_{z,h}(x - z)}{\mbox{Vol}(V_{z,h})}p_j(x)dx\\
        &= \int_{\Omega - z} \frac{K_{z,h}(u)}{\mbox{Vol}(V_{z,h})}p_j(z + u)du 
    \end{align*}
    Because $K_{z,h}(u) = 0$ whenever $\|u\| > h$, we get that
    \[
    |p_j(z + u) - q_j(u)| \le L\|u\|^\beta \le Lh^\beta.
    \]
    This implies that
    \[
    \left|\mathbb{E}[\widehat{p}_j(z)] - \int_{V_{z,h}} \frac{K_{z,h}(u)}{\mbox{Vol}(V_{z,h})}q_{z,j}(u)du\right| ~\le~ Lh^\beta\int_{V_{z,h}}\frac{|K_{z,h}(u)|}{\mbox{Vol}(V_{z,h})}du \le Lh^\beta\frac{\mathfrak{K}\mbox{Vol}(V_{z,h})}{\mbox{Vol}(V_{z,h})} = L\mathfrak{K}h^\beta.
    \]
    Note that assumption~\ref{eq:Holder-smooth-densities} implies $q_{z,j}(0) = p_j(z)$ and $q_{z,j}(u) = p_j(z) + \sum_{1\le \|\alpha\|_1 \le \llfloor s\rrfloor} a_{\alpha}\prod_{j=1}^d u_j^{\alpha_j}$ for some coefficients $a_{\alpha}$ which can depend on $z.$ Hence, the assumption on the kernel yields
    \[
    \int_{V_{z,h}} \frac{K_{z,h}(u)}{\mbox{Vol}(V_{z,h})}q_{z,j}(u)du = p_j(z)\frac{\int_{V_{z,h}} K_{z,h}(u)du}{\mbox{Vol}(V_{z,h})} = p_j(z).
    \]
    Therefore,
    \[
    \sup_{z\in\Omega}\,\left|\mathbb{E}[\widehat{p}_j(z)] - p_j(z)\right| ~\le~ L\mathfrak{K}h^\beta.
    \]
    The bound is uniform over all $z\in\Omega$.

    To bound the variance, note that
    \begin{align*}
        \mbox{Var}(\widehat{p}_j(z)) &= \frac{1}{n\mbox{Vol}^2(V_{z,h})}\mbox{Var}\left(K_{z,h}(X_{j1} - z)\right)\\
        &\le \frac{1}{n\mbox{Vol}^2(V_{z,h})}\int_{\Omega} K_{z,h}^2(x - z)p_j(z)dz\\
        &\le \frac{1}{n\mbox{Vol}^2(V_{z,h})}\int_{V_{z,h}} K_{z,h}^2(u)p_j(z + u)du\\
        &\le \frac{\mathfrak{K}^2\|p_j\|_{\infty}}{n\mbox{Vol}^2(V_{z,h})}\int_{V_{z,h}}du = \frac{\mathfrak{K}^2\|p_j\|_{\infty}}{n\mbox{Vol}(V_{z,h})}.
    \end{align*}
    The final part of the result follows from the fact that any set of full affine rank has positive Lebesgue density at each point and that every convex set with non-empty interior has a full affine rank. See, for example,~\url{https://math.stackexchange.com/questions/3491213/is-there-a-lower-bound-to-density-at-boundary-points-of-a-convex-set} and~\cite{Choudhary2025Uniform}.
    This completes the proof.
% \end{proof}
%%%%%%%%%%%%%%%%%%%%%%%%%%%%%%%%%%%%%%%
%%%%%%%%%%%%%%%%%%%%%%%%%%%%%%%%%%%%%%%
\subsection{Proof of Theorem~\ref{thm:asymp-normality-bounded-case}}\label{appsec:proof-asym-normality-bounded-case}
    Observe that
    \begin{align*}
        &\widehat{v}^{\mathrm{den}}(s, z) - v(s, z)\\ &\quad= \widehat{v}^{\mathrm{den}}(s, z)\left(1 - \frac{\widehat{p}_s(z)}{p_s(z)}\right) + \frac{\widehat{f}_s(z) - f_s(z)}{p_s(z)}\\
        &\quad= v(s, z)\left(1 - \frac{\widehat{p}_s(z)}{p_s(z)}\right) + \frac{\widehat{f}_s(z) - f_s(z)}{p_s(z)} + (\widehat{v}^{\mathrm{den}}(s, z) - v(s, z))\left(1 - \frac{\widehat{p}_s(z)}{p_s(z)}\right)\\
        &\quad= \int_{S_s(z)} \frac{\delta - v(s, z)}{p_s(z)}\{\widehat{p}_0(z - s\delta)\widehat{p}_1(z + (1-s)\delta) - p_0(z - s\delta)p_1(z + (1-s)\delta)\}d\delta\\
        &\qquad+ (\widehat{v}^{\mathrm{den}}(s, z) - v(s, z))\left(1 - \frac{\widehat{p}_s(z)}{p_s(z)}\right)\\
        &\quad= \int_{S_s(z)} \frac{\delta - v(s, z)}{p_s(z)}p_1(z + (1-s)\delta)\{\widehat{p}_0(z - s\delta) - p_0(z - s\delta)\}d\delta\\
        &\qquad+ \int_{S_s(z)} \frac{\delta - v(s, z)}{p_s(z)}p_0(z - s\delta)\{\widehat{p}_1(z + (1-s)\delta) - p_1(z + (1-s)\delta)\}\delta\\
        &\qquad+ \int_{S_s(z)} \frac{\delta - v(s, z)}{p_s(z)}\{\widehat{p}_0(z - s\delta) - p_0(z - s\delta)\}\{\widehat{p}_1(z + (1-s)\delta) - p_1(z + (1-s)\delta)\}d\delta\\
        &\qquad+ (\widehat{v}^{\mathrm{den}}(s, z) - v(s, z))\left(1 - \frac{\widehat{p}_s(z)}{p_s(z)}\right).
    \end{align*}
    With $z_s = \mathfrak{R}(s, z)$, define 
    \begin{align*}
    \widehat{\gamma}(s, x) &:= \int_{S_s(z_s)} \frac{\delta - v(s, z_s)}{p_s(z)}p_1(z_s + (1-s)\delta)\{\widehat{p}_0(z_s - s\delta) - p_0(z_s - s\delta)\}d\delta\\
        &\qquad+ \int_{S_s(z)} \frac{\delta - v(s, z)}{p_s(z_s)}p_0(z_s - s\delta)\{\widehat{p}_1(z_s + (1-s)\delta) - p_1(z_s + (1-s)\delta)\}\delta.
    \end{align*}
    Then we get 
    \[
    \sup_{s\in[0,1]}\left\|\widehat{v}^{\mathrm{den}}(s, \mathfrak{R}(s,x)) - v(s,\,\mathfrak{R}(s,x)) - \widehat{\gamma}(s,x)\right\| = O_p\left(\max_{s\in[0,1]}\|\Phi(1)(\Phi(s))^{-1}\|_{\mathrm{op}}\right)r_n^2.
    \]
    From Proposition~\ref{prop:Volterra}, $\max_{s\in[0,1]}\|\Phi(1)(\Phi(s))^{-1}\|_{\mathrm{op}} = O(1)$.
    
    Observe that 
   \begin{align*}
    \mathbb{E}[\widehat{\gamma}(s, x)] &:= \int_{S_s(z_s)} \frac{\delta - v(s, z_s)}{p_s(z_s)}p_1(z_s + (1-s)\delta)\{\mathbb{E}[\widehat{p}_0(z_s - s\delta)] - p_0(z_s - s\delta)\}d\delta\\
        &\qquad+ \int_{S_s(z_s)} \frac{\delta - v(s, z_s)}{p_s(z_s)}p_0(z_s - s\delta)\{\mathbb{E}[\widehat{p}_1(z_s + (1-s)\delta)] - p_1(z_s + (1-s)\delta)\}\delta.
    \end{align*}
    From Proposition~\ref{prop:density-estimator-bounded} (under assumption~\ref{eq:Holder-smooth-densities}), we get
    \begin{align*}
    \sup_{s\in[0,1]}\|\mathbb{E}[\widehat{\gamma}(s, x)]\| &\le 3\mbox{diam}(\Omega)L\mathfrak{K}h^\beta\int_{S_s(z)} \frac{\sum_{j=0}^1 p_j(z + (j-s)\delta)\mathbf{1}\{z + (1-j- s)\delta\in\Omega\}}{p_s(z)}\\
    &\le \frac{3L\mathfrak{K}\mbox{diam}(\Omega)}{\underline{\mathfrak{p}}}h^{\beta}.
    \end{align*}
    Therefore, defining
    \[
    \widetilde{E}^*(t, x) := \Phi(t)\int_0^t (\Phi(s))^{-1} (\widehat{\gamma}(s,x) - \mathbb{E}[\widehat{\gamma}(s, x)])ds, 
    \]
    we get
    \[
    \sup_{t\in[0,1]}\left\|\widehat{\mathfrak{R}}(t, x) - \mathfrak{R}(t, x) - \widetilde{E}^*(t, x)\right\| = o_p\left(\sup_{t\in[0,1]}\|\widetilde{E}^*(t, x)\|\right) + O_p(r_n^2 + h^{\beta}).
    \]
    We first find the rate of convergence of $\sup_{t\in[0,1]}\|\widetilde{E}^*(t, x)\|$. Note that
    % , by symmetrization and~\eqref{eq:supremum-density-rate} (or more appropriately, Theorem 1 of~\cite{einmahl2005uniform}),
    \begin{align*}
    \sup_{t\in[0,1]}\|\widetilde{E}^*(t, x)\| &\le C\sup_{t\in[0,1]}\,\left\|\int_0^t (\Phi(s))^{-1} (\widehat{\gamma}(s,x) - \mathbb{E}[\widehat{\gamma}(s, x)])ds\right\|\\
    &\le C\sup_{t\in[0,t^*]}\,\left\|\int_0^t (\Phi(s))^{-1} (\widehat{\gamma}(s,x) - \mathbb{E}[\widehat{\gamma}(s, x)])ds\right\|\\
    &\quad+ C\sup_{t\in[t^*, 1]}\,\left\|\int_0^t (\Phi(s))^{-1} (\widehat{\gamma}(s,x) - \mathbb{E}[\widehat{\gamma}(s, x)])ds\right\|\\
    &\le 2C\sup_{t\in[0,t^*]}\,\left\|\int_0^t (\Phi(s))^{-1} (\widehat{\gamma}(s,x) - \mathbb{E}[\widehat{\gamma}(s, x)])ds\right\|\\
    &\quad+ C\sup_{t\in[t^*, 1]}\,\left\|\int_{t^*}^t (\Phi(s))^{-1} (\widehat{\gamma}(s,x) - \mathbb{E}[\widehat{\gamma}(s, x)])ds\right\|%\\
    % &\le 2Ct^*\max_{j\in\{0,1\}}\,\|\widehat{p}_j - \mathbb{E}[\widehat{p}_j]\|_{\infty} + C\sup_{t\in[t^*,1]}\,\left\|\int_{t^*}^t (\Phi(s))^{-1} (\widehat{\gamma}(s,x) - \mathbb{E}[\widehat{\gamma}(s, x)])ds\right\|\\
    % &\le 2Ct^*\sqrt{\frac{\log(1/h)}{nh^d}}  + C\sup_{t\in[t^*,1]}\,\left\|\int_{t^*}^t (\Phi(s))^{-1} (\widehat{\gamma}(s,x) - \mathbb{E}[\widehat{\gamma}(s, x)])ds\right\|. 
    \end{align*}
    Set
    \begin{align*}
        Z_n(t) &:= \int_{0}^t (\Phi(s))^{-1} (\widehat{\gamma}(s,x) - \mathbb{E}[\widehat{\gamma}(s, x)])ds,\quad\mbox{for}\quad 0 \le t \le t^*,\\
        W_n(t) &:= \int_{t^*}^t (\Phi(s))^{-1} (\widehat{\gamma}(s,x) - \mathbb{E}[\widehat{\gamma}(s, x)])ds,\quad\mbox{for}\quad t^* \le t \le 1.        
    \end{align*}
    We apply the Kolmogorov continuity theorem (i.e., a simple version of generic chaining) to control the supremum of these terms. (The following proof follows Section 1.3 of~\cite{Talagrand2022}.) 
    We shall control $\sup_{t\in[t^*, 1]}\|W_n(t)\|$ first.
    For each $k \ge 1$, set $\mathcal{G}_k = \{\lfloor 2^kt\rfloor/2^{k}:\, t\in[0, 1)\}\cap[t^*, 1]$. Clearly, the cardinality of $\mathcal{G}_k$ is at most $2^{k}$. For each $t\in[t^*, 1]$, choose $\pi_k(t)\in\mathcal{G}_k$ such that $|t - \pi_k(t)| \le 2^{-k}$. It is easy to see that $|\pi_k(t) - \pi_{k-1}(t)| \le 3/2^k$ for all $k \ge 1$ and $t\in[t^*, 1]$. Set
    \[
    \mathcal{U}_k = \{(s, t)\in\mathcal{G}_k\times\mathcal{G}_k:\, |s - t| \le 3/2^k\}.
    \]
    Since, for any given $s\in\mathcal{G}_k$, there are at most six elements in $\mathcal{G}_k$ that are $3/2^k$ distant from $s$, we conclude that the cardinality of $\mathcal{U}_k$ is at most $3\times 2^{k+1}$. This implies that
    \begin{equation}\label{eq:joint-terms-cardinality}
        |\{(\pi_k(t),\,\pi_{k-1}(t)):\,t\in[t^*, 1]\}| \le 3\times 2^{k+1}.
    \end{equation}
    This holds because $\mathcal{G}_{k-1}\subset \mathcal{G}_k$.
    Observe now that, setting $\pi_0(t) = t^*$,
    \begin{equation}\label{eq:main-generic-chaining-bound}
        \begin{split}
        &\mathbb{E}\left[\sup_{t\in[t^*,1]}\|W_n(t)\|\right]\\ 
        &\le \sum_{k\ge1}\, \mathbb{E}\left[\sup_{t\in[t^*,1]}\|W_n(\pi_k(t)) - W_n(\pi_{k-1}(t))\|\right]\\
        &\le \sum_{k\ge1}\, \left(\mathbb{E}\left[\sup_{(t_1, t_2)\in\mathcal{U}_k}\|W_n(t_2) - W_n(t_1)\|^2\right]\right)^{1/2}\\
        &\le d^{1/2}\sum_{k\ge1}\, (3\times 2^{k+1})^{1/2}\max_{1\le j\le d}\left(\sum_{(t_1, t_2)\in\mathcal{U}_k}\mathbb{E}\left[|e_j^{\top}(W_n(t_2) - W_n(t_1)|^2\right]\right)^{1/2}.
        \end{split}
    \end{equation}
    (Here we used the fact that $\mathbb{E}[\max_{1\le j\le M}|W_j|^2] \le \mathbb{E}[\sum_jW_j^2].$)
    For $t^* \le t_1 \le t_2 \le 1$, $e_j^{\top}(W_n(t_2) - W_n(t_1))$ is a mean-zero random variable and can be written as
    \begin{align*}
        &e_j^{\top}(W_n(t_2) - W_n(t_1))\\
        &= \frac{1}{n}\sum_{i=1}^n \int_{t_1}^{t_2} \{e_j^{\top}U_{0i}(s, x) + e_j^{\top}U_{1i}(s, x) - \mathbb{E}[e_j^{\top}U_{0i}(s, x) + e_j^{\top}U_{1i}(s, x)]\}ds,
    \end{align*}
    where, with $z_s = \mathfrak{R}(s, x)$,
    \begin{align*}
        U_{0i}(s, x)
        &= (\Phi(s))^{-1}\int_{S_{s}(z_s)} \frac{\delta - v(s, z_s)}{p_s(z_s)}p_1(z_s + (1- s)\delta)\frac{K_{z_s-s\delta,h}(X_{0i} - (z_s - s\delta))}{\mbox{Vol}(V_{z_s-s\delta,h})}d\delta\\
        U_{1i}(s, x) &= (\Phi(s))^{-1}\int_{S_{s}(z_s)} \frac{\delta - v(s, z_s)}{p_s(z_s)}p_0(z_s - s\delta)\frac{K_{z_s+(1-s)\delta,h}(X_{1i} - (z_s + (1-s)\delta))}{\mbox{Vol}(V_{z_s-s\delta,h})}d\delta.
    \end{align*}
    Given the independence of $(X_{01}, \ldots, X_{0n})$ and $(X_{11}, \ldots, X_{1n})$, and the similar structure of $U_{0i}, U_{1i}$, it suffices to control the second moment of $\int_{t_1}^{t_2} e_j^{\top}U_{0i}(s, x)ds$. For this, we use
    \begin{align*}
    &\mathbb{E}\left|\int_{t_1}^{t_2} e_j^{\top}U_{0i}(s, x)ds\right|^2\\ 
    &= \int_{\Omega} \left(\int_{t_1}^{t_2}\int_{S_{s}(z_s)} \frac{e_j^{\top}(\delta - v(s, z_s))}{p_s(z_s)}p_1(z_s + (1- s)\delta)\frac{K_{z_s-s\delta,h}(z - (z_s - s\delta))}{\mbox{Vol}(V_{z_s-s\delta,h})}d\delta ds\right)^2p_0(z)dz\\
    &\le \frac{C}{h^{2d}}\int_{\Omega}\left(\int_{t_1}^{t_2}\int_{S_{s}(z_s)}\mathbf{1}\{\|z - (z_s - s\delta)\| \le h\}d\delta ds\right)^2p_0(z)dz,
    \end{align*}
    where the last inequality follows several fact (a) $\|p_1\|_{\infty} \le \overline{\mathfrak{p}},$ (b) $|e_j^{\top}(\delta - v(s, z))| \le 3\mbox{diam}(\Omega)$, (c) there exists $\varepsilon_x > 0$ such that $\mathfrak{R}(s, x)\in\Omega^{-\varepsilon_x}$ for almost all $x\in\Omega^\circ$, (d) $p_s(z)$ is bounded away from for $z\in\Omega^{-\varepsilon}$, (e) $|K_{z,h}(u)| \le \mathfrak{K}\mathbf{1}\{u\in V_{z,h}\}$, and (f) $\mbox{Vol}(V_{z,h}) \ge Ch^d$ for all $z\in\Omega$. 
    Note that
    \begin{align*}
        \int_{S_s(z_s)} \mathbf{1}\{\|z - (z_s - s\delta)\| \le h\}d\delta &= \mbox{Vol}\left(\mathcal{B}\left(\frac{z_s - z}{s},\,\frac{h}{s}\right)\cap\left(\frac{z_s - \Omega}{s}\right)\cap\left(\frac{\Omega - z_s}{1 - s}\right)\right).
    \end{align*}
    Because $s \ge t_1 \ge h$, we have
    \begin{equation}\label{eq:volume-bound-large-s}
    \mbox{Vol}\left(\mathcal{B}\left(\frac{z_s - z}{s},\,\frac{h}{s}\right)\cap\left(\frac{z_s - \Omega}{s}\right)\cap\left(\frac{\Omega - z_s}{1 - s}\right)\right) \le \mbox{Vol}\left(\mathcal{B}\left(\frac{z - z_s}{s},\,\frac{h}{s}\right)\right) \le C(h/s)^d.
    \end{equation}
    This implies that
    \[
    \sup_{z\in\Omega}\int_{t_1}^{t_2}\int_{S_{s}(z_s)}\mathbf{1}\{\|z - (z_s - s\delta)\| \le h\}d\delta ds \le Ch^d\int_{t_1}^{t_2} \frac{ds}{s^d} \le C(t_2 - t_1)(h/t_1)^d. 
    \]
    Therefore, 
    \begin{align*}
        &\mathbb{E}\left|\int_{t_1}^{t_2} e_j^{\top}U_{0i}(s, x)ds\right|^2\\
        &\le \frac{C'}{h^{2d}}\frac{h^d(t_2 - t_1)}{t_1^d}\int_{\Omega}\int_{t_1}^{t_2} \int_{S_s(z_s)} \mathbf{1}\{\|z - (z_s - s\delta)\| \le h\}p_0(z)d\delta ds dz\\
        &= \frac{C'(t_2 - t_1)}{h^dt_1^d}\int_{t_1}^{t_2}\int_{S_s(z_s)} \mathbb{P}(X_0 \in \mathcal{B}(z_s - s\delta, h))d\delta ds\\
        &\le \frac{C^{''}(t_2 - t_1)}{t_1^d} \int_{t_1}^{t_2}ds = C^{''}\frac{(t_2 - t_1)^2}{t_1^d}.
    \end{align*}
    Hence, we conclude that for $d \ge 2$
    \begin{equation}\label{eq:bound-increment-W}
    \begin{split}
    \max_{1\le j\le d}\left(\sum_{(t_1, t_2)\in\mathcal{U}_k}\mathbb{E}\left[|e_j^{\top}(W_n(t_2) - W_n(t_1)|^2\right]\right)^{1/2} &\le C\left(\sum_{t_1\in\mathcal{G}_k}\sum_{t_2\in\mathcal{G}_k: |t_2 - t_1| \le 3/2^k} \frac{(t_2 - t_1)^2}{t_1^d}\right)^{1/2}\\
    &\le \frac{C}{2^k}\left(\sum_{t_1\in\mathcal{G}_k} \frac{1}{t_1^d}\right)^{1/2}\\
    &\le \frac{C}{2^k}\left(\sum_{i=\lceil 2^kt^*\rceil}^{2^k} \left(\frac{2^k}{i}\right)^{d}\right)^{1/2}\\
    &\le \frac{C}{2^k}\left(2^{kd}\int_{2^kt^*}^{2^k} \frac{1}{s^d}ds\right)^{1/2}\\
    &\le \frac{C}{2^k}\left(2^{kd}\frac{(2^kt^*)^{-d+1}}{d-1}\right)^{1/2} \le \frac{C}{2^{k/2}(t^*)^{(d-1)/2}}.
    \end{split}
    \end{equation}
    Substituting this in~\eqref{eq:main-generic-chaining-bound}, we obtain
    \[
    \mathbb{E}\left[\sup_{t\in[t^*,1]}\|W_n(t)\|\right] \le C_1\frac{1}{\sqrt{n(t^*)^{d-1}}}\sum_{k\ge 1} \frac{2^{k/2}}{2^k} \le \frac{C_2}{\sqrt{n(t^*)^{d-1}}}.
    \]
    If $d = 1$, then 
    \[
    \sum_{i=\lceil 2^kt^*\rceil}^{2^k} (2^k/i)^d \le 2^{k}\int_{2^kt^*}^{2^k} ds/s = 2^k\log(1/t^*).
    \]
    Hence, we conclude
    \[
    \mathbb{E}\left[\sup_{t\in[t^*,1]}\|W_n(t)\|\right] \le C_1\frac{1}{\sqrt{n(t^*)^{d-1}}}\sum_{k\ge 1} \frac{2^{k/2}}{2^k} \le \frac{C}{\sqrt{n}}\times\begin{cases}\sqrt{\log(1/t^*)}, &\mbox{if }d = 1,\\
    (t^*)^{-(d-1)/2}, &\mbox{if }d \ge 2.\end{cases}
    \]

    The derivation of the bound for $\sup_{t\in[0,t^*]}\|Z_n(t)\|$ follows the same structure and reuse $\mathcal{G}_k, \mathcal{U}_k$ for similar sets. Define
    \begin{align*}
        \mathcal{G}_k &:= \{\lfloor 2^kt\rfloor/2^k:\, t\in[0,t^*)\},\quad
        \mathcal{U}_k := \{(s, t)\in\mathcal{G}_k\times\mathcal{G}_K:\, |s - t| \le 3/2^k\}.
    \end{align*}
    Clearly, the cardinality of $\mathcal{G}_k$ is at most $\lceil 2^kt^*\rceil$ and that of $\mathcal{U}_k$ is at most $6\lceil 2^kt^*\rceil$. Note that if $k \le \log_2(1/t^*)$, then $\mathcal{G}_k = \{0\}$. Using the same decomposition as~\eqref{eq:main-generic-chaining-bound}, it suffices to control $\mathbb{E}[(e_j^{\top}(Z_n(t_2) - Z_n(t_1)))^2]$ for any $(t_1, t_2)\in\mathcal{U}_k$. As before, we get
    \[
    \mathbb{E}[(e_j^{\top}(Z_n(t_2) - Z_n(t_1)))^2] \le \frac{C}{nh^{2d}}\int_{\Omega}\left(\int_{t_1}^{t_2}\int_{S_s(z_s)} \mathbf{1}\{\|z - (z_s - s\delta)\| \le h\}d\delta ds\right)^2p_0(z)dz
    \]
    However, instead of~\eqref{eq:volume-bound-large-s}, we use
    \[
    \mbox{Vol}\left(\mathcal{B}\left(\frac{z_s - z}{s},\,\frac{h}{s}\right)\cap\left(\frac{z_s - \Omega}{s}\right)\cap\left(\frac{\Omega - z_s}{1 - s}\right)\right) \le \mbox{Vol}\left(\left(\frac{z_s - \Omega}{s}\right)\cap\left(\frac{\Omega - z_s}{1 - s}\right)\right) \le C,
    \]
    for some constant independent of $s$ and depending only on $d$ and the diameter of $\Omega$. This implies
    \begin{align*}
    \mathbb{E}[(e_j^{\top}(Z_n(t_2) - Z_n(t_1)))^2] &\le \frac{C(t_2 - t_1)}{nh^{2d}}\int_{t_1}^{t_2}\int_{S_s(z_s)}\int_{\Omega} \mathbf{1}\{\|z - (z_s - s\delta)\| \le h\}p_0(z)dz\\
    &\le \frac{C(t_2 - t_1)^2}{nh^d}.
    \end{align*}
    Note that in decomposition~\eqref{eq:main-generic-chaining-bound}, the summands for $k \le \log_2(1/t^*)$ can be ignored because $\mathcal{G}_k = \{0\}$ implies $\pi_k(t) = \pi_{k-1}(t)$. Therefore,
    \begin{align*}
    \mathbb{E}\left[\sup_{t\in[0, t^*]}\|Z_n(t)\|\right] &\le C\sum_{k > \log_2(1/t^*)}\, (\lceil 2^kt^*\rceil)^{1/2}\max_{1\le j\le d}\sup_{(t_1, t_2)\in\mathcal{U}_k}\left(\mathbb{E}\left[|e_j^{\top}(Z_n(t_2) - Z_n(t_1)|^2\right]\right)^{1/2}\\
    &\le C\sum_{k\ge \log_2(1/t^*)} \frac{2^{k/2}(t^*)^{1/2}}{\sqrt{nh^d}}\left(\frac{1}{2^k}\right) \le C\frac{t^*}{\sqrt{nh^d}}.
    \end{align*}
    Hence,
    \[
    \sup_{t\in [0,1]}\|\widetilde{E}^*(t, x)\| = O_p\left(\frac{t^*}{\sqrt{nh^d}}\right) + n^{-1/2}\begin{cases}\sqrt{\log(1/t^*)}, &\mbox{if }d = 1,\\
    (t^*)^{-(d-1)/2}, &\mbox{if }d\ge 2.
    \end{cases}
    \]
    For $d = 1$, choosing $t^* = \sqrt{h\log(1/h)}$ and for $d = 1$, choosing $t^* = h^{d/(d + 1)}$, we get
    \[
    \sup_{t\in [0,1]}\|\widetilde{E}^*(t, x)\| = O_p(n^{-1/2})\times\begin{cases}\sqrt{\log(1/h)}, &\mbox{if }d = 1,\\
    \sqrt{h^{-d(d-1)/(d+1)}}, &\mbox{if }d\ge 2.
    \end{cases}
    \]
    
    % We can write the square of the double integral as four integrals to get
    % {\small
    % \begin{align*}
    %     &\mathbb{E}\left|\int_{t_1}^{t_2} e_j^{\top}U_{0i}(s, x)ds\right|^2\\ 
    %     &\le C\int_{\Omega}\int_{t_1}^{t_2}\int_{t_1}^{t_2}\int_{S_{s}(z_s)}\int_{S_{s'}(z_{s'})}\frac{|K_{z_s-s\delta,h}(z - (z_s - s\delta))|}{\mbox{Vol}(V_{z_s,h})}\frac{|K_{z_{s'}-{s'}\delta',h}(z - (z_{s'} - {s'}\delta'))|}{\mbox{Vol}(V_{z_{s'},h})}p_0(z)d\delta d\delta' ds'dsdz
    % \end{align*}}
    % Using the change of variable $z\mapsto u = z - (z_s - s\delta)$ (and all other variables unchanged), we get
    % {\small
    % \begin{align*}
    %     &\mathbb{E}\left|\int_{t_1}^{t_2} e_j^{\top}U_{0i}(s, x)ds\right|^2\\ 
    %     &\le C\int_{\Omega}\int_{t_1}^{t_2}\int_{t_1}^{t_2}\int_{S_{s}(z_s)}\int_{S_{s'}(z_{s'})}\frac{|K_{z_s-s\delta,h}(u)|}{\mbox{Vol}(V_{z_s,h})}\frac{|K_{z_{s'}-{s'}\delta',h}(u + z_s - s\delta - (z_{s'} - {s'}\delta'))|}{\mbox{Vol}(V_{z_{s'},h})}p_0(u + z_s - s\delta)d\delta d\delta' ds'dsdu.
    % \end{align*}}
    % Additionally, considering the change of variable $s'\mapsto \omega = s' - s$ (and all other variables unchanged), we get
    % {\small
    % \begin{align*}
    %     &\mathbb{E}\left|\int_{t_1}^{t_2} e_j^{\top}U_{0i}(s, x)ds\right|^2\\ 
    %     &\le C\int_{\Omega}\int_{t_1}^{t_2}\int_{t_1}^{t_2}\int_{S_{s}(z_s)}\int_{S_{s'}(z_{s'})}\frac{|K_{z_s-s\delta,h}(u)|}{\mbox{Vol}(V_{z_s,h})}\frac{|K_{z_{s + \omega}-(s + \omega)\delta',h}(u + (z_s - s\delta) - (z_{s+\omega} - (s+\omega)\delta'))|}{\mbox{Vol}(V_{z_{s + \omega},h})}p_0(u + z_s - s\delta)d\delta d\delta' d\omega dsdu.
    % \end{align*}}
\subsection{Proof of Proposition~\ref{prop:well-defined-strictly-convex-set}}\label{appsubsec:well-defined-strictly-convex}
    The continuity of $z\mapsto v(0, z)$ and $z\mapsto v(1, z)$ is obvious, because they are linear functions. 
    \paragraph{Fix $t\in(0, 1/2]$ and $z\in\Omega^\circ$.} (The case of $t\in[1/2,1)$ is similar.) There exists $\varepsilon > 0$,
    \begin{equation}\label{eq:bounded-away-from-zero-p_t}
    p_t(z') = \int_{S_t(z')} p_0(z' - t\delta)p_1(z' + (1-t)\delta)d\delta > 0\quad\mbox{for all}\quad z'\in B(z, \varepsilon)\subseteq\Omega.
    \end{equation}
    This follows from the assumption that the densities are bounded away from zero on $\Omega$. This calculation allows us to assume that $z, z'\in\Omega^\circ$ and condition~\eqref{eq:inclusions-away-from-boundary} holds.
    Moreover, for any $t\in[0, 1],$ $S_t(z)$ is a compact subset of $\mathbb{R}^d$.
    Hence, it suffices to prove that for a continuous function $g:\mathbb{R}^d\to\mathbb{R}_+$ (bounded on bounded sets),
    \begin{equation}\label{eq:continuity-of-h}
    h(z') = \int_{S_t(z')} g(\delta)p_0(z' - t\delta)p_1(z' + (1-t)\delta)d\delta ~\to~ h(z)\quad\mbox{as}\quad\|z' - z\|\to0.
    \end{equation}
    For notational convenience, set $e(t, z, \delta) = g(\delta)p_0(z - t\delta)p_1(z + (1-t)\delta),$ which is non-negative and positive if and only if $\delta\in S_t(z)$.
    Note that
    \begin{align*}
        h(z') - h(z) &= \int_{S_t(z')\setminus S_t(z)} e(t, z', \delta)d\delta - \int_{S_t(z)\setminus S_t(z')} e(t, z, \delta)d\delta\\
        &\quad + \int_{S_t(z')\cap S_t(z)} \{e(t, z', \delta) - e(t, z, \delta)\}d\delta.
    \end{align*}
    Because $S_t(z)$ is bounded for all $z$ and $g$ is bounded on bounded sets, we get
    \[
    \left|\int_{S_t(z')\setminus S_t(z)} e(t, z', \delta)d\delta\right| \le \|g\|_{\infty}\|p_0\|_{\infty}\|p_1\|_{\infty}\mbox{Leb}(S_t(z')\setminus S_t(z)).
    \]
    From Proposition~\ref{prop:properties-of-S_t}(4), we get for $\|z - z'\| = \gamma > 0$ small enough, 
    \[
    \mbox{Leb}(S_t(z')\setminus S_t(z)) \le \mbox{Leb}((S_t(z))^{2\gamma/(t(1-t))}\setminus S_t(z)) \to 0\quad\mbox{as}\quad \gamma\to0.
    \]
    Similarly, we obtain that
    \[
    \left|\int_{S_t(z)\setminus S_t(z')} e(t, z, \delta)d\delta\right| \to 0\quad\mbox{as}\quad \gamma \to 0.
    \]
    Finally, setting
    \[
    \omega_j(\varepsilon) = \sup_{\|x - x'\| \le \varepsilon}|p_j(x) - p_j(x')|\quad\mbox{for}\quad j = 0, 1,
    \]
    note that
    \[
    \left|\int_{S_t(z')\cap S_t(z)} \{e(t, z', \delta) - e(t, z, \delta)\}d\delta\right| \le \|g\|_{\infty}\|p_1\|_{\infty}\omega_0(\gamma) + \|g\|_{\infty}\|p_0\|_{\infty}\omega_1(\gamma),
    \]
    which converges to zero as $\gamma\to0.$ (Recall that assumption~\ref{eq:bounded-away-densities} implies continuity of $p_0, p_1$ on the compact $\Omega$, and hence they are uniformly continuous on $\Omega$ which yields $\omega_j(\gamma)\to 0$ as $\gamma\to0$.)
    Therefore,~\eqref{eq:continuity-of-h} holds, which by~\eqref{eq:bounded-away-from-zero-p_t} implies continuity of $z\mapsto v(t, z)$ on $\Omega^\circ.$
    \paragraph{Fix $t\in(0, 1)$ and $z\in\mathrm{SC}(\Omega)$.} Proposition~\ref{prop:properties-of-S_t}(5) implies that $S_t(z) = \{0\}$. Hence, $p_t(z) = 0$ because $p_0(z - t\delta)p_1(z + (1-t)\delta) = 0$ almost everywhere $\delta\in\mathbb{R}^d$. To show that $v(t, z) = 0$ is the continuous extension, we show that for any $z'\in\Omega$ that is ``close'' to $z$, $S_t(z')$ is a small set (in the sense of its diameter). 
    
    For any $z'\in\Omega$, we shall show that
    \begin{equation}\label{eq:upper-bound-in-delta'}
    \sup_{\delta'\in S_t(z')}\|\delta'\| ~\le~ \frac{\|z' - z\| + \mathfrak{m}_z^{-1}(\|z - z'\|/2;\,\Omega)}{\min\{t, 1 - t\}}.
    \end{equation}
    If $\|\delta'\| \le \|z' - z\|/\min\{t, 1 - t\}$, then nothing to prove. Assume that $\|\delta'\| \ge \|z' - z\|/\min\{t, 1 - t\}$.
    Applying the definition of modulus of uniform convexity~\eqref{eq:modulus-at-a-point}, we obtain
    \begin{align*}
    \mbox{dist}\left(\frac{z + z' -t\delta'}{2},\,\partial\Omega\right) &\ge \mathfrak{m}_z\left(\|z' - t\delta' - z\|; \Omega\right)\\
    \mbox{dist}\left(\frac{z + z' + (1-t)\delta'}{2},\,\partial\Omega\right) &\ge \mathfrak{m}_z\left(\|z' + (1-t)\delta' - z\|; \Omega\right).
    \end{align*}
    From Lemma~\ref{lem:concavity-of-distance-to-boundary}, these inequalities imply
    \begin{align*}
    \left\|\frac{z + z'}{2} - z\right\| &\ge \mbox{dist}\left(\frac{z + z'}{2},\,\partial\Omega\right)\\
    &\ge (1-t)\mbox{dist}\left(\frac{z + z' - t\delta'}{2},\,\partial\Omega\right) + t\mbox{dist}\left(\frac{z + z' + (1-t)\delta'}{2},\,\partial\Omega\right)\\
    &\ge (1-t)\mathfrak{m}_z(\|z' - t\delta' - z\|;\, \Omega) + t\mathfrak{m}_z(\|z' + (1-t)\delta' - z\|;\, \Omega)\\
    &\ge \mathfrak{m}_z(\min\{t, 1-t\}\|\delta'\| - \|z' - z\|;\, \Omega),
    \end{align*}
    where the last inequality follows from the monotonicity of $\varepsilon\mapsto \mathfrak{m}_z(\varepsilon,\,\Omega)$. Therefore, 
    \[
    \min\{t, 1 - t\}\|\delta'\| - \|z' - z\| \le \mathfrak{m}_z^{-1}\left(\frac{\|z - z'\|}{2};\,\Omega\right).
    \]
    This implies~\eqref{eq:upper-bound-in-delta'}, which in turn implies that the diameter of $S_t(z')$ converges to zero. Hence, for any $z'\in\Omega^\circ$ with $\|z' - z\| \le \gamma$, we get
    \[
    \|v(t, z')\| \le \frac{2\max\{\gamma,\mathfrak{m}_z^{-1}(\gamma/2; \Omega)\}}{t(1-t)} \to 0,\quad\mbox{as}\quad \gamma\to0.
    \]
    (Recall that $v(t, z')$ is the expected value of a random vector supported on $S_t(z')$.)
    Hence, defining $v(t, z) = 0$ for $t\in(0, 1)$ and $z\in\partial\Omega$ is a point of strict convexity of $\Omega$. Even if we consider $z'\in\partial\Omega$ with $\|z' - z\|\to0$, Lemma~\ref{lem:eventually-strongly-convex} implies that $z'$ eventually is a point of strict convexity and so, zero is a continuous extension of the velocity field.

\section{Boundary Effects}

The boundary of $\Omega$ has a strong effect on $v_t(z)$.
When $\Omega$ has boundaries,
the velocity $v_t(z)$ is not differentiable
no matter how smooth the densities are.
Furthermore, the Lipschitz constant that explodes
at $t\to 0$ and $t\to 1$.
We note that boundaries
are known to play an important role
for optimal transport maps as well
(Caffarelli 1992, 1996; De Philipps and Figalli 2014)
so it is not surprising that they do so here as well.

Let $\beta$ be a positive integer.
We will say that a function $g$ is $\beta$-smooth at $x$ if,
for all $y$ in a neighborhood of $x$, and some $L < \infty$,
$$
|D^s g(x) - D^s g(y)| \leq L ||x-y||\ \ \ \mathrm{for\ all\ }|s|\leq \beta-1.
$$
Thus, for $u$ close to $x$,
$$
|g(u)-g_{x,\beta}(u)|\leq L||u-x||^\beta
$$
where
$$
g_{x,\beta}(u) = \sum_{|s| < \beta} \frac{(u-x)^s}{s!}.
$$
We will use the following facts.

\bigskip

{\bf Fact 1}: If
$f(z,\delta)$ and $g(z,\delta)$ are $\beta$-smooth at $(z,\delta)$,
$a(z)$ and $b(z)$ are linear
and
$\int_{a_1(z)}^{b_1(z)} \cdots \int_{a_d(z)}^{b_d(z)} g(z,\delta) d\delta_1\ \cdots d\delta_d > 0$,
then
$$
\frac{\int_{a_1(z)}^{b_1(z)} \cdots \int_{a_d(z)}^{b_d(z)} f(z,\delta) d\delta_1\ \cdots d\delta_d}
{\int_{a_1(z)}^{b_1(z)} \cdots \int_{a_d(z)}^{b_d(z)} g(z,\delta) d\delta_1\ \cdots d\delta_d}
$$
is $\beta$-smooth at $z$.
This follows since, repeated application of the Leibniz rule shows that
the derivates of the numerator and denominator
are linear combinations of
$D^s f(z,\delta)$,
$D^s g(z,\delta)$,
$D^s a_j(z)$ and $D^s b_j(z)$.

\bigskip

{\bf Fact 2}:
Suppose that $a(z)$ is linear for $u < z$
and linear for $u>z$ and that
the left and right derivatives of $a(u)$ at $z$ are not equal.
Then
$m(z)=\int_{a(z)}^{b(z)} f(z,\delta) d\delta/ \int_{a(z)}^{b(z)} g(z,\delta) d\delta $ is
not differentiable at $z$.
Similarly for $b(z)$.
This follows by
comparing
$m(z+\epsilon)-m(z)$ and
$m(z)-m(z-\epsilon)$.

\bigskip

We assume throught this section that
$p_0$ and $p_1$ are $\beta$-smooth at all $z$.
We will consider two examples.
In each case we show the following:

(P1) For
$t\in (0,1/2) \bigcap (1/2,1)$,
$v_t(z)$ is
not differentiable on
a set $\Lambda_t$ of measure 0.
Specifically,
\begin{equation}\label{eq::Lambda}
\Lambda_t = \Bigl\{z\in\Omega:\ z = (1-t)b_1 + t b_2,\ \ \mathrm{for\ some\ }b_1,b_2\in\partial \Omega\Bigr\}.
\end{equation}

(P2) $v_t(z)$ is $\beta$ smooth on
$\Lambda_t^c$.

(P3) There is a positive constant $c$ such that
$$
C_t \equiv \sup_{z_1,z_2}\frac{|v_t(z_2)-v_t(z_1)|}{|z_2-z_1|} \geq \frac{c}{t(1-t)}.
$$

\bigskip

{\bf Case 1: Interval.}
$\Omega = [\ell,r]$.
Then
$$
v_t(z) =
\frac{\int_{a_t(z)}^{b_t(z)} \delta p_0(z-t\delta) p_1(z+(1-t)\delta)d\delta}
{\int_{a_t(z)}^{b_t(z)}  p_0(z-t\delta) p_1(z+(1-t)\delta)d\delta}
$$
where
$$
a_t(z) = \max
\Biggl\{\frac{z-r}{t}, \frac{\ell-z}{1-t} \Biggr\},\ \ \ 
b_t(z) = \min
\Biggl\{\frac{z-\ell}{t}, \frac{r-z}{1-t} \Biggr\}.
$$
If $t < 1/2$ let
$z_0 = tr + (1-t)\ell$ and $z_1 = t\ell + (1-t)r$.
If $t>1/2$ let
$z_1 = tr + (1-t)\ell$ and $z_0 = t\ell + (1-t)r$.
We'll focus on the case where $t < 1/2$.
Then
$$
S_t(z) = [a_t(z),b_t(z)] =
\begin{cases}
\left[\frac{\ell-z}{1-t},\ \   \frac{z-\ell}{t}\right] & \ell \leq z \leq z_0\\
\medskip
\left[\frac{\ell-z}{1-t},\ \   \frac{r-z}{1-t}\right] & z_0 \leq z \leq z_1\\
\medskip
\left[\frac{z-r}{t},\ \   \frac{r-z}{1-t}\right] & z_1 \leq z \leq r.
\end{cases}
$$
Now $a_t(z)$ and $b_t(z)$ are linear
on the open intervals
$(\ell,z_0)$, $(z_0,z_1)$ and $(z_1,r)$
so $v_t(z)$ is $\beta$-smooth there.
On the other hand, $a_t(z)$ is not differentiable at $z_1$ and
$b_z(t)$ is not differentiable at $z_0$ and so
$v_t(z)$ is not differentiable on $\{z_0,z_1\}$.
So $v_t(z)$ is not differentiable on
$\Lambda_t(z) = \{z_0,z_1\}$ and is $\beta$-smooth otherwise.
Now
$$
C_t \geq \lim_{z\to\ell}\frac{v_t(z_0)-v_t(z)}{z_0 - z}.
$$
For $z < z_0$,
and some $\delta',\delta'' \in (0,z_0)$,
\begin{align*}
v_t(z) &=
\frac{\int_{ \frac{\ell-z}{1-t} }^{\frac{z-\ell}{t}} \delta r_t(z,\delta) d\delta}
{\int_{ \frac{\ell-z}{1-t} }^{\frac{z-\ell}{t}}  r_t(z,\delta) d\delta}=
\frac{r_t(z,\delta') \int_{ \frac{\ell-z}{1-t} }^{\frac{z-\ell}{t}} \delta  d\delta}
{r_t(z,\delta'') \int_{ \frac{\ell-z}{1-t} }^{\frac{z-\ell}{t}}   d\delta}\\
&=
\frac{r_t(z,\delta')}{r_t(z,\delta'')}
\frac{(z-\ell)(1-2t)}{2 t (1-t)} \to 0
\end{align*}
as $z\to \ell$.
On the other hand,
$$
v_t(z_0) =
\frac{\int_{-t(r-\ell)}^{r-\ell} \delta r_t(z,\delta)d\delta}
{\int_{-t(r-\ell)}^{r-\ell}r_t(z,\delta)d\delta} =
c+ O(t)
$$
where
$$
c =
\frac{\int_{0}^{r-\ell} \delta r_t(z,\delta)d\delta}
{\int_{0}^{r-\ell}r_t(z,\delta)d\delta} >0.
$$
So
$$
\lim_{z\to \ell}
\frac{v_t(z_0)-v_t(z)}{z_0 - \ell} = 
\frac{c + O(t)}{t(r-z)} \asymp \frac{1}{t}.
$$
Similar arguments hold for the $z\to r$ and for $t > 1/2$.

\bigskip

{\bf Case 2. Rectangle.}
$\Omega = [\ell_1,r_1] \times \cdots \times [\ell_d,r_d]$.
Without loss of generality, assume that $0 < t < 1/2$.
Now
$$
S_t(z) = [a_{t1}(z),b_{t1}(z)] \times \cdots \times [a_{td}(z),b_{td}(z)]
$$
where
$$
a_{tj}(z) = \max
\Biggl\{\frac{z_j-r_j}{t}, \frac{\ell_j-z}{1-t} \Biggr\},\ \ \ 
b_{tj}(z) = \min
\Biggl\{\frac{z-\ell_j}{t}, \frac{r_j-z}{1-t} \Biggr\}.
$$
Let
$z_{0j} = tr_j + (1-t)\ell_j$ and $z_{1j} = t\ell_j + (1-t)r_j$.
Then
$\ell_j \leq z_{0j} \leq z_{1j}\leq r_j$.
The interval $[\ell_j,r_j]$
is partitioned into three intervals
$[\ell_j,z_{0j}]$, $[z_{0j},z_{1j}]$, $[z_{1j},r_j]$.
Then
$$
[a_{tj}(z),b_{tj}(z)] =
\begin{cases}
\left[\frac{\ell_j-z_j}{1-t},\ \   \frac{z_j-\ell_j}{t}\right] & \ell_j \leq z_j \leq z_{0j}\\
\medskip
\left[\frac{\ell_j-z_j}{1-t},\ \   \frac{r_j-z_j}{1-t}\right] & z_{0j} \leq z_j \leq z_{1j}\\
\medskip
\left[\frac{z_j-r_j}{t},\ \   \frac{r_j-z_j}{1-t}\right] & z_{1j} \leq z_j \leq r_j.
\end{cases}
$$
Hence we can partition $S_t(z)$
into sets of the form
$C_1 \times \cdots \times C_d$
where each $C_j$ is one of
$[\ell_j,z_{0j}]$, $[z_{0j},z_{1j}]$, $[z_{1j},r_j]$.
There are $N=3^d$ such sets
which we denote by $R_1,\ldots, R_N$.
Then 
$a_{jt}(z)$ and $b_{jt}(z)$ are linear functions
over the interiors of these sets
but have a change of slope at the boundaries.
So $v_t(z)$ is smooth
over the interiors of $R_1,\ldots, R_N$.
See Figure \ref{fig::Rectangle_Plot}.

Let $\Lambda_t(z)$ be the boundaries of these sets.
Note that this is precisely the set defined by (\ref{eq::Lambda}).
Fix a point $z$ on the boundary of $R_k$, say.
There exists at least one $j$ such that
$z_j = z_{0j}$ or 
$z_j = z_{1j}$.
Rewrite $S_t(z)$ as
$S_t(z) = [a_{tj}(z),b_{tj}(z)] \times B$
where
$B = \prod_{k \neq j }[a_{tk}(z),b_{tk}(z)]$.
Hence,
$$
\int_{S_t(z)} \delta r_t(z,\delta) d\delta =
\int_{a_{tj}(z)}^{b_{tj}(z)} M_t(z_j,\delta)
$$
where
$M_t(z_j,\delta) = \int_B \delta r_t(z,\delta) d\delta_{(-j)}$
where $\delta_{(-j)} = (\delta_1,\ldots, \delta_{j-1},\delta_{j+1},\ldots, \delta_d)$.
By Fact 2,
$\int_{a_{tj}(z)}^{b_{tj}(z)} M_t(z_j,\delta)$ is not differentiable.
A similar argument applies to the denominator
$\int_{S_t(z)}r_t(z,\delta) d\delta$.

\begin{figure}
\begin{center}
\includegraphics[scale=.5]{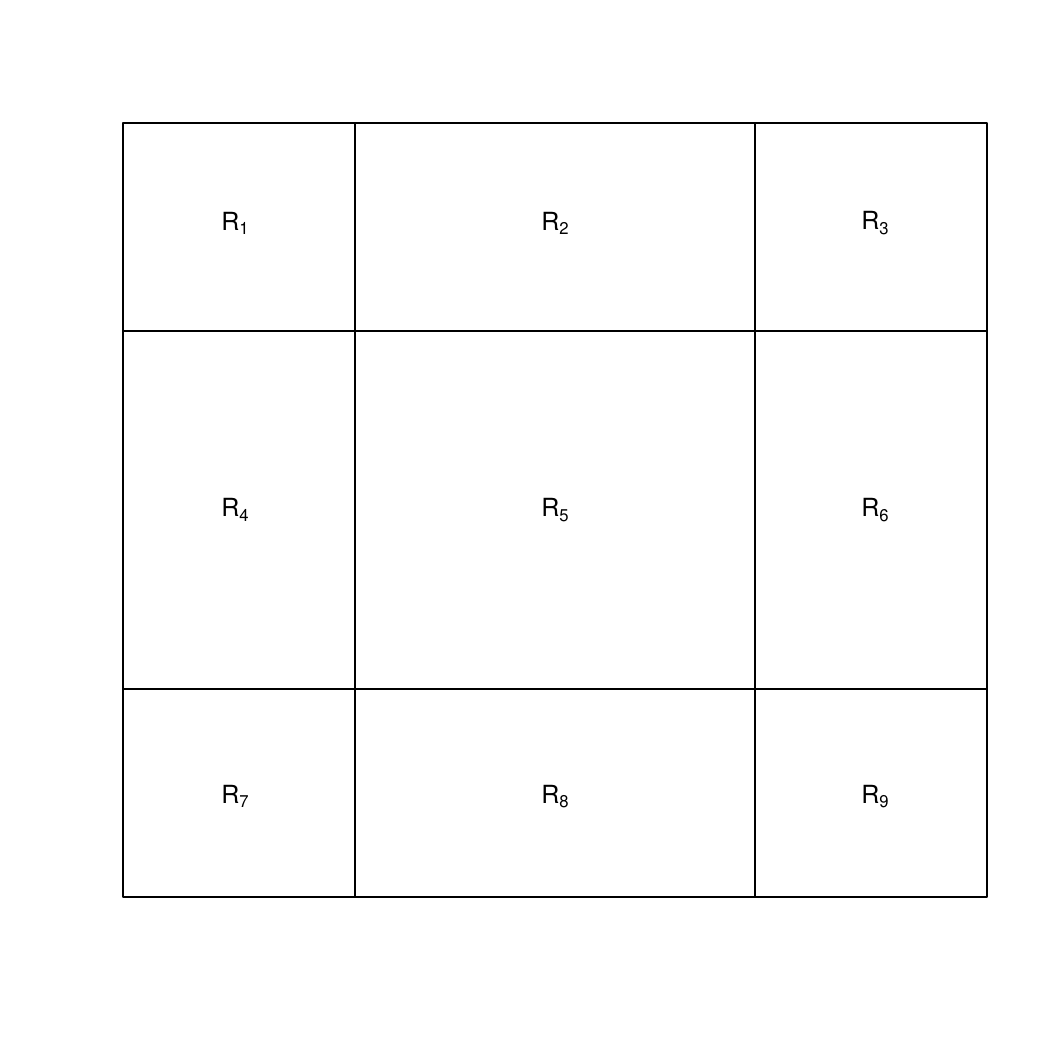}
\end{center}
\vspace{-1cm}
\caption{\em This plot shows the sets $R_1,\ldots, R_N$ in the case of a hyper-rectangle.
The set of non-differentiability $\Lambda_t$ consists of the vertical and horizontal lines.
The function $v_t(z)$ is $\beta$-smooth in the interior of each $R_j$.}
\label{fig::Rectangle_Plot}
\end{figure}

Now we bound Lipschitz constant.
Let $z_n$ be a sequence in $\Omega$ such that
$z_n \to \ell=(\ell_1,\ldots, \ell_d)$ as $n\to \infty$.
It is easy to check that
$v_t(z_n) \to 0$ as $z_n \to\ell$
as in the univariate case.
Next,
\begin{align*}
v_t(z_0) &=
\frac{\int_{-t(r_1-\ell_1)}^{r_1-\ell_1} \times \int_{-t(r_d-\ell_d)}^{r_d-\ell_d}
\delta r_t(z,\delta)d\delta}
{\int_{-t(r_1-\ell_1)}^{r_1-\ell_1} \times \int_{-t(r_d-\ell_d)}^{r_d-\ell_d}
r_t(z,\delta)d\delta}\\
& \to
\frac{\int_{0}^{r_1-\ell_1} \times \int_{0}^{r_d-\ell_d} \delta p_1(z+\delta)d\delta}
{\int_{0}^{r_1-\ell_1} \times \int_{0}^{r_d-\ell_d} p_1(z+\delta)d\delta}
\equiv c.
\end{align*}
Also,
$||z_0 - \ell||=t||r-\ell||$.
Hence,
$$
C_t \geq \frac{v(z_0)-v(\ell)}{||z_0-\ell||} \asymp \frac{c}{t||r-\ell||} 
$$
as $t\to 0$.

\end{document}